\documentclass[11pt]{amsart}
\usepackage{amsmath, amssymb, amscd, mathrsfs, url, pinlabel,verbatim}
\usepackage[pagebackref]{hyperref}
\usepackage[margin=1.25in,marginparwidth=1in,centering,letterpaper,dvips]{geometry}
\usepackage{color,dcpic,latexsym,graphicx,epstopdf,comment}
\usepackage[all]{xy}
\usepackage[dvipsnames]{xcolor}
\usepackage{tikz,pgfplots,tikz-cd}
\usepackage{enumitem,upquote,tabularx,textcomp}
\usepackage[all]{hypcap}
\usepackage[color=blue!20!white,textsize=tiny]{todonotes}

\title{Floer homology and non-fibered knot detection}
\author[John A. Baldwin]{John A. Baldwin}
\address{Department of Mathematics \\ Boston College}
\email{john.baldwin@bc.edu}

\author[Steven Sivek]{Steven Sivek}
\address{Department of Mathematics\\Imperial College London}
\email{s.sivek@imperial.ac.uk}

\thanks{JAB was supported by NSF FRG Grant DMS-1952707.}

\makeatletter
\newtheorem*{rep@theorem}{\rep@title}
\newcommand{\newreptheorem}[2]{%
\newenvironment{rep#1}[1]{%
 \def\rep@title{#2 \ref{##1}}%
 \begin{rep@theorem}}%
 {\end{rep@theorem}}}
\makeatother

\newtheorem {theorem}{Theorem}
\newreptheorem{theorem}{Theorem}
\newtheorem {lemma}[theorem]{Lemma}
\newtheorem {proposition}[theorem]{Proposition}

\numberwithin{equation}{section}
\numberwithin{theorem}{section}

\theoremstyle{definition}
\newtheorem{definition}[theorem]{Definition}

\newtheorem{remark}[theorem]{Remark}
\newtheorem*{remark*}{Remark}

\newlist{pcases}{enumerate}{1}
\setlist[pcases]{
  label=\bf{Case~\arabic*:}\protect\thiscase.~,
  ref=\arabic*,
  align=left,
  labelsep=0pt,
  leftmargin=0pt,
  labelwidth=0pt,
  parsep=0pt
}
\newcommand{\case}[1][]{%
  \if\relax\detokenize{#1}\relax
    \def\thiscase{}%
  \else
    \def\thiscase{~#1}%
  \fi
  \item
}

\newcommand{\Z}{\mathbb{Z}}

\newcommand{\R}{\mathbb{R}}

\newcommand{\F}{\mathbb{F}}
\newcommand{\Q}{\mathbb{Q}}
\newcommand{\spc}{\operatorname{Spin}^c}
\newcommand{\spinc}{\mathfrak{s}}
\newcommand{\spint}{\mathfrak{t}}

\newcommand\cf{\mathit{CF}}
\newcommand\cfhat{\widehat{\cf}}

\newcommand\hfk{\mathit{HFK}}
\newcommand\hfkhat{\widehat{\hfk}}

\newcommand\cfkinfty{\mathit{CFK}^\infty}

\newcommand\SFH{\mathit{SFH}}

\newcommand\slfrak{\mathfrak{sl}}

\DeclareFontFamily{U}{mathx}{\hyphenchar\font45}
\DeclareFontShape{U}{mathx}{m}{n}{
      <5> <6> <7> <8> <9> <10>
      <10.95> <12> <14.4> <17.28> <20.74> <24.88>
      mathx10
      }{}
\DeclareSymbolFont{mathx}{U}{mathx}{m}{n}
\DeclareFontSubstitution{U}{mathx}{m}{n}
\DeclareMathAccent{\widecheck}{0}{mathx}{"71}

\newcommand{\inr}{\operatorname{int}}

\newcommand{\hfhat}{\widehat{\mathit{HF}}}
\newcommand{\hfp}{\mathit{HF}^+}

\newcommand{\gr}{\operatorname{gr}}

\newcommand{\tr}{\operatorname{tr}}

\newcommand{\coker}{\operatorname{coker}}
\newcommand{\pt}{\mathrm{pt}}

\newcommand{\mirror}[1]{\overline{#1}}

\newcommand{\Kh}{\mathit{Kh}}

\newcommand{\Khr}{\overline{\Kh}}

\newcommand{\dcover}{\Sigma_2}

\newcommand{\Wh}{\operatorname{Wh}}

\makeatletter
\DeclareFontFamily{OMX}{MnSymbolE}{}
\DeclareSymbolFont{MnLargeSymbols}{OMX}{MnSymbolE}{m}{n}
\SetSymbolFont{MnLargeSymbols}{bold}{OMX}{MnSymbolE}{b}{n}
\DeclareFontShape{OMX}{MnSymbolE}{m}{n}{
    <-6>  MnSymbolE5
   <6-7>  MnSymbolE6
   <7-8>  MnSymbolE7
   <8-9>  MnSymbolE8
   <9-10> MnSymbolE9
  <10-12> MnSymbolE10
  <12->   MnSymbolE12
}{}
\DeclareFontShape{OMX}{MnSymbolE}{b}{n}{
    <-6>  MnSymbolE-Bold5
   <6-7>  MnSymbolE-Bold6
   <7-8>  MnSymbolE-Bold7
   <8-9>  MnSymbolE-Bold8
   <9-10> MnSymbolE-Bold9
  <10-12> MnSymbolE-Bold10
  <12->   MnSymbolE-Bold12
}{}

\let\llangle\@undefined
\let\rrangle\@undefined
\DeclareMathDelimiter{\llangle}{\mathopen}%
                     {MnLargeSymbols}{'164}{MnLargeSymbols}{'164}
\DeclareMathDelimiter{\rrangle}{\mathclose}%
                     {MnLargeSymbols}{'171}{MnLargeSymbols}{'171}
\makeatother

\usetikzlibrary{calc,intersections,patterns,decorations.markings}
\tikzset{every picture/.style=thick}
\tikzset{link/.style = { white, double = black, line width = 1.75pt, double distance = 1.25pt, looseness=1.75 }}
\tikzset{linkred/.style = { white, double = red, line width = 1.75pt, double distance = 1.25pt, looseness=1.75 }}
\tikzset{thinlink/.style = { white, double = black, line width = 1.25pt, double distance = 0.75pt, looseness=1.75 }}
\tikzset{link2/.style = { white, double = blue, line width = 1.75pt, double distance = 1.25pt, looseness=1.75 }}
\tikzset{crossing/.style = {draw, circle, dotted, minimum size=0.5cm, inner sep=0, outer sep=0}}
\pgfplotsset{compat=1.12}

\begin{document}

\begin{abstract}
We prove for the first time that knot Floer homology and Khovanov homology can detect non-fibered knots, and that HOMFLY homology detects infinitely many  knots; these theories were previously known to detect a mere six knots, all fibered.  These results rely on our main technical theorem, which gives a complete classification of genus-1 knots in the 3-sphere whose knot Floer homology in the top Alexander grading is 2-dimensional. We  discuss applications of this classification  to problems in Dehn surgery which are carried out in two sequels. These include a proof that $0$-surgery characterizes infinitely many knots,  generalizing results of Gabai from his 1987 resolution of the Property R Conjecture.\end{abstract}

\maketitle
\section{Introduction} \label{sec:intro}

A fundamental question for any knot invariant asks which knots it detects, if any. The most famous open version of this question asks whether the Jones polynomial detects the unknot. In this paper, we study the closely related detection question for knot Floer homology and Khovanov homology, as well as for Khovanov--Rozansky's HOMFLY homology.

Considerable attention has been paid to this question over the last twenty years, and yet  we have only managed to prove that  these homology theories detect  six knots: the unknot \cite{osz-genus, km-unknot}, the  two trefoils and the figure eight \cite{ghiggini,bs-trefoil,bdlls}, and the two  cinquefoils \cite{frw-cinquefoil,bhs-cinquefoil}. Each of these  detection results required substantial new ideas, which have in several cases reverberated far beyond knot detection,  but one thing they have in common is that each (save for that of the unknot)   relied crucially on the knot in question being fibered. This paper expands the  knot detection landscape dramatically. In particular, we prove for the first time that knot Floer homology and Khovanov homology can  detect non-fibered knots, and that HOMFLY homology detects infinitely many  knots.

Our detection results are summarized in the list below.  See Figure~\ref{fig:main-knots} for diagrams of the knots in this list, which are each non-fibered of  Seifert genus one. In particular, $\Wh^\pm(T_{2,3},2)$ is the 2-twisted Whitehead double of the right-handed trefoil with a positive or a negative clasp, respectively, and the $P(-3,3,2n+1)$ are pretzel knots. We prove that:
\begin{itemize}
\item Knot Floer homology detects $5_2$ and $\Wh^+(T_{2,3},2)$.
\item Knot Floer homology detects membership in each of the  sets \[\{ 15n_{43522},  \Wh^-(T_{2,3},2) \} \,\textrm{ and }\, \{ P(-3,3,2n+1) \mid n\in\Z \}.\]
\item Khovanov homology detects $5_2$.
\item Khovanov homology together with the degree of the Alexander polynomial detects $P(-3,3,2n+1)$ for each $n\in \Z$.
\item HOMFLY homology detects $P(-3,3,2n+1)$ for each $n\in \Z$.
\end{itemize}
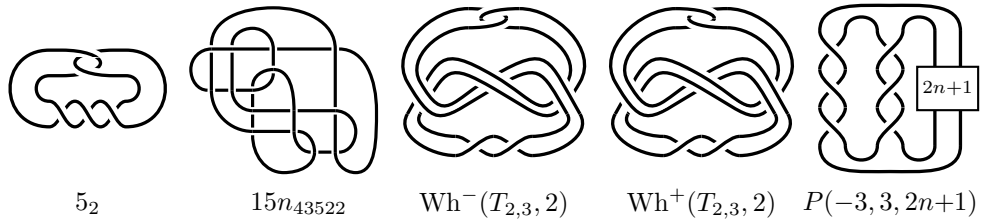
\begin{figure}
\begin{tikzpicture}[scale=0.85,font=\small]
\begin{scope}[xshift=-1.55cm,yshift=1.7cm]
\draw[link] (0.2,-0.4) to[out=0,in=180,looseness=0.75] ++(0.4,-0.4) to[out=0,in=0] ++(0,1.2) -- ++(-0.2,0) arc (90:180:0.6 and 0.2) ++(0.4,0) arc (360:270:0.6 and 0.2) -- ++(-0.2,0) to[out=180,in=180] ++(0,-0.4) to[out=0,in=180,looseness=0.75] ++(0.4,-0.4);
\draw[link] (-0.2,-0.4) to[out=180,in=0,looseness=0.75] ++(-0.4,-0.4) to[out=180,in=180] ++(0,1.2) -- ++(0.2,0) arc (90:0:0.6 and 0.2) ++(-0.4,0) arc (180:270:0.6 and 0.2) -- ++(0.2,0) to[out=0,in=0] ++(0,-0.4) to[out=180,in=0,looseness=0.75] ++(-0.4,-0.4);
\begin{scope}
\clip ([shift={(-\pgflinewidth,-0.5*\pgflinewidth)}]-0.5,-0.9) rectangle ([shift={(\pgflinewidth,0.5*\pgflinewidth)}]0.5,-0.3);
\draw[link,looseness=0.75] (-0.6,-0.4) to[out=0,in=180] ++(0.4,-0.4) to[out=0,in=180] ++(0.4,0.4) to[out=0,in=180] ++(0.4,-0.4);
\draw[link,looseness=0.75] (-0.6,-0.8) to[out=0,in=180] ++(0.4,0.4) to[out=0,in=180] ++(0.4,-0.4) to[out=0,in=180] ++(0.4,0.4);
\foreach \i in {-1,1} {
  \begin{scope}
  \clip ([shift={(-\pgflinewidth,-\pgflinewidth)}]0.4*\i,-0.4) ++(-0.1,0) rectangle ++([shift={(\pgflinewidth,1.5*\pgflinewidth)}]0.2,-0.4);
  \draw[link,looseness=0.75] (0.4*\i,-0.4) ++ (-0.2,0) to[out=0,in=180] ++(0.4,-0.4);
  \end{scope}
}
\end{scope}
\begin{scope}
\clip ([shift={(-0.5*\pgflinewidth,-0.5*\pgflinewidth)}]-0.5,0.15) rectangle ([shift={(0.5*\pgflinewidth,1*\pgflinewidth)}]0.5,0.25);
\draw[link] (-0.4,0.4) arc (90:-90:0.6 and 0.2);
\draw[link] (0.4,0.4) arc (90:270:0.6 and 0.2);
\end{scope}
\node at (0,-2) {$5_2$};
\end{scope}

\begin{scope}
\begin{scope}[scale=0.3,very thick] % code produced by PLink Viewer
    \draw (5.78, 2.76) .. controls (6.31, 2.76) and (6.66, 2.25) .. (6.66, 1.68);
    \draw (6.66, 1.68) .. controls (6.66, 0.97) and (5.86, 0.61) .. 
          (5.05, 0.61) .. controls (4.09, 0.61) and (3.43, 1.54) .. (3.43, 2.56);
    \draw (3.43, 2.96) .. controls (3.43, 3.19) and (3.43, 3.41) .. (3.43, 3.64);
    \draw (3.43, 4.03) .. controls (3.43, 4.26) and (3.43, 4.49) .. (3.43, 4.71);
    \draw (3.43, 5.11) .. controls (3.43, 5.64) and (3.95, 5.99) .. (4.51, 5.99);
    \draw (4.51, 5.99) .. controls (5.39, 5.99) and (5.59, 4.87) .. (5.59, 3.84);
    \draw (5.59, 3.84) .. controls (5.59, 3.48) and (5.59, 3.12) .. (5.59, 2.76);
    \draw (5.59, 2.76) .. controls (5.59, 2.20) and (5.94, 1.68) .. (6.46, 1.68);
    \draw (6.86, 1.68) .. controls (7.09, 1.68) and (7.31, 1.68) .. (7.54, 1.68);
    \draw (7.94, 1.68) .. controls (8.46, 1.68) and (8.81, 2.20) .. 
          (8.81, 2.76) .. controls (8.81, 3.35) and (8.33, 3.84) .. (7.74, 3.84);
    \draw (7.74, 3.84) .. controls (7.09, 3.84) and (6.43, 3.84) .. (5.78, 3.84);
    \draw (5.39, 3.84) .. controls (4.74, 3.84) and (4.08, 3.84) .. (3.43, 3.84);
    \draw (3.43, 3.84) .. controls (2.84, 3.84) and (2.36, 4.32) .. (2.36, 4.91);
    \draw (2.36, 4.91) .. controls (2.36, 5.56) and (2.36, 6.21) .. (2.36, 6.87);
    \draw (2.36, 7.26) .. controls (2.36, 7.79) and (2.87, 8.14) .. 
          (3.43, 8.14) .. controls (4.03, 8.14) and (4.51, 7.66) .. (4.51, 7.06);
    \draw (4.51, 7.06) .. controls (4.51, 6.77) and (4.51, 6.48) .. (4.51, 6.19);
    \draw (4.51, 5.79) .. controls (4.51, 5.26) and (3.99, 4.91) .. (3.43, 4.91);
    \draw (3.43, 4.91) .. controls (3.14, 4.91) and (2.85, 4.91) .. (2.56, 4.91);
    \draw (2.16, 4.91) .. controls (1.93, 4.91) and (1.71, 4.91) .. (1.48, 4.91);
    \draw (1.08, 4.91) .. controls (0.56, 4.91) and (0.21, 5.43) .. 
          (0.21, 5.99) .. controls (0.21, 6.58) and (0.69, 7.06) .. (1.28, 7.06);
    \draw (1.28, 7.06) .. controls (1.64, 7.06) and (2.00, 7.06) .. (2.36, 7.06);
    \draw (2.36, 7.06) .. controls (3.01, 7.06) and (3.66, 7.06) .. (4.31, 7.06);
    \draw (4.71, 7.06) .. controls (5.65, 7.06) and (6.60, 7.06) .. (7.54, 7.06);
    \draw (7.94, 7.06) .. controls (9.32, 7.06) and (9.89, 5.43) .. 
          (9.89, 3.84) .. controls (9.89, 2.37) and (9.89, 0.61) .. 
          (8.81, 0.61) .. controls (8.22, 0.61) and (7.74, 1.09) .. (7.74, 1.68);
    \draw (7.74, 1.68) .. controls (7.74, 2.34) and (7.74, 2.99) .. (7.74, 3.64);
    \draw (7.74, 4.03) .. controls (7.74, 5.04) and (7.74, 6.05) .. (7.74, 7.06);
    \draw (7.74, 7.06) .. controls (7.74, 8.50) and (6.13, 9.21) .. 
          (4.51, 9.21) .. controls (2.92, 9.21) and (1.28, 8.65) .. (1.28, 7.26);
    \draw (1.28, 6.87) .. controls (1.28, 6.21) and (1.28, 5.56) .. (1.28, 4.91);
    \draw (1.28, 4.91) .. controls (1.28, 3.72) and (2.25, 2.76) .. (3.43, 2.76);
    \draw (3.43, 2.76) .. controls (4.08, 2.76) and (4.74, 2.76) .. (5.39, 2.76);
\end{scope}
\node at (1.75,-0.3) {$15n_{43522}$};
\end{scope}

\begin{scope}[xshift=4.8cm,yshift=2.6cm] % nice (imho) drawing of Wh^-(T(2,3),2)
\draw[link] (-0.4,0.15) arc (90:0:0.6 and 0.15) ++ (-0.4,0) arc (180:270:0.6 and 0.15); % clasp
\draw[link] (0.4,0.15) arc (90:180:0.6 and 0.15) ++ (0.4,0) arc (360:270:0.6 and 0.15);
\draw[link] (-1.35,-1.5) to[out=60,in=150] (0,-0.9) to[out=330,in=240] (1.15,-1) to[out=60,in=0,looseness=1] (0.4,-0.15); % clasp top right to one of the middle strands
\draw[link] (-1.15,-1.5) to[out=60,in=150] (0,-1.1) to[out=330,in=240] (1.35,-1) to[out=60,in=0,looseness=1] (0.4,0.15);
\draw[link] (-0.4,0.15) to[out=180,in=120,looseness=1] (-1.35,-1) to[out=300,in=210] (0,-1.1) to[out=30,in=120] (1.15,-1.5); % clasp top left to the other middle strands
\draw[link] (-0.4,-0.15) to[out=180,in=120,looseness=1] (-1.15,-1) to[out=300,in=210] (0,-0.9) to[out=30,in=120] (1.35,-1.5);
\begin{scope} % fix middle crossing
\clip (-0.5,-1.5) rectangle (0.5,-0.5);
\draw[link] (-1.35,-1.5) to[out=60,in=150] (0,-0.9) to[out=330,in=240] (1.15,-1);
\draw[link] (-1.15,-1.5) to[out=60,in=150] (0,-1.1) to[out=330,in=240] (1.35,-1);
\end{scope}
% bottom strand, minus clasp
\draw[link,looseness=1] (-1.35,-1.5) to[out=240,in=180] (-0.6,-2.15) ++ (1.2,0) to[out=0,in=300] (1.35,-1.5);
\draw[link,looseness=1] (-1.15,-1.5) to[out=240,in=180] (-0.6,-1.85) ++ (1.2,0) to[out=0,in=300] (1.15,-1.5);
% clasp
\draw[link,looseness=0.75] (-0.6,-1.85) to[out=0,in=180] ++(0.6,-0.3) ++(0,0.3) to[out=0,in=180] ++(0.6,-0.3); % clasp
\draw[link,looseness=0.75] (-0.6,-2.15) to[out=0,in=180] ++(0.6,0.3) ++(0,-0.3) to[out=0,in=180] ++(0.6,0.3);
\node at (0,-2.9) {$\Wh^-(T_{2,3},2)$};
\end{scope}

\begin{scope}[xshift=8.05cm,yshift=2.6cm] % nice (imho) drawing of Wh^+(T(2,3),2)
\draw[link] (0.4,0.15) arc (90:180:0.6 and 0.15) ++ (0.4,0) arc (360:270:0.6 and 0.15); % clasp
\draw[link] (-0.4,0.15) arc (90:0:0.6 and 0.15) ++ (-0.4,0) arc (180:270:0.6 and 0.15);
\draw[link] (-1.35,-1.5) to[out=60,in=150] (0,-0.9) to[out=330,in=240] (1.15,-1) to[out=60,in=0,looseness=1] (0.4,-0.15); % clasp top right to one of the middle strands
\draw[link] (-1.15,-1.5) to[out=60,in=150] (0,-1.1) to[out=330,in=240] (1.35,-1) to[out=60,in=0,looseness=1] (0.4,0.15);
\draw[link] (-0.4,0.15) to[out=180,in=120,looseness=1] (-1.35,-1) to[out=300,in=210] (0,-1.1) to[out=30,in=120] (1.15,-1.5); % clasp top left to the other middle strands
\draw[link] (-0.4,-0.15) to[out=180,in=120,looseness=1] (-1.15,-1) to[out=300,in=210] (0,-0.9) to[out=30,in=120] (1.35,-1.5);
\begin{scope} % fix middle crossing
\clip (-0.5,-1.5) rectangle (0.5,-0.5);
\draw[link] (-1.35,-1.5) to[out=60,in=150] (0,-0.9) to[out=330,in=240] (1.15,-1);
\draw[link] (-1.15,-1.5) to[out=60,in=150] (0,-1.1) to[out=330,in=240] (1.35,-1);
\end{scope}
% bottom strand, minus clasp
\draw[link,looseness=1] (-1.35,-1.5) to[out=240,in=180] (-0.6,-2.15) ++ (1.2,0) to[out=0,in=300] (1.35,-1.5);
\draw[link,looseness=1] (-1.15,-1.5) to[out=240,in=180] (-0.6,-1.85) ++ (1.2,0) to[out=0,in=300] (1.15,-1.5);
% clasp
\draw[link,looseness=0.75] (-0.6,-1.85) to[out=0,in=180] ++(0.6,-0.3) ++(0,0.3) to[out=0,in=180] ++(0.6,-0.3); % clasp
\draw[link,looseness=0.75] (-0.6,-2.15) to[out=0,in=180] ++(0.6,0.3) ++(0,-0.3) to[out=0,in=180] ++(0.6,0.3);
\node at (0,-2.9) {$\Wh^+(T_{2,3},2)$};
\end{scope}

\begin{scope}[xshift=11cm,yshift=2.4cm]
\draw[link,looseness=1] (-1.1,0) to[out=90,in=180] ++(1.1,0.4) to[out=0,in=90] ++(1.1,-0.4);
\draw[link,looseness=1.5] (-0.7,0) to[out=90,in=90] ++(0.5,0) ++(0.4,0) to[out=90,in=90] ++(0.5,0);
\foreach \i in {0,-0.6,-1.2} {
  \draw[link,looseness=0.75] (-0.7,\i) to[out=270,in=90] ++(-0.4,-0.6) (-0.2,\i) to[out=270,in=90] ++(0.4,-0.6);
  \draw[link,looseness=0.75] (-1.1,\i) to[out=270,in=90] ++(0.4,-0.6) (0.2,\i) to[out=270,in=90] ++(-0.4,-0.6);
}
\draw[link] (0.7,0) -- ++(0,-1.8) ++(0.4,0) -- ++(0,1.8);
\node[draw,thick,rectangle,fill=white,inner sep=2pt,minimum height=1.5em] at (0.9,-0.9) {\tiny$2n{+}1$};
\draw[link,looseness=1.5] (-0.7,-1.8) to[out=270,in=270] ++(0.5,0) ++(0.4,0) to[out=270,in=270] ++(0.5,0);
\draw[link,looseness=1] (-1.1,-1.8) to[out=270,in=180] ++(1.1,-0.4) to[out=0,in=270] ++(1.1,0.4);
\node at (0,-2.7) {$P(-3,3,2n{+}1)$};
\end{scope}

\end{tikzpicture}
\caption{All of the genus-1 nearly fibered knots in $S^3$, up to taking mirrors; the labeled box on the right indicates the number of signed \emph{half}-twists. }
\label{fig:main-knots}
\end{figure}

These new detection results rely on our surprising main result, Theorem~\ref{thm:main-hfk}, which gives a complete classification of what we call \emph{nearly fibered} genus-1 knots in $S^3$. 
We  motivate and explain Theorem~\ref{thm:main-hfk} below, and then state precise versions of the detection results above. We next outline the proof of Theorem~\ref{thm:main-hfk}, which combines in novel ways arguments involving sutured manifolds \cite{gabai-foliations1}, involutions, the cyclic surgery theorem \cite{cgls}, and foundational work of Birman and Menasco on braids \cite{birman-menasco-iii,birman-menasco-iv}. Finally, we discuss applications of this theorem to problems in Dehn surgery, which are carried out in our papers \cite{bs-characterizing, bs-zero}. Perhaps the most striking of these  is our proof in \cite{bs-zero} that $0$-surgery characterizes infinitely many knots, where this was previously only known for the unknot, trefoils, and figure eight by Gabai's  celebrated 1987 work on the Property R Conjecture \cite{gabai-foliations3}.

\subsection{Our results} \label{ssec:results} Recall that knot Floer homology assigns to a knot $K\subset S^3$ a bigraded vector space over $\Q$,
\[\hfkhat(K;\Q) = \bigoplus_{m,a} \hfkhat_m(K,a;\Q),\] where $m$ and $a$ are the Maslov and Alexander gradings, respectively.  Letting \[\hfkhat(K,a;\Q) = \bigoplus_m \hfkhat_m(K,a;\Q),\]
knot Floer homology detects the Seifert genus of $K$ by the formula
\begin{equation}\label{eqn:genus-detection} g(K) = \max \{a \mid \hfkhat(K,a;\Q) \neq 0 \} \end{equation}
\cite{osz-genus}. 
Moreover, $K$ is fibered if and only if \[\dim\hfkhat(K,g(K);\Q)=1\] \cite{ghiggini,ni-hfk}.  The knot Floer homology detection results for the unknot, trefoils, and figure eight  follow readily from these properties, as the first is the only knot of genus zero and the others are the only fibered knots of genus one. Detection for the cinquefoils is substantially more involved \cite{frw-cinquefoil}, but also hinges on the fact that the cinquefoils are fibered.

We focus in this paper on what we call \emph{nearly fibered} knots. These are non-fibered knots which are  as close as possible, from the knot Floer homology perspective, to being fibered:

\begin{definition}\label{def:nearly-fibered-main}
A knot $K \subset S^3$ is \emph{nearly fibered} if $\dim\hfkhat(K,g(K); \Q) =2.$
\end{definition}

Our main result is the complete classification of genus-1 nearly fibered knots:

\begin{theorem} \label{thm:main-hfk}
If $K \subset S^3$ is a genus-1 nearly fibered knot, then $K$ is one of the knots
\[5_2,  \,\, 15n_{43522},\,\, \Wh^-(T_{2,3},2),\,\,
\Wh^+(T_{2,3},2), \,\,P(-3,3,2n+1) \ (n\in\Z)\] shown in Figure \ref{fig:main-knots}, or the  mirror of one of these knots.
\end{theorem}

The knot Floer homologies of these knots are displayed for reference in Table~\ref{fig:hfk-table}, with the computations explained in Appendix~\ref{sec:hfk-15n43522}.  Together with Theorem~\ref{thm:main-hfk} and the symmetry \[\hfkhat_m(K,a;\Q) \cong \hfkhat_{-m}(\mirror{K},-a;\Q)\]  under taking mirrors, these computations immediately imply the promised detection results for knot Floer homology, stated as Theorems~\ref{thm:main-hfk-detection-1}, \ref{thm:main-hfk-detection-2}, and \ref{thm:main-hfk-detection-3} below. The first of these makes precise our claim that knot Floer homology detects the knots $5_2$ and $\Wh^+(T_{2,3},2)$:

\begin{theorem} \label{thm:main-hfk-detection-1}
Let $K\subset S^3$ be a knot, and let $J\in\{5_2,\Wh^+(T_{2,3},2)\}$. If \[\hfkhat(K;\Q) \cong \hfkhat(J;\Q)\] as bigraded vector spaces, then $K=J$.
\end{theorem}

The next two theorems make precise our claim that knot Floer homology detects membership in each of the sets \[\{ 15n_{43522},  \Wh^-(T_{2,3},2) \}\,\textrm{ and }\,\{P(-3,3,2n+1)\mid n\in\Z\}.\] Note from  Table~\ref{fig:hfk-table} that knot Floer homology cannot distinguish the  knots in either set.

\begin{theorem} \label{thm:main-hfk-detection-2}
Let $K\subset S^3$ be a knot, and let $J\in\{ 15n_{43522},  \Wh^-(T_{2,3},2) \}$. If \[\hfkhat(K;\Q) \cong \hfkhat(J;\Q)\] as bigraded vector spaces, then $K\in\{ 15n_{43522},  \Wh^-(T_{2,3},2) \}$.
\end{theorem}

\begin{theorem} \label{thm:main-hfk-detection-3}
Let $K\subset S^3$ be a knot, and let $J\in \{P(-3,3,2n+1)\mid n\in\Z\}$. If \[\hfkhat(K;\Q) \cong \hfkhat(J;\Q)\] as bigraded vector spaces, then $K\in \{P(-3,3,2n+1)\mid n\in\Z\}$.\end{theorem}

As alluded to above, Theorem~\ref{thm:main-hfk-detection-1} is the first result which shows that knot Floer homology can detect non-fibered knots. We note that it is also the first knot Floer detection result for knots whose   Floer homology is not thin (i.e., not supported in a single $\delta = m-a$ grading).

\begin{table}
\[ \arraycolsep=1em
\begin{array}{cccc}
K & \hfkhat(K,1;\Q) & \hfkhat(K,0;\Q) & \hfkhat(K,-1;\Q) \\[0.25em] \hline &&& \\[-0.75em]
5_2 & \Q^2_{(2)} & \Q^3_{(1)} & \Q^2_{(0)} \\[0.5em]
15n_{43522} & \Q^2_{(0)} & \Q^4_{(-1)} \oplus \Q^{\vphantom{0}}_{(0)} & \Q^2_{(-2)} \\[0.5em]
\Wh^-(T_{2,3},2) & \Q^2_{(0)} & \Q^4_{(-1)} \oplus \Q^{\vphantom{0}}_{(0)} & \Q^2_{(-2)} \\[0.5em] \hline&&&\\[-0.75em]
P(-3,3,2n+1) & \Q^2_{(1)} & \Q^5_{(0)} & \Q^2_{(-1)} \\[0.5em]
\Wh^+(T_{2,3},2) & \Q^2_{(-1)} & \Q^4_{(-2)} \oplus \Q^{\vphantom{0}}_{(0)} & \Q^2_{(-3)}
\end{array} \]
\caption{Knot Floer homologies of genus-1 nearly fibered knots, grouped by whether $\det(K)$ is 7 or 9.  The subscripts denote Maslov gradings.}
\label{fig:hfk-table}
\end{table}

We now turn to our detection results for Khovanov homology. Recall that reduced Khovanov homology also assigns to a knot $K\subset S^3$ a bigraded vector space over $\Q$, \[\Khr(K;\Q) = \bigoplus_{h,q} \Khr^{h,q}(K;\Q),\] where $h$ and $q$ are the homological and quantum gradings, respectively. We use Theorem~\ref{thm:main-hfk} together with Dowlin's spectral sequence from Khovanov homology to knot Floer homology \cite{dowlin} to prove that reduced Khovanov homology detects $5_2$:

\begin{theorem} \label{thm:main-kh}
Let $K \subset S^3$ be a knot, and suppose  that \[ \Khr(K;\Q) \cong \Khr(5_2;\Q) \]
as bigraded vector spaces. Then $K = 5_2$.\end{theorem}

As mentioned previously, Theorem \ref{thm:main-kh} is the first result showing that Khovanov homology can detect non-fibered knots. Using the same strategy, we can also nearly show for the first time that Khovanov homology detects infinitely many knots:

\begin{theorem} \label{thm:main-kh-pretzel}
Let $K \subset S^3$ be a knot, and suppose for some $n\in\Z$ that 
\[ \Khr(K;\Q) \cong \Khr(P(-3,3,2n+1);\Q) \]
as bigraded vector spaces.  If in addition the Alexander polynomial $\Delta_K(t)$ has degree 1, then $K=P(-3,3,2n+1)$.
\end{theorem}

We expect that  $\Khr(K;\Q)$ alone should detect each of these pretzel knots. Indeed, their  reduced Khovanov homologies are all 9-dimensional but (unlike their knot Floer homologies) are distinguished by their bigradings. The only  remaining obstacle  is  to show  that there are no fibered knots of genus at least two with the same reduced Khovanov homology as one of these pretzels. We are currently unable to show this, which is the reason for the additional Alexander polynomial hypothesis in Theorem \ref{thm:main-kh-pretzel}.  

On the other hand, we can achieve the desired detection result using the reduced version of Khovanov--Rozansky's HOMFLY homology \cite{khovanov-rozansky-2}. This theory  assigns to a knot $K\subset S^3$ a triply-graded vector space over $\Q$, 
\[\bar{H}(K;\Q) = \bigoplus_{i,j,k} \bar{H}^{i,j,k}(K;\Q),\] which determines the HOMFLY polynomial of $K$. We use the fact that the HOMFLY polynomial encodes the Alexander polynomial, together with  recent results of Wang \cite{wang-split}, to bypass the  obstacle described above and prove for the first time that  HOMFLY homology detects infinitely many knots:

\begin{theorem} \label{thm:main-homfly-detection}
Let $K \subset S^3$ be a knot, and suppose for some $n\in\Z$ that 
\[ \bar{H}(K;\Q) \cong \bar{H}(P(-3,3,2n+1);\Q) \]
as triply-graded vector spaces.  Then $K=P(-3,3,2n+1)$.
\end{theorem}

\begin{remark}
Some of the knots in Theorem \ref{thm:main-hfk} may be more familiar under other names. For instance, $6_1$ is  the pretzel knot $P(-3,3,1)$.  The knot $15n_{43522}$ is one of the simplest hyperbolic knots, as tabulated in the census \cite{ckm-500}, where it is labeled $k8_{218}$.  The twisted Whitehead doubles $\Wh^+(T_{2,3},2)$ and $\Wh^-(T_{2,3},2)$ appear in the  tabulation \cite{htw-16} as the knots $15n_{115646}$ and $16n_{696530}$, respectively.
\end{remark}

We  outline  our proof of Theorem \ref{thm:main-hfk} in some detail below. For the reader interested in fewer details, the key new idea is that if $K$ is a genus-1 nearly fibered knot, then the fact that \[\dim \hfkhat(K,1;\Q)=2\] is small allows us to determine the complement of a genus-1 Seifert surface $F$ for $K$.  This complement is not simply a product $F\times[-1,1]$ since $K$ is not fibered, but work of Juh\'asz \cite{juhasz-polytope} provides us with product annuli that we can use to cut the complement into simpler pieces and identify it anyway.  In each case, the complement of $F$ admits an involution which extends over the complement of $K$, and by taking quotients we can reduce the classification problem in Theorem \ref{thm:main-hfk} to a difficult but ultimately solvable question about 3-braids.

\subsection{Proof outline}\label{ssec:proof}

 Let $K\subset S^3$ be a genus-1 nearly fibered knot, so that \[\dim \hfkhat(K,1;\Q) =2.\] Let $F$ be   a genus-1 Seifert surface for $K$. Let us identify a closed tubular neighborhood of $F$ with the product $F\times[-1,1]$, and consider  the sutured Seifert surface complement \[S^3(F):=(M,\gamma)=(S^3\setminus \inr(F\times[-1,1]),\partial F\times[-1,1]).\] Then $S^3$ is recovered by gluing  this neighborhood back in, \begin{equation*}\label{eqn:gluing-map}S^3 = S^3(F) \cup (F\times[-1,1]),\end{equation*} and   $K$ is the image of the suture \[s(\gamma) = \partial F\times \{0\}\] in this glued manifold.  Our strategy is to first identify the  complement $S^3(F)$ abstractly,  in a way which does not remember its  embedding into $S^3$, and  then classify the gluings   that recover $S^3$  from this abstract point of view, so as to ultimately  determine the knot $K$. 
 
 It will be helpful to consider the following  slightly different perspective. Let \[M_F:= S^3(F) \cup (D^2\times [-1,1]),\] in which we identify each circle \[\partial F\times\{t\}\subset \partial S^3(F)\]with the corresponding $\partial D^2 \times \{t\}$. Then $M_F$  has two toroidal boundary components, and can be viewed as the 3-manifold obtained from $S^3_0(K)$ by removing a neighborhood of the capped off Seifert surface. This manifold contains a distinguished arc \[\alpha := \{0\}\times[-1,1]\subset D^2\times[-1,1]\subset M_F,\] whose complement recovers $S^3(F)$, where the suture $s(\gamma)$ is identified with a meridian of $\alpha$.

Work of Juh\'asz \cite{juhasz-decomposition} tells us that the sutured Floer homology of $S^3(F)$ has dimension
\[ \dim\SFH(S^3(F);\Q) =\dim \hfkhat(K,1;\Q) =2. \] This dimension is sufficiently small  that  another theorem of Juh\'asz \cite{juhasz-polytope} guarantees the existence of an essential product annulus $A$ inside of $S^3(F)$.  Because $F$ has genus $1$ we can guarantee that the components of $\partial A$ are homologically essential in their respective copies of $F$, or equivalently in the tori of $\partial M_F$, so by Dehn filling along curves dual to $\partial A$ we can identify $M_F$ as the complement of a 2-component cable link, in which $A$ is the cabling annulus. 

A similar argument shows that the manifold obtained by decomposing $S^3(F)$ along $A$ also contains an essential product annulus $B$, since such decompositions preserve the dimension of sutured Floer homology. We prove that  $B$ separates $S^3(F)\setminus A$ into two pieces, and argue based on the dimensions of the sutured Floer homologies  of these pieces that the component containing  $\gamma$ must be a product sutured manifold. We then use this to show that the arc $\alpha$ in $M_F$ can be isotoped into the cabling annulus $A$. 

It  follows that the manifold obtained by cutting $M_F$ open along the cabling annulus $A$ can alternatively be obtained  by first removing a neighborhood of $\alpha$ to form $S^3(F)$, and then decomposing $S^3(F)$ along a product disk to remove the rest of the annulus $A$. Since $S^3(F)$ is a subset of  $S^3$, this cut-open manifold with torus boundary must then be the complement of a knot $C\subset S^3$, with sutures isotopic to $\partial A$. Moreover, its sutured Floer homology is also 2-dimensional, since product disk decomposition preserves dimension. Using this, we argue that $C$ is an unknot or a trefoil, and  conclude the following:

{
\renewcommand{\thetheorem}{\ref{thm:identify-y-c}}
\begin{theorem}
Up to orientation reversal, $M_F$ must be the complement of the $(2,4)$-cable of either the unknot or the right-handed trefoil, and $\alpha$ is an arc in the cabling annulus.
\end{theorem}
\addtocounter{theorem}{-1}
}

This then gives us two possibilities (up to orientation reversal) for $S^3(F)$, which we recall is obtained from $M_F$ by removing a neighborhood of $\alpha$. The next important observation is that in both cases, there is an involution \[\iota:S^3(F)\to S^3(F)\] which fixes $\gamma$ setwise and restricts to a hyperelliptic involution on the once-punctured tori $R_+(\gamma)$ and $R_-(\gamma)$, as shown in Figures \ref{fig:E-T24-quotient} and \ref{fig:E-C24-T23-quotient-1}. The quotient of $S^3(F)$ by this involution is a sutured 3-ball with connected suture. It is natural to identify this quotient 3-ball  with the complement of a thickened disk in $S^3$, \[S^3(F)/\iota \cong S^3(D^2) = (S^3\setminus \inr(D^2\times[-1,1]), \partial D^2 \times [-1,1]),\] and the quotient map realizes $S^3(F)$ as the branched double cover of this ball along a  tangle $\tau\subset S^3(D^2)$, as shown in  Figures \ref{fig:E-T24-quotient} and \ref{fig:E-C24-T23-quotient-2}.

As  discussed at the beginning, $S^3$ is recovered by gluing $F\times[-1,1]$ to $S^3(F)$ by a map which in particular identifies $\partial F\times[-1,1]$ with $\gamma$. For \emph{any} such gluing map $\varphi$, the facts that the once-punctured torus admits a unique hyperelliptic involution up to isotopy, and that this commutes with $\varphi$ up to isotopy -- note that these facts require our assumption that $g(F)=1$ -- imply that $\iota$ extends to an involution $\hat\iota$ of the glued manifold \[Y_\varphi = S^3(F)\cup_\varphi (F\times [-1,1]),\]  whose restriction to the piece $F\times [-1,1]$ is a hyperelliptic involution on each $F\times\{t\}$. The quotient map \[Y_\varphi\to Y_\varphi/\hat{\iota}\] therefore restricts on this piece to a branched double covering \[F\times[-1,1]\to D^2\times[-1,1]\] along some 3-braid \[\beta \subset D^2\times [-1,1].\] It follows that $Y_\varphi$ is the branched double cover of \[S^3(D^2)\cup (D^2\times[-1,1]) \cong S^3\] along the link $\tau\cup \beta$. Moreover, $K$ is the lift of the braid axis \[\kappa = \partial D^2 \times \{0\} = s(\gamma)/\hat{\iota}\] in this double cover, as shown in Figures \ref{fig:E-T24-quotient} and \ref{fig:E-T24-isotopy} in the case that $M_F$ is the complement of the $(2,4)$-cable of the unknot. In particular, $Y_\varphi \cong S^3$ if and only if $\tau\cup \beta$ is an unknot.

This leads to our strategy for identifying $K$:
\begin{enumerate}
\item Identify all 3-braids $\beta$ such that $\tau \cup \beta$ is an unknot.
\item For each such $\beta$, lift $\kappa$ to the branched double cover
\[ \dcover(S^3, \tau \cup \beta) \cong \dcover(S^3, U) \cong S^3, \]
and this lift $\tilde \kappa$ is the corresponding knot $K$.
\end{enumerate}
The first step is generally difficult, and takes up a lot of room in this paper. Our approach is to find a crossing of $\tau$ whose various resolutions are all relatively simple, and then understand surgeries between the branched double covers of these resolutions, making heavy use of the cyclic surgery theorem \cite{cgls} throughout.  We eventually conclude the following:

{
\renewcommand{\thetheorem}{\ref{thm:M_F-E24}}
\begin{theorem}
If $M_F$ is the complement of the $(2,4)$-cable of the unknot, then $K$ must be $5_2$, $15n_{43522}$, or a pretzel knot $P(-3,3,2n+1)$, up to mirroring.
\end{theorem}
\addtocounter{theorem}{-1}
}

{
\renewcommand{\thetheorem}{\ref{thm:M_F-C24-T23}}
\begin{theorem}
If $M_F$ is the complement of the $(2,4)$-cable of the right-handed trefoil, then $K$ must be a twisted Whitehead double $\Wh^{\pm}(T_{2,3},2)$.
\end{theorem}
\addtocounter{theorem}{-1}
}

Given that taking the mirror of $K$ corresponds to reversing the orientation of $M_F$, this completes the  proof of Theorem~\ref{thm:main-hfk}.

\begin{remark}
One of the main inspirations for this work and for our approach was a paper by Cantwell and Conlon \cite{cantwell-conlon-52}, who showed (among other things) that if $K$ is either $5_2$ or $P(-3,3,2n+1)$, then $M_F$ is the complement of the $(2,4)$-torus link.  
\end{remark}

\subsection{Other applications}

One of the strengths of knot Floer homology is  its relationship to the Heegaard Floer homology of Dehn surgeries on knots. Indeed, the fact that knot Floer homology detects the unknot can be used to give another proof  that the unknot is uniquely characterized by each of its nontrivial Dehn surgeries (this was first proved via different but similar means by Kronheimer--Mrowka--Ozsv{\'a}th--Szab{\'o}  in \cite{kmos}). Likewise, Ozsv{\'a}th--Szab{\'o} used the fact that  knot Floer homology detects the trefoils and figure eight to prove that these knots are also characterized by each of their nontrivial  surgeries \cite{osz-characterizing}.

In \cite{bs-characterizing}, we use Theorem~\ref{thm:main-hfk} to prove that Dehn surgeries of nearly all rational slopes uniquely characterize the knot $5_2$:

\begin{theorem}[\cite{bs-characterizing}] \label{thm:main-characterizing}
 Let $K \subset S^3$ be a knot, and suppose that $r$ is a rational number for which there is an orientation-preserving homeomorphism
\[ S^3_r(K) \cong S^3_r(5_2). \]
If $r$ is not a positive integer, then $K=5_2$.
\end{theorem}

This is the strongest result to date concerning characterizing slopes for any hyperbolic knot other than the figure eight. Note that we cannot hope to extend Theorem~\ref{thm:main-characterizing} to all positive integers, since, for example,  $S^3_1(5_2) \cong S^3_1(P(-3,3,8)),$ as shown in \cite{bs-characterizing}.

Using Theorem \ref{thm:main-characterizing},  we  can then determine all of the ways in which the Brieskorn sphere $\Sigma(2,3,11)$ can arise from Dehn surgery on a knot in $S^3$:

\begin{theorem}[\cite{bs-characterizing}]
Given a knot $K \subset S^3$ and a rational number  $r$, there exists an orientation-preserving homeomorphism \[S^3_r(K) \cong \Sigma(2,3,11)\] if and only if $(K,r)$ is either $(T_{-2,3},-\frac{1}{2})$ or $(5_2,-1)$.
\end{theorem}

We note that similar results were achieved for $\Sigma(2,3,5)$ by Ghiggini in \cite{ghiggini}, and for $\Sigma(2,3,7)$ by Ozsv\'ath--Szab\'o in \cite{osz-characterizing}.

Similarly, the only knots for which 0-surgery was previously known to be characterizing are the unknot, trefoils, and figure eight, by a 1987 theorem of Gabai \cite{gabai-foliations3}.  (This is an immediate corollary of Gabai's proof that $S^3_0(K)$ determines the Seifert genus of $K$ as well as whether or not $K$ is fibered.)  Combining the case $r=0$ of Theorem~\ref{thm:main-characterizing} with the main result of \cite{bs-zero} lets us add the infinitely many knots of Theorem~\ref{thm:main-hfk} to this list.

\begin{theorem}[\cite{bs-characterizing,bs-zero}] \label{thm:main-zero}
Let $K \subset S^3$ be a genus-1 nearly fibered knot.  If for some knot $J \subset S^3$ there is an orientation-preserving homeomorphism \[ S^3_0(K) \cong S^3_0(J), \] then $J=K$.
\end{theorem}

\subsection{Coefficients} Every Floer theory and link homology theory in this paper will be considered with coefficients in $\Q$ unless specified otherwise (as in Appendix \ref{sec:hfk-15n43522}). For this reason, we will typically omit the coefficients from our notation for these  theories going forward.

\subsection{Organization}

In \S\ref{sec:sfh-intro}, we review necessary background on sutured Floer homology.  In \S\ref{sec:nearly-fibered}--\S\ref{sec:C}, we classify  the possible pairs $(M_F,\alpha)$, eventually proving Theorem~\ref{thm:identify-y-c}.  In \S\ref{sec:unknot-24}, we determine the knots $K$ arising when $M_F$ is the complement of a cabled unknot, proving Theorem \ref{thm:M_F-E24}. In \S\ref{sec:trefoil-24}, we do the same when $M_F$ is the complement of a cabled trefoil, proving Theorem \ref{thm:M_F-C24-T23}. This proves Theorem~\ref{thm:main-hfk},  and the knot Floer homology detection results in Theorems \ref{thm:main-hfk-detection-1}, \ref{thm:main-hfk-detection-2}, and \ref{thm:main-hfk-detection-3} follow immediately. In \S\ref{sec:kh-to-hfk}, we use Dowlin's spectral sequence to prove  the Khovanov homology detection results in Theorems \ref{thm:main-kh} and \ref{thm:main-kh-pretzel}. We then apply Theorem~\ref{thm:main-kh-pretzel} in \S\ref{sec:homfly} to prove the HOMFLY homology detection result in Theorem~\ref{thm:main-homfly-detection}.  We  finish with Appendix~\ref{sec:hfk-15n43522}, detailing the computations which appear in Table~\ref{fig:hfk-table}.

\subsection{Acknowledgements}

We thank Anna Beliakova, Nathan Dunfield, Matt Hedden, Jen Hom, Siddhi Krishna, Zhenkun Li, Tye Lidman, Nikolai Saveliev,  Josh Wang, and Fan Ye for helpful and interesting conversations related to this work.  We are grateful to the referees for all of their feedback, and in particular for a substantial simplification to the proof of Lemma~\ref{lem:boundary-b-not-meridian}. We also note that Zhenkun and Fan independently discovered some of the results in \S\ref{sec:nearly-fibered}--\S\ref{sec:C}.

\section{Sutured Floer homology background} \label{sec:sfh-intro}

In this section, we briefly review some facts about sutured Floer homology which will be of use in this paper, and establish some notation. See \cite{gabai-foliations1, juhasz-sutured, juhasz-decomposition} for more background.

Following Gabai \cite{gabai-foliations1}, a \emph{sutured manifold} is a pair $(M,\gamma)$, where $M$ is a compact, oriented $3$-manifold and $\gamma \subset \partial M$ is a union of annuli $A(\gamma)$ and tori $T(\gamma)$, all of which are pairwise disjoint.  We identify an oriented simple closed curve inside each annulus that is isotopic to the core of that annulus, and take the \emph{sutures} $s(\gamma)$ to be their union.  We orient the components of $R(\gamma) = \partial M - \inr(\gamma)$ so that their boundary orientations agree with the orientations of $s(\gamma)$, and then let $R_+(\gamma)$ and $R_-(\gamma)$ consist of those components of $R(\gamma)$ whose orientations agree or disagree with the boundary orientation of $\partial M$, respectively.

Juh\'asz \cite[Definition~2.2]{juhasz-sutured} calls $(M,\gamma)$ a \emph{balanced} sutured manifold if $M$ has no closed components, the subsurfaces $R_+(\gamma)$ and $R_-(\gamma)$ have the same Euler characteristic, and every component of $\partial M$ contains an annulus of $A(\gamma)$.  In this case the set of tori $T(\gamma)$ must be empty.

Sutured Floer homology, as defined by Juh{\'a}sz in \cite{juhasz-sutured}, assigns to a balanced sutured manifold $(M,\gamma)$ a vector space over $\Q$,
\[ \SFH(M,\gamma) = \bigoplus_{\spinc \in \spc(M,\gamma)} \SFH(M,\gamma,\spinc), \]
generalizing the hat version of Heegaard Floer homology.  For example, given a knot $K \subset Y$ we consider the sutured knot complement
\[ Y(K) := (Y\setminus N(K), \gamma_\mu), \] whose sutures $s(\gamma_\mu)$ are the union of two oppositely oriented meridians of $K$. Moreover, given a Seifert surface $F$ for $K$, we identify a closed tubular neighborhood of $F$ with the product  $F\times[-1,1]$, and define the sutured Seifert surface complement by
\[ Y(F) := (M,\gamma) = (Y \setminus \inr(F\times[-1,1]), \partial F \times [-1,1]), \] with suture  \[s(\gamma) = \partial F\times\{0\}\] and \[R_\pm(\gamma) = F\times\{\pm 1\}.\] 
Then sutured Floer homology recovers the knot Floer homology of $K$, as well as its summand in the top Alexander grading with respect to $F$, by \begin{align}
\SFH(Y(K)) &\cong \hfkhat(Y,K) \label{eq:sfh-hfk}, \\
\SFH(Y(F)) &\cong \hfkhat(Y,K,[F],g(F)), \label{eq:sfh-hfk-top}
\end{align}
as shown in \cite[Proposition~9.2]{juhasz-sutured} and \cite[Theorem~1.5]{juhasz-decomposition}, respectively. 

Juh\'asz also proved \cite{juhasz-sutured,juhasz-decomposition} that sutured Floer homology detects whether a balanced sutured manifold is taut and  whether it is a product, as stated in Theorem \ref{thm:sfh-taut} below. Recall for this theorem that a  sutured manifold $(M,\gamma)$ is \emph{taut} if it is irreducible and if $R(\gamma)$ is incompressible and Thurston norm-minimizing in \[H_2(M,\gamma).\] It  is a \emph{product} sutured manifold if it is of the form
\[ (M,\gamma) \cong (\Sigma \times [-1,1], \partial \Sigma \times [-1,1]) \] with $s(\gamma) = \partial \Sigma\times\{0\}$, 
where $\Sigma$ is a compact, oriented surface with no closed components.

\begin{theorem} \label{thm:sfh-taut}
Let $(M,\gamma)$ be a balanced sutured manifold.
\begin{itemize}
\item If $(M,\gamma)$ is irreducible and not taut, then $\SFH(M,\gamma) \cong 0$.
\item If $(M,\gamma)$ is taut, then $\dim\SFH(M,\gamma)\geq 1$.
\item If $(M,\gamma)$ is taut and not a product, then $\dim\SFH(M,\gamma)\geq 2$.
\end{itemize}
\end{theorem}
\begin{proof}
These claims are  \cite[Proposition~9.18]{juhasz-sutured} (whose proof is attributed to Yi Ni), \cite[Theorem~1.4]{juhasz-decomposition}, and \cite[Theorem~9.7]{juhasz-decomposition}, respectively.
\end{proof}

\begin{remark}
\label{rmk:taut}
If $K\subset S^3$ is a knot and $F$ is a genus-minimizing Seifert surface for $K$, then the sutured Seifert surface complement $S^3(F)$  is taut.
\end{remark}

Sutured Floer homology behaves well with respect to sutured manifold decompositions \[(M,\gamma) \stackrel{S}{\leadsto} (M',\gamma')\] for certain  surfaces $S\subset (M,\gamma)$, as stated precisely in \cite[Theorem~1.3]{juhasz-decomposition}.  In this paper, we will be concerned with decompositions along:
\begin{itemize}
\item \emph{product disks}, which are properly embedded disks
\[ S \subset (M,\gamma) \]
such that $\partial S$ meets the sutures $s(\gamma)$ in two points; and
\item \emph{product annuli}, which are properly embedded annuli
\[ S \subset (M,\gamma) \]
such that $\partial S$ has one component in $R_+(\gamma)$ and the other component in $R_-(\gamma)$.
\end{itemize}
The two theorems below state that sutured Floer homology is preserved under product disk decomposition, and under product annulus decomposition with mild additional hypotheses.

\begin{theorem}[{\cite[Lemma~9.13]{juhasz-sutured}}] \label{thm:sfh-disk-decomposition}
Let $(M,\gamma)$ be a balanced sutured manifold. If $(M',\gamma')$ is obtained by decomposing $(M,\gamma)$ along a product disk, then \[\SFH(M,\gamma) \cong \SFH(M',\gamma').\]
\end{theorem}

\begin{theorem}[{\cite[Lemma~8.9]{juhasz-decomposition}}] \label{thm:sfh-annulus-decomposition}
Let $(M,\gamma)$ be a balanced sutured manifold such that $H_2(M) \cong 0.$  Let \[S \subset (M,\gamma)\] be a product annulus where at least one component of $\partial S$ is nonzero in $H_1(R(\gamma))$. If $(M',\gamma')$ is obtained by decomposing $(M,\gamma)$ along $S$, then
\[ \SFH(M',\gamma') \cong \SFH(M,\gamma). \]
\end{theorem}

\begin{remark}\label{rmk:taut-annulus}
These two theorems are closely related to the fact that decompositions along product disks and along product annuli preserve tautness \cite[Lemma~3.12]{gabai-foliations1}.
\end{remark}

We say that a product annulus $S\subset (M,\gamma)$ is \emph{essential} if it is incompressible and if it is not isotopic to any component of $\gamma$ by an isotopy which keeps $\partial S$ in $R(\gamma)$ at all times. As discussed in \S\ref{ssec:proof}, our proof of Theorem \ref{thm:main-hfk} relies on finding essential product annuli  in the sutured complement of a genus-1  Seifert surface for a nearly fibered knot. Our main source of such annuli will be the following result:

\begin{theorem} \label{thm:find-annulus}
Let $(M,\gamma)$ be a taut balanced sutured manifold with $H_2(M) \cong 0$, and suppose that $(M,\gamma)$ is not a product $(\Sigma\times[-1,1],\partial \Sigma\times[-1,1])$ in which $\Sigma$ is either an annulus or a pair of pants.  If
\[ \dim \SFH(M,\gamma) < 4 \]
and 
\[ \dim \SFH(M,\gamma) \leq \tfrac{1}{2} b_1(\partial M), \]
then $(M,\gamma)$ contains an essential product annulus $S$.
\end{theorem}

\begin{proof}
Since $(M,\gamma)$ is taut and $\dim\SFH(M,\gamma)<4$,  \cite[Corollary~2.2]{juhasz-seifert} says that $(M,\gamma)$ is \emph{horizontally prime} (see \cite[Definition~1.7]{juhasz-seifert}).  If $(M,\gamma)$ is also \emph{reduced}, meaning that it does not contain an essential product annulus, and if it is not one of the forbidden products, then \cite[Theorem~3]{juhasz-polytope} says that
\[ \dim \SFH(M,\gamma) \geq \tfrac{1}{2} b_1(\partial M) + 1. \]
By hypothesis, this is not the case, so since $(M,\gamma)$ is not such a product, it is not reduced.  (The products were not excluded in the statement of \cite[Theorem~3]{juhasz-polytope}, but the proof assumes that there are no essential product disks in $(M,\gamma)$, which by \cite[Lemma~2.13]{juhasz-polytope} holds if  and only if $(M,\gamma)$ is not  one of these products.  See \cite[Remark~5.10]{ghosh-li}.)
\end{proof}

Lastly, we record the following for eventual use in our proof of Theorem \ref{thm:identify-y-c}.

\begin{proposition} \label{prop:longitude-fh}
Let $K \subset S^3$ be a nontrivial knot, and let \[(S^3 \setminus N(K), \gamma_0)\] denote the balanced sutured manifold whose sutures $s(\gamma_0)$ are a union of two oppositely oriented Seifert longitudes.  Then
\[ \dim \SFH(S^3 \setminus N(K), \gamma_0) \geq 4. \]
\end{proposition}

\begin{proof} For any balanced sutured manifold $(M,\gamma)$, a choice of  homology orientation for the pair $(M,R_-(\gamma))$ gives rise to an absolute lift of the relative $\Z/2\Z$-grading on $\SFH(M,\gamma)$, and therefore to a well-defined Euler characteristic \[\chi(\SFH(M,\gamma,\spinc))\in \Z\] for each $\spinc\in\spc(M,\gamma)$, as  described in \cite{fjr-decategorification}. Fixing an $H_1(M)$-affine isomorphism \[\iota: \spc(M,\gamma)\to H_1(M),\] these Euler characteristics can be packaged as an element \[\tau(M,\gamma) = \sum_{\spinc\in\spc(M,\gamma)} \chi(\SFH(M,\gamma,\spinc)) \cdot \iota (\spinc)\] of the group ring $\Z[H_1(M)]$.

Let us write $E_K = S^3 \setminus N(K)$ for convenience. Then \[ \tau(E_K,\gamma_0) = 0 \]
as shown in \cite[Example~8.1]{fjr-decategorification}, which means that
\[ \chi(\SFH(E_K,\gamma_0, \spinc)) = 0 \]
for each $\spinc \in \spc(E_K,\gamma_0)$.  In particular, $\dim \SFH(E_K,\gamma_0,\spinc)$ is always even.

Since $K$ is nontrivial, its complement $E_K$ is irreducible. Thus, if we let
\[ S = \{ \spinc \in \spc(E_K,\gamma_0) \mid \SFH(E_K,\gamma_0,\spinc) \not\cong 0 \}, \]
then \cite[Theorem~1.4]{fjr-decategorification} tells us that for all $\alpha \in H_2(E_K,\partial E_K;\R)$, we have
\[ \max_{\spinc,\spint \in S} \langle \spinc-\spint, \alpha \rangle = x^s(\alpha) \]
where $x^s$ is the sutured Thurston norm on $(E_K,\gamma_0)$.  If $\alpha$ is the class of a  Seifert surface for $K$, with genus $g= g(K) \geq 1$, then we compute by \cite[Lemma~7.3]{fjr-decategorification} that
\[ x^s(\alpha) = x(\alpha) = 2g-1, \]
and since this is nonzero there must be two different $\spc$ structures $\spinc$ on $(E_K,\gamma_0)$, each pairing differently with $\alpha$, for which $\SFH(E_K,\gamma_0,\spinc)$ is nonzero.  But then $\SFH(E_K,\gamma_0)$ has dimension at least two in each of these two $\spc$ structures, so we conclude that
\[ \dim \SFH(E_K,\gamma_0) \geq 4 \]
as desired.
\end{proof}

\section{Nearly fibered knots and essential annuli} \label{sec:nearly-fibered}

Let $K\subset S^3$ be a nearly fibered knot of genus $g$, as in Definition \ref{def:nearly-fibered-main}. Then \[ \dim\hfkhat(K,g)=2. \] Since this dimension is less than 4,   \cite[Theorem~2.3]{juhasz-seifert} says that $K$ has a unique genus-$g$ Seifert surface $F$, up to isotopy. In this section, we  will use Theorem \ref{thm:find-annulus} to study essential product annuli in the sutured Seifert surface complement \[S^3(F) = (S^3\setminus \inr(F\times[-1,1]), \partial F \times [-1,1]).\] The  lemma below guarantees the existence of such annuli with  nice boundary properties.

\begin{lemma} \label{lem:product-annulus}
Let $K\subset S^3$ be a nearly fibered knot, and let $F$ be a Seifert surface for $K$ of genus $g = g(K)$.   Then there is an essential product annulus $A$ in the sutured manifold \[(M,\gamma) = S^3(F)\] whose boundary components \[A_\pm=\partial A \cap R_\pm(\gamma)\] 
are not both  boundary-parallel in their respective surfaces $R_\pm(\gamma)$.
\end{lemma}

\begin{proof}
Let us check that the hypotheses of Theorem \ref{thm:find-annulus} are met. First, note that $S^3(F)$ is not one of the excluded products, since $R_+(\gamma)= F\times\{1\}$  is  not an annulus or  pair of pants. Next, we have that \begin{equation}\label{eqn:homology-zero}H_2(S^3(F)) \cong \tilde{H}^0(F) \cong 0\end{equation} by Alexander duality. We also know that $S^3(F)$ is irreducible (in fact, this sutured manifold is taut, per Remark \ref{rmk:taut}), and that
\[ \dim \SFH(S^3(F)) = \dim \hfkhat(K,g) = 2 \]
by \eqref{eq:sfh-hfk-top}. Note that $g\geq 1$ since the unknot is not nearly fibered. Therefore, \[ \dim \SFH(M,\gamma) =2\leq 2g = \tfrac{1}{2}b_1(\partial M),\] and so Theorem~\ref{thm:find-annulus} provides an essential product annulus $A\subset (M,\gamma)=S^3(F)$. 

Let us suppose for a contradiction that both  boundary components \[A_\pm\subset R_\pm(\gamma)\] of $A$ are   boundary-parallel in their respective surfaces.   We  recover the knot complement \[E_K = S^3 \setminus N(K)\]  from $M$ by gluing $R_+(\gamma)$ to $R_-(\gamma)$ by some homeomorphism, and  we can assume that this gluing map sends $A_+$ to $A_-$ since these curves are boundary-parallel in $R_\pm(\gamma)$, respectively. Then  $A$ becomes a torus $T \subset E_K$ which meets $F$ in a boundary-parallel circle.  

We first claim that $T$ is incompressible in $E_K$.  Indeed, its fundamental group is spanned by a longitude $\lambda$ of $K$ and the image $c$ of a curve
\[ \{\pt\} \times [-1,1] \subset S^1 \times [-1,1] \cong A, \]
which is homologically essential in $E_K$ since it is dual to $F$.  If some product $\lambda^i c^j$ is nullhomotopic in $E_K$ then its homology class satisfies \[0 = [\lambda^ic^j]\cdot F = j,\] so it is a power $\lambda^i$ of the longitude of $K$, but then $i=0$ since $K$ is a nontrivial knot in $S^3$. Therefore, $\lambda^i c^j$ is nullhomotopic in $T$ as well. 

We next claim that $T$ is not boundary-parallel. Indeed, if it were, then  $T$ and $\partial E_K$ would cobound a thickened torus intersecting $F$ in a properly embedded annulus, in which case  cutting $E_K$ back open along $F$ would  give a thickened annulus in $(M,\gamma)$ which is the trace of an isotopy between $A$ and $\gamma$ that keeps $\partial A$ in $R(\gamma)$ at all times.  But $A$ is \emph{essential}, which by definition implies that no such isotopy  exists, a contradiction.

We have shown that under these circumstances $K$ must be a satellite knot, and the torus $T$ splits its exterior into two pieces: the exterior $E_C$ of the companion $C$, and the exterior $E_P$ of the pattern $P\subset S^1\times D^2$.  But then $T$ splits the Seifert surface $F$ into two pieces as well, one of which is an annulus in $E_P$ cobounded by the image  of $A_\pm$ and the boundary $\partial F$.  This annulus gives an isotopy of the pattern $P$ into $T$, where it is identified with a longitude of $C$, so $P$ must be a cable pattern with winding number one.  But this means that $P$ is isotopic to the core of $S^1\times D^2$, so $T$ is boundary-parallel and we have a contradiction.  We conclude that $A_\pm$ cannot both be boundary-parallel, as desired.
\end{proof}

\subsection{The manifold $M_F$}\label{ssec:M_F}
While Lemma~\ref{lem:product-annulus} applies to nearly fibered knots of any genus, we are especially interested in the genus-1 case.  In this setting we introduce the following construction, as in \S\ref{ssec:proof}, which we will refer to repeatedly throughout the paper.

\begin{definition} \label{def:M_F}
Let $F$ be a genus-1 Seifert surface for a nontrivial knot $K\subset S^3$.  We define
\[ M_F = S^3(F) \cup \big(D^2 \times [-1,1]\big) \]
to be the manifold obtained by gluing $D^2\times[-1,1]$ to $S^3(F)$ by a diffeomorphism 
\[  \partial D^2 \times [-1,1]   \cong  \partial F \times [-1,1] \]
which preserves the interval coordinate.  The boundary $\partial M_F$ is a disjoint union of two tori,
\[ T_\pm = (F\times \{\pm1\}) \cup (D^2 \times \{\pm 1\}). \] Let $\alpha$ be the properly embedded arc in $M_F$ given by  \[\alpha = \{0\}\times [-1,1]\subset D^2\times [-1,1].\] Then $(M,\gamma) = S^3(F)$ is clearly recovered  by removing the neighborhood \[N(\alpha) = D^2\times [-1,1]\] of $\alpha$ from $M_F$, with suture $s(\gamma)$  given by  the meridian \[\mu_\alpha = \partial D^2\times\{0\}\] of the arc $\alpha$.

\end{definition}

As noted in \S\ref{ssec:proof}, $M_F$ can also be described as the manifold obtained from the 0-surgery $S^3_0(K)$ by removing a tubular neighborhood of the torus $\hat{F}$ formed by capping off the Seifert surface $F$ with a disk in the solid surgery torus. This perspective shows the following:

\begin{lemma}\label{lem:MF-inc} Let $F$ be a genus-1 Seifert surface for a nontrivial knot $K\subset S^3$. Then  the manifold $M_F$ is irreducible, and  the tori $T_+$ and $T_-$ are incompressible.
\end{lemma}

\begin{proof}
\cite[Corollary~8.2]{gabai-foliations3} says that $S^3_0(K)$ admits a taut foliation with $\hat{F}$ a compact leaf. Cutting   open along this leaf  then gives a taut foliation on $M_F$ for which $T_\pm$ are compact leaves, from which the lemma follows.
\end{proof}

We end this section with the following lemma:

\begin{lemma} \label{lem:dehn-fill}
Let $F$ be a genus-1 Seifert surface for a nearly fibered knot $K\subset S^3$, and let \[A\subset M_F\] be the image of the annulus provided by Lemma \ref{lem:product-annulus} under the inclusion of $S^3(F)$ into $M_F.$  Then the boundary components \[A_\pm = \partial A\cap T_\pm\] are each homologically essential in their respective tori $T_\pm$.
\end{lemma}

\begin{proof}
Lemma \ref{lem:product-annulus} says that at least one of the boundary components of $A$, which we can take to be $A_+$ without loss of generality, is not boundary-parallel in  $R_+(\gamma)$, where \[(M,\gamma)=S^3(F).\] Since $R_+(\gamma)$ is a once-punctured torus in the case at hand, and the torus $T_+$ is obtained by capping off $R_+(\gamma)$ with a disk, it follows that  $A_+$ is homologically essential in $T_+$.
 
It remains to show that $A_-$ is homologically essential in $T_-$. If not, then this means that $A_-$ must be boundary-parallel when viewed as a curve in the once-punctured torus $R_-(\gamma)$. In this case, $A_-$ bounds the disk $D\subset T_-$ which caps off $R_-(\gamma)$ to form $T_-$. Then the union \[A\cup D\]  is a disk bounded by $T_+$. Pushing this disk slightly into the interior of $M_F$ gives a compressing disk for $T_+$. But this contradicts the fact that $T_+$ is incompressible, per Lemma \ref{lem:MF-inc}. It follows that $A_-$ is homologically essential in $T_-$, completing the proof of the lemma.
\end{proof}

This lemma is notable in part for the following consequence, as mentioned in \S\ref{ssec:proof}:

\begin{remark}\label{rmk:cabling}It follows from Lemma \ref{lem:dehn-fill}  that if $F$ is a genus-1 Seifert surface for a nearly fibered knot, then $M_F$ is the complement of a 2-component cable link in some 3-manifold, with \[A\subset M_F\] being the cabling annulus. Indeed, since  the curves $A_\pm$ are homologically essential in $T_\pm$, there are curves $c_\pm\subset T_\pm$ which are homologically dual to $A_\pm$. Then $M_F$ is the complement \[M_F \cong Y\setminus N(L),\] where  $Y$ is the closed 3-manifold obtained by Dehn filling the  tori $T_\pm$ along the curves $c_\pm$, and $L$ is the 2-component link given by the union of the cores of the solid tori in this filling. Recall that our eventual goal is to prove that $M_F$ is the complement of 2-component cables of the unknot or trefoils, per Theorem \ref{thm:identify-y-c}.
\end{remark}

\begin{remark}
As indicated in Lemma \ref{lem:dehn-fill}, we will henceforth view the annulus $A$ of Lemma \ref{lem:product-annulus} as living in $S^3(F)$ or $M_F$ interchangeably.
\end{remark}

\section{On the manifold $M_F$ and the arc $\alpha$} \label{sec:M_F}

Let $K \subset S^3$ be a nearly fibered knot, with a Seifert surface $F$ of genus 1.  Let \[\alpha\subset M_F\] be the arc in Definition \ref{def:M_F} whose complement recovers $S^3(F)$. Per Remark \ref{rmk:cabling}, $M_F$ is the complement of a 2-component cable link, with cabling annulus \[A\subset M_F\] as provided in Lemma \ref{lem:dehn-fill}. By construction, $\alpha$ is disjoint from $A$. Our goal in this section is to prove that it can be isotoped to lie in this cabling annulus, however. This is a key step towards our eventual classification of $M_F$ and thus $S^3(F)$  in the next section. 

\begin{proposition} \label{prop:t-in-cabling-annulus}
Let $F$ be a genus-1 Seifert surface for a nearly fibered knot $K \subset S^3$. Let \[A \subset M_F\] be the annulus provided by Lemma~\ref{lem:dehn-fill}. Then the arc $\alpha$  admits an isotopy, keeping $\partial\alpha$ in  $\partial M_F$ at all times, which carries $\alpha$ to a properly embedded arc in $A$.
\end{proposition}
Proposition \ref{prop:t-in-cabling-annulus} will follow from a combination of several  lemmas in this section.
To start, note that we can view $M_F$ as a (non-balanced) sutured manifold $(M_F,\gamma_F)$, where $\gamma_F = A(\gamma_F) \sqcup T(\gamma_F)$ is empty and the two boundary tori $T_\pm$ are oriented so that
\begin{align*}
R_+(\gamma_F) &= T_+, &
R_-(\gamma_F) &= T_-.
\end{align*}
Choose an orientation for $A$ and consider the sutured manifold decomposition
\[ (M_F,\gamma_F) \stackrel{A}{\leadsto} (M_A, \gamma_A) \] along $A$, illustrated in Figure~\ref{fig:cable-annulus-decomposition}. In particular, \[(M_A, \gamma_A)\] is a balanced sutured manifold with torus boundary, whose sutures $s(\gamma_A)$ are the union of two oppositely oriented curves of the same slope as the boundary components of $A$. 

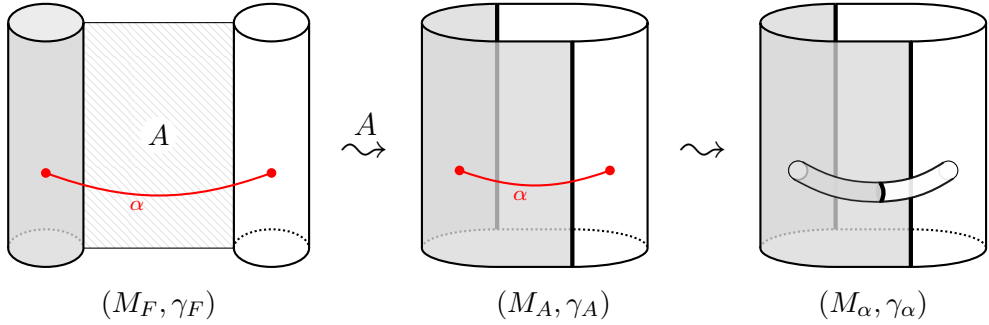
\begin{figure}
\begin{tikzpicture}
% cabling annulus
\draw[thin,pattern=north west lines,pattern color=gray!20] (1,-3) rectangle (3,0);
\node[circle,fill=white,inner sep=1pt] at (2,-1.5) {$A$};
% left complement
\path[fill=gray!30, fill opacity=0.5] (0,-3) -- ++(0,3) arc (180:0:0.5 and 0.25) -- ++(0,-3) arc (0:180:0.5 and 0.25);
\draw (0,0) arc (180:0:0.5 and 0.25);
\draw[densely dotted] (0,-3) arc (180:0:0.5 and 0.25);
\draw[fill=gray!30, fill opacity=0.75] (0,0) -- ++(0,-3) arc(180:360:0.5 and 0.25) -- ++(0,3) arc(0:-180:0.5 and 0.25);
% right complement
\draw (3,0) arc (180:0:0.5 and 0.25);
\draw[densely dotted] (3,-3) arc (180:0:0.5 and 0.25);
\draw (3,0) -- ++(0,-3) arc(180:360:0.5 and 0.25) -- ++(0,3) arc(0:-180:0.5 and 0.25);
% tangle
\draw[red] (0.5,-2) to[bend right=20] node[pos=0.4,below,inner sep=2pt] {\scriptsize$\alpha$} ++(3,0);
\draw[red,fill=red] (0.5,-2) circle (0.05) ++(3,0) circle (0.05);
\node at (2,-3.75) {$(M_F,\gamma_F)$};
\node at (4.75,-1.5) {\LARGE$\stackrel{A}{\leadsto}$};
\node at (7.25,-3.75) {$(M_A,\gamma_A)$};
%
% decompose along A
%
\begin{scope}[xshift=-0.5cm]
\draw (7,0.25) -- ++(1,0) arc (90:-90:1 and 0.25) ++(1,0.25) -- ++(0,-3) ++(-1,-0.25) arc (-90:0:1 and 0.25);
\path[fill=gray!30, fill opacity=0.75] (6,0) arc (180:90:1 and 0.25) -- ++(0,-3) arc (90:180:1 and 0.25) -- ++(0,3);
\draw[densely dotted] (6,-3) arc (180:90:1 and 0.25) -- ++(1,0) arc (90:0:1 and 0.25);
\draw (6,0) arc (180:90:1 and 0.25);
\draw[ultra thick] (7,0.25) -- ++(0,-3);
\draw[fill=gray!30, fill opacity=0.75] (6,0) arc (180:270:1 and 0.25) -- ++(1,0) -- ++(0,-3) -- ++(-1,0) arc (270:180:1 and 0.25) -- ++(0,3);
\draw[ultra thick] (8,-0.25) -- ++(0,-3);
% tangle
\path (7,-2) arc (270:240:1 and 0.25) coordinate (tl);
\path (8,-2) arc (270:300:1 and 0.25) coordinate (tr);
\draw[red] (tl) to[bend right=20] node[pos=0.4,below,inner sep=2pt] {\scriptsize$\alpha$} (tr);
\draw[red,fill=red] (tl) circle (0.05) (tr) circle (0.05);
\end{scope}
%
% remove T
%
\begin{scope}[xshift=4cm]
\draw (7,0.25) -- ++(1,0) arc (90:-90:1 and 0.25) ++(1,0.25) -- ++(0,-3) ++(-1,-0.25) arc (-90:0:1 and 0.25);
\path[fill=gray!30, fill opacity=0.75] (6,0) arc (180:90:1 and 0.25) -- ++(0,-3) arc (90:180:1 and 0.25) -- ++(0,3);
\draw[densely dotted] (6,-3) arc (180:90:1 and 0.25) -- ++(1,0) arc (90:0:1 and 0.25);
\draw (6,0) arc (180:90:1 and 0.25);
\draw[ultra thick] (7,0.25) -- ++(0,-3);
\draw[fill=gray!30, fill opacity=0.75] (6,0) arc (180:270:1 and 0.25) -- ++(1,0) -- ++(0,-3) -- ++(-1,0) arc (270:180:1 and 0.25) -- ++(0,3);
\draw[ultra thick] (8,-0.25) -- ++(0,-3);
% tangle
\path (7,-2) arc (270:240:1 and 0.25) coordinate (tl);
\path (8,-2) arc (270:300:1 and 0.25) coordinate (tr);
\draw[thin,fill=white] (tl) circle (0.125) (tr) circle (0.125);
\path (tl) ++(60:0.125) coordinate (tl-top) (tr) ++(120:0.125) coordinate (tr-top);
\path (tl) ++(240:0.125) coordinate (tl-bot) (tr) ++(300:0.125) coordinate (tr-bot);
\path (tl-top) to[bend right=30] coordinate[pos=0.55] (top-mid) (tr-top);
\path (tl-bot) to[bend right=30] coordinate[pos=0.55] (bot-mid) (tr-bot);
%\draw[red] (tl) to[bend right=20] node[pos=0.4,below,inner sep=2pt] {\scriptsize$T$} (tr);
%\draw[fill] (tl) circle (0.05) (tr) circle (0.05);
\draw[ultra thick] (top-mid) to[bend right=30] (bot-mid);
\draw[fill=gray!30,opacity=0.75] (tl-top) to[bend right=15] (top-mid) to[bend left=15] (bot-mid) to[bend left=15] (tl-bot) arc (240:420:0.125);
\draw[thin,fill=gray!30,opacity=0.75] (tl-top) to[bend right=15] (top-mid) to[bend left=30] (bot-mid) to[bend left=15] (tl-bot) arc (240:60:0.125);
\draw[fill=white,opacity=0.9] (tr-top) to[bend left=15] (top-mid) to[bend left=30] (bot-mid) to[bend right=15] (tr-bot);
\draw[ultra thick] (top-mid) to[bend left=30] (bot-mid);
\node at (5.25,-1.5) {\LARGE$\stackrel{\vphantom{A}}{\leadsto}$};
\node at (7.5,-3.75) {$(M_\alpha,\gamma_\alpha)$};
\end{scope}
\end{tikzpicture}
\caption{Decomposing $(M_F,\gamma_F)$ along the annulus $A$ to form $(M_A,\gamma_A)$, and then removing the arc $\alpha$ to obtain the sutured manifold $(M_\alpha,\gamma_\alpha)$. The thick curves in the middle and right pictures indicate the sutures for these manifolds; there are no sutures on the left because $A(\gamma_F)$ is empty.}
\label{fig:cable-annulus-decomposition}
\end{figure}

Since $\alpha$ is disjoint from  $A$ in $M_F$, we can also view $\alpha$ as a properly embedded arc in $M_A$. From this perspective, we then define the sutured arc complement
\[ (M_\alpha,\gamma_\alpha) := ( M_A \setminus N(\alpha), \gamma_A \cup  N(\mu_\alpha) ) \] pictured on the right side in Figure~\ref{fig:cable-annulus-decomposition}, 
where \[N(\mu_\alpha):=N(\alpha)\cap \partial M_\alpha\] is a neighborhood in $\partial M_\alpha$ of the meridian $\mu_\alpha$ of  $\alpha$.

Note that $(M_\alpha,\gamma_\alpha)$ can alternatively be obtained from $(M,\gamma) = S^3(F)$ via sutured manifold decomposition 
\begin{equation}\label{eqn:annulus-decomposition} S^3(F) \stackrel{A}{\leadsto} (M_\alpha,\gamma_\alpha) \end{equation}
along the product annulus $A$ (to be precise, the annulus whose image in $M_F$ is $A$), and the image of $\gamma$ under this decomposition is $N(\mu_\alpha)$. It follows from  Lemma~\ref{lem:dehn-fill} that at least one (in fact, both) of the boundary components of  \[A\subset S^3(F)\] is homologically essential in $R(\gamma)$. Moreover, we have by Alexander duality as in \eqref{eqn:homology-zero} that \[H_2(S^3(F))\cong 0.\] The product annulus decomposition in \eqref{eqn:annulus-decomposition} therefore preserves sutured Floer homology, \begin{equation} \label{eq:sfh-m_t}
\SFH(M_\alpha,\gamma_\alpha) \cong \SFH(S^3(F)) \cong \Q^2,
\end{equation}
by Theorem~\ref{thm:sfh-annulus-decomposition}. 
Since $S^3(F)$ is taut, it follows that  $(M_\alpha,\gamma_\alpha)$ is taut as well (Remark \ref{rmk:taut-annulus}).

\begin{lemma} \label{lem:find-B-homology}
We have $H_2(M_\alpha;R) \cong 0$ and $H_2(M_A;R) \cong 0$ for any commutative ring $R$.
\end{lemma}

\begin{proof} Forgetting about the sutures, note  that the Seifert surface complement \[S^3(F)\cong S^3\setminus \inr(F\times[-1,1])\] can be recovered from $M_\alpha$ by gluing a thickened annulus $N(A')$  along $\gamma_A$ by a map which identifies $\partial A'$ with $s(\gamma_A).$
The Mayer--Vietoris sequence associated to the decomposition
\[ S^3(F) \cong M_\alpha \cup_{\gamma_A} N(A') \]
with coefficients in $R$ (which we momentarily suppress for convenience) reads in part:
\[
\underbrace{H_2( \gamma_A)}_{\cong 0} \to H_2(M_\alpha) \oplus \underbrace{H_2(N(A'))}_{\cong 0} \to \underbrace{H_2(S^3(F))}_{\cong 0} 
\to H_1( \gamma_A) \to H_1(M_\alpha) \oplus H_1(N(A')).
\]
We have that \[H_2(S^3(F); R) \cong \tilde{H}^0(F;R) \cong 0,\] by Alexander duality, so the leftmost portion of the sequence tells us that $H_2(M_\alpha;R) \cong 0$, proving the first claim.

Moreover, the map $H_1( \gamma_A;R) \to H_1(N(A');R)$  sends  the  class  $[s(\gamma_A)]$ to \[[\partial A']=0\in H_1(N(A');R).\] Since the rightmost map in the  sequence is injective, and \[r\cdot[s(\gamma_A)]\neq 0 \in H_1( \gamma_A;R) \textrm{ for all } r\in R\setminus \{0\},\] it follows that \[r\cdot[s(\gamma_A)]\neq 0 \in H_1( M_\alpha;R) \textrm{ for all } r\in R\setminus \{0\}. \] 
 Note that the meridian $\mu_\alpha$ of $\alpha$ and the sutures $s(\gamma_A)$ cobound the pair of pants $R_+(\gamma_\alpha)\subset \partial M_\alpha.$ It follows that \[ [\mu_\alpha] = \pm [s(\gamma_{A})]\in H_1(M_\alpha; R),\] and therefore that \begin{equation}\label{eqn:nonzero-alpha}r\cdot[\mu_\alpha]\neq 0 \in H_1( M_\alpha;R) \textrm{ for all } r\in R\setminus \{0\}. \end{equation}

To prove the second claim, note that $M_A$ is recovered from $M_\alpha$ by gluing back the neighborhood $N(\alpha)$  along the annular neighborhood $N(\mu_\alpha)$ of $\mu_\alpha$,
\begin{equation}\label{eqn:MAdisk} M_A \cong M_\alpha \cup_{N(\mu_\alpha)} N(\alpha). \end{equation} Let us  consider the Mayer--Vietoris sequence corresponding to this decomposition.
Since \[H_2(M_\alpha;R)\cong H_2(N(\alpha);R)\cong H_1(N(\alpha);R)\cong 0,\] the portion of the sequence beginning at $H_2(M_\alpha;R) \oplus H_2(N(\alpha);R)$ has the form
\[ 0 \to H_2(M_A;R) \to \underbrace{H_1(N(\mu_\alpha);R)}_{\cong R} \to H_1(M_\alpha;R), \]
with $H_1(N(\mu_\alpha);R)$   generated by the class $[\mu_\alpha]$.  Then it follows from \eqref{eqn:nonzero-alpha} that the rightmost map is injective, and we conclude by exactness that $H_2(M_A;R) \cong 0$, as desired.\end{proof}

The next lemma provides the product annulus $B$ mentioned in \S\ref{ssec:proof}:

\begin{lemma} \label{lem:M_T-annulus}
There exists an essential product annulus $B \subset (M_\alpha,\gamma_\alpha)$.
\end{lemma}

\begin{proof}
We know that $(M_\alpha,\gamma_\alpha)$ is a taut balanced sutured manifold, with  $H_2(M_\alpha;\Z)=0$ by Lemma~\ref{lem:find-B-homology}, and its boundary $\partial M_\alpha$ is a connected genus-2 surface.  Then
\[ \dim \SFH(M_\alpha,\gamma_\alpha) = 2 \]
by \eqref{eq:sfh-m_t}, so Theorem~\ref{thm:find-annulus} provides the desired annulus.
\end{proof}

Given the  product annulus $B$ from Lemma~\ref{lem:M_T-annulus},  let us denote its boundary circles by  \[B_\pm = \partial B \cap R_\pm(\gamma_\alpha).\] Neither   $B_+$ nor $B_-$ bounds a disk in $R(\gamma_\alpha)$, since $B$ is essential and hence incompressible. It follows that  $B_+$ and $B_-$ are each boundary-parallel curves in the pairs of pants $R_+(\gamma_\alpha)$ and $R_-(\gamma_\alpha)$, respectively. In particular, $B_\pm$ are each isotopic in $\partial M_\alpha$ either to  a component of $s(\gamma_A)$ or to the meridian $\mu_\alpha$ of $\alpha$. We rule out the latter possibility below:

\begin{lemma} \label{lem:boundary-b-not-meridian}
Neither $B_+$ nor $B_-$ is isotopic in $\partial M_\alpha$ to the meridian $\mu_\alpha$ of $\alpha$.
\end{lemma}

\begin{proof}
Suppose  that $B_+$ is isotopic in $\partial M_\alpha$ to $\mu_\alpha$  but  $B_-$ is not. From the discussion above,  $B_-$ must then be isotopic in $\partial M_\alpha$ to a component of $s(\gamma_A)$. Recall that $M_A$ is obtained from $M_\alpha$ by gluing back a thickened disk (namely, the neighborhood $N(\alpha)$) along a neighborhood  of the meridian $\mu_\alpha$, as in \eqref{eqn:MAdisk}. It follows that under the inclusion \[M_\alpha\hookrightarrow M_A,\] the boundary component $B_+$ of the annulus $B$ gets capped off with a disk $D$, so that  \[B\cup D\subset M_A\] is a disk bounded by the curve $B_-\subset \partial M_A$. This disk then gives rise under the inclusion \[M_A\hookrightarrow M_F\] to a disk in $M_F$ bounded by the image \[B_-\subset T_-\subset \partial M_F.\] But $B_-$ is isotopic in $T_-$ to the boundary component $A_-$ of the annulus $A$, which by Lemma \ref{lem:dehn-fill} is homologically essential. The fact that this curve bounds a disk in $M_F$ then contradicts the fact that the torus $T_-$ is incompressible, as shown in Lemma \ref{lem:MF-inc}. 

Swapping the roles of $B_+$ and $B_-$ leads to the same contradiction, so let us now assume that the curves $B_\pm$ are both are isotopic in $\partial M_\alpha$ to $\mu_\alpha$.
In this case, reversing the decompositions $S^3(K) \stackrel{F}{\leadsto} S^3(F) \stackrel{A}{\leadsto} (M_\alpha,\gamma_\alpha)$, we can glue $B_+$ to $B_-$ to turn the annulus $B$ into a closed, embedded surface $\Sigma_B$ in $S^3(K)$ that meets $F$ transversely in a single boundary-parallel curve.  Then $\Sigma_B$ must be a torus, since if it were a Klein bottle it could not embed in $S^3(K) \subset S^3$; as a torus in $S^3$, it must bound a solid torus $V_B$ on one side or the other.

If $V_B \subset S^3$ were contained in the knot complement $S^3(K)$, then
\[ V_B \cap F \subset V_B \]
would be a properly embedded, punctured torus (consisting of $F$ minus a collar neighborhood of its boundary) in the solid torus $V_B$; but then it must compress inside $V_B$ and hence in $S^3(K)$, contradicting the incompressibility of $F$.  Thus $V_B$ must not lie entirely in $S^3(K)$, and this means that it must contain $\partial\big(S^3(K)\big) = \partial N(K)$ as well as the knot $K$.  We now argue exactly as in the proof of Lemma~\ref{lem:product-annulus}: the torus $\partial V_B = \Sigma_B$ must be incompressible in $S^3(K)$, realizing $K$ as a satellite knot, but then the annulus $F \cap V_B$ provides an isotopy from $K$ to its companion knot, so the satellite pattern must have been trivial.  This means that $\Sigma_B = \partial V_B$ is boundary-parallel in $S^3(K)$.  Decomposing again along $F$ and then $A$, we conclude that our original annulus $B$ must have been parallel to an annular neighborhood of $\mu_\alpha$ in $\partial M_\alpha$.  But this contradicts the claim from Lemma~\ref{lem:M_T-annulus} that $B$ is essential, so we are done.
\end{proof}

The proof of Lemma~\ref{lem:boundary-b-not-meridian} in the case where both of $B_\pm$ are isotopic to $\mu_\alpha$ was substantially longer in the original version of this paper; we thank one of the referees for providing the much simpler argument used here.

\begin{lemma} \label{lem:annulus-b-separating}
The annulus $B$ separates $M_A$, and its oriented boundary meets the torus $\partial M_A$ in a pair of parallel but oppositely oriented essential curves.
\end{lemma}

\begin{proof} 
Let us orient $B$ as well as  its boundary curves $B_+$ and $B_-$ so that \[\partial B= B_+ \sqcup -B_-.\] Recall from Lemma~\ref{lem:find-B-homology} that $H_2(M_A) = 0$.  Therefore, the long exact sequence of the pair $(M_A,\partial M_A)$ reads in part:
\[ 0 \to H_2(M_A,\partial M_A) \xrightarrow{\partial_*} H_1(\partial M_A) \to H_1(M_A). \]
If $B$ is nonseparating in $M_A$ then it is nonzero in $H_2(M_A,\partial M_A)$. It then follows from the exact sequence above that the class $[\partial B]$ is nonzero in $H_1(\partial M_A)$, and hence that \begin{equation}\label{eqn:B-not-equal}[B_+] \neq [B_-]\in H_1(\partial M_A).\end{equation}  Let us suppose for a contradiction that this is the case.

As discussed  before Lemma~\ref{lem:boundary-b-not-meridian}, $B_+$ and $B_-$ are each isotopic in $\partial M_\alpha$ either to components of the sutures $s(\gamma_A)$ or to a meridian of the arc $\alpha$, as unoriented curves. We ruled out the latter possibility in Lemma~\ref{lem:boundary-b-not-meridian}. Therefore, when viewed as curves in $\partial M_A$, $B_\pm$ are each isotopic to components of  $s(\gamma_A)$ (and are thus core circles of $R_\pm(\gamma_A)$). In particular, $B_+$ and $B_-$ are isotopic to one another as unoriented curves in $\partial M_A$. Given \eqref{eqn:B-not-equal}, it must therefore be the case that $B_+$ and $B_-$ are parallel, oppositely oriented curves in $\partial M_A$.  

Forgetting their orientation, these curves cobound an annulus in $\partial M_A$, whose union with $B$ is then a Klein bottle \[\Sigma \subset M_A.\]  Since $M_A$ is orientable, the Klein bottle $\Sigma$ must be one-sided and in particular nonseparating. This implies  that the mod-$2$ intersection pairing
\[ H_1(M_A,\partial M_A;\Z/2\Z) \times H_2(M_A;\Z/2\Z) \to \Z/2\Z \]
is nonzero.  But this contradicts the fact that $H_2(M_A;\Z/2\Z)=0$, by Lemma~\ref{lem:find-B-homology}.  Therefore, $[B_+] = [B_-]$, and then $B$ has the desired properties.
\end{proof}

\begin{lemma} \label{lem:isotope-t-into-dNC}
The arc $\alpha \subset M_A$ can be isotoped rel endpoints so that it lies in $\partial M_A$ and meets the sutures $s(\gamma_A)$ transversely in a single point.
\end{lemma}

\begin{proof}
Lemma~\ref{lem:annulus-b-separating} implies that decomposing $(M_\alpha,\gamma_\alpha)$ along the product annulus $B$ produces a disconnected balanced sutured manifold
\[ (M_\alpha,\gamma_\alpha) \stackrel{B}{\leadsto} (M_2,\gamma_2) \sqcup (M_3,\gamma_3), \]
where we have labeled the components so that $(M_2,\gamma_2)$ has two  sutures and $(M_3,\gamma_3)$ has three, as depicted in Figure~\ref{fig:B-essential-boundary}.
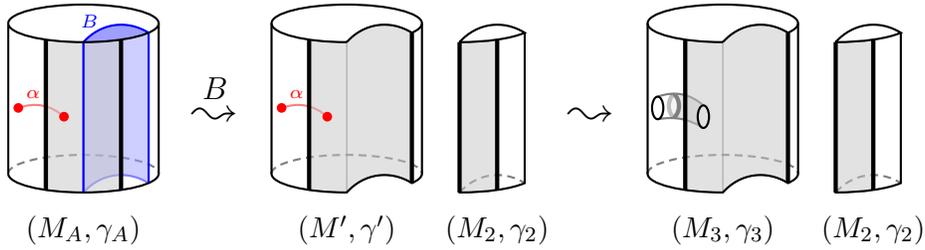
\begin{figure}
\begin{tikzpicture}
\begin{scope}
\draw[densely dashed] (-1,-2) arc(180:0:1 and 0.25);
\begin{scope} % the annulus B
\coordinate (bleft) at (270:1 and 0.25);
\coordinate (bright) at (330:1 and 0.25);
\clip (-1,0) -- ++(0,-2) arc (180:360:1 and 0.25) -- ++(0,2) arc (0:180:1 and 0.25); % make sure endpoints of boundary of B don't go too far
\draw[blue,fill=blue!75, fill opacity=0.5] (bleft) to[bend left=40] coordinate[pos=0.3] (blabel) (bright) -- ++(0,-2) to[bend right=40] ($(bleft)+(0,-2)$) -- ++(0,2);
\end{scope}
\begin{scope} % the tangle T
\path (210:1 and 0.25) -- ++(0,-1) coordinate (tl);
\path (255:1 and 0.25) -- ++(0,-1) coordinate (tr);
\draw[red] (tl) to[bend left=30] coordinate[pos=0.3] (tmid) (tr);
\end{scope}
\begin{scope} % sutures
\fill[gray!30,fill opacity=0.75] (240:1 and 0.25) coordinate (leftsuture) -- ++(0,-2) arc (240:300:1 and 0.25) -- ++(0,2) coordinate (rightsuture) arc (300:240:1 and 0.25);
\fill[white,fill opacity=0.5] (leftsuture) arc (240:180:1 and 0.25) -- ++(0,-2) arc (180:240:1 and 0.25) -- ++(0,2);
\fill[white,fill opacity=0.5] (rightsuture) arc (300:360:1 and 0.25) -- ++(0,-2) arc (360:300:1 and 0.25) -- ++(0,2);
\draw[ultra thick] (leftsuture) -- ++(0,-2) (rightsuture) -- ++(0,-2); 
\end{scope}
\draw[red,fill=red] (tl) circle (0.05) (tr) circle (0.05); % draw endpoints of T
\path[red] (tmid) -- ++(0,0.15) node {\tiny$\alpha$}; % label T
\path[blue] (blabel) ++(-0.15,0.1) node {\tiny$B$}; % label B
\draw[blue] (bleft) -- ++(0,-2) (bright) -- ++(0,-2); % redraw boundary of B
\draw (0,0) ellipse (1 and 0.25);
\draw (1,0) -- ++(0,-2) arc(0:-180: 1 and 0.25) -- ++(0,2);
\node at (0,-2.75) {$(M_A,\gamma_A)$};
\end{scope}
\node at (1.75,-1) {\LARGE$\stackrel{B}{\leadsto}$};
\begin{scope}[xshift=3.5cm]
\draw[densely dashed] (-1,-2) arc(180:0:1 and 0.25);
\begin{scope} % the annulus B
\coordinate (bleft) at (270:1 and 0.25);
\coordinate (bright) at (330:1 and 0.25);
\clip (-1,0) -- ++(0,-2) arc (180:360:1 and 0.25) -- ++(0,2) arc (0:180:1 and 0.25); % make sure endpoints of boundary of B don't go too far
\draw[thin] (bleft) to[bend left=40] coordinate[pos=0.3] (blabel) (bright) -- ++(0,-2) to[bend right=40] ($(bleft)+(0,-2)$) -- ++(0,2);
\end{scope}
\begin{scope} % the tangle T
\path (210:1 and 0.25) -- ++(0,-1) coordinate (tl);
\path (255:1 and 0.25) -- ++(0,-1) coordinate (tr);
\draw[red] (tl) to[bend left=30] coordinate[pos=0.3] (tmid) (tr);
\end{scope}
\begin{scope} % sutures
\fill[gray!30,fill opacity=0.75] (240:1 and 0.25) coordinate (leftsuture) -- ++(0,-2) arc (240:270:1 and 0.25) to[bend left=40] ($(bright)+(0,-2)$) -- ++(0,2) coordinate (rightsuture) to[bend right=40] (bleft) arc (270:240:1 and 0.25);
\fill[white,fill opacity=0.5] (leftsuture) arc (240:180:1 and 0.25) -- ++(0,-2) arc (180:240:1 and 0.25) -- ++(0,2);
\fill[white,fill opacity=0.5] (rightsuture) arc (330:360:1 and 0.25) -- ++(0,-2) arc (360:330:1 and 0.25) -- ++(0,2);
\draw[ultra thick] (leftsuture) -- ++(0,-2) (rightsuture) -- ++(0,-2); 
\end{scope}
\draw[red,fill=red] (tl) circle (0.05) (tr) circle (0.05); % draw endpoints of T
\path[red] (tmid) -- ++(0,0.15) node {\tiny$\alpha$}; % label T
\coordinate (bleft) at (270:1 and 0.25);
\draw (bleft) arc (270:-30:1 and 0.25) coordinate (bright) to[bend right=40] (bleft);
\draw (1,0) -- ++(0,-2) arc(0:-30: 1 and 0.25) to[bend right=40] ($(bleft)+(0,-2)$) arc (270:180:1 and 0.25) -- ++(0,2);
\begin{scope} %% the other component
\path (1.5,0) ++(270:1 and 0.25) coordinate (bbl) arc(270:300:1 and 0.25) coordinate (bbm) arc (300:330:1 and 0.25) coordinate (bbr);
\draw[densely dashed] ($(bbl)+(0,-2)$) to[bend left=40] ($(bbr)+(0,-2)$);
\fill[gray!30,fill opacity=0.75] (bbl) arc(270:300:1 and 0.25) -- ++(0,-2) arc (300:270:1 and 0.25) -- ++(0,2);
\fill[white,fill opacity=0.5] (bbm) arc (300:330:1 and 0.25) -- ++(0,-2) arc (330:300:1 and 0.25) -- ++(0,2);
\begin{scope} % do some clipping to fix corners
\clip ($(bbl)+(-0.015,0.5)$) rectangle ($(bbr)+(0.015,-2.5)$);
\draw (bbl) arc (270:330:1 and 0.25) to[bend right=40] (bbl);
\end{scope}
\draw (bbl) -- ++(0,-2) arc (270:330:1 and 0.25) -- ++(0,2);
\draw[ultra thick] (bbl) -- ++(0,-2) (bbm) -- ++(0,-2); % sutures
\node at (0,-2.75) {$(M',\gamma')$};
\path let \p1 = (bbm) in node at (\x1,-2.75) {$(M_2,\gamma_2)$};
\end{scope}
\end{scope}
\node at (6.75,-1) {\LARGE$\stackrel{\vphantom{B}}{\leadsto}$};
\begin{scope}[xshift=8.5cm]
\draw[densely dashed] (-1,-2) arc(180:0:1 and 0.25);
\begin{scope} % the annulus B
\coordinate (bleft) at (270:1 and 0.25);
\coordinate (bright) at (330:1 and 0.25);
\clip (-1,0) -- ++(0,-2) arc (180:360:1 and 0.25) -- ++(0,2) arc (0:180:1 and 0.25); % make sure endpoints of boundary of B don't go too far
\draw[thin] (bleft) to[bend left=40] coordinate[pos=0.3] (blabel) (bright) -- ++(0,-2) to[bend right=40] ($(bleft)+(0,-2)$) -- ++(0,2);
\end{scope}
\begin{scope} % the tangle T
\path (210:1 and 0.25) -- ++(0,-1) coordinate (tl);
\path (255:1 and 0.25) -- ++(0,-1) coordinate (tr);
\path ($(tl)+(0,0.15)$) to[bend left=30] coordinate[pos=0.33] (tmtop) ($(tr)+(0,0.15)$);
\path ($(tl)-(0,0.15)$) to[bend left=20] coordinate[pos=0.33] (tmbot) ($(tr)-(0,0.15)$);
\fill[gray!40,fill opacity=0.75] (tmtop) to[bend right=45] (tmbot) to[bend left=13] ($(tr)-(0,0.15)$) arc (270:450:0.075 and 0.15) to[bend right=20] (tmtop); % back of T tube
\begin{scope}
\clip ($(tl)+(0,0.15)$) to[bend left=30] ($(tr)+(0,0.15)$) -- ($(tr)-(0,0.15)$) to[bend right=20] ($(tl)-(0,0.15)$) -- cycle;
\draw[ultra thick] (tmtop) to[bend right=45] (tmbot);
\fill[gray!40,fill opacity=0.75] (tmtop) to[bend left=45] (tmbot) to[bend left=13] ($(tr)-(0,0.15)$) arc (270:90:0.075 and 0.15) to[bend right=20] (tmtop); % front of T tube
\fill[white,fill opacity=0.5] (tmtop) to[bend left=45] (tmbot) to[bend left=13] ($(tl)-(0,0.15)$) arc (270:450:0.075 and 0.15) to[bend right=20] (tmtop);
\draw[ultra thick] (tmtop) to[bend left=45] (tmbot);
\end{scope}
\draw ($(tl)+(0,0.15)$) to[bend left=30] ($(tr)+(0,0.15)$);
\draw ($(tl)-(0,0.15)$) to[bend left=20] ($(tr)-(0,0.15)$);
\end{scope}
\begin{scope} % sutures
\fill[gray!30,fill opacity=0.75,even odd rule] (240:1 and 0.25) coordinate (leftsuture) -- ++(0,-2) arc (240:270:1 and 0.25) to[bend left=40] ($(bright)+(0,-2)$) -- ++(0,2) coordinate (rightsuture) to[bend right=40] (bleft) arc (270:240:1 and 0.25) (tr) ellipse (0.075 and 0.15);
\fill[white,fill opacity=0.5] (leftsuture) arc (240:180:1 and 0.25) -- ++(0,-2) arc (180:240:1 and 0.25) -- ++(0,2);
\fill[white,fill opacity=0.5] (rightsuture) arc (330:360:1 and 0.25) -- ++(0,-2) arc (360:330:1 and 0.25) -- ++(0,2);
\draw[ultra thick] (leftsuture) -- ++(0,-2) (rightsuture) -- ++(0,-2); 
\end{scope}
\draw (tl) ellipse (0.075 and 0.15);
\draw[fill=gray!20,fill opacity=0.75] (tr) ellipse (0.075 and 0.15);
%\path[red] (tmid) -- ++(0,0.15) node {\tiny$\alpha$}; % label T
\coordinate (bleft) at (270:1 and 0.25);
\draw (bleft) arc (270:-30:1 and 0.25) coordinate (bright) to[bend right=40] (bleft);
\draw (1,0) -- ++(0,-2) arc(0:-30: 1 and 0.25) to[bend right=40] ($(bleft)+(0,-2)$) arc (270:180:1 and 0.25) -- ++(0,2);
\begin{scope} %% the other component
\path (1.5,0) ++(270:1 and 0.25) coordinate (bbl) arc(270:300:1 and 0.25) coordinate (bbm) arc (300:330:1 and 0.25) coordinate (bbr);
\draw[densely dashed] ($(bbl)+(0,-2)$) to[bend left=40] ($(bbr)+(0,-2)$);
\fill[gray!30,fill opacity=0.75] (bbl) arc(270:300:1 and 0.25) -- ++(0,-2) arc (300:270:1 and 0.25) -- ++(0,2);
\fill[white,fill opacity=0.5] (bbm) arc (300:330:1 and 0.25) -- ++(0,-2) arc (330:300:1 and 0.25) -- ++(0,2);
\begin{scope} % do some clipping to fix corners
\clip ($(bbl)+(-0.015,0.5)$) rectangle ($(bbr)+(0.015,-2.5)$);
\draw (bbl) arc (270:330:1 and 0.25) to[bend right=40] (bbl);
\end{scope}
\draw (bbl) -- ++(0,-2) arc (270:330:1 and 0.25) -- ++(0,2);
\draw[ultra thick] (bbl) -- ++(0,-2) (bbm) -- ++(0,-2); % sutures
\end{scope}
\node at (0,-2.75) {$(M_3,\gamma_3)$};
\path let \p1 = (bbm) in node at (\x1,-2.75) {$(M_2,\gamma_2)$};
\end{scope}
\end{tikzpicture}
\caption{We decompose $(M_A,\gamma_A)$ along $B$ to obtain $(M_2,\gamma_2)\sqcup (M',\gamma')$.  Removing   $\alpha$  and adding a meridional suture produces $(M_2,\gamma_2) \sqcup (M_3,\gamma_3)$, which is also the result of decomposing $(M_\alpha,\gamma_\alpha)$ along $B$.}
\label{fig:B-essential-boundary}
\end{figure}
Indeed, in $M_A$ the components of $\partial B$ are core circles of the annuli $R_+(\gamma_A)$ and $R_-(\gamma_A)$, so decomposing $(M_A,\gamma_A)$ along the separating $B$ produces a disjoint union of two sutured manifolds, with two sutures each, \[ (M_A,\gamma_A) \stackrel{B}{\leadsto} (M_2,\gamma_2) \sqcup (M',\gamma'). \]  One of these components is disjoint from the arc $\alpha$, so we label it $(M_2,\gamma_2)$. We then remove a tubular neighborhood of $\alpha$ from the other component $(M',\gamma')$ and add a meridional suture $\mu_\alpha$ to get $(M_3,\gamma_3)$.

Since the components of $\partial B$ are homologically essential in $R(\gamma_A)$, we have that \begin{align*}
\SFH(M_\alpha,\gamma_\alpha) &\cong \SFH((M_2,\gamma_2) \sqcup (M_3,  \gamma_3)) \\
&\cong \SFH(M_2,\gamma_2) \otimes \SFH(M_3,\gamma_3),
\end{align*}
by Theorem~\ref{thm:sfh-annulus-decomposition}.
Since the left side is 2-dimensional, per \eqref{eq:sfh-m_t}, it follows that
\[ \dim\SFH(M_i,\gamma_i) =1 \]
for some $i\in\{2,3\}$.  Then Theorem~\ref{thm:sfh-taut} tells us that the corresponding $(M_i,\gamma_i)$ is a product sutured manifold (note that $(M_i,\gamma_i)$ is taut since $(M_\alpha,\gamma_\alpha)$ is taut, per Remark \ref{rmk:taut-annulus}).

Suppose first that $(M_2,\gamma_2)$ is a product sutured manifold.  Since $\partial M_2$ is a torus and the sutures $s(\gamma_2)$ consist of two parallel essential curves on this torus, $R_+(\gamma_2)$ is an annulus and so there is a homeomorphism
\[ (M_2,\gamma_2) \cong \big((S^1\times I) \times [-1,1], (S^1\times \partial I) \times [-1,1]\big). \]
But if this is the case, then $B$ could have been isotoped onto the component of $\gamma_\alpha$ which became a  component of $\gamma_2$, by an isotopy keeping $\partial B$ in $R(\gamma_\alpha)$ at all times. This  contradicts the fact that $B$ is  essential.

It follows that  $(M_3,\gamma_3)$ is a product sutured manifold.  Since $R_+(\gamma_3)$ is a pair of pants $P$, we have that
\[ (M_3,\gamma_3) \cong (P\times[-1,1], \partial P \times [-1,1]). \]
One component of the sutures $s(\gamma_3)$ is a meridian $\mu_\alpha$ of $\alpha$, and $(M',\gamma')$ is recovered by gluing back a thickened disk $D^2\times I$ along an annular neighborhood of this meridian. The meridian $\mu_\alpha$ corresponds to a certain boundary component of $P$. Letting \[S^1\times [0,1] = P\cup D^2\] be the annulus formed by capping off this boundary component with a disk, we then have the identification
\[ (M',\gamma') \cong \big((S^1\times [0,1]) \times [-1,1], (S^1\times \{0,1\}) \times [-1,1]\big), \]
where the arc $\alpha\subset M'$ is given by \[ \alpha = \{\pt\} \times [-1,1] \subset M' \]
for some point \[\pt \in D^2\subset (S^1\times [0,1]),\] as depicted  in Figure~\ref{fig:mprime-product}.
\begin{figure}
\begin{tikzpicture}
\begin{scope}
\begin{scope}
\clip (-2,0) -- ++(0,-3) arc (180:360:2 and 0.5) -- ++(0,3) arc (0:180:2 and 0.5);
\draw[densely dashed] (2,-3) arc (0:180:2 and 0.5); % back outer curve
\draw[densely dashed] (0.8,-3) arc (0:180:0.8 and 0.2); % back inner curve
\fill[gray!30,fill opacity=0.75] (2,0) arc (0:180:2 and 0.5) -- ++(0,-1.5) arc (180:0:2 and 0.5) -- ++(0,1.5); % fill outer tube, back
\draw[ultra thick] (2,-1.5) arc (0:180:2 and 0.5); % draw outer suture, back
\fill[gray!40,fill opacity=0.75] (0.8,0) arc (0:180:0.8 and 0.2) -- ++(0,-1.5) arc (180:0:0.8 and 0.2) -- ++(0,1.5); % fill inner tube, back
\fill[white,opacity=0.5] (0.8,-1.5) arc (0:180:0.8 and 0.2) -- ++(0,-1.5) arc (180:0:0.8 and 0.2) -- ++(0,1.5);
\draw[ultra thick] (0.8,-1.5) arc (0:180:0.8 and 0.2); % draw inner suture, back
\fill[gray!40,fill opacity=0.75] (0.8,0) arc (360:180:0.8 and 0.2) -- ++(0,-1.5) arc (180:360:0.8 and 0.2) -- ++(0,1.5); % fill inner tube, front
\fill[white,opacity=0.5] (0.8,-1.5) arc (360:180:0.8 and 0.2) -- ++(0,-1.5) arc (180:360:0.8 and 0.2) -- ++(0,1.5);
\draw[ultra thick] (0.8,-1.5) arc (360:180:0.8 and 0.2); % draw inner suture, front
\draw (0.8,0) -- ++(0,-3) arc (360:180:0.8 and 0.2) -- ++(0,3); % draw inner tube
\draw[red] (1.4,0) coordinate (tt) -- coordinate[pos=0.75] (tm) ++(0,-3) coordinate (tb); % draw T
\draw[red,fill=red] (tb) circle (0.05); % draw T bottom endpoint
\fill[gray!30,fill opacity=0.75] (-2,0) -- ++(0,-1.5) arc (180:360:2 and 0.5) -- ++(0,1.5) arc (360:180:2 and 0.5); % fill outer tube
\fill[white,fill opacity=0.5] (2,-1.5) arc (360:180:2 and 0.5) -- ++(0,-1.5) arc (180:360:2 and 0.5) -- ++(0,1.5);
\draw[fill=gray!30,fill opacity=0.75,even odd rule] (0,0) ellipse (2 and 0.5) ellipse (0.8 and 0.2); % fill top annulus
\draw[ultra thick] (2,-1.5) arc (360:180:2 and 0.5); % draw outer suture, front
\draw[red,fill=red] (tt) circle (0.05); % top endpoint of T
\path (tm) node[red,right,inner sep=2pt] {\tiny$\alpha$};
\end{scope}
\draw (0,0) ellipse (2 and 0.5); % redraw top outer ellipse
\draw (-2,0) -- ++(0,-3) arc (180:360:2 and 0.5) -- ++(0,3); % draw outside of outer tube
\end{scope}
\node at (3,-1.5) {\LARGE$\to$};
\begin{scope}[xshift=6cm]
\draw (0,0) ellipse (2 and 0.5); % draw top outer ellipse before clipping
\begin{scope}
\clip (-2,0) -- ++(0,-3) arc (180:360:2 and 0.5) -- ++(0,3) arc (0:180:2 and 0.5);
\draw[densely dashed] (2,-3) arc (0:180:2 and 0.5); % back outer curve
\draw[densely dashed] (0.8,-3) arc (0:180:0.8 and 0.2); % back inner curve
\fill[gray!30,fill opacity=0.75] (2,0) arc (0:180:2 and 0.5) -- ++(0,-1.5) arc (180:0:2 and 0.5) -- ++(0,1.5); % fill outer tube, back
\draw[ultra thick] (2,-1.5) arc (0:180:2 and 0.5); % draw outer suture, back
\fill[gray!40,fill opacity=0.75] (0.8,0) arc (0:180:0.8 and 0.2) -- ++(0,-1.5) arc (180:0:0.8 and 0.2) -- ++(0,1.5); % fill inner tube, back
\fill[white,opacity=0.5] (0.8,-1.5) arc (0:180:0.8 and 0.2) -- ++(0,-1.5) arc (180:0:0.8 and 0.2) -- ++(0,1.5);
\draw[ultra thick] (0.8,-1.5) arc (0:180:0.8 and 0.2); % draw inner suture, back
\fill[gray!40,fill opacity=0.75] (0.8,0) arc (360:180:0.8 and 0.2) -- ++(0,-1.5) arc (180:360:0.8 and 0.2) -- ++(0,1.5); % fill inner tube, front
\fill[white,opacity=0.5] (0.8,-1.5) arc (360:180:0.8 and 0.2) -- ++(0,-1.5) arc (180:360:0.8 and 0.2) -- ++(0,1.5);
\draw[ultra thick] (0.8,-1.5) arc (360:180:0.8 and 0.2); % draw inner suture, front
\draw (0.8,0) -- ++(0,-3) arc (360:180:0.8 and 0.2) -- ++(0,3); % draw inner tube
\begin{scope} % outer portion of T, clipped to keep bottom vertex inside the figure
\clip (-2,0) arc (180:0:2 and 0.5) -- ++(0,-3) arc (360:180:2 and 0.5) -- ++(0,3);
\draw[red] (1.4,0) coordinate (tt) -- (-30:2 and 0.5) coordinate (tcorner) -- coordinate[pos=0.7] (tm) ++(0,-3) -- (1.4,-3) coordinate (tb); % draw T
\draw[red,fill=red] (tb) circle (0.05); % draw T bottom endpoint
\end{scope}
\fill[gray!30,fill opacity=0.75] (-2,0) -- ++(0,-1.5) arc (180:360:2 and 0.5) -- ++(0,1.5) arc (360:180:2 and 0.5); % fill outer tube
\fill[white,fill opacity=0.5] (2,-1.5) arc (360:180:2 and 0.5) -- ++(0,-1.5) arc (180:360:2 and 0.5) -- ++(0,1.5);
\draw[fill=gray!30,fill opacity=0.75,even odd rule] (0,0) ellipse (2 and 0.5) ellipse (0.8 and 0.2); % fill top annulus
\draw[ultra thick] (2,-1.5) arc (360:180:2 and 0.5); % draw outer suture, front
\draw[red,fill=red] (tt) circle (0.05);
\draw[red] (tt) -- (tcorner) -- ++(0,-3); % outer portion of T
\path (tm) node[red,left,inner sep=2pt] {\tiny$\alpha$};
\end{scope}
\draw (-2,0) -- ++(0,-3) arc (180:360:2 and 0.5) -- ++(0,3); % draw outside of outer tube
\end{scope}
\end{tikzpicture}
\caption{Left, the product sutured manifold $(M',\gamma')$, together with the arc $\alpha$.  Right, the same manifold with $\alpha$ isotoped into $\partial M'$.}
\label{fig:mprime-product}
\end{figure}
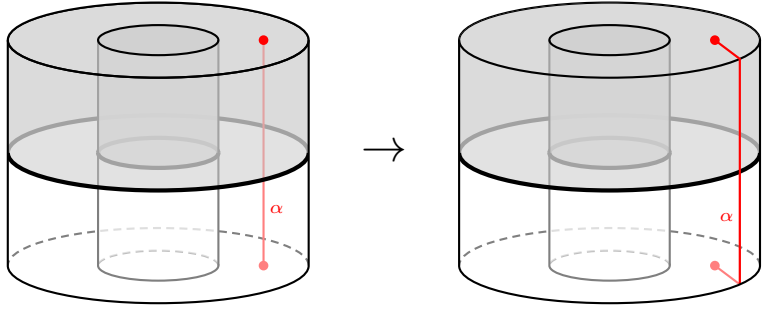

The portion of $\partial M'$ which came from the annulus $B$ (i.e., which was  in the interior of $M_A$) is contained in a tubular neighborhood $N \subset \partial M'$ of one of the two components of $\gamma'$, let us say the component \[(S^1\times\{0\})\times[-1,1].\] Since $\alpha$ is disjoint from $B$, we can isotope this arc  into $\partial M' \setminus N$ while keeping its endpoints fixed, so that it meets the other component \[(S^1\times \{1\})\times[-1,1]\] of $\gamma'$  in an arc  $\{\pt\} \times [-1,1]$, as indicated in Figure \ref{fig:mprime-product}.  Gluing $(M',\gamma')$ and $(M_2,\gamma_2)$ back together to form $(M_A,\gamma_A)$, this gives an isotopy in $M_A$ which fixes the endpoints of $\alpha$ while carrying $\alpha$ to an arc in $\partial M_A$ which meets the sutures $s(\gamma_A)$ in one point.\end{proof}

With Lemma~\ref{lem:isotope-t-into-dNC} in hand, we may now complete the proof of Proposition~\ref{prop:t-in-cabling-annulus}.

\begin{proof}[Proof of Proposition~\ref{prop:t-in-cabling-annulus}]
Viewing $\alpha$ as an arc in $M_A$, Lemma~\ref{lem:isotope-t-into-dNC} says that we can isotope it rel its endpoints to lie in $\partial M_A$, so that it meets the sutures $s(\gamma_A)$ transversely in a single point, as shown on the left side of Figure~\ref{fig:glue-a-back-together}.
\begin{figure}
\begin{tikzpicture}
\begin{scope}[xshift=-7cm]
\draw (7,0.25) -- ++(1,0) arc (90:-90:1 and 0.25) ++(1,0.25) -- ++(0,-3) ++(-1,-0.25) arc (-90:0:1 and 0.25);
\path[fill=gray!30, fill opacity=0.75] (6,0) arc (180:90:1 and 0.25) -- ++(0,-3) arc (90:180:1 and 0.25) -- ++(0,3);
\draw[densely dotted] (6,-3) arc (180:90:1 and 0.25) -- ++(1,0) arc (90:0:1 and 0.25);
\draw (6,0) arc (180:90:1 and 0.25);
\draw[ultra thick] (7,0.25) -- ++(0,-3);
\draw[fill=gray!30, fill opacity=0.75] (6,0) arc (180:270:1 and 0.25) -- ++(1,0) -- ++(0,-3) -- ++(-1,0) arc (270:180:1 and 0.25) -- ++(0,3);
\draw[ultra thick] (8,-0.25) -- ++(0,-3);
% tangle
\path (7,-2) arc (270:240:1 and 0.25) coordinate (tl);
\path (8,-2) arc (270:300:1 and 0.25) coordinate (tr);
\draw[red] (tl) -- node[pos=0.4,below,inner sep=2pt] {\scriptsize$\alpha$} (tr);
\draw[red,fill=red] (tl) circle (0.05) (tr) circle (0.05);
\node at (7.5,-3.75) {$M_A$};
\end{scope}
\node at (3,-1.5) {\LARGE$\hookrightarrow$};
%
% glue A back together
%
\begin{scope}[xshift=4cm]
% cabling annulus
\draw[thin,pattern=north west lines,pattern color=gray!20] (1,-3) rectangle (3,0);
\node[circle,fill=white,inner sep=1pt] at (2,-0.75) {$A$};
% left complement
\path[fill=gray!30, fill opacity=0.5] (0,-3) -- ++(0,3) arc (180:0:0.5 and 0.25) -- ++(0,-3) arc (0:180:0.5 and 0.25);
\draw (0,0) arc (180:0:0.5 and 0.25);
\draw[densely dotted] (0,-3) arc (180:0:0.5 and 0.25);
\draw[fill=gray!30, fill opacity=0.75] (0,0) -- ++(0,-3) arc(180:360:0.5 and 0.25) -- ++(0,3) arc(0:-180:0.5 and 0.25);
% right complement
\draw (3,0) arc (180:0:0.5 and 0.25);
\draw[densely dotted] (3,-3) arc (180:0:0.5 and 0.25);
\draw (3,0) -- ++(0,-3) arc(180:360:0.5 and 0.25) -- ++(0,3) arc(0:-180:0.5 and 0.25);
% tangle
\draw[red,fill=red] (0.5,-2) coordinate (tl) circle (0.05) ++(3,0) coordinate (tr) circle (0.05);
\path (tl) arc (270:360:0.5 and 0.25) coordinate (phi-l);
\path (tr) arc (270:180:0.5 and 0.25) coordinate (phi-r);
\fill[red!25,fill opacity=0.75] (tl) arc (270:360:0.5 and 0.25) -- (phi-r) arc (180:270:0.5 and 0.25) -- (tl);
\draw[red] (tl) to node[pos=0.4,below,inner sep=2pt] {\scriptsize$\alpha$} (tr);
%\draw[red,fill=red] (phi-l) circle (0.04) (phi-r) circle (0.04);
\draw[red,fill=red] (phi-l) circle (0.04) -- node[pos=0.4,above,inner sep=2pt] {\scriptsize$\phi(\alpha)$} (phi-r) circle (0.04);
\draw[red,densely dotted] (tl) arc (270:360:0.5 and 0.25) coordinate (phi-l);
\draw[red,densely dotted] (tr) arc (270:180:0.5 and 0.25) coordinate (phi-r);
\node at (2,-3.75) {$M_F$};
\end{scope}
\end{tikzpicture}
\caption{Viewing $M_A$ as a submanifold of $M_F$, the arc $\alpha \subset \partial M_A$ lies in a push-off of the annulus $A$.  On the right we see the region swept out by the isotopy of $\alpha$ into $A$.}
\label{fig:glue-a-back-together}
\end{figure}
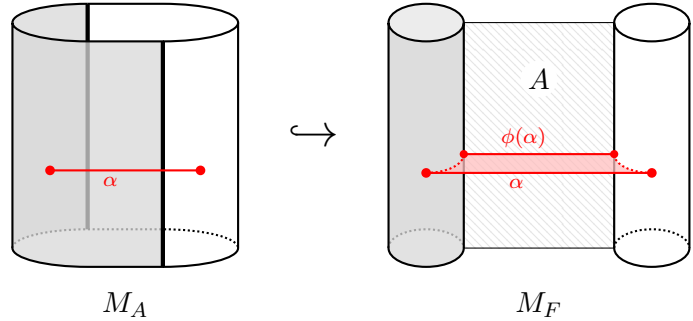
Recall that $M_A$ was formed from $M_F$ by removing the interior of a tubular neighborhood $A \times [-1,1]$, where the original cabling annulus $A$ is identified as $A\times\{0\}$.  We can arrange the interval coordinate so that $\alpha \subset A\times \{1\}$, and then the desired isotopy is simply $\phi_t(x) = (x,1-t)$ for $x\in \alpha$.
\end{proof}

\section{Identifying the manifolds $M_F$ and $S^3(F)$} \label{sec:C}

Let $K \subset S^3$ be nearly fibered, with a genus-1 Seifert surface $F$. According to Proposition \ref{prop:t-in-cabling-annulus}, we can assume that the  arc \[\alpha\subset M_F\] in Definition \ref{def:M_F}, whose complement recovers $S^3(F)$, lies in the annulus \[A\subset M_F\] of Lemma \ref{lem:dehn-fill}; moreover, Remark~\ref{rmk:cabling} says that $A$ is a cabling annulus.  In this section, we use these facts to identify the manifold $M_F$ and hence the sutured Seifert surface complement $S^3(F)$.  Specifically, we prove the following:

\begin{theorem} \label{thm:identify-y-c}
Let $K\subset S^3$ be a nearly fibered knot, with genus-1 Seifert surface $F$.  Then, up to possibly replacing $K$ with its mirror, the manifold $M_F$ is the complement of either:
\begin{enumerate}
\item the $(2,4)$-cable of the unknot in $S^3$, or
\item the $(2,4)$-cable of the right-handed trefoil in $S^3$.
\end{enumerate}
In each case the arc $\alpha$ is a properly embedded arc in the cabling annulus. 
\end{theorem}

\begin{proof}
As defined in \S\ref{sec:M_F}, the manifold $(M_A,\gamma_A)$ is obtained from $(M_F,\gamma_F)$ by decomposing along the cabling annulus $A$ provided in Lemma \ref{lem:dehn-fill}, \[ (M_F,\gamma_F) \stackrel{A}{\leadsto} (M_A,\gamma_A). \]Recall from Definition \ref{def:M_F} that  \[(M,\gamma) = S^3(F)\] can be recovered from $M_F$ by removing a neighborhood $N(\alpha)$ of the arc $\alpha$, where the suture $s(\gamma)$ is identified with a meridian $\mu_\alpha$ of $\alpha$. 
 By Proposition \ref{prop:t-in-cabling-annulus}, we can  assume that $\alpha\subset A$. Therefore, when we remove a neighborhood of $\alpha$ from $M_F$ to form $S^3(F)$, what remains of the cabling annulus $A$ is a product disk $D\subset S^3(F)$. Thus, $(M_A,\gamma_A)$ can alternatively be obtained via the product disk decomposition  \[ S^3(F) \stackrel{D}{\leadsto} (M_A,\gamma_A), \] as indicated in Figure \ref{fig:product-disk}.

\begin{figure}
\begin{tikzpicture}
\begin{scope} % M_F
% left complement
\path[fill=gray!30, fill opacity=0.5] (0,-3) -- ++(0,3) arc (180:0:0.5 and 0.25) -- ++(0,-3) arc (0:180:0.5 and 0.25);
\draw (0,0) arc (180:0:0.5 and 0.25);
\draw[densely dotted] (0,-3) arc (180:0:0.5 and 0.25);
\draw[fill=gray!30, fill opacity=0.75] (0,0) -- ++(0,-3) arc(180:360:0.5 and 0.25) -- ++(0,3) arc(0:-180:0.5 and 0.25);
% right complement
\draw (3,0) arc (180:0:0.5 and 0.25);
\draw[densely dotted] (3,-3) arc (180:0:0.5 and 0.25);
\draw (3,0) -- ++(0,-3) arc(180:360:0.5 and 0.25) -- ++(0,3) arc(0:-180:0.5 and 0.25);
% cabling annulus
\path (0.5,0) ++ (-45:0.5 and 0.25) coordinate (rect-nw) -- ++(0,-1.35) coordinate (alpha-l);
\path (3.5,-3) ++ (225:0.5 and 0.25) coordinate (rect-se) -- ++(0,1.65) coordinate (alpha-r);
\draw[thin,pattern=north west lines,pattern color=gray!20] (rect-nw) rectangle (rect-se);
\node[circle,fill=white,inner sep=1pt] at (2,-0.75) {$A$};
\draw[red,fill=red] (alpha-l) circle (0.05) -- node[pos=0.5,above,inner sep=2pt] {\scriptsize$\alpha$} (alpha-r) circle (0.05);
\node at (2,-3.75) {$(M_F,\gamma_F)$};
\end{scope}
\node at (4.625,-1.25) {\huge$\stackrel{\vphantom{D}}{\leadsto}$};
\begin{scope}[xshift=5.25cm] % S^3(F)
% left complement
\path[fill=gray!30, fill opacity=0.5] (0,-3) -- ++(0,3) arc (180:0:0.5 and 0.25) -- ++(0,-1.05) to[out=270,in=180] ++(0.2,-0.2) -- ++(0.8,0) arc(90:270:0.1 and 0.25) -- ++(-0.8,0) to[out=180,in=90] ++(-0.2,-0.2) -- ++(0,-1.05) arc (0:180:0.5 and 0.25);
\draw (0,0) arc (180:0:0.5 and 0.25);
\draw[densely dotted] (0,-3) arc (180:0:0.5 and 0.25);
\path (0,0) -- ++(0,-3) arc(180:360:0.5 and 0.25) -- ++(0,1.05) to[out=90,in=180] ++(0.2,0.2) -- ++(0.8,0) coordinate (suture-bottom);
\begin{scope} % suture back
\clip (1,-1.25) rectangle (3,-1.75);
\draw[ultra thick,densely dotted] (suture-bottom) ++(0.02,-0.02) arc (270:90:0.1 and 0.27);
\end{scope}
\path[fill=gray!30, fill opacity=0.75] (0,0) -- ++(0,-3) arc(180:360:0.5 and 0.25) -- ++(0,1.05) to[out=90,in=180] ++(0.2,0.2) -- ++(0.8,0) arc (-90:90:0.1 and 0.25) -- ++(-0.8,0) to[out=180,in=270] ++(-0.2,0.2) -- ++(0,1.05) arc(0:-180:0.5 and 0.25);
\draw (0,0) -- ++(0,-3) arc(180:360:0.5 and 0.25) -- ++(0,1.05) to[out=90,in=180] ++(0.2,0.2) -- ++(0.8,0) arc (-90:90:0.1 and 0.25) -- ++(-0.8,0) to[out=180,in=270] ++(-0.2,0.2) -- ++(0,1.05) arc(0:-180:0.5 and 0.25);
\begin{scope} % suture front
\clip (1,-1.25) rectangle (3,-1.75);
\draw[ultra thick] (suture-bottom) ++(0,-0.02) arc (-90:90:0.1 and 0.27);
\end{scope}
% right complement
\draw (3,0) arc (180:0:0.5 and 0.25);
\draw[densely dotted] (3,-3) arc (180:0:0.5 and 0.25);
\draw (3,0) -- ++(0,-1.05) to[out=270,in=0] ++(-0.2,-0.2) -- ++(-1.3,0) ++(0,-0.5) -- ++(1.3,0) to[out=0,in=90] ++(0.2,-0.2) -- ++(0,-1.05) arc(180:360:0.5 and 0.25) -- ++(0,3) arc(0:-180:0.5 and 0.25);
% product disk
\path (0.5,0) ++ (-45:0.5 and 0.25) coordinate (rect-nw) ++(0,-3) coordinate (rect-sw);
\path (3.5,-3) ++ (225:0.5 and 0.25) coordinate (rect-se) ++(0,3) coordinate (rect-ne) ++(-0.2,-1.2) coordinate (d-corner-1);
\path (rect-se) ++(-0.2,1.55) coordinate (d-corner-2);
%\draw[thin,pattern=north west lines,pattern color=gray!20] (rect-nw) rectangle (rect-se);
\draw[thin,pattern=north west lines,pattern color=gray!20] (rect-nw) -- ++(0,-1) to[out=270,in=180] ++(0.2,-0.2) -- (d-corner-1) to[out=0,in=270] ++(0.2,0.2) -- (rect-ne) -- cycle;
\draw[thin,pattern=north west lines,pattern color=gray!20] (rect-sw) -- ++(0,1.35) to[out=90,in=180] ++(0.2,0.2) -- (d-corner-2) to[out=0,in=90] ++(0.2,-0.2) -- (rect-se) -- cycle;
\node[circle,fill=white,inner sep=1pt] at (2,-0.75) {$D$};
%\draw[red,fill=red] (alpha-l) circle (0.05) -- node[pos=0.5,above,inner sep=2pt] {\scriptsize$\alpha$} (alpha-r) circle (0.05);
\node at (2,-3.75) {$S^3(F)$};
\end{scope}
\node at (9.875,-1.25) {\huge$\stackrel{D}{\leadsto}$};
\begin{scope}[xshift=10.5cm] % M_A
\draw (1,0.25) -- ++(1,0) arc (90:-90:1 and 0.25) ++(1,0.25) -- ++(0,-3) ++(-1,-0.25) arc (-90:0:1 and 0.25);
\path[fill=gray!30, fill opacity=0.75] (0,0) arc (180:90:1 and 0.25) -- ++(0,-3) arc (90:180:1 and 0.25) -- ++(0,3);
\draw[densely dotted] (0,-3) arc (180:90:1 and 0.25) -- ++(1,0) arc (90:0:1 and 0.25);
\draw (0,0) arc (180:90:1 and 0.25);
\draw[ultra thick] (1,0.25) -- ++(0,-3);
\draw[fill=gray!30, fill opacity=0.75] (0,0) arc (180:270:1 and 0.25) -- ++(1,0) -- ++(0,-3) -- ++(-1,0) arc (270:180:1 and 0.25) -- ++(0,3);
\draw[ultra thick] (2,-0.25) -- ++(0,-3);
\node at (1.5,-3.75) {$(M_A,\gamma_A)$};
\end{scope}

\end{tikzpicture}
\caption{A schematic picture  which shows that decomposing $(M_F,\gamma_F)$ along the cabling annulus $A$ is the same as first removing a neighborhood of $\alpha\subset A$ and then decomposing along the product disk $D$.}
\label{fig:product-disk}
\end{figure}
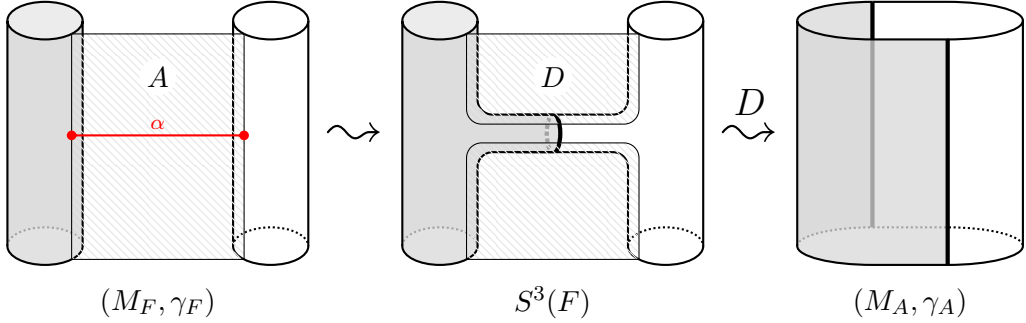

This shows in particular that $M_A$ is a subset of $S^3$, as \[M_A\subset S^3(F)\subset S^3.\] Since $M_A$ has torus boundary, it follows that $M_A$ can be identified with the complement of a knot $C\subset S^3$, and moreover that we have an identification of sutured manifolds, \[(M_A,\gamma_A)\cong (S^3\setminus N(C),\gamma_r),\] where the sutures $s(\gamma_r)$ are a union of two parallel oppositely oriented curves of slope $r$, with respect to the Seifert framing of $C$. Furthermore, we have \begin{equation}\label{eqn:dim-0} \SFH(S^3\setminus N(C),\gamma_r) \cong \SFH(S^3(F))\cong\Q^2, \end{equation} by Theorem~\ref{thm:sfh-disk-decomposition}. It remains to determine the slope $r$ and the knot $C$.

Suppose first that $r=0$. Then $C$ is  the unknot, because otherwise we would have \[\dim \SFH(S^3\setminus N(C),\gamma_0)\geq 4,\] by Proposition~\ref{prop:longitude-fh}, contradicting \eqref{eqn:dim-0}.  But then $M_F$ is the complement of the $(2,0)$-cable of the unknot in $S^3$, which contradicts the fact in Lemma \ref{lem:MF-inc} that $M_F$ is  irreducible. 

The above argument shows that  $r\neq 0$.
Note that we can identify $(S^3\setminus N(C),\gamma_r)$ as the sutured complement of the core $C' \subset S^3_r(C)$ of $r$-surgery on $C$, whose sutures are a union of two oppositely oriented meridians of $C'$.  With this in mind, equation \eqref{eq:sfh-hfk} becomes
\begin{align*}
\hfkhat(S^3_r(C), C') &\cong \SFH(S^3_r(C)(C')) \\
&\cong \SFH(S^3 \setminus N(C), \gamma_r) \cong \Q^2.
\end{align*}
Since $r\neq 0$, the core $C'$ is rationally nullhomologous in $S^3_r(C)$. It follows that there is a spectral sequence
\[ \Q^2 \cong \hfkhat(S^3_r(C),C') \ \Longrightarrow\ \hfhat(S^3_r(C)) \]
leading to the chain of inequalities
\begin{align*}
1 \leq |H_1(S^3_r(C);\Z)| &\leq \dim \hfhat(S^3_r(C)) \\
&\leq \dim \hfkhat(S^3_r(C),C') = 2.
\end{align*}
We conclude that  \begin{equation}\label{eqn:dim-2}\dim\hfhat(S^3_r(C))=2,\end{equation} as this dimension has the same parity as \[\dim\hfkhat(S^3_r(C), C')=2.\] It also has the same parity as
\[ \chi\big( \hfhat(S^3_r(C)) \big) = |H_1(S^3_r(C);\Z)|, \]
which then implies that  \begin{equation}\label{eqn:euler}|H_1(S^3_r(C);\Z)| = 2. \end{equation} Combining \eqref{eqn:dim-2} and \eqref{eqn:euler}, we have shown that  $S^3_r(C)$ is an L-space. Moreover,  if $r=p/q$ with $q\geq0$ and $\gcd(p,q)=1$ then $|p|=2$.

We now recall from \cite[Proposition~9.6]{osz-rational} (see \cite[\S2]{hom-cabling} for details) that if $C \subset S^3$ is a nontrivial knot, then $r$-surgery on $C$ can only be an L-space if  \[|r| \geq 2g(C)-1.\] Moreover, if we also have that $r>0$, then $C$ must additionally be fibered \cite{ghiggini,ni-hfk} and strongly quasipositive \cite{hedden-positivity}.  Note that when $C$ is knotted, we have that \[0 < |r| < 1 \leq 2g(C)-1\]  for slopes $r=\pm2/q$ unless $q=1$, so there are three cases to consider:
\begin{enumerate}
\item $C$ is an unknot and $r=2/q$ for some odd $q\in\Z$. \label{i:yc-case1}

\item $C$ is knotted and $r=2$.  Then $S^3_2(C)$ is an L-space, so $g(C) = 1.$ Then $C$ must be the right-handed trefoil since this is the only genus-1, fibered, strongly quasipositive knot in the 3-sphere.  \label{i:yc-case2}

\item $C$ is knotted and $r=-2$.  Then $S^3_{-2}(C)$ is an L-space, so again $g(C) \leq 1$.  But now $C$ must be the \emph{left}-handed trefoil, since its mirror $\mirror{C}$ admits a positive L-space surgery and is therefore the right-handed trefoil, as discussed above. \label{i:yc-case3}
\end{enumerate}

In case \eqref{i:yc-case1} it follows that
\[ (M_A,\gamma_A) \cong (S^3 \setminus N(U), \gamma_2), \]
since any two choices of $\gamma_{2/q}$ are related by a homeomorphism of the solid torus $S^3 \setminus N(U)$.  We conclude that 
\[ M_F \cong S^3 \setminus N(C_{2,4}(U)) \cong S^3 \setminus N(T_{2,4}). \]
Similarly, in case \eqref{i:yc-case2} we have that \[ (M_A,\gamma_A) \cong (S^3 \setminus N(T_{2,3}), \gamma_2), \] and  therefore conclude that \[ M_F \cong S^3 \setminus N(C_{2,4}(T_{2,3})). \]
This leaves only case \eqref{i:yc-case3}, in which \[ (M_A,\gamma_A) \cong (S^3 \setminus N(T_{-2,3}), \gamma_{-2}). \] Then we have 
\begin{align*}
M_F &\cong S^3 \setminus N(C_{2,-4}(T_{-2,3})) \\
&\cong -\left( S^3 \setminus N(C_{2,4}(T_{2,3})) \right).
\end{align*}
But in this case we can replace $K$ with its mirror $\mirror{K}$, and doing so replaces $M_F$ with $-M_F$, so again case~\eqref{i:yc-case2} applies here and we are done.
\end{proof}

\section{The $(2,4)$-cable of the unknot} \label{sec:unknot-24}

In this lengthy section, we determine all knots $K\subset S^3$ which arise from the first case of Theorem~\ref{thm:identify-y-c}, in which $M_F$ is the complement of the $(2,4)$-cable of the unknot. Our goal is to prove the following:

\begin{theorem} \label{thm:M_F-E24}
Let $K\subset S^3$ be a nearly fibered knot with genus-1 Seifert surface $F$, and suppose that \[M_F \cong S^3 \setminus N(T_{2,4}).\]  Then $K$ is one of the knots
\[ 5_2, \ 15n_{43522}, \text{ or } P(-3,3,2n+1)\ (n\in\Z) \]
or their mirrors.
\end{theorem}

The key observation is that under the hypotheses of Theorem \ref{thm:M_F-E24}, $M_F$ admits an involution which is rotation by $180^\circ$ about an axis of symmetry containing the arc $\alpha \subset M_F$.  This then gives rise to an involution $\iota$ of the sutured Seifert surface complement \[(M,\gamma) = S^3(F)\] obtained by removing a neighborhood of $\alpha$ from $M_F$, where $s(\gamma)$  is identified with a meridian $\mu_\alpha$ of  $\alpha$. This involution is depicted on the left side of Figure~\ref{fig:E-T24-quotient},
while the right side illustrates the quotient \[S^3(F)/\iota,\] which is a sutured 3-ball with connected suture. As suggested by the figure, it is natural to identify this quotient 3-ball  with the complement of a thickened disk in $S^3$, \[S^3(F)/\iota \cong S^3(D^2) = (S^3\setminus \inr(D^2\times[-1,1]), \partial D^2 \times [-1,1]),\] and the quotient map realizes $S^3(F)$ as the branched double cover of this ball along a  tangle $\tau\subset S^3(D^2)$, as shown in  Figure \ref{fig:E-T24-quotient}.

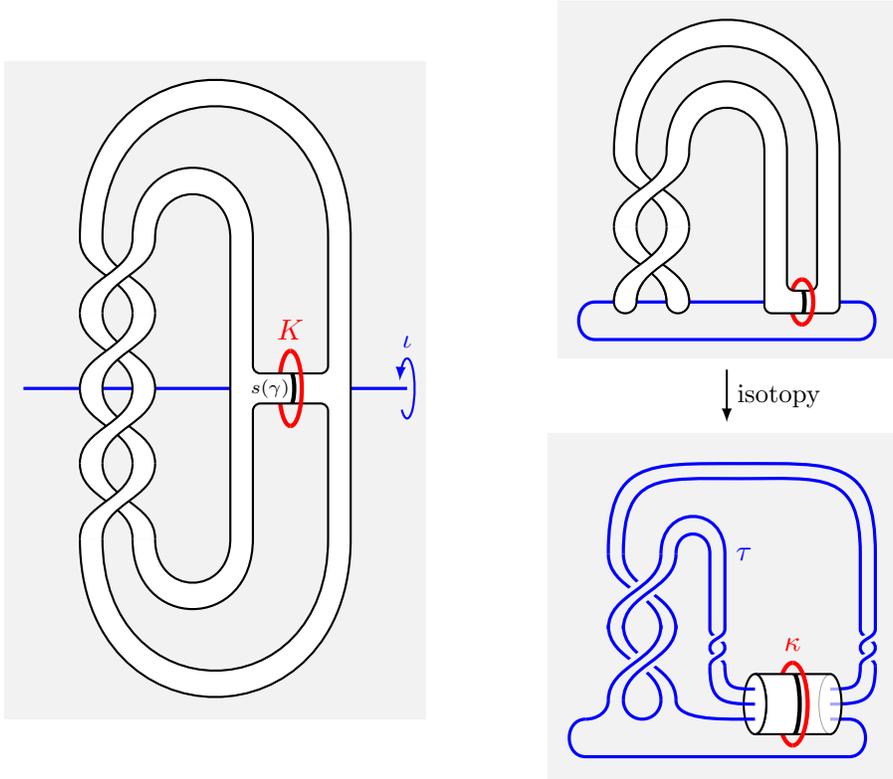
\begin{figure}
\begin{tikzpicture}
% Draw S^3(F)
\begin{scope}[yshift=-0.8cm]
\path[fill=gray!10] (-1,-2.4) rectangle (4.6,6.35);
\draw[blue,very thick] (-0.75,2) -- (4.35,2); % axis of involution
\draw[blue,-latex] (4.35,2) ++ (-135:0.1 and 0.4) arc (-135:165:0.1 and 0.4);
\path[blue] (4.35,2) ++ (0,0.4) node[above] {\small$\iota$};
% predraw some of the twist region to fix the colors
\begin{scope}
\path[fill=white] (0.7,3) to[out=270,in=90] (0,2) to[out=270,in=90] (0.7,1) -- (1,1) to[out=90,in=270] (0.3,2) to[out=90,in=270] (1,3) -- cycle;
\path[fill=white] (0,3) to[out=270,in=90] (0.7,2) to[out=270,in=90] (0,1) -- (0.3,1) to[out=90,in=270] (1,2) to[out=90,in=270] (0.3,3) -- cycle;
\end{scope}
\foreach \y in {4,3,2,1} { % twist region
\begin{scope}
\path[fill=white] (0,\y) to[out=270,in=90] ++(0.7,-1) -- ++(0.3,0) to[out=90,in=270] ++(-0.7,1) -- cycle;
\draw (0,\y) to[out=270,in=90] ++(0.7,-1) ++(0.3,0) to[out=90,in=270] ++(-0.7,1);
\path[fill=white] (1,\y) to[out=270,in=90] ++(-0.7,-1) -- ++(-0.3,0) to[out=90,in=270] ++(0.7,1) -- cycle;
\draw (1,\y) to[out=270,in=90] ++(-0.7,-1) ++ (-0.3,0) to[out=90,in=270] ++(0.7,1);
\end{scope}
}
\draw[red,ultra thick] (2.8,2.5) arc (90:270:0.15 and 0.5); % back of K
\begin{scope}[every path/.append style={looseness=2}] % rest of exterior
\path[fill=white] (1,0) to[out=270,in=270] ++(1,0) -- ++(0,4) to[out=90,in=90] ++(-1,0) -- ++(-0.3,0) to[out=90,in=90] ++(1.6,0) -- ++(0,-4) to[out=270,in=270] ++(-1.6,0) -- cycle;
\path[fill=white] (0.3,0) to[out=270,in=270] ++(3,0) -- ++(0,4) to[out=90,in=90] ++(-3,0) -- ++(-0.3,0) to[out=90,in=90] ++(3.6,0) -- ++(0,-4) to[out=270,in=270] ++(-3.6,0) -- cycle;
\path[fill=white] (2,1.8) rectangle (3.6,2.2);
\path[fill=white] (2.4,2.2) arc (270:180:0.1) -- ++(0,-0.1) -- ++(0.1,0); % NW corner of N(alpha)
\path[fill=white] (3.2,2.2) arc (-90:0:0.1) -- ++(0,-0.1) -- ++(-0.1,0); % NE corner of N(alpha)
\path[fill=white] (2.4,1.8) arc (90:180:0.1) -- ++(0,0.1) -- ++(0.1,0); % SW corner of N(alpha)
\path[fill=white] (3.2,1.8) arc (90:0:0.1) -- ++(0,0.1) -- ++(-0.1,0); % SE corner of N(alpha)
\draw (1,0) to[out=270,in=270] ++(1,0) -- ++(0,4) to[out=90,in=90] ++(-1,0);
\draw (0,0) to[out=270,in=270] ++(3.6,0) -- ++(0,4) to[out=90,in=90] ++(-3.6,0);
\draw (0.3,0) to[out=270,in=270] ++(3,0) -- ++(0,1.7) arc (0:90:0.1) -- ++(-0.8,0) arc (90:180:0.1) -- ++(0,-1.7) to[out=270,in=270] ++(-1.6,0);
\draw (0.3,4) to[out=90,in=90] ++(3,0) -- ++(0,-1.7) arc (0:-90:0.1) -- ++(-0.8,0) arc (270:180:0.1) -- ++(0,1.7) to[out=90,in=90] ++(-1.6,0);
\end{scope}
\begin{scope}
\clip (2.2,1.8) rectangle (3,2.2);
\draw[ultra thick] (2.8,2) ++(0,0.22) arc (90:-90:0.05 and 0.22) node[midway,left,inner sep=1pt] {\tiny$s(\gamma)$};
\end{scope}
\draw[red,ultra thick] (2.8,2.5) node[above] {$K$} arc (90:-90:0.15 and 0.5);
\end{scope}

% Quotient by the involution
\begin{scope}[xshift=7.1cm, yshift=0.35cm]
\path[fill=gray!10] (-0.75,1.25) rectangle (3.75,6);
\draw[red,ultra thick] (2.5,2) ++(0,0.3) arc (90:270:0.15 and 0.3);
\draw[blue,very thick,looseness=1.5] (-0.25,2) -- (3.25,2) to[out=0,in=0] ++(0,-0.5) -- ++(-3.5,0) to[out=180,in=180] ++(0,0.5); % branch locus
%\draw[blue,-latex] (4.35,2) ++ (-135:0.1 and 0.4) arc (-135:165:0.1 and 0.4);
\foreach \y in {4,3} { % twist region
\begin{scope}
\path[fill=white] (0,\y) to[out=270,in=90] ++(0.7,-1) -- ++(0.3,0) to[out=90,in=270] ++(-0.7,1) -- cycle;
\draw (0,\y) to[out=270,in=90] ++(0.7,-1) ++(0.3,0) to[out=90,in=270] ++(-0.7,1);
\path[fill=white] (1,\y) to[out=270,in=90] ++(-0.7,-1) -- ++(-0.3,0) to[out=90,in=270] ++(0.7,1) -- cycle;
\draw (1,\y) to[out=270,in=90] ++(-0.7,-1) ++ (-0.3,0) to[out=90,in=270] ++(0.7,1);
\end{scope}
}
\begin{scope} % add in bottom of twists
\clip (-0.25,1.75) rectangle (1.25,2.25);
\path[fill=white] (0,3) to[out=270,in=90] ++(0.7,-1) arc (180:360:0.15) to[out=90,in=270] ++(-0.7,1);
\path[fill=white] (0.7,3) to[out=270,in=90] ++(-0.7,-1) arc (180:360:0.15) to[out=90,in=270] ++(0.7,1);
\draw (0,3) to[out=270,in=90] ++(0.7,-1) arc (180:360:0.15) to[out=90,in=270] ++(-0.7,1);
\draw (0.7,3) to[out=270,in=90] ++(-0.7,-1) arc (180:360:0.15) to[out=90,in=270] ++(0.7,1);
\end{scope}
\begin{scope}[every path/.append style={looseness=2}] % rest of exterior
\path[fill=white] (0,4) to[out=90,in=90] ++(3,0) -- ++(0,-2) -- ++(0,-0.05) arc (0:-90:0.1) -- ++(-0.8,0) arc (270:180:0.1) -- ++(0,0.05) -- ++(0,2) to[out=90,in=90] ++(-1,0) -- ++(-0.3,0) to[out=90,in=90] ++(1.6,0) -- ++(0,-1.75) arc (180:270:0.1) -- ++(0.2,0) arc (-90:0:0.1) -- ++(0,1.75) to[out=90,in=90] ++(-2.4,0) -- cycle;
\draw (0,4) to[out=90,in=90] ++(3,0) -- ++(0,-2) -- ++(0,-0.05) arc (0:-90:0.1) -- ++(-0.8,0) arc (270:180:0.1) -- ++(0,0.05) -- ++(0,2) to[out=90,in=90] ++(-1,0);
\draw (0.7,4) to[out=90,in=90] ++(1.6,0) -- ++(0,-1.75) arc (180:270:0.1) -- ++(0.2,0) arc (-90:0:0.1) -- ++(0,1.75) to[out=90,in=90] ++(-2.4,0);
\draw[red,ultra thick] (2.5,2) ++(0,0.3) arc (90:-90:0.15 and 0.3);
\end{scope}
\begin{scope} % draw band corresponding to image of suture
\clip (2.2,1.85) rectangle (2.8,2.15);
\draw[ultra thick] (2.5,2) ++(0,0.17) arc (90:-90:0.035 and 0.17); % node[midway,left,inner sep=1pt] {\tiny$\gamma$};
\end{scope}
\end{scope}

\draw[-latex] (8.6,1.45) -- node[right] {\small{isotopy}} (8.6,0.75);

% Isotope a bit
\begin{scope}[xshift=6.975cm, yshift=-5cm]
\path[fill=gray!10] (-0.75,1) rectangle (4,5.6);
\draw[red,ultra thick] (2.5,2) ++(0,0.55) arc (90:270:0.2 and 0.55);
% start filling in complement
\path[fill=white] (2,2.4) -- ++(1,0) arc (90:-90:0.15 and 0.4) -- ++(-1,0) arc (-90:90:0.15 and 0.4);
% parts of branch locus near right end of complement
\draw[blue,very thick,looseness=1.5] (-0.25,2) ++(0,-0.2) to[out=180,in=180] ++(0,-0.5) -- ++(3.5,0) to[out=0,in=0] ++(0,0.5) -- ++(-0.25,0);
\draw[blue,very thick,looseness=1.5] (3,2) to[out=0,in=270] (3.6,2.5);
\draw[blue,very thick,looseness=1.5] (3,2.2) to[out=0,in=270] (3.4,2.5);
\begin{scope} % draw complement, after isotopy
\draw[thin] (3,2.4) arc (90:270:0.15 and 0.4);
\draw[fill=white,fill opacity=0.66] (2,2.4) -- ++(1,0) arc (90:-90:0.15 and 0.4) -- ++(-1,0) arc (-90:90:0.15 and 0.4);
\draw[fill=white] (2,2) ellipse (0.15 and 0.4);
\begin{scope} % draw image of suture
\clip (2.4,1.6) rectangle (2.8,2.4);
\draw[ultra thick] (2.5,2.42) arc (90:-90:0.095 and 0.42);% node[midway,left,inner sep=1pt] {\tiny$\gamma$};
\end{scope}
\end{scope}
\draw[red,ultra thick] (2.5,2) ++(0,0.55) node[above] {$\kappa$} arc (90:-90:0.2 and 0.55);
% rest of the branch locus
\begin{scope}[every path/.append style={blue,very thick}]
\draw (-0.25,1.8) to[out=0,in=270] (0.05,2.05);
\draw (0.25,2.05) to[out=270,in=180] (0.5,1.8) to[out=0,in=270] (0.75,2.05);
\draw (0.95,2.05) to[out=270,in=180,looseness=0.75] (1.8,1.8) -- (2,1.8);
\foreach \y in {2.05,3} { % draw twists
  \foreach \x in {0.75,0.95} { \draw(\x,\y) to [out=90,in=270] ++(-0.7,1); }
  \path[fill=gray!10] (0,\y) ++(-0.1,0.2) -- ++(1,0.8) -- ++(0.45,0) -- ++(-1,-0.8) -- ++(-0.45,0);
  \foreach \x in {0.05,0.25} { \draw (\x,\y) to[out=90,in=270] ++(0.7,1); }
}
\path[blue] (1.6,4) node[right] {$\tau$};
\begin{scope}[every path/.append style={looseness=1.5}] % rest of exterior
\draw (0.25,4) to[out=90,in=180] (1.8,5) to[out=0,in=90] (3.4,4) -- (3.4,3) to[out=270,in=90,looseness=1] (3.6,2.75) ++(-0.2,0) to[out=270,in=90,looseness=1] (3.6,2.5);
\foreach \x/\y in {3.5/3,3.5/2.75} {
  \path[fill=gray!10] (\x,\y) ++(0.15,-0.02) -- ++(-0.3,-0.1) -- ++(0,-0.12) -- ++(0.3,0.1) -- cycle;
}
\draw (0.05,4) to[out=90,in=180] (1.8,5.2) to[out=0,in=90] (3.6,4) -- (3.6,3) to[out=270,in=90,looseness=1] (3.4,2.75) ++(0.2,0) to[out=270,in=90,looseness=1] (3.4,2.5);
\draw (0.95,4) to[out=90,in=90,looseness=2] (1.4,4) -- (1.4,3) to[out=270,in=90,looseness=1] ++(0.2,-0.25) ++(-0.2,0) to[out=270,in=90,looseness=1] ++(0.2,-0.25);
\foreach \x/\y in {1.5/3,1.5/2.75} {
  \path[fill=gray!10] (\x,\y) ++(0.15,-0.02) -- ++(-0.3,-0.1) -- ++(0,-0.12) -- ++(0.3,0.1) -- cycle;
}
\draw (0.75,4) to[out=90,in=90,looseness=2] (1.6,4) -- (1.6,3) to[out=270,in=90,looseness=1] ++(-0.2,-0.25) ++(0.2,0) to[out=270,in=90,looseness=1] ++(-0.2,-0.25);
\draw (1.6,2.5) to[out=270,in=180] (2,2.2);
\draw (1.4,2.5) to[out=270,in=180] (2,2);
\end{scope}
\end{scope}
\end{scope}
\end{tikzpicture}
\caption{Taking the quotient of $S^3(F) \cong M_F\setminus N(\alpha)$  by an involution $\iota$ in the case where $M_F \cong S^3 \setminus N(T_{2,4})$.  On the left, $S^3(F)$ is the complement in $S^3$ of the white region, the involution is rotation by $180^\circ$ about the horizontal axis (in blue), and the meridian of $\alpha$ (in red) is isotopic in $S^3(F)$ to a pushoff of $K$.  The quotient (right) is a 3-ball, viewed as the complement in $S^3$ of the white region; when we isotope this white region to become a standard $D^2 \times [-1,1]$, the branch locus is carried along to become the tangle $\tau$.}
\label{fig:E-T24-quotient}
\end{figure}

Now, under the identification \[\gamma =\partial F\times[-1,1],\] we can assume that $\iota$ restricts on each $\partial F\times\{t\}\subset \gamma$ to a  rotation of $\partial F$ which is independent of $t$.  Recall that $S^3$ is recovered by gluing $F\times[-1,1]$ back into \[S^3(F) = S^3\setminus \inr(F\times[-1,1])\] by a map which in particular identifies \[\partial F\times[-1,1]\cong \gamma\] via the identity. Any such gluing map \[\varphi:\partial (F\times [-1,1])\to \partial S^3(F)\] is then determined by its restrictions to the once-punctured tori
\begin{align*}
\varphi_+&:F\times\{+1\}\to R_+(\gamma),\\
\varphi_-&:F\times\{-1\}\to R_-(\gamma).
\end{align*}
Note that $\iota$ restricts to a hyperelliptic involution on each of the once-punctured tori \[R_\pm(\gamma)\subset \partial S^3(F).\] Pulling back the involution $\iota$ via the maps $\varphi_\pm$ then induces hyperelliptic involutions \[\iota_\pm:F\times\{\pm 1\}\to F\times\{\pm 1\}\] which agree on the boundary under the canonical identification of these two surfaces. Since  once-punctured tori admit  unique hyperelliptic involutions up to isotopy, we can  extend $\iota_\pm$ to all of $F\times[-1,1]$ by a map restricting to a hyperelliptic involution on each $F\times\{t\}$.

In summary, we have shown that $\iota$ extends to an involution $\hat\iota$ of the glued manifold \[Y_\varphi = S^3(F)\cup_\varphi (F\times [-1,1]),\]  whose restriction to the piece $F\times [-1,1]$ is a hyperelliptic involution on each $F\times\{t\}$. The  quotient map \[Y_\varphi\to Y_\varphi/\hat{\iota}\] therefore restricts on this piece to a branched double covering \[F\times[-1,1]\to D^2\times[-1,1]\] along some 3-braid \[\beta \subset D^2\times [-1,1].\] It follows that $Y_\varphi$ is the branched double cover of \[S^3(D^2)\cup (D^2\times[-1,1]) \cong S^3\] along the link $\tau\cup \beta$. Moreover, $K$ is the lift $\tilde\kappa$ of the braid axis \[\kappa = \partial D^2 \times \{0\} = s(\gamma)/\hat{\iota}\] of $\beta$ in this double cover. Since  $Y_\varphi\cong S^3$ if and only if $\tau\cup \beta$ is an unknot, we  conclude the following:

\begin{lemma} \label{lem:E24-as-branched-cover}
Suppose that $K\subset S^3$ is a nearly fibered knot with genus-1 Seifert surface $F$, and that \[M_F \cong S^3 \setminus N(T_{2,4}).\]  Then there is a 3-braid $\beta \in B_3$ such that $\tau\cup\beta$ is an unknot in $S^3$, and such that the lift
\[ \tilde\kappa \subset \dcover(\tau\cup\beta) \cong S^3 \]
of $\kappa$ is isotopic to $K$.
\end{lemma}

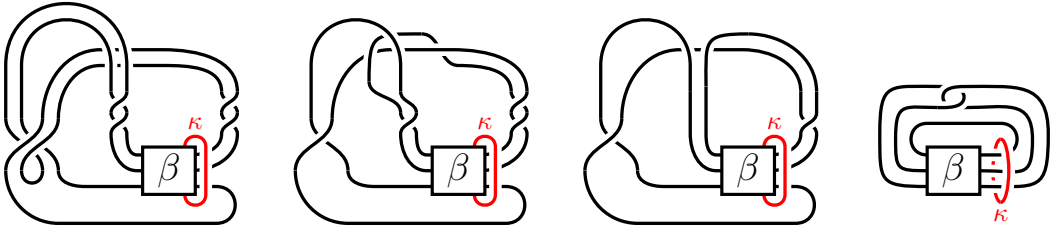
\begin{figure}
\begin{tikzpicture}[scale=0.7]
\begin{scope} % start of the isotopy
\draw[link,looseness=1] (1,0) to[out=90,in=270] ++(-0.7,1) ++(-0.3,0) to[out=270,in=90] ++(0.7,-1);
\draw[link,looseness=1] (0,0) to[out=90,in=270] ++(0.7,1) ++(0.3,0) to[out=270,in=90] ++(-0.7,-1);
\draw[link,looseness=2] (0.3,0) to[out=270,in=270] (0.7,0);
\draw[link,looseness=1] (0.7,1) to[out=90,in=180] (1.8,2.3) -- (3,2.3) to[out=0,in=90] (4.4,1.5);
\draw[link,looseness=1] (1,1) to[out=90,in=180] (1.8,2.0) -- (3,2.0) to[out=0,in=90] (4.1,1.5);
\draw[link,looseness=1] (4.1,1.5) to[out=270,in=90] ++(0.3,-0.4) ++(-0.3,0) to[out=270,in=90] ++(0.3,-0.4) to[out=270,in=0] (3.5,0);
\draw[link,looseness=1] (4.4,1.5) to[out=270,in=90] ++(-0.3,-0.4) ++(0.3,0) to[out=270,in=90] ++(-0.3,-0.4) to[out=270,in=0] (3.5,0.3);
% twists leading into left edge of braid
\draw[link] (0,1) -- ++(0,1) to[out=90,in=90] (2.3,2) -- ++(0,-0.5);
\draw[link] (0.3,1) -- ++(0,1) to[out=90,in=90] (2,2) -- ++(0,-0.5);
\draw[link,looseness=1] (2,1.5) to[out=270,in=90] ++(0.3,-0.4) ++(-0.3,0) to[out=270,in=90] ++(0.3,-0.4) to[out=270,in=180] (2.5,0.3) -- ++(1,0);
\draw[link,looseness=1] (2.3,1.5) to[out=270,in=90] ++(-0.3,-0.4) ++(0.3,0) to[out=270,in=90] ++(-0.3,-0.4) to[out=270,in=180] (2.5,0.0) -- ++(1,0);
\draw[link,looseness=1] (1,0) to[out=270,in=180] (1.5,-0.3) -- (4,-0.3) to[out=0,in=0,looseness=1.75] (4,-1) -- (1,-1) to[out=180,in=270] (0,0);
% box for the braid
\draw[linkred] (3.4,0.45) to[out=90,in=90] node[midway,above,inner sep=2pt,red] {\small$\kappa$} ++(0.4,0) -- ++(0,-0.9) to[out=270,in=270] ++(-0.4,0);
\draw[very thick,fill=white] (2.6,-0.45) rectangle (3.6,0.45);
\node at (3.1,0) {\Large$\beta$};
\end{scope}

\begin{scope}[xshift=5.5cm] % step 2
\draw[link,looseness=1] (1,0) to[out=90,in=270] ++(-0.7,1);% ++(-0.3,0) to[out=270,in=90] ++(0.7,-1);
\draw[link,looseness=1] (0,0) to[out=90,in=270] ++(0.7,1) ;%++(0.3,0) to[out=270,in=90] ++(-0.7,-1);
%\draw[link,looseness=2] (0.3,0) to[out=270,in=270] (0.7,0);
\draw[link,looseness=1] (4.1,1.5) to[out=90,in=0] (3.1,2) to[out=180,in=0] ++ (-0.7,0.6);
\draw[link,looseness=1] (0.7,1) to[out=90,in=180] (1.8,2.3) -- (3,2.3) to[out=0,in=90] (4.4,1.5);
\draw[link,looseness=1] (4.1,1.5) to[out=270,in=90] ++(0.3,-0.4) ++(-0.3,0) to[out=270,in=90] ++(0.3,-0.4) to[out=270,in=0] (3.5,0);
\draw[link,looseness=1] (4.4,1.5) to[out=270,in=90] ++(-0.3,-0.4) ++(0.3,0) to[out=270,in=90] ++(-0.3,-0.4) to[out=270,in=0] (3.5,0.3);
% twists leading into left edge of braid
\draw[link] (1.4,1.7) to[out=90,in=180] ++(1,0.9);
%\draw[link] (0,1) -- ++(0,1) to[out=90,in=90] (2.3,2) -- ++(0,-0.5);
\draw[link] (0.3,1) -- ++(0,1) to[out=90,in=90] (2,2) -- ++(0,-0.5);
\draw[link,looseness=1] (2,1.5) to[out=270,in=90] ++(0.3,-0.4) ++(-0.3,0) to[out=270,in=90] ++(0.3,-0.4) to[out=270,in=180] (2.5,0.3) -- ++(1,0);
\draw[link,looseness=1] (2,1.1) to[out=90,in=270] ++(-0.6,0.6);
\draw[link,looseness=1] (2.3,1.5) ++(-0.3,-0.4) ++(0.3,0) to[out=270,in=90] ++(-0.3,-0.4) to[out=270,in=180] (2.5,0.0) -- ++(1,0);
\draw[link,looseness=1] (1,0) to[out=270,in=180] (1.5,-0.3) -- (4,-0.3) to[out=0,in=0,looseness=1.75] (4,-1) -- (1,-1) to[out=180,in=270] (0,0);
% box for the braid
\draw[linkred] (3.4,0.45) to[out=90,in=90] node[midway,above,inner sep=2pt,red] {\small$\kappa$} ++(0.4,0) -- ++(0,-0.9) to[out=270,in=270] ++(-0.4,0);
\draw[very thick,fill=white] (2.6,-0.45) rectangle (3.6,0.45);
\node at (3.1,0) {\Large$\beta$};
\end{scope}

\begin{scope}[xshift=11cm] % step 3
\draw[link,looseness=1] (1,0) to[out=90,in=270] ++(-0.7,1);% ++(-0.3,0) to[out=270,in=90] ++(0.7,-1);
\draw[link,looseness=1] (0,0) to[out=90,in=270] ++(0.7,1) ;%++(0.3,0) to[out=270,in=90] ++(-0.7,-1);
%\draw[link,looseness=2] (0.3,0) to[out=270,in=270] (0.7,0);
\draw[link,looseness=1] (4.4,1.5) to[out=90,in=0] (3.1,2.6);% to[out=180,in=0] ++ (-0.7,0);
\draw[link,looseness=1] (0.7,1) to[out=90,in=180] (1.8,2.3) -- (3,2.3) to[out=0,in=90] (4.1,1.5);
\draw[link,looseness=1] (4.1,1.5) -- ++(0,-0.4) to[out=270,in=90] ++(0.3,-0.4) to[out=270,in=0] (3.5,0);
\draw[link,looseness=1] (4.4,1.5) -- ++(0,-0.4) to[out=270,in=90] ++(-0.3,-0.4) to[out=270,in=0] (3.5,0.3);
% twists leading into left edge of braid
\draw[link] (2.3,1.6) to[out=90,in=180] ++(0.8,1);
\draw[link] (0.3,1) -- ++(0,1) to[out=90,in=90] (2,2) -- ++(0,-1.3);
\draw[link,looseness=1] (2.3,1.6) -- ++(0,-0.9) to[out=270,in=180] (2.5,0.3) -- ++(1,0);
\draw[link,looseness=1] (2,1.5) -- ++(0,-0.8) to[out=270,in=180] (2.5,0.0) -- ++(1,0);
\draw[link,looseness=1] (1,0) to[out=270,in=180] (1.5,-0.3) -- (4,-0.3) to[out=0,in=0,looseness=1.75] (4,-1) -- (1,-1) to[out=180,in=270] (0,0);
% box for the braid
\draw[linkred] (3.4,0.45) to[out=90,in=90] node[midway,above,inner sep=2pt,red] {\small$\kappa$} ++(0.4,0) -- ++(0,-0.9) to[out=270,in=270] ++(-0.4,0);
\draw[very thick,fill=white] (2.6,-0.45) rectangle (3.6,0.45);
\node at (3.1,0) {\Large$\beta$};
\end{scope}

\begin{scope}[xshift=17.5cm] % final step
\draw[linkred] (1.4,0.6) arc (90:270:0.15 and 0.6);
\draw[link] (0,0.3) to[out=180,in=180] ++(0,0.6) -- ++(1.4,0) to[out=0,in=0] ++(0,-0.6) -- ++(-0.4,0);
\draw[link] (0,0) to[out=180,in=270] ++(-0.6,0.6) to[out=90,in=180] ++(0.9,0.6) to[out=0,in=270,looseness=1] ++(0.4,0.2) to[out=90,in=0,looseness=1] ++(-0.4,0.2) to[out=180,in=90] ++(-1.2,-0.9) to[out=270,in=180] ++(0.9,-1);
\draw[link] (1,0) to[out=0,in=270] ++(1,0.6) to[out=90,in=0,looseness=1.4] ++(-1.3,0.6) to[out=180,in=270,looseness=1] ++(-0.4,0.2) to[out=90,in=180,looseness=1] ++(0.4,0.2) to[out=0,in=90,looseness=1.4] ++(1.6,-0.9) to[out=270,in=0] ++(-1.3,-1);
\draw[link] (0.7,1.4) to[out=270,in=0,looseness=1] ++(-0.4,-0.2) -- ++(-0.1,0); % redraw a crossing
\draw[linkred] (1.4,0.6) arc (90:-90:0.15 and 0.6) node[below,red,inner sep=2pt] {\small$\kappa$};
\draw[very thick,fill=white] (0,-0.45) rectangle (1,0.45);
\node at (0.5,0) {\Large$\beta$};
\end{scope}

\end{tikzpicture}
\caption{An isotopy of the unknot $U = \tau \cup \beta$ in the complement of $\kappa$.}
\label{fig:E-T24-isotopy}
\end{figure}

Figure~\ref{fig:E-T24-isotopy} shows an isotopy of the unknot $U=\tau\cup\beta$ into a simpler form, which we will use in the subsections below.

In the sequel we will often write $K = K_\beta$ when $K$ arises from a given braid $\beta \in B_3$ in the sense of Lemma~\ref{lem:E24-as-branched-cover}.  We write each 3-braid as a word in
\[
x =
\begin{tikzpicture}[baseline=-0.3em, link/.append style={looseness=0.75}]
\draw[link] (0,0) to[out=0,in=180] ++(0.6,0.3);
\draw[link] (0,0.3) to[out=0,in=180] ++(0.6,-0.3);
\draw[link] (0,-0.3) -- ++(0.6,0);
\end{tikzpicture}
\quad\text{and}\quad
y = 
\begin{tikzpicture}[baseline=-0.3em, link/.append style={looseness=0.75}]
\draw[link] (0,-0.3) to[out=0,in=180] ++(0.6,0.3);
\draw[link] (0,0) to[out=0,in=180] ++(0.6,-0.3);
\draw[link] (0,0.3) -- ++(0.6,0);
\end{tikzpicture}
\]
where $x$ and $y$ denote positive crossings between the top two strands and the bottom two strands, respectively.  The following observations will help simplify our analysis in the following subsections.

\begin{lemma} \label{lem:reverse-braid}
Let $r: B_3 \to B_3$ be the map which reverses a braid word, defined recursively by \[r(1)=1 \,\textrm{ and }\, r(gw)=r(w)g\] for any $g\in\{x^{\pm1},y^{\pm1}\}$.  If $\beta$ is a 3-braid for which $\tau \cup \beta$ is unknotted, then $\tau \cup r(\beta)$ is also unknotted and $K_\beta \cong K_{r(\beta)}$.
\end{lemma}

\begin{proof}
We can rotate the diagram of the unknot $U=\tau\cup\beta$ on the right side of Figure~\ref{fig:E-T24-isotopy} about a vertical axis, and this preserves the tangle $\tau$ and the linked curve $\kappa$ while replacing the braid $\beta$ with its reverse $r(\beta)$.  It follows that $\tau\cup r(\beta)$ is also unknotted, since it is isotopic to the unknot $U$.  This isotopy also carries $\kappa$ to itself, so up to isotopy $\kappa$ must lift to both $K_\beta$ and $K_{r(\beta)}$ in the branched double cover of $U$, hence $K_\beta \cong K_{r(\beta)}$.
\end{proof}

\begin{lemma} \label{lem:mirror-braid}
Let $m: B_3 \to B_3$ be the map which mirrors a braid word, defined recursively by \[m(1) = 1\,\textrm{ and }m(gw) = g^{-1} m(w)\] for any $g\in \{x^{\pm1},y^{\pm1}\}$.    If $\beta$ is a 3-braid for which $\tau\cup\beta$ is unknotted, then $\tau \cup \big(m(\beta)y\big)$ is also unknotted, and $K_{m(\beta)y}$ is the mirror of $K_\beta$.
\end{lemma}

\begin{proof}
\begin{figure}
\begin{tikzpicture}
\begin{scope}
\draw[linkred] (2.2,0.6) arc (90:270:0.15 and 0.6);
\draw[link] (0,0.3) to[out=180,in=180] ++(0,0.6) -- ++(2.2,0) to[out=0,in=0] ++(0,-0.6) -- ++(-1.2,0);
\draw[link] (1,0) -- ++(0.2,0) (1,-0.3) -- ++(0.2,0);
\draw[link] (0,0) to[out=180,in=270] ++(-0.6,0.6) to[out=90,in=180] ++(0.9,0.6) to[out=0,in=270,looseness=1] ++(0.4,0.2) to[out=90,in=0,looseness=1] ++(-0.4,0.2) to[out=180,in=90] ++(-1.2,-0.9) to[out=270,in=180] ++(0.9,-1);
\draw[link] (1.8,0) to[out=0,in=270] ++(1,0.6) to[out=90,in=0,looseness=1.4] ++(-1.3,0.6) -- ++(-0.8,0) to[out=180,in=270,looseness=1] ++(-0.4,0.2) to[out=90,in=180,looseness=1] ++(0.4,0.2) -- ++(0.8,0) to[out=0,in=90,looseness=1.4] ++(1.6,-0.9) to[out=270,in=0] ++(-1.3,-1);
\draw[link,looseness=0.75] (1.2,-0.3) to[out=0,in=180] ++(0.6,0.3);
\draw[link,looseness=0.75] (1.2,0) to[out=0,in=180] ++(0.6,-0.3);
\draw[link] (0.7,1.4) to[out=270,in=0,looseness=1] ++(-0.4,-0.2) -- ++(-0.1,0); % redraw a crossing
\draw[linkred] (2.2,0.6) arc (90:-90:0.15 and 0.6) node[below,red,inner sep=2pt] {\small$\kappa$};
\draw[very thick,fill=white] (0,-0.45) rectangle (1,0.45);
\node at (0.5,0) {$m(\beta)$};
\end{scope}

\begin{scope}[xshift=4.75cm]
\draw[linkred] (2.2,0.6) arc (90:270:0.15 and 0.6);
\draw[link] (0,0.3) to[out=180,in=180] ++(0,0.6) -- ++(2.2,0) to[out=0,in=0] ++(0,-0.6) -- ++(-1.2,0);
\draw[link] (1,0) -- ++(0.8,0) (1,-0.3) -- ++(0.8,0);
\draw[link] (0,0) to[out=180,in=270] ++(-0.6,0.6) to[out=90,in=180] ++(0.9,0.6) to[out=0,in=270,looseness=1] ++(0.4,0.2) to[out=90,in=0,looseness=1] ++(-0.4,0.2) to[out=180,in=90] ++(-1.2,-0.9) to[out=270,in=180] ++(0.9,-1);
\draw[link] (1.8,0) to[out=0,in=270] ++(1,0.6) to[out=90,in=0,looseness=1.4] ++(-1.3,0.6) ++(-0.8,0) to[out=180,in=270,looseness=1] ++(-0.4,0.2) to[out=90,in=180,looseness=1] ++(0.4,0.2) ++(0.8,0) to[out=0,in=90,looseness=1.4] ++(1.6,-0.9) to[out=270,in=0] ++(-1.3,-1);
\draw[link,looseness=0.75] (0.7,1.2) -- ++(0.2,0) to[out=0,in=180] ++(0.6,0.4);
\draw[link,looseness=0.75] (0.7,1.6) -- ++(0.2,0) to[out=0,in=180] ++(0.6,-0.4);
\draw[link] (0.7,1.4) to[out=270,in=0,looseness=1] ++(-0.4,-0.2) -- ++(-0.1,0); % redraw a crossing
\draw[linkred] (2.2,0.6) arc (90:-90:0.15 and 0.6) node[below,red,inner sep=2pt] {\small$\kappa$};
\draw[very thick,fill=white] (0,-0.45) rectangle (1,0.45);
\node at (0.5,0) {$m(\beta)$};
\end{scope}

\begin{scope}[xshift=9.5cm]
\draw[linkred] (1.4,0.6) arc (90:270:0.15 and 0.6);
\draw[link] (0,0.3) to[out=180,in=180] ++(0,0.6) -- ++(1.4,0) to[out=0,in=0] ++(0,-0.6) -- ++(-0.4,0);
\draw[link] (0,0) to[out=180,in=270] ++(-0.6,0.6) to[out=90,in=180] ++(0.9,0.6) to[out=0,in=270,looseness=1] ++(0.4,0.2) to[out=90,in=0,looseness=1] ++(-0.4,0.2) to[out=180,in=90] ++(-1.2,-0.9) to[out=270,in=180] ++(0.9,-1);
\draw[link] (1,0) to[out=0,in=270] ++(1,0.6) to[out=90,in=0,looseness=1.4] ++(-1.3,0.6) to[out=180,in=270,looseness=1] ++(-0.4,0.2) to[out=90,in=180,looseness=1] ++(0.4,0.2) to[out=0,in=90,looseness=1.4] ++(1.6,-0.9) to[out=270,in=0] ++(-1.3,-1);
\draw[link] (0.3,1.6) to[out=0,in=90,looseness=1] ++(0.4,-0.2); % redraw a crossing
\draw[linkred] (1.4,0.6) arc (90:-90:0.15 and 0.6) node[below,red,inner sep=2pt] {\small$\kappa$};
\draw[very thick,fill=white] (0,-0.45) rectangle (1,0.45);
\node at (0.5,0) {$m(\beta)$};
\end{scope}
\end{tikzpicture}
\caption{An isotopy takes the tangle $\tau \cup (m(\beta)y)$ to the mirror of the tangle $\tau \cup \beta$.}
\label{fig:beta-mirror}
\end{figure}
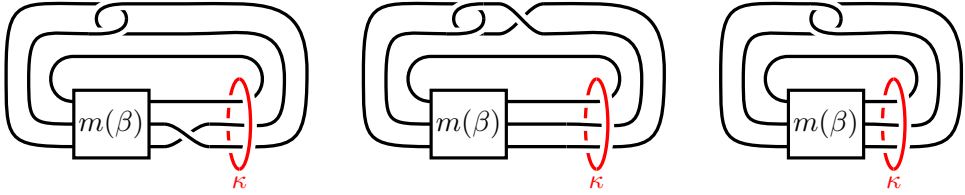
In Figure~\ref{fig:beta-mirror} we perform an isotopy of $U = \tau \cup \big(m(\beta)y\big)$ in the complement of $\kappa$, and we quickly find ourselves with a mirror image (reflecting across the plane of the page) of the diagram used to recover $K_\beta$.  Thus if $\tau \cup \beta$ is unknotted then so is $\tau \cup \big(m(\beta)y\big)$, and the unknot $\kappa$ for $\tau \cup \big(m(\beta)y\big)$ lifts to the mirror of the lift $K_\beta$ of the corresponding knot in the $\tau \cup \beta$ diagram.
\end{proof}

We remark that the mirror of $\beta$ is equal to the reverse of $\beta^{-1}$, i.e., $m(\beta) = r(\beta^{-1})$.

\begin{lemma} \label{lem:y-conjugation}
If $\beta \in B_3$ produces an unknot $U=\tau\cup\beta$, then so does $y^a \beta y^{-a}$ for any $a \in \Z$, and moreover $K_{y^a\beta y^{-a}} \cong K_\beta$.
\end{lemma}

\begin{proof}
It is straightforward to see that $\tau \cup (y\beta y^{-1})$ is isotopic to $\tau \cup \beta$ in the complement of $\kappa$, so the lemma follows by induction on $a$.
\end{proof}

We now outline the proof of Theorem~\ref{thm:M_F-E24}.

\begin{proof}[Proof of Theorem~\ref{thm:M_F-E24}]
By Lemma~\ref{lem:E24-as-branched-cover}, it suffices to classify the braids $\beta \in B_3$ such that $U = \tau\cup\beta$ is unknotted, and to determine $K = K_\beta$ for each of them.

Supposing that $U$ is an unknot, in Subsection~\ref{ssec:bdc-beta} we will identify an arc (see Figure~\ref{fig:E-T24-resolutions}) that lifts to a knot $\gamma$ in the branched double cover $\dcover(U) \cong S^3$.  We will argue via the cyclic surgery theorem \cite{cgls} that $\gamma$ must be an unknot or a torus knot, and study various surgeries on $\gamma$ which must be lens spaces or connected sums of lens spaces.  In Subsections~\ref{ssec:gamma-unknot} and \ref{ssec:gamma-torus-knot}, we will study the cases $\gamma \cong U$ and $\gamma \cong T_{p,q}$ separately, proving in Propositions~\ref{prop:gamma-unknot} and \ref{prop:gamma-torus} that $\beta$ must be one of
\[ x^{-1},\ xy, \text{ or } x^ny^{-1}xy \ (n\in\Z) \]
or 
\[ x^3y^{-1}x^2y \text{ or } x^{-3}yx^{-2} \]
respectively, up to reversal and conjugation by powers of $y$.  Lemmas~\ref{lem:reverse-braid} and \ref{lem:y-conjugation} tell us that it is enough to consider these particular braids.

After classifying these braids, we devote Subsection~\ref{ssec:E24-examples} to determining the knot $K_\beta$ for each of
\[ \beta = x^{-1},\ x^ny^{-1}xy, \text{ or } x^3y^{-1}x^2y. \]
These cases occupy Propositions~\ref{prop:recover-5_2}, \ref{prop:recover-pretzels}, and \ref{prop:recover-15n43522}, respectively, and they recover the knots $5_2$, $P(-3,3,2n+1)$, and $15n_{43522}$.  The only remaining braids are
\[ \beta = xy = m(x^{-1}) y \]
and
\[ \beta = x^{-3}yx^{-2} = m(x^3y^{-1}x^2y) y, \]
but then Lemma~\ref{lem:mirror-braid} says that the corresponding $K_\beta$ are the mirrors of knots which we already found, so the proof is complete.
\end{proof}

The remainder of this lengthy section is devoted to proving the results cited in the proof of Theorem~\ref{thm:M_F-E24}.

\subsection{Resolutions and the 3-braid $\beta$} \label{ssec:bdc-beta}

In Figure~\ref{fig:E-T24-resolutions} we take a fixed crossing (indicated by a dashed arc) of the unknot diagram from Figure~\ref{fig:E-T24-isotopy} and modify it in several ways, changing the crossing to produce a new knot $L^\beta$ and also resolving the crossing in two different ways to produce the links $L^\beta_0$ and $L^\beta_1$.  It is clear from the diagrams that $L^\beta$ is a two-bridge knot, and that $L^\beta_0 \cong \widehat{\beta}$ and $L^\beta_1\cong \widehat{\beta y^{-1}}$ are both closures of 3-braids.
\begin{figure}
\begin{tikzpicture}
\begin{scope}
\draw[link] (0,0.3) to[out=180,in=180] ++(0,0.6) -- ++(1,0) to[out=0,in=0] ++(0,-0.6);
\draw[link] (0,0) to[out=180,in=270] ++(-0.6,0.6) to[out=90,in=180] ++(0.9,0.6) to[out=0,in=270,looseness=1] ++(0.4,0.2) to[out=90,in=0,looseness=1] ++(-0.4,0.2) to[out=180,in=90] ++(-1.2,-0.9) to[out=270,in=180] ++(0.9,-1);
\draw[link] (1,0) to[out=0,in=270] ++(0.6,0.6) to[out=90,in=0] ++(-0.9,0.6) to[out=180,in=270,looseness=1] ++(-0.4,0.2) to[out=90,in=180,looseness=1] ++(0.4,0.2) to[out=0,in=90] ++(1.2,-0.9) to[out=270,in=0] ++(-0.9,-1);
\draw[link] (0.7,1.4) to[out=270,in=0,looseness=1] ++(-0.4,-0.2); % redraw a crossing
\draw[very thick,fill=white] (0,-0.45) rectangle (1,0.45);
\node at (0.5,0) {\Large$\beta$};
\draw[red,densely dashed] (0.5,1.6) ++ (0.25,0) arc (0:180:0.25 and 0.15);

\draw[link] (1,-2) -- ++(-1,1);
\draw[link] (0,-2) -- ++(1,1);
\draw[red,densely dashed] (0.5,-1.5) ++ (45:0.5) arc (45:135:0.5);
\node[below] at (0.5,-2) {$U^{\vphantom{\beta}}$};
\end{scope}

\begin{scope}[xshift=3.5cm]
\draw[link] (0,0.3) to[out=180,in=180] ++(0,0.6) -- ++(1,0) to[out=0,in=0] ++(0,-0.6);
\draw[link, looseness=3.5] (0,0) to[out=180,in=180] ++(0,-0.3) (1,0) to[out=0,in=0] ++(0,-0.3);
\draw[very thick,fill=white] (0,-0.45) rectangle (1,0.45);
\node at (0.5,0) {\Large$\beta$};

\draw[link] (0,-2) -- ++(1,1);
\draw[link] (1,-2) -- ++(-1,1);
\node[below] at (0.5,-2) {$L^\beta$};
\end{scope}

\begin{scope}[xshift=7cm]
\draw[link] (0,0.3) to[out=180,in=180] ++(0,0.6) -- ++(1,0) to[out=0,in=0] ++(0,-0.6);
\draw[link] (0,0) to[out=180,in=180] ++(0,1.2) -- ++(1,0) to[out=0,in=0] ++(0,-1.2);
\draw[link] (0,-0.3) to[out=180,in=180] ++(0,1.8) -- ++(1,0) to[out=0,in=0] ++(0,-1.8);
\draw[very thick,fill=white] (0,-0.45) rectangle (1,0.45);
\node at (0.5,0) {\Large$\beta$};

\draw[link] (0,-2) to[out=45,in=135] ++(1,0) (0,-1) to[out=-45,in=-135] ++(1,0);
\node[below] at (0.5,-2) {$L_0^\beta$};
\end{scope}

\begin{scope}[xshift=10.5cm]
\begin{scope}[xshift=-0.25cm]
\draw[link] (0,0.3) to[out=180,in=180] ++(0,0.6) -- ++(1.5,0) to[out=0,in=0] ++(0,-0.6) -- ++(-0.5,0);
\draw[link] (0,-0.3) to[out=180,in=180] ++(0,1.8) -- ++(1.5,0) to[out=0,in=0] ++(0,-1.8) -- ++(-0.1,0) to[out=180,in=0,looseness=0.75] ++(-0.3,0.3) -- ++(-0.1,0);
\draw[link] (0,0) to[out=180,in=180] ++(0,1.2) -- ++(1.5,0) to[out=0,in=0] ++(0,-1.2) -- ++(-0.1,0) to[out=180,in=0,looseness=0.75] ++(-0.3,-0.3) -- ++(-0.1,0);
\draw[very thick,fill=white] (0,-0.45) rectangle (1,0.45);
\node at (0.5,0) {\Large$\beta$};
\end{scope}
\draw[link] (0,-2) to[out=45,in=-45] ++(0,1) (1,-2) to[out=135,in=-135] ++(0,1);
\node[below] at (0.5,-2) {$L_1^\beta$};
\end{scope}\end{tikzpicture}
\caption{Resolving the topmost crossing in the clasp of $U = \tau \cup \beta$ in several different ways.}
\label{fig:E-T24-resolutions}
\end{figure}
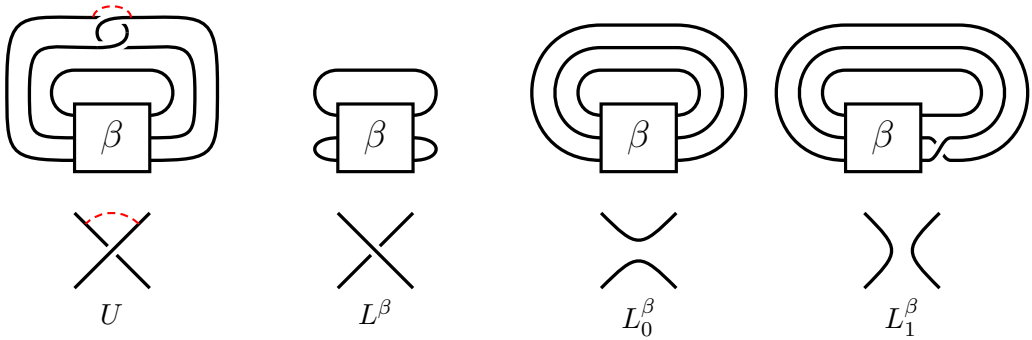

The dashed arc on the left side of Figure~\ref{fig:E-T24-resolutions} lifts to a simple closed curve $\gamma$ in the branched double cover $\dcover(U) \cong S^3$.  Then the Montesinos trick \cite{montesinos} says that $\dcover(L^\beta)$ can be realized as a half-integral surgery on $\gamma$:
\begin{equation} \label{eq:dcover-lbeta}
\dcover(L^\beta) \cong S^3_{(2n+1)/2}(\gamma) \text{ for some } n \in \Z.
\end{equation}
(Indeed, the branch loci $U$ and $L^\beta$ agree outside a neighborhood of the indicated arc, so $\dcover(L^\beta)$ and $\dcover(U)$ agree outside the branched double cover of that neighborhood, which in either case is a solid torus.  This says that $\dcover(\beta)$ comes from some surgery on $\gamma$ in $\dcover(U) \cong S^3$, and then it must be half-integral because the peripheral curves in $S^3 \setminus N(\gamma)$ whose fillings produce $\dcover(U)$ and $\dcover(L^\beta)$ have distance two in $\partial N(\gamma)$.)
Similarly, the $0$- and $1$-resolutions of that crossing correspond to consecutive integral surgeries on $\gamma$, which are each distance-1 from the $\frac{2n+1}{2}$-surgery corresponding to the crossing change: that is,
\begin{align} \label{eq:dcover-lbeta-i}
\dcover(L^\beta_0) &\cong S^3_n(\gamma), &
\dcover(L^\beta_1) &\cong S^3_{n+1}(\gamma).
\end{align}
To see that $\dcover(L^\beta_0)$ and $\dcover(L^\beta_1)$ are homeomorphic to $S^3_n(\gamma)$ and $S^3_{n+1}(\gamma)$ respectively, and not vice versa, we note that the ordered triple $(\dcover(L^\beta), \dcover(L^\beta_1), \dcover(L^\beta_0))$ forms a \emph{surgery triad} \cite[Proposition~2.1]{osz-branched}, meaning that these three manifolds are all Dehn fillings of $S^3\setminus N(\gamma)$ along oriented curves $\alpha, \alpha_1, \alpha_0 \subset \partial N(\gamma)$ such that
\begin{equation} \label{eq:triad-intersections}
\alpha \cdot \alpha_1 = \alpha_1 \cdot \alpha_0 = \alpha_0 \cdot \alpha = -1.
\end{equation}
(Note that following \cite[Figure~1]{osz-branched}, their ``$L_0$'' and ``$L_1$'' are our $L^\beta_1$ and $L^\beta_0$.)  Up to reversing the orientation of all three curves simultaneously we can assume that $\alpha=(2n+1)\mu+2\lambda$, where $\mu$ and $\lambda$ are a meridian and longitude of $\gamma$ and $\partial N(\gamma)$ is oriented so that $\mu\cdot\lambda=-1$, and then there is no way to choose signs for $\alpha_1 = \pm(n\mu+\lambda)$ and $\alpha_0 = \pm((n+1)\mu+\lambda)$ so that \eqref{eq:triad-intersections} is satisfied.  On the other hand,
\[ (\alpha,\alpha_1,\alpha_0) = \big( (2n+1)\mu+2\lambda, -(n+1)\mu-\lambda, -n\mu-\lambda \big) \]
does satisfy \eqref{eq:triad-intersections}, so $\dcover(L^\beta_0)$ and $\dcover(L^\beta_1)$ must correspond to $n$- and $(n+1)$-surgeries in that order as claimed.

From this discussion we immediately deduce the following.

\begin{lemma} \label{lem:possible-gamma}
The knot $\gamma \subset \dcover(U) \cong S^3$ is either an unknot or a nontrivial torus knot.
\end{lemma}

\begin{proof}
Since $L^\beta$ is a 2-bridge knot, we know that $\dcover(L^\beta)$ is a lens space.  But the cyclic surgery theorem \cite{cgls} says that a non-integral surgery on $\gamma\subset S^3$ can only produce a lens space if $\gamma$ is an unknot or a nontrivial torus knot $T_{p,q}$.
\end{proof}

We will handle each of the two possible outcomes of Lemma~\ref{lem:possible-gamma} separately in the following subsections.  The remainder of this subsection is devoted to some computations that will prove useful in that work.

To set the stage, we cut the given 2-bridge diagram of $L^\beta$ along a pair of vertical lines passing just by $\beta$ on either side.  Taking the double cover branched over each piece of $L^\beta$ in turn gives a genus-1 Heegaard splitting of $\dcover(L^\beta)$, illustrated in Figure~\ref{fig:L-beta-heegaard}.
\begin{figure}
\begin{tikzpicture}
\begin{scope}
\draw[link] (-0.3,0.3) to[out=180,in=180] ++(0,0.6) -- ++(1.6,0) to[out=0,in=0] ++(0,-0.6) -- cycle;
\draw[link, looseness=3.5] (-0.3,0) to[out=180,in=180] ++(0,-0.3) -- (1.3,-0.3) to[out=0,in=0] ++(0,0.3) -- cycle;
\draw[very thick,fill=white] (0,-0.45) rectangle (1,0.45);
\node at (0.5,0) {\Large$\beta$};
\draw[densely dotted] (-0.3,1.2) -- (-0.3,-0.6) (1.3,1.2) -- (1.3,-0.6);
\node[below] at (0.5,-0.6) {$\dcover(L^\beta)$};
\end{scope}

\begin{scope}[xshift=4cm]
\draw[link] (-0.3,0.3) to[out=180,in=180] ++(0,0.6);
\draw[link, looseness=3.5] (-0.3,0) to[out=180,in=180] ++(0,-0.3);
\draw[densely dotted] (-0.3,1.2) -- (-0.3,-0.6);
\node[below] at (-0.3,-0.6) {$S^1\times D^2$};
\end{scope}

\begin{scope}[xshift=5.25cm]
\foreach \y in {0.9,0.3,0,-0.3} { \draw[link] (-0.3,\y) -- (1.3,\y); }
\draw[very thick,fill=white] (0,-0.45) rectangle (1,0.45);
\node at (0.5,0) {\Large$\beta$};
\draw[densely dotted] (-0.3,1.2) -- (-0.3,-0.6) (1.3,1.2) -- (1.3,-0.6);
\node[below] at (0.5,-0.6) {$T^2\times I$};
\end{scope}

\begin{scope}[xshift=6.5cm]
\draw[link] (1.3,0.3) to[out=0,in=0] ++(0,0.6);
\draw[link, looseness=3.5] (1.3,0) to[out=0,in=0] ++(0,-0.3);
\draw[densely dotted] (1.3,1.2) -- (1.3,-0.6);
\node[below] at (1.3,-0.6) {$S^1\times D^2$};
\end{scope}

\node[below] at (2.125,-0.6) {$\cong\vphantom{S^1}$};
\node[below] at (4.75,-0.6) {$\cup\vphantom{S^1}$};
\node[below] at (6.75,-0.6) {$\cup\vphantom{S^1}$};
\end{tikzpicture}
\caption{A genus-1 Heegaard splitting of $\dcover(L^\beta)$.}
\label{fig:L-beta-heegaard}
\end{figure}
The solid tori $S^1\times D^2$ on either side of this splitting would be glued together to form $S^1\times S^2$ if the braid $\beta$ were trivial.  But in general, the effect of gluing the middle $T^2\times I$ to either $S^1\times D^2$ is to reparametrize its boundary: the braid generators $x$ and $y$ act as positive Dehn twists along essential curves in $S^1\times S^1$, which we have labeled $c_x$ and $c_y$ and oriented in Figure~\ref{fig:bdc-ball-arcs}.  Gluing after this reparametrization produces the desired Heegaard splitting of $\dcover(L^\beta)$.
\begin{figure}
\begin{tikzpicture}
\begin{scope} % draw the ball
\draw (0,0) circle (2);
\clip (0,0) circle (2);
\draw[very thin] (-2,0) arc (180:0:2 and 0.5); % back of equator
\foreach \i in {1,2,3,4} { \path[rotate=-80] (25+20*\i:2 and 0.625) coordinate (p\i); }
\begin{scope}
\draw[rotate=-80,very thin,fill=gray!10,fill opacity=0.75] (0,0) ellipse (2 and 0.625);
\clip[rotate=-80] (0,0) ellipse (2 and 0.625);
\path[fill=blue!15] (p1) to[out=175,in=160,looseness=4.5] (p2) -- ++(1,0) -- ++(0,1) -- cycle;
\draw[blue] (p1) to[out=175,in=160,looseness=4.5] (p2);
\draw[blue] (p3) to[out=205,in=195,looseness=4] (p4);
\draw[rotate=-80,semithick] (0:2 and 0.5) arc (0:180:2 and 0.625);
\end{scope}
\draw[semithick] (-2,0) arc (180:360:2 and 0.5); % front of equator
\draw (0,0) circle (2); % redraw boundary circle
\foreach \i in {1,2,3} { \draw[red,fill=red] (p\i) circle (0.05); }
\draw[rotate=-80,very thick,red] (45:2 and 0.625) arc (45:65:2 and 0.625) node[midway,right,inner sep=2pt] {\small$\alpha_y$} arc (65:85:2 and 0.625) node[pos=0.7,right,inner sep=2pt] {\small$\alpha_x$};
\end{scope}
\node at (3,0) {\huge$\xrightarrow{\dcover}$};
\begin{scope}[xshift=7cm] % the torus
% the y circle
\begin{scope}
\draw[rotate=15,red,fill=blue!15,thin] (0,-0.975) ellipse (0.125 and 0.55);
\draw[rotate=15,red,very thick] (0,-0.975) ++ (0,-0.55) arc (-90:90:0.125 and 0.55) node[right,pos=0.55,inner sep=3pt] {$c_y$};
\draw[rotate=15,red,->] (0,-0.975) ++(35:0.125 and 0.55) arc (35:25:0.125 and 0.55);
\clip[rotate=15] (0,-0.975) ellipse (0.125 and 0.55);
\path[rotate=15] (0,-0.975) ++ (-38:0.125 and 0.55) coordinate (b1);
\path[rotate=15] (0,-0.975) ++ (142:0.125 and 0.55) coordinate (b2);
\draw[blue] (b1) -- (b2);
\end{scope}
% phantom circle on the other side
\begin{scope}
\draw[rotate=15,very thin] (0,0.95) ++ (-90:0.125 and 0.575) arc (-90:90:0.125 and 0.575);
\draw[rotate=15,ultra thin] (0,0.95) ++ (90:0.125 and 0.575) arc (90:270:0.125 and 0.575);
\clip[rotate=15] (0,0.95) ellipse (0.125 and 0.575);
\path[rotate=15] (0,0.95) ++ (33:0.125 and 0.575) coordinate (b3);
\path[rotate=15] (0,0.95) ++ (213:0.125 and 0.575) coordinate (b4);
\draw[blue] (b3) -- (b4);
\end{scope}
% the x circle
\draw[red,very thick] (0,0) ellipse (2.625 and 1.25);
\draw[red,->] (0,0) ++(120:2.625 and 1.25) arc (120:130:2.625 and 1.25);
\node[below,red,inner sep=3pt] at (60:2.625 and 1.25) {$c_x$};
% the torus itself
\draw (0,0) ellipse (3 and 1.5);
\begin{scope} % add genus
\draw (0,3.1) ++ (240:3.5) arc (240:300:3.5);
\clip (0,3.1) -- ++(240:3.5) arc (240:300:3.5) -- cycle;
\draw (0,-3.15) ++ (60:3.5) arc (60:120:3.5);
\end{scope}
\end{scope}
\end{tikzpicture}
\caption{Lifting arcs $\alpha_x$ and $\alpha_y$ in a 3-ball to closed curves $c_x$ and $c_y$ in a solid torus, viewed as its branched double cover over a pair of properly embedded arcs.}
\label{fig:bdc-ball-arcs}
\end{figure}
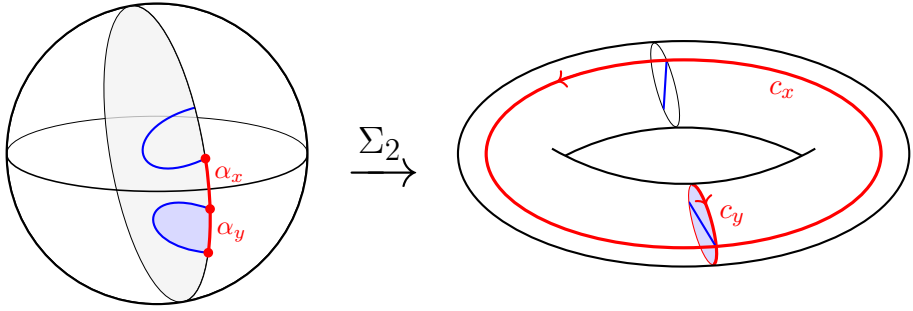

The braid generators $x$ and $y$ act on the homology of the leftmost $S^1\times D^2$ by
\begin{align*}
[c_x] \cdot x = [\tau_{c_x}(c_x)] &= [c_x], &
[c_x] \cdot y = [\tau_{c_y}(c_x)] &= [c_x]+[c_y], \\
[c_y] \cdot x = [\tau_{c_x}(c_y)] &= [c_y]-[c_x], &
[c_y] \cdot y= [\tau_{c_y}(c_y)] &= [c_y].
\end{align*}
Equivalently, we can view them as fixing that $S^1\times D^2$, but acting on the rightmost $S^1\times D^2$ by the inverse of the above action:
\begin{align*}
x \cdot [c_x] = [\tau_{c_x}(c_x)] &= [c_x], &
y \cdot [c_x] = [\tau_{c_y}(c_x)] &= [c_x]-[c_y], \\
x \cdot [c_y] = [\tau_{c_x}(c_y)] &= [c_x]+[c_y], &
y \cdot [c_y] = [\tau_{c_y}(c_y)] &= [c_y].
\end{align*}
Thus if we fix the ordered basis $([c_x],[c_y])$, then the (left) action of $B_3$ on the rightmost $H_1(\partial(S^1\times D^2)) \cong \Z^2$ is given by a homomorphism
\[ \rho: B_3 \to SL_2(\Z) \]
defined by
\begin{align} \label{eq:B3-action}
\rho(x) &= \begin{pmatrix} 1 & 1 \\ 0 & 1 \end{pmatrix}, &
\rho(y) &= \begin{pmatrix} 1 & 0 \\ -1 & 1 \end{pmatrix}.
\end{align}
One can verify that this is well-defined, since $\rho(xyx)=\rho(yxy) = \left(\begin{smallmatrix} 0 & 1 \\ -1 & 0 \end{smallmatrix}\right)$; and that $\rho(\Delta^2) = -I$, where $\Delta^2 = (xyx)^2 = (xy)^3$ is the full twist which generates the center of $B_3$.

\begin{lemma} \label{lem:kernel-rho}
The kernel of $\rho$ is generated by $\Delta^4$.
\end{lemma}

\begin{proof}
If $w \in B_3$ satisfies $\rho(w) = I$, then the same is true for every conjugate of $w$, and Murasugi \cite{murasugi-3-braid} showed that $w$ is conjugate to one of
\begin{enumerate}
\item $\Delta^{2d} xy^{-a_1} xy^{-a_2} \cdots xy^{-a_n}$, where all $a_i$ are nonnegative and at least one is positive;
\item $\Delta^{2d} y^m$ for some $m\in\Z$; or
\item $\Delta^{2d} x^my^{-1}$ where $m=-1,-2,-3$.
\end{enumerate}
In the second and third cases we compute that
\[ \rho(\Delta^{2d}y^m) = (-1)^d \begin{pmatrix} 1 & 0 \\ -m & 1 \end{pmatrix}
\quad\text{and}\quad
\rho(\Delta^{2d}x^my^{-1}) = (-1)^d \begin{pmatrix} m+1 & m \\ 1 & 1 \end{pmatrix}, \]
so the only such braids in the kernel are $\Delta^{4d} = \Delta^{2\cdot 2d} y^0$.  For the first case we have
\[ \rho(xy^{-a_1} \cdots xy^{-a_n}) = \begin{pmatrix} a_1+1 & 1 \\ a_1 & 1 \end{pmatrix} \cdots \begin{pmatrix} a_n+1 & 1 \\ a_n & 1 \end{pmatrix}, \]
and a straightforward induction on $n \geq 1$ shows that its entries are nonnegative integers, and that the top right entry is strictly positive.  In particular it cannot be $\pm I$ since it is not diagonal, so
\[ \rho(\Delta^{2d}xy^{-a_1} \cdots xy^{-a_n}) = (-1)^d\rho(xy^{-a_1} \cdots xy^{-a_n}) \]
is not the identity either.  We conclude that $\rho(w) = I$ if and only if $w$ is conjugate to some power of $\Delta^4$, and then it must actually be that power of $\Delta^4$ since $\Delta^2$ is central.\end{proof}

\begin{lemma} \label{lem:bdc-2-bridge}
If the representation \eqref{eq:B3-action} satisfies 
\[ \rho(\beta) = \begin{pmatrix} a & b \\ c & d \end{pmatrix}, \]
then we have $\dcover(L^\beta) \cong S^3_{b/d}(U)$.
\end{lemma}

\begin{proof}
The curve $c_y$ bounds a disk in the rightmost $S^1\times D^2$ of Figure~\ref{fig:L-beta-heegaard}, so then $\beta \cdot [c_y] = b[c_x] + d[c_y]$ bounds a disk in the rightmost $(T^2\times I) \cup (S^1\times D^2)$.  Thus we can obtain the branched double cover of $L^\beta$ by Dehn filling the leftmost $S^1\times D^2$ along $b[c_x]+d[c_y]$.  Thinking of the left $S^1\times D^2$ as the complement of an unknot in $S^3$, the oriented curves $c_x$ and $c_y$ correspond to a meridian and longitude of that unknot, respectively, so this amounts to a Dehn filling of slope $\frac{b}{d}$.
\end{proof}

\begin{lemma} \label{lem:bdc-trace}
We have $\tr \rho(\beta) = 2 \pm |H_1(\dcover(L^\beta_0);\Z)|$, where we define $|H_1|=0$ if $H_1$ is infinite.
\end{lemma}

\begin{proof}
Inspecting the diagram for $L_0^\beta \cong \widehat\beta$ in Figure~\ref{fig:E-T24-resolutions}, we see that its branched double cover admits an open book decomposition whose binding is the lift of the braid axis; the pages are punctured tori (i.e., the double cover of a disk with three branch points), and the monodromy acts on the homology of the pages by $\rho(\beta)$.  It follows that
\[ H_1(\dcover(L_0^\beta);\Z) \cong \coker( \rho(\beta) - I: \Z^2 \to \Z^2 ). \]
Thus if this order is finite then it equals $\left|\det(\rho(\beta)-I)\right|$, and otherwise $\det(\rho(\beta)-I) = 0$.  Writing $\rho(\beta)=\left(\begin{smallmatrix}a&b\\c&d\end{smallmatrix}\right)$ with $ad-bc=1$, we compute this order up to sign as
\begin{align*}
\det(\rho(\beta)-I) &= \det \begin{pmatrix}a-1&b\\c&d-1\end{pmatrix} \\
&= (a-1)(d-1)-bc = (ad-bc) - (a+d) + 1,
\end{align*}
which is equal to $2-\tr(\rho(\beta))$, so $\tr(\rho(\beta)) = 2 \pm |H_1(\dcover(L_0^\beta))|$ as claimed.
\end{proof}

According to \eqref{eq:dcover-lbeta} and \eqref{eq:dcover-lbeta-i}, we have some $n\in\Z$ such that
\[ \dcover(L^\beta) \cong S^3_{(2n+1)/2}(\gamma) \quad\text{and}\quad \dcover(L^\beta_0) \cong S^3_n(\gamma), \]
so $|H_1(\dcover(L^\beta_0))| = |n|$ and we can write the conclusion of Lemma~\ref{lem:bdc-trace} more simply as
\[ \tr \rho(\beta) = 2\pm n. \]

\begin{lemma} \label{lem:rho-beta-from-l-beta}
Let $\beta$ be a 3-braid such that the link $L^\beta$ of Figure~\ref{fig:E-T24-resolutions} satisfies
\[ \dcover(L^\beta) \cong S^3_{p/q}(U), \qquad 0 < q \leq p. \]
Let $\bar{q}$ be any integer with $q\cdot\bar{q} \equiv 1 \pmod{p}$, and write
\[ q \cdot \bar{q} = rp + 1 \]
for some $r\in\Z$.  Then either
\begin{equation} \label{eq:rho-beta-factor-1}
\rho(\beta) = (-1)^e \begin{pmatrix} 1 & 0 \\ k & 1 \end{pmatrix} \begin{pmatrix} \bar{q} & p \\ r & q \end{pmatrix} \begin{pmatrix} 1 & 0 \\ \ell & 1 \end{pmatrix} = \rho(\Delta^{4d+2e} y^{-k}) \begin{pmatrix} \bar{q} & p \\ r & q \end{pmatrix} \rho(y^{-\ell})
\end{equation}
or
\begin{equation} \label{eq:rho-beta-factor-2}
\rho(\beta) = (-1)^e \begin{pmatrix} 1 & 0 \\ k & 1 \end{pmatrix} \begin{pmatrix} q & p \\ r & \bar{q} \end{pmatrix} \begin{pmatrix} 1 & 0 \\ \ell & 1 \end{pmatrix} = \rho(\Delta^{4d+2e}y^{-k}) \begin{pmatrix} q & p \\ r & \bar{q} \end{pmatrix} \rho(y^{-\ell})
\end{equation}
where $d\in\Z$ and $e\in\{0,1\}$.
\end{lemma}

\begin{proof}
Suppose that we have
\[ \rho(\beta) = \begin{pmatrix} a & b \\ c & d \end{pmatrix}. \]
Then by Lemma~\ref{lem:bdc-2-bridge} and the classification of lens spaces up to orientation-preserving homeomorphism, we must have
\[ \begin{pmatrix} b \\ d \end{pmatrix} = \pm \begin{pmatrix} p \\ q+kp \end{pmatrix} \text{ or } \pm \begin{pmatrix} p \\ \bar{q}+kp \end{pmatrix} \]
for some $k\in\Z$.  In this case, since $\det \rho(\beta) = 1$, we know that $\rho(\beta)$ must have the form
\[ \rho(\beta) = \pm\begin{pmatrix} \bar{q}+\ell p & p \\ r+k\bar{q} + \ell(q+kp) & q+kp \end{pmatrix}
\text{ or }
\pm\begin{pmatrix} q+\ell p & p \\ r+kq + \ell(\bar{q}+kp) & \bar{q}+kp \end{pmatrix}
\]
for some integers $k$ and $\ell$.  These matrices factor exactly as in \eqref{eq:rho-beta-factor-1} and \eqref{eq:rho-beta-factor-2}, completing the proof.
\end{proof}

In either case of Lemma~\ref{lem:rho-beta-from-l-beta}, we have 
\begin{equation} \label{eq:trace-rho-beta}
\tr \rho(\beta) = (-1)^e(q+\bar{q} + (k+\ell)p),
\end{equation}
which by Lemma~\ref{lem:bdc-trace} is equal to $2 \pm |H_1(\dcover(L^\beta_0);\Z)|$.  In other words, we must have
\begin{equation} \label{eq:beta-trace-constraint}
(-1)^e(q+\bar{q}+(k+\ell)p) = 2 \pm n,
\end{equation}
which will be useful in the following subsections.

\subsection{The case $\gamma=U$} \label{ssec:gamma-unknot}

For now we suppose that the curve $\gamma \subset \dcover(U) \cong S^3$ from Subsection~\ref{ssec:bdc-beta} is unknotted.  We recall from \eqref{eq:dcover-lbeta} that $\dcover(L^\beta) \cong S^3_{(2n+1)/2}(U)$ for some integer $n$.  Thus in Lemma~\ref{lem:rho-beta-from-l-beta} we can take
\[ (p,q,\bar{q},r) = (2n+1, 2,n+1,1) \text{ or } (2n+1,n+1,2,1). \]
This gives
\begin{equation} \label{eq:rho-beta-u1}
\begin{aligned}
\rho(\beta) &= \rho(\Delta^{4d+2e}y^{-k}) \begin{pmatrix} n+1 & 2n+1 \\ 1 & 2 \end{pmatrix} \rho(y^{-\ell}) \\
&= \rho(\Delta^{4d+2e}y^{-k}) \begin{pmatrix} 1 & n \\ 0 & 1 \end{pmatrix} \begin{pmatrix} 1 & 0 \\ 1 & 1 \end{pmatrix} \begin{pmatrix} 1 & 1 \\ 0 & 1 \end{pmatrix} \rho(y^{-\ell}) \\
&= \rho(\Delta^{4d+2e}y^{-k} x^n y^{-1} x y^{-\ell})
\end{aligned}
\end{equation}
in the first case, and
\begin{equation} \label{eq:rho-beta-u2}
\begin{aligned}
\rho(\beta) &= \rho(\Delta^{4d+2e}y^{-k}) \begin{pmatrix} 2 & 2n+1 \\ 1 & n+1 \end{pmatrix} \rho(y^{-\ell}) \\
&= \rho(\Delta^{4d+2e}y^{-k}) \begin{pmatrix} 1 & 1 \\ 0 & 1 \end{pmatrix} \begin{pmatrix} 1 & 0 \\ 1 & 1 \end{pmatrix} \begin{pmatrix} 1 & n \\ 0 & 1 \end{pmatrix} \rho(y^{-\ell}) \\
&= \rho(\Delta^{4d+2e} y^{-k} x y^{-1} x^n y^{-\ell} )
\end{aligned}
\end{equation}
in the second.

In each of \eqref{eq:rho-beta-u1} and \eqref{eq:rho-beta-u2}, the braid $\beta$ is uniquely determined up to the value of $d\in\Z$, since Lemma~\ref{lem:kernel-rho} says that $\Delta^4$ generates $\ker(\rho)$.  In fact, we can disregard the braids arising from \eqref{eq:rho-beta-u2}, because up to conjugation by powers of $y$, they are all obtained by \emph{reversing} the braids from \eqref{eq:rho-beta-u1}: we have
\begin{align*}
r(\Delta^{4d+2e}y^{-k} x^ny^{-1}xy^{-\ell}) &= \Delta^{4d+2e} y^{-\ell} x y^{-1} x^n y^{-k} \\
&= y^{k-\ell} \cdot \left(\Delta^{4d+2e} y^{-k} x y^{-1} x^n y^{-\ell}\right) \cdot y^{-(k-\ell)}.
\end{align*}
It follows by Lemmas~\ref{lem:reverse-braid} and \ref{lem:y-conjugation} that every knot $K$ with $M_F \cong S^3 \setminus N(T_{2,4})$ and $\gamma$ unknotted has the form $K \cong K_\beta$, where
\[ \beta = \Delta^{4d+2e} y^{-k} x^n y^{-1} x y^{-\ell} \]
is one of the braids in \eqref{eq:rho-beta-u1}.  Recalling that $\Delta^2$ generates the center of $B_3$, we can now rewrite them as
\begin{equation} \label{eq:rho-beta-u-conj}
y^k \beta y^{-k} = \Delta^{4d+2e} x^n y^{-1} x y^{-(k+\ell)}
\end{equation}
with Lemma~\ref{lem:y-conjugation} in mind.

\begin{lemma} \label{lem:pin-down-d}
Suppose that $\beta$ is a 3-braid of the form \eqref{eq:rho-beta-u-conj}, and that its closure $L_0^\beta = \hat\beta$ has branched double cover $S^3_n(U)$.  Then the following must be true:
\begin{itemize}
\item If $n\neq \pm1$ then $6(2d+e) = (k+\ell)\pm1$.
\item If $n = \pm1$ then $6(2d+e) + n - (k+\ell) \in \{-2,0,2\}$.
\end{itemize}
\end{lemma}

\begin{proof}
Since $S^3_n(U)$ is a lens space, Hodgson and Rubinstein \cite{hodgson-rubinstein} proved that it is the branched double cover of exactly one link in $S^3$, which we know to be the $(2,n)$ torus link.  Thus $\hat\beta \cong T_{2,n}$, and so Birman and Menasco \cite{birman-menasco-iii} proved that up to conjugacy, we have
\[ \beta \sim \begin{cases} x^n y^{\pm1}, & n \neq \pm1 \\ xy,\ xy^{-1},\ \text{or } x^{-1}y^{-1}, & n=\pm1. \end{cases} \]
Now we can read from \eqref{eq:rho-beta-u-conj} that $\beta$ has exponent sum
\[ \varepsilon(\beta) = 6(2d+e) + n - (k+\ell), \]
where $\varepsilon: B_3 \to \Z$ is the homomorphism defined by $\varepsilon(x)=\varepsilon(y)=1$.  This exponent sum is invariant under conjugation, so $\varepsilon(\beta)$ must also be equal to $n\pm1$ if $n\neq \pm1$ and one of $2,0,-2$ otherwise.  The lemma follows immediately.
\end{proof}

\begin{proposition} \label{prop:gamma-unknot}
Let $\beta \in B_3$ be a braid for which $U=\tau\cup \beta$ is unknotted and the curve $\gamma \subset \dcover(U) \cong S^3$ is also unknotted.
Up to reversal, there is some integer $a\in\Z$ such that $y^a \beta y^{-a}$ is one of the 3-braids
\[ x^{-1},\ xy,\ \text{or }x^ny^{-1}xy \ (n\in\Z). \]
\end{proposition}

\begin{proof}
As discussed above, it suffices to consider $\beta$ as in \eqref{eq:rho-beta-u1}.  We fix $n\in\Z$ so that $\dcover(L^\beta)$, $\dcover(L^\beta_0)$, and $\dcover(L^\beta_1)$ are all surgeries on $\gamma$ of slopes $\frac{2n+1}{2}$, $n$, and $n+1$ respectively, as guaranteed by \eqref{eq:dcover-lbeta} and \eqref{eq:dcover-lbeta-i}.  Then in particular $\dcover(L^\beta_0) \cong S^3_n(U)$, with first homology of order $|n|$, so Lemma~\ref{lem:bdc-trace} now says that
\[ 2 \pm |n| = \tr(\rho(\beta)) = (-1)^e(n+3+(k+\ell)(2n+1)) \]
for $\beta$ as in \eqref{eq:rho-beta-u1}.  After multiplying through by $(-1)^e$, we have four cases, where in each case we can determine the value of $e\in\{0,1\}$ from the sign of the constant term $(-1)^e\cdot 2$.  These cases are:

\vspace{1em}
\noindent\underline{Case 1}: $n+3 + (k+\ell)(2n+1) = n + 2$, so $e=0$.

This simplifies to $(k+\ell)(2n+1) = -1$, so $(k+\ell,n)$ is either $(1,-1)$ or $(-1,0)$.  Then \eqref{eq:rho-beta-u-conj} becomes
\begin{align*}
y^k \beta y^{-k} &= \Delta^{4d} x^{-1} y^{-1} x y^{-1} = \Delta^{4d} yx^{-1}y^{-2} \\
\text{or }
y^k \beta y^{-k} &= \Delta^{4d} y^{-1} x y, % k+\ell=-1, n = 0
\end{align*}
respectively, where we have simplified the first braid using the relation $x^{-1}y^{-1}x = yx^{-1}y^{-1}$.  Lemma~\ref{lem:pin-down-d} says that $d=0$ in each case, so now \eqref{eq:rho-beta-u-conj} becomes
\begin{equation} \label{eq:beta-U-case1}
y^a\beta y^{-a} = x^{-1} y^{-1} \text{\ or\ } x
\end{equation}
for some $a\in\Z$.

\vspace{1em}
\noindent\underline{Case 2}: $n+3 + (k+\ell)(2n+1) = -n + 2$, so $e=0$.

After rearranging we get
\[ (k+\ell+1)(2n+1)=0, \]
and $2n+1$ is nonzero so we must have $k+\ell=-1$.  Now we apply Lemma~\ref{lem:pin-down-d} to see that if $n\neq\pm1$ then $12d = -1\pm1$, while if $n=\pm1$ then $12d + (\pm1) - (-1) \in \{-2,0,2\}$.  Thus in either case $d=0$, and so \eqref{eq:rho-beta-u-conj} becomes
\begin{equation} \label{eq:beta-U-case2}
y^k\beta y^{-k} = x^n y^{-1} x y.
\end{equation}

\vspace{1em}
\noindent\underline{Case 3}: $n+3 + (k+\ell)(2n+1) = n - 2$, so $e=1$.

This simplifies to $(k+\ell)(2n+1) = -5$, so $(k+\ell,n)$ is one of $(5,-1)$, $(-5,0)$, $(1,-3)$, or $(-1,2)$. In each of these cases, equation \eqref{eq:rho-beta-u-conj} and Lemma~\ref{lem:pin-down-d} give us
\begin{align*}
y^k \beta y^{-k} &= \Delta^{4d+2} x^{-1} y^{-1} x y^{-5} = \Delta^{4d+2} yx^{-1}y^{-6}, && 6(2d+1)-1-5 \in \{-2,0,2\} \\
y^k \beta y^{-k} &= \Delta^{4d+2} y^{-1} x y^{5}, && 6(2d+1) = -5 \pm 1 \\
y^k \beta y^{-k} &= \Delta^{4d+2} x^{-3} y^{-1} x y^{-1}, && 6(2d+1) = 1 \pm 1\\
\text{or }
y^k \beta y^{-k} &= \Delta^{4d+2} x^2 y^{-1} x y, && 6(2d+1) = -1 \pm 1,
\end{align*}
respectively.  The third and fourth braids are ruled out by Lemma~\ref{lem:pin-down-d} because there is no such $d\in\Z$, whereas the first and second braids must have $d=0$ and $d=-1$ respectively.  Thus in the first case we have
\begin{align*}
y^k\beta y^{-k} = \Delta^2 yx^{-1}y^{-6} &= y \cdot yxyxyx \cdot x^{-1}y^{-6} \\
&= y^2 \cdot xyx \cdot y^{-5} = y^2 \cdot yxy \cdot y^{-5} \\
&= y^3 \cdot xy^{-1} \cdot y^{-3},
\end{align*}
while we can rearrange the second case to get
\begin{align*}
y^{k+1} \beta y^{-(k+1)} = (xyxyxy)^{-1} xy^4 &= y^{-1}(xyx)^{-1} y^{-1} x^{-1} \cdot xy^4 \\
&= y^{-1}(yxy)^{-1} y^3 \\
&= y^{-2} x^{-1} y^{2}.
\end{align*}
Thus up to conjugation by powers of $y$, the possible braids in this case are
\begin{equation} \label{eq:beta-U-case3}
y^a\beta y^{-a} = xy^{-1} \text{ or } x^{-1}.
\end{equation}

\vspace{1em}
\noindent\underline{Case 4}: $n+3 + (k+\ell)(2n+1) = -n - 2$, so $e=1$.

This condition is equivalent to
\[ (k+\ell+1)(2n+1) = -4, \]
and $2n+1$ is odd so it must be $\pm1$, hence $(k+\ell,n)$ is either $(-5,0)$ or $(3,-1)$.  The first of these already appeared in case 3, leading to
\[ y^{k+1}\beta y^{-(k+1)} = y^{-2}x^{-1}y^2. \]
In the second case, equation \eqref{eq:rho-beta-u-conj} becomes
\[ y^k \beta y^{-k} = \Delta^{4d+2} x^{-1} y^{-1} x y^{-3} = \Delta^{4d+2} y x^{-1} y^{-4}, \]
while Lemma~\ref{lem:pin-down-d} says that $6(2d+1)+(-1)-3 \in \{-2,0,2\}$, hence $d=0$.  Thus
\begin{align*}
y^{k-1} \beta y^{1-k} = \Delta^2 x^{-1}y^{-3} &= yxyxyx \cdot x^{-1}y^{-3} \\
&= y\cdot xyx\cdot y^{-2} = y\cdot yxy \cdot y^{-2} \\
&= y^2 \cdot xy \cdot y^{-2}.
\end{align*}
Thus in this case the possible braids all have the form
\begin{equation} \label{eq:beta-U-case4}
y^a\beta y^{-a} = x^{-1} \text{\ or\ } xy.
\end{equation}

\vspace{1em}
We now combine the lists of braids enumerated in \eqref{eq:beta-U-case1}, \eqref{eq:beta-U-case2}, \eqref{eq:beta-U-case3}, and \eqref{eq:beta-U-case4} to see that for some $a\in\Z$, the braid $y^a \beta y^{-a}$ must be one of
\[ x,\ x^{-1},\ xy,\ xy^{-1},\ x^{-1}y^{-1},\ \text{or }x^ny^{-1}xy \ (n\in\Z). \]
But we can eliminate $xy^{-1}$ and $x^{-1}y^{-1}$ from this list, because filling the tangle $\tau$ in with either of these produces a right-handed trefoil, as shown in Figure~\ref{fig:braids-xY-XY}.
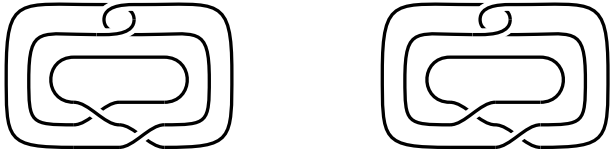
\begin{figure}
\begin{tikzpicture}
\begin{scope}
\draw[link] (0,0.3) to[out=180,in=180] ++(0,0.6) -- ++(1.2,0) to[out=0,in=0] ++(0,-0.6);
\draw[link] (0,0) to[out=180,in=270] ++(-0.6,0.6) to[out=90,in=180] ++(0.9,0.6) ++(0.1,0) to[out=0,in=270,looseness=1] ++(0.4,0.2) to[out=90,in=0,looseness=1] ++(-0.4,0.2) -- ++(-0.1,0) to[out=180,in=90] ++(-1.2,-0.9) to[out=270,in=180] ++(0.9,-1);
\draw[link] (1.2,0) to[out=0,in=270] ++(0.6,0.6) to[out=90,in=0] ++(-0.9,0.6) -- ++(-0.1,0) to[out=180,in=270,looseness=1] ++(-0.4,0.2) to[out=90,in=180,looseness=1] ++(0.4,0.2) -- ++(0.1,0) to[out=0,in=90] ++(1.2,-0.9) to[out=270,in=0] ++(-0.9,-1);
\draw[link] (0.8,1.4) to[out=270,in=0,looseness=1] ++(-0.4,-0.2) -- ++(-0.1,0); % redraw a crossing
\draw[link,looseness=0.75] (0,0) to[out=0,in=180] ++(0.6,0.3) ++(0,-0.3) to[out=0,in=180] ++(0.6,-0.3);
\draw[link,looseness=0.75] (0,0.3) to[out=0,in=180] ++(0.6,-0.3) ++(0,-0.3) to[out=0,in=180] ++(0.6,0.3);
\draw[link] (0,-0.3) -- ++(0.6,0) ++(0,0.6) -- ++(0.6,0);
\end{scope}

\begin{scope}[xshift=5cm]
\draw[link] (0,0.3) to[out=180,in=180] ++(0,0.6) -- ++(1.2,0) to[out=0,in=0] ++(0,-0.6);
\draw[link] (0,0) to[out=180,in=270] ++(-0.6,0.6) to[out=90,in=180] ++(0.9,0.6) ++(0.1,0) to[out=0,in=270,looseness=1] ++(0.4,0.2) to[out=90,in=0,looseness=1] ++(-0.4,0.2) -- ++(-0.1,0) to[out=180,in=90] ++(-1.2,-0.9) to[out=270,in=180] ++(0.9,-1);
\draw[link] (1.2,0) to[out=0,in=270] ++(0.6,0.6) to[out=90,in=0] ++(-0.9,0.6) -- ++(-0.1,0) to[out=180,in=270,looseness=1] ++(-0.4,0.2) to[out=90,in=180,looseness=1] ++(0.4,0.2) -- ++(0.1,0) to[out=0,in=90] ++(1.2,-0.9) to[out=270,in=0] ++(-0.9,-1);
\draw[link] (0.8,1.4) to[out=270,in=0,looseness=1] ++(-0.4,-0.2) -- ++(-0.1,0); % redraw a crossing
\draw[link,looseness=0.75] (0,0.3) to[out=0,in=180] ++(0.6,-0.3) to[out=0,in=180] ++(0.6,-0.3);
\draw[link,looseness=0.75] (0,0) to[out=0,in=180] ++(0.6,+0.3) ++(0,-0.6) to[out=0,in=180] ++(0.6,0.3);
\draw[link] (0,-0.3) -- ++(0.6,0) ++(0,0.6) -- ++(0.6,0);
\end{scope}
\end{tikzpicture}
\caption{The knot $\tau\cup\beta$ is a right-handed trefoil when $\beta$ is $xy^{-1}$ or $x^{-1}y^{-1}$.}
\label{fig:braids-xY-XY}
\end{figure}
The braid $x$ is also redundant, because if $y^a \beta y^{-a} = x$ then
\[ y^{a-1} \beta y^{-(a-1)} = y^{-1} x y = x^0 y^{-1} x y \]
belongs to the family $x^n y^{-1} xy$.  Thus we can remove it, and we are now left with exactly the list of braids promised in this proposition.
\end{proof}

\subsection{The case where $\gamma$ is a torus knot} \label{ssec:gamma-torus-knot}

In this subsection we will suppose that $\gamma \cong T_{p,q}$ for some $p$ and $q$.  Our goal is to prove the following.

\begin{proposition} \label{prop:gamma-torus}
Suppose that $\beta \in B^3$ is a 3-braid for which $U = \tau \cup \beta$ is unknotted, and the curve $\gamma \subset S^3$ is a nontrivial knot.  Then for some $a\in\Z$ we have
\[ y^a \beta y^{-a} = x^3 y^{-1} x^2 y \text{ or } x^{-3}yx^{-2} \]
up to braid reversal.
\end{proposition}

We recall from Subsection~\ref{ssec:bdc-beta} that $L^\beta$ is a 2-bridge link, so that
\[ S^3_{(2n+1)/2}(\gamma) \cong \dcover(L^\beta) \]
is a lens space.  The only half-integral lens space surgeries on $\gamma \cong T_{p,q}$ are those of slopes $pq \pm \frac{1}{2}$ \cite{moser}, so we must have $n+\frac{1}{2} \in \{pq-\frac{1}{2}, pq+\frac{1}{2}\}$, hence exactly one of
\begin{align}
n &= pq-1: &
S^3_n(\gamma) &\cong S^3_{(pq-1)/q^2}(U), &
S^3_{n+1}(\gamma) &\cong S^3_{p/q}(U) \# S^3_{q/p}(U) \label{eq:n-pq-1} \\
n &= pq: &
S^3_n(\gamma) &\cong S^3_{p/q}(U) \# S^3_{q/p}(U), &
S^3_{n+1}(\gamma) &\cong S^3_{(pq+1)/q^2}(U). \label{eq:n-pq}
\end{align}
occurs.  These surgeries were determined by Moser \cite[Proposition~3.2]{moser}, though we follow the notational conventions of Gordon \cite[Corollary~7.4]{gordon}.

We now observe that whether $n=pq-1$ or $n=pq$, we have found a 3-braid $\beta' \in \{\beta,\beta y^{-1}\}$ whose closure has branched double cover
\[ \dcover(\widehat{\beta'}) \cong S^3_{p/q}(U) \# S^3_{q/p}(U), \]
which is not prime.  A theorem of Kim and Tollefson \cite{kim-tollefson} says that the link $\widehat{\beta'}$ is therefore a nontrivial connected sum
\[ \widehat{\beta'} \cong L_1 \# L_2, \]
where $\dcover(L_1) \cong S^3_{p/q}(U)$ and $\dcover(L_2) \cong S^3_{q/p}(U)$.  Now since $L_1 \# L_2$ has braid index at most $3$ and the summands $L_i$ are nontrivial, the ``braid index theorem'' of Birman and Menasco \cite{birman-menasco-iv} shows that $L_1$ and $L_2$ are each closures of 2-braids.  Thus we can write
\begin{align*}
L_1 &\cong T_{a,2}, & 
L_2 &\cong T_{b,2}, &
(a,b &\neq \pm1,0)
\end{align*}
where $a$ and $b$ cannot be $\pm1$ or $0$ because the branched double covers are nontrivial rational homology spheres, hence are neither $S^3$ nor $S^1\times S^2$.  Then we have
\begin{align*}
S^3_{p/q}(U) &\cong \begin{cases} S^3_{a/1}(U) & a>0 \\ S^3_{|a|/(|a|-1)}(U) & a < 0, \end{cases} &
S^3_{q/p}(U) &\cong \begin{cases} S^3_{b/1}(U) & b>0 \\ S^3_{|b|/(|b|-1)}(U) & b < 0. \end{cases} &
\end{align*}
In particular, this is only possible if $|p| = |a|$ and $|q| = |b|$, and if moreover
\begin{align*}
|q| &\equiv \pm1 \pmod{|p|}, &
|p| &\equiv \pm1 \pmod{|q|}.
\end{align*}

\begin{lemma} \label{lem:p-pm1-q}
Let $P,Q \geq 2$ be coprime positive integers satisfying
\[ P \equiv \pm1 \pmod{Q} \quad\text{and}\quad Q \equiv \pm1 \pmod{P}. \]
Then $P = Q \pm 1$.
\end{lemma}

\begin{proof}
Write $P = kQ \pm1$, where $P,Q \geq 2$ implies that $k \geq 1$.  If $k \geq 2$ then we have $P \geq 2Q-1$, so either $P=3$ and then $Q=2$ (hence $P=Q+1$), or $P > 3$ and then we have
\[ 1 < Q \leq \frac{P+1}{2} < P-1. \]
(The last two inequalities are equivalent to $P \geq 2Q-1$ and $P > 3$ respectively.)  But if $1 < Q < P-1$ then we cannot possibly have $Q \equiv \pm1 \pmod{P}$, so there are no other solutions with $k \geq 2$ and thus we must have $P=Q\pm1$.
\end{proof}

Lemma~\ref{lem:p-pm1-q} says that for $\gamma \cong T_{p,q}$, if we write $P=|p|$ and $Q=|q|$ then $P = Q \pm 1$, and \eqref{eq:n-pq-1} and \eqref{eq:n-pq} tell us that either
\[ S^3_{(pq-1)/q^2}(U) \quad\text{or}\quad S^3_{(pq+1)/q^2}(U) \]
is the branched double cover of a 3-braid, depending on whether $n=pq-1$ or $n=pq$ respectively.  Reversing orientation if exactly one of $p$ and $q$ is negative replaces that 3-braid with its mirror, which is still a 3-braid, and the surgered manifold is then
\[ -S^3_{(pq\pm1)/q^2}(U) \cong S^3_{(-pq\mp1)/q^2}(U) \cong S^3_{(PQ \mp 1)/Q^2}(U), \]
so in any case we see that one of
\[ S^3_{(PQ-1)/Q^2}(U) \quad\text{or}\quad S^3_{(PQ+1)/Q^2}(U) \]
is the branched double cover of a 3-braid.  This gives us strong restrictions on $P$ and $Q$ by the following result of Murasugi.

\begin{proposition}[{\cite[Proposition~7.2]{murasugi-braid}}] \label{prop:murasugi-criterion}
Let $L_{r/s}$ be the 2-bridge link with branched double cover $L(r,s) = S^3_{r/s}(U)$, where $0<s<r$ and $s$ is odd.  Then $L_{r/s}$ has braid index $2$ if and only if $s=1$, and it has braid index $3$ if and only if either
\begin{enumerate}
\item there are integers $c,d > 0$ such that $(r,s) = (2cd+3c+3d+4, 2c+3)$, or
\item there are $c,d>0$ such that $(r,s) = (2cd+c+d+1, 2c+1)$.
\end{enumerate}
\end{proposition}

\begin{remark} \label{rem:murasugi-corollary}
We note that in the first and second cases of Proposition~\ref{prop:murasugi-criterion} we have
\[ r = \frac{(2d+3)s - 1}{2} \quad\text{and}\quad r = \frac{(2d+1)s + 1}{2}, \]
respectively, so if $L_{r/s}$ has braid index $3$ then $s$ divides either $2r+1$ or $2r-1$.  In particular, if the braid index is \emph{at most} $3$ then we can draw the same conclusion, since braid index $2$ implies $s=1$.
\end{remark}

Putting all of this together, we can now show the following.

\begin{lemma} \label{lem:possible-gamma-torus}
Suppose that the link $\tau \cup \beta$ is unknotted, and that $\gamma$ is not an unknot.  Then $\gamma$ or its mirror must be one of the torus knots $T_{2,3}$, $T_{3,4}$, or $T_{4,5}$, and 
$(\gamma,n)$ must be one of the pairs indicated in Table~\ref{fig:gamma-table}.
\end{lemma}

\begin{table}
\[ \arraycolsep=1em
\begin{array}{c||cc|cc|cc}
\gamma & T_{2,3} & T_{-2,3} & T_{3,4} & T_{-3,4} & T_{4,5} & T_{-4,5} \\[0.25em]
\hline &&&&&&\\[-0.75em]
n & 5,6 & -6,-7 & 12 & -13 & 19 & -20
\end{array}
\]
\caption{Possible torus knots $\gamma$ and the associated values of $n$ for which $\dcover(L^\beta)$ is $\frac{2n+1}{2}$-surgery on $\gamma$, as tabulated in Lemma~\ref{lem:possible-gamma-torus}.}
\label{fig:gamma-table}
\end{table}

\begin{proof}
Lemma~\ref{lem:possible-gamma} says that $\gamma \cong T_{p,q}$ for some $p$ and $q$, and we have argued that if $P=|p|$ and $Q=|q|$ then $P=Q\pm1$; without loss of generality, we write $P=Q+1 \geq 3$.  We consider each parity of $P$ separately, and determine in each case which lens space $S^3_{(PQ\pm1)/Q^2}(U)$ must arise as the branched double cover of a 3-braid.  Up to orientation, we know that the corresponding $S^3_{(pq\pm1)/q^2}(U)$ is either $n$-surgery (i.e., $\dcover(L_0^\beta)$) or $(n+1)$-surgery (i.e., $\dcover(L_1^\beta)$) on $\gamma$, so the value of $n$ follows immediately and then the precise lens spaces are determined by the relations
\begin{align*}
S^3_{pq\pm1}(T_{p,q}) &\cong S^3_{(pq\pm1)/q^2}(U), &
S^3_{pq}(T_{p,q}) &= S^3_{p/q}(U) \# S^3_{q/p}(U)
\end{align*}
and the relations $S^3_{r/s}(U) \cong S^3_{r/(s+kr)}(U)$ and $S^3_{r/s}(U) \cong -S^3_{-r/s}(U)$ for all $r,s,k$.

\vspace{1em}
\noindent \underline{Case 1}:
$P$ is odd.  Then $Q$ is even, so if $\epsilon = \pm1$ then
\[ L(PQ+\epsilon, Q^2) = L(Q^2+Q+\epsilon, Q^2) \cong -L(Q^2+Q+\epsilon, Q+\epsilon) \]
and $Q+\epsilon$ is odd.  According to Murasugi's result, and in particular Remark~\ref{rem:murasugi-corollary}, it follows that
\[ s = Q+\epsilon \]
divides one of
\[ 2r \pm 1 = 2(Q^2+Q+\epsilon) \pm 1, \]
hence it also divides
\[ (2r \pm 1 - 2s) - 2s(Q-\epsilon) = (2Q^2 \pm 1) - 2(Q^2-1) = 2\pm1. \]
Thus $s$ must be either $1$ or $3$.  Then $2 \leq Q = s-\epsilon$ says that $(P,Q)$ is either $(3,2)$ or $(5,4)$.  We determine the following possibilities:
\begin{itemize}
\item If $s=1$ then $(P,Q)=(3,2)$ and $\epsilon = -1$, so the lens space in question is $L(5,4)$.

\item If $s=3$ and $\epsilon=+1$ then $(P,Q) = (3,2)$ and the lens space is $L(7,4)$.

\item If $s=3$ and $\epsilon=-1$ then $(P,Q) = (5,4)$ and the lens space is $L(19,16)$.
\end{itemize}

\vspace{1em}
\noindent\underline{Case 2}:
$P$ is even.  Then $Q$ is odd, so if $\epsilon=\pm1$ then 
\[ L(PQ+\epsilon, Q^2) = L(Q^2+Q+\epsilon, Q^2) \]
(with $Q^2$ odd) arises as the branched double cover of a 3-braid closure if
\[ s = Q^2 \quad\text{divides}\quad 2r\pm1 = 2(Q^2+Q+\epsilon)\pm1. \]
This is equivalent to $Q^2$ dividing $2Q + (2\epsilon \pm1) \leq 2Q+3$, but given that $Q$ is odd and $Q \geq 2$, we have $Q^2 > 2Q+3$ unless $Q=3$.  So $(P,Q) = (4,3)$ and $\epsilon=1$, and the lens space in question must be $L(13,9)$.

\vspace{1em}
This completes the identification of the lens spaces in question when $\gamma = T_{p,q}$ and $p,q$ are both positive.  If one of $p$ and $q$ is negative, then we can apply the same argument to the mirror of $\gamma$ to determine the value of $-n$ and the proposition follows.
\end{proof}

In fact, we can rule out most of the pairs $(\gamma,n)$ appearing in Lemma~\ref{lem:possible-gamma-torus} as well.

\begin{lemma} \label{lem:gamma-trefoil}
If $\gamma$ is a nontrivial torus knot, then $(\gamma,n,k+\ell,e)$ is either
\[ (T_{2,3},5,-2,1) \quad\text{or}\quad (T_{-2,3},-6,-1,0). \]
\end{lemma}

\begin{proof}
In Table~\ref{fig:gamma-table-traces} we tabulate the possible pairs $(\gamma,n)$ from Lemma~\ref{lem:possible-gamma-torus}, together with
\begin{itemize}
\item the corresponding lens spaces
\[ \dcover(L^\beta) \cong S^3_{(2n+1)/2}(\gamma) \cong L(p,q) := S^3_{p/q}(U) \]
for some integers $p$ and $q$;
\item the integers $p$ and $q$, as well as $\bar{q}$ such that $q\cdot\bar{q} \equiv 1 \pmod{p}$; and
\item the resulting trace of $\rho(\beta)$, as determined by \eqref{eq:trace-rho-beta}, given that Lemma~\ref{lem:rho-beta-from-l-beta} says that $\rho(\beta)$ must have one of the two forms \eqref{eq:rho-beta-factor-1} or \eqref{eq:rho-beta-factor-2}.
\end{itemize}
\begin{table}
\[ \arraycolsep=0.75em
\begin{array}{ccccccc}
\gamma & n & \dcover(L^\beta) & p & q & \bar{q} & \tr \rho(\beta) \\[0.25em] \hline \\[-0.75em]
T_{2,3} & 5 & L(11,8) & 11 & 8 & 7 & (-1)^e(15+11(k+\ell)) \\[0.25em]
T_{-2,3} & -6 & L(11,3) & 11 & 3 & 4 & (-1)^e(7+11(k+\ell)) \\[0.25em] \hline \\[-0.75em]

T_{2,3} & 6 & L(13,8) & 13 & 8 & 5 & (-1)^e(13+13(k+\ell)) \\[0.25em]
T_{-2,3} & -7 & L(13,5) & 13 & 5 & 8 & (-1)^e(13+13(k+\ell)) \\[0.25em] \hline \\[-0.75em]

T_{3,4} & 12 & L(25,18) & 25 & 18 & 7 & (-1)^e(25+25(k+\ell)) \\[0.25em]
T_{-3,4} & -13 & L(25,7) & 25 & 7 & 18 & (-1)^e(25+25(k+\ell)) \\[0.25em] \hline \\[-0.75em]

T_{4,5} & 19 & L(39,32) & 39 & 32 & 11 & (-1)^e(43+39(k+\ell)) \\[0.25em]
T_{-4,5} & -20 & L(39,7) & 39 & 7 & 28 & (-1)^e(35+39(k+\ell))
\end{array} \]
\caption{Possible values of $\tr \rho(\beta)$ for each torus knot $\gamma$ and integer $n$.}
\label{fig:gamma-table-traces}
\end{table}
The lens spaces $\dcover(L^\beta)$ in Table~\ref{fig:gamma-table-traces} are determined by the formulas
\[ S^3_{(2rs\pm 1)/2}(T_{r,s}) \cong S^3_{(2rs\pm1)/(2r^2)}(U), \]
which again follow from \cite{moser} or \cite{gordon}.

Lemma~\ref{lem:bdc-trace} tells us that $\tr \rho(\beta) = 2\pm n$, so we inspect Table~\ref{fig:gamma-table-traces} to see whether this is possible.  We have
\begin{align*}
(\gamma,n) &= (T_{2,3},6): & \tr \rho(\beta) &\equiv 0 \pmod{13}, & 2\pm n &\equiv 8,9\pmod{13} \\
(\gamma,n) &= (T_{3,4},12): & \tr \rho(\beta) &\equiv 0 \pmod{25}, & 2\pm n &\equiv 14,15\pmod{25} \\
(\gamma,n) &= (T_{4,5},19): & \tr \rho(\beta) &\equiv 4,35 \pmod{39}, & 2\pm n &\equiv 21,22\pmod{39} %\\
\end{align*}
and the computations for $(T_{-2,3},-7)$, $(T_{-3,4},-13)$, and $(T_{-4,5},-20)$ are identical, so there is no solution in any of these cases.  This leaves only
\begin{align*}
(\gamma,n) &= (T_{2,3},5): & (-1)^e(15+11(k+\ell)) = 2\pm 5
\end{align*}
with solution $(k+\ell,e) = (-2,1)$, and
\begin{align*}
(\gamma,n) &= (T_{-2,3},-6): & (-1)^e(7+11(k+\ell)) = 2\pm (-6)
\end{align*}
with solution $(k+\ell,e) = (-1,0)$.
\end{proof}

\begin{proposition} \label{prop:gamma-lht}
If $\gamma = T_{-2,3}$, then up to reversal, there is some integer $a$ such that
\[ y^a \beta y^{-a} = x^3 y^{-1} x^2 y. \]
\end{proposition}

\begin{proof}
In this case we have $(n,k+\ell,e) = (-6,-1,0)$ and $\dcover(L^\beta) = L(11,3)$ by Lemma~\ref{lem:gamma-trefoil}, so we can write
\[ \rho(y^k\beta y^{-k}) = \rho(\Delta^{4d}) \begin{pmatrix} 4 & 11 \\ 1 & 3 \end{pmatrix} \rho(y) \quad\text{or}\quad \rho(\Delta^{4d}) \begin{pmatrix} 3 & 11 \\ 1 & 4 \end{pmatrix} \rho(y) \]
by \eqref{eq:rho-beta-factor-1} and \eqref{eq:rho-beta-factor-2}.  We compute that 
\begin{align*}
\begin{pmatrix} 4 & 11 \\ 1 & 3 \end{pmatrix} &= \begin{pmatrix} 1 & 3 \\ 0 & 1 \end{pmatrix} \begin{pmatrix} 1 & 0 \\ 1 & 1 \end{pmatrix} \begin{pmatrix} 1 & 2 \\ 0 & 1 \end{pmatrix} = \rho(x^3 y^{-1} x^2) \\
\begin{pmatrix} 3 & 11 \\ 1 & 4 \end{pmatrix} &= \begin{pmatrix} 1 & 2 \\ 0 & 1 \end{pmatrix} \begin{pmatrix} 1 & 0 \\ 1 & 1 \end{pmatrix} \begin{pmatrix} 1 & 3 \\ 0 & 1 \end{pmatrix} = \rho(x^2 y^{-1} x^3),
\end{align*}
and since $\ker(\rho)$ is generated by $\Delta^4$ it follows that
\[ y^k\beta y^{-k} = \Delta^{4d} x^3y^{-1}x^2y \text{ or } \Delta^{4d} x^2y^{-1}x^3y \]
for some $d\in \Z$.  These two families of braids are reverses of each other, since 
\[ \beta = \Delta^{4d} y^{-k} x^2 y^{-1} x^3 y^{k+1} \quad\Longrightarrow\quad r(\beta) = \Delta^{4d} y^{k+1} (x^3y^{-1} x^2 y)y^{-(k+1)}, \]
so we need only consider the first family, namely
\[ \beta = \Delta^{4d} y^{-k} x^3 y^{-1} x^2 y^{k+1}. \]

In order to determine $d$, we recall that the link $L^\beta_1$ from Figure~\ref{fig:E-T24-resolutions} is the closure of $\beta y^{-1}$, and by \eqref{eq:dcover-lbeta-i} we have
\[ \dcover(L^\beta_1) \cong S^3_{n+1}(\gamma) = S^3_{-5}(T_{-2,3}) \cong S^3_{-5/4}(U) \cong L(5,1). \]
As a lens space, this must be the branched double cover of a unique knot \cite{hodgson-rubinstein}, so we have $\widehat{\beta y^{-1}} \cong T_{2,5}$.  Then Birman and Menasco's classification theorem from \cite{birman-menasco-iii} says that $\beta y^{-1}$ must be conjugate to $x^5 y^{\pm1}$, so that $\beta$ has exponent sum
\[ \varepsilon(\beta) = \varepsilon(\beta y^{-1}) + 1 = 6 \pm 1. \]
On the other hand, we can read off the explicit form for $\beta$ above that $\varepsilon(\beta) = 12d + 5$, so we must have $d=0$.
\end{proof}

\begin{proposition} \label{prop:gamma-rht}
If $\gamma = T_{2,3}$, then up to reversal, there is some integer $a$ such that
\[ y^a \beta y^{-a} = x^{-3}y x^{-2}. \]
\end{proposition}

\begin{proof}
In this case we have $(n,k+\ell,e) = (5,-2,1)$ and $\dcover(L^\beta) = L(11,8)$ by Lemma~\ref{lem:gamma-trefoil}, so we can write
\[ \rho(y^k\beta y^{-k}) = \rho(\Delta^{4d+2}) \begin{pmatrix} 7 & 11 \\ 5 & 8 \end{pmatrix} \rho(y^2) \quad\text{or}\quad \rho(\Delta^{4d+2}) \begin{pmatrix} 8 & 11 \\ 5 & 7 \end{pmatrix} \rho(y^2) \]
by \eqref{eq:rho-beta-factor-1} and \eqref{eq:rho-beta-factor-2}.  We compute that 
\begin{align*}
\begin{pmatrix} 7 & 11 \\ 5 & 8 \end{pmatrix} &= \begin{pmatrix} 1 & 1 \\ 0 & 1 \end{pmatrix} \begin{pmatrix} 1 & 0 \\ 2 & 1 \end{pmatrix} \begin{pmatrix} 1 & 1 \\ 0 & 1 \end{pmatrix} \begin{pmatrix} 1 & 0 \\ 1 & 1 \end{pmatrix} \begin{pmatrix} 1 & 1 \\ 0 & 1 \end{pmatrix} = \rho(xy^{-2}xy^{-1}x) \\
\begin{pmatrix} 8 & 11 \\ 5 & 7 \end{pmatrix} &= \begin{pmatrix} 1 & 1 \\ 0 & 1 \end{pmatrix} \begin{pmatrix} 1 & 0 \\ 1 & 1 \end{pmatrix} \begin{pmatrix} 1 & 1 \\ 0 & 1 \end{pmatrix} \begin{pmatrix} 1 & 0 \\ 2 & 1 \end{pmatrix} \begin{pmatrix} 1 & 1 \\ 0 & 1 \end{pmatrix} = \rho(xy^{-1}xy^{-2}x),
\end{align*}
so now since $\Delta^4$ generates $\ker(\rho)$ we have
\[ y^k\beta y^{-k} = \Delta^{4d+2} xy^{-2}xy^{-1}xy^2 \text{ or } \Delta^{4d+2} xy^{-1}xy^{-2}xy^2 \]
for some $d\in\Z$.

In order to determine $d$, we note that the braid closure $\widehat\beta = L^\beta_0$ satisfies
\[ \dcover(L^\beta_0) \cong S^3_5(T_{2,3}) \cong S^3_{5/4}(U) \cong \dcover(T_{-2,5}), \]
so $L^\beta_0 \cong T_{-2,5}$ since every lens space is the branched double cover of a unique knot \cite{hodgson-rubinstein}.  Then $\beta$ must be conjugate to either $x^{-5}y$ or $x^{-5}y^{-1}$ \cite{birman-menasco-iii}, so its exponent sum is $\varepsilon(\beta) = -5 \pm 1$.  But in either of the above families we have $\varepsilon(\beta) = 12d+8$, so in fact $d=-1$.  Moreover, if we reverse the second family above then we get
\[ \beta = \Delta^{-2} y^{-k} xy^{-1}xy^{-2}xy^{k+2}  \quad\Longrightarrow\quad r(\beta) = y^{k+2} ( \Delta^{-2} xy^{-2}xy^{-1}x y^2 ) y^{-(k+2)}, \]
so up to reversal it suffices to consider only the family of braids
\[ y^k\beta y^{-k} = \Delta^{-2} xy^{-2}xy^{-1}xy^2. \]
We can simplify this somewhat by writing
\begin{align*}
y^k \beta y^{-k} &= y^{-1}x^{-1}y^{-1}x^{-1} y^{-1}x^{-1} \cdot xy^{-2}xy^{-1}xy^2 \\
&= y^{-1}x^{-1}y^{-1}\cdot \underbrace{x^{-1}y^{-3}x}_{=yx^{-3}y^{-1}} \cdot y^{-1}xy^2 \\
&= y^{-1}x^{-4}y^{-2}xy^2 \\ % so this gives x^{-4} y^{-2} xy up to conjugation
&= y^{-1}x^{-3} y \cdot \underbrace{y^{-1} x^{-1} y^{-2}}_{=x^{-2}y^{-1}x^{-1}} \cdot xy^2 \\
&= y^{-1} x^{-3} y x^{-2} y
\end{align*}
and so
\[ y^{k+1} \beta y^{-(k+1)} = x^{-3} y x^{-2} \]
as claimed.
\end{proof}

We can now complete the main result of this subsection.

\begin{proof}[Proof of Proposition~\ref{prop:gamma-torus}]
Lemmas~\ref{lem:possible-gamma-torus} and \ref{lem:gamma-trefoil} tell us that if $\gamma$ is knotted then it must be a trefoil.  If it is a left-handed trefoil then Proposition~\ref{prop:gamma-lht} says that up to reversal, $\beta$ is conjugate to $x^3 y^{-1} x^2 y$ by some power of $y$.  Otherwise it is a right-handed trefoil, so by Proposition~\ref{prop:gamma-rht}, either $\beta$ or its reverse is conjugate to $x^{-3} y x^{-2}$ by some power of $y$.
\end{proof}

\subsection{Some knots arising from specific braids} \label{ssec:E24-examples}

In this subsection we consider several families of 3-braids $\beta$ that arise in Propositions~\ref{prop:gamma-unknot} and \ref{prop:gamma-torus}, producing unknots when inserted into the tangle $\tau$ of Figure~\ref{fig:E-T24-isotopy}.  We will determine the corresponding nearly fibered knots $K = K_\beta$ which arise as lifts of $\kappa$ to $\dcover(U) \cong S^3$.  The results are summarized in Table~\ref{fig:unknot-braid-table};
\begin{table}
\[ \arraycolsep=1em
\begin{array}{c||ccc}
\beta & y^a x^{-1} y^a & y^ax^ny^{-1}xy^{1-a} & y^ax^3y^{-1}x^2y^{1-a} \\[0.25em] \hline \\[-0.75em]
K_\beta & 5_2 & P(-3,3,2n+1) & 15n_{43522}
\end{array}
\]
\caption{Some braids $\beta$ such that $\tau \cup \beta$ is unknotted, and the resulting knots $K=K_\beta$.}
\label{fig:unknot-braid-table}
\end{table}
the proofs in each case occupy Propositions~\ref{prop:recover-5_2}, \ref{prop:recover-pretzels}, and \ref{prop:recover-15n43522}, respectively.

\begin{proposition} \label{prop:recover-5_2}
The family of braids $\beta = y^a x^{-1} y^{-a}$ produces $K_\beta \cong 5_2$.
\end{proposition}

\begin{proof}
By Lemma~\ref{lem:y-conjugation} it suffices to take $a=0$, so $\beta = x^{-1}$.  We insert this into the tangle $\tau$ from Figure~\ref{fig:E-T24-isotopy}, apply an isotopy so that $U = \tau \cup \beta$ bounds a planar disk and $\kappa$ winds around it, and then cut $\kappa$ open along that disk and glue two copies together to construct the lift $K_\beta = \tilde\kappa$.  This process is illustrated in Figure~\ref{fig:braid-5_2}, where we isotope $U \cup \kappa$ into a convenient position and then take the branched double cover with respect to $U$ at the last step; the resulting diagram of $\tilde\kappa$ is isotopic to $5_2$ as claimed.
\begin{figure}
\begin{tikzpicture}
\begin{scope} % initial tangle
\draw[linkred] (1.0,0.6) arc (90:270:0.15 and 0.6);
\draw[link] (0,0.3) to[out=180,in=180] ++(0,0.5) -- ++(1.4,0) to[out=0,in=0] ++(0,-0.5) -- ++(-0.4,0);
\draw[link] (0,0) to[out=180,in=270] ++(-0.6,0.6) to[out=90,in=180] ++(0.9,0.6) to[out=0,in=270,looseness=1] ++(0.4,0.2) to[out=90,in=0,looseness=1] ++(-0.4,0.2) to[out=180,in=90] ++(-1.2,-0.9) to[out=270,in=180] ++(0.9,-1);
\draw[link] (1,0) to[out=0,in=270] ++(1,0.6) to[out=90,in=0,looseness=1.4] ++(-1.3,0.6) to[out=180,in=270,looseness=1] ++(-0.4,0.2) to[out=90,in=180,looseness=1] ++(0.4,0.2) to[out=0,in=90,looseness=1.4] ++(1.6,-0.9) to[out=270,in=0] ++(-1.3,-1);
\draw[link] (0.7,1.4) to[out=270,in=0,looseness=1] ++(-0.4,-0.2) -- ++(-0.1,0); % redraw a crossing
\draw[linkred] (1.0,0.6) arc (90:-90:0.15 and 0.6) node[below,red,inner sep=2pt] {\small$\kappa$};
%\draw[very thick,fill=white] (0,-0.45) rectangle (1,0.45);
%\node at (0.5,0) {\Large$\beta$};
\draw[link,looseness=0.75] (0,0.3) to[out=0,in=180] ++(0.6,-0.3) -- ++(0.4,0);
\draw[link,looseness=0.75] (0,0) to[out=0,in=180] ++(0.6,0.3) -- ++(0.4,0);
\draw[link] (0,-0.3) -- ++(1,0);
\end{scope}

\begin{scope}[xshift=4cm] % isotopy step 1
\draw[linkred] (0.85,0.5) arc (90:270:0.3 and 0.55);
\draw[link] (0,0.3) to[out=180,in=180] ++(0,0.4) -- ++(1.2,0) to[out=0,in=0] ++(0,-0.3) -- ++(-0.2,0) to[out=180,in=180] ++(0,-0.2);
\draw[link] (1.3,0.9) to[out=180,in=270] (0.7,1.4) to[out=90,in=0,looseness=1] ++(-0.4,0.2) to[out=180,in=90] ++(-1.2,-0.9) to[out=270,in=180] ++(0.9,-1);
\draw[link] (1,0) to[out=0,in=270] ++(1,0.6) to[out=90,in=0,looseness=1.4] ++(-1.3,0.6) to[out=180,in=270,looseness=1] ++(-0.4,0.2) to[out=90,in=180,looseness=1] ++(0.4,0.2) to[out=0,in=90,looseness=1.4] ++(1.6,-0.9) to[out=270,in=0] ++(-1.3,-1);
\draw[link] (1.3,0.9) to[out=180,in=270] (0.7,1.4); % redraw a crossing
\draw[linkred] (0.85,0.5) arc (90:-90:0.3 and 0.55) node[below,red,inner sep=2pt] {\small$\kappa$};
\draw[link] (1,0.2) -- ++(0.3,0) to [out=0,in=270] ++(0.3,0.35) to[out=90,in=0] ++(-0.3,0.35);
%%\draw[very thick,fill=white] (0,-0.45) rectangle (1,0.45);
%%\node at (0.5,0) {\Large$\beta$};
\draw[link,looseness=0.75] (0,0.3) to[out=0,in=180] ++(0.6,-0.3) -- ++(0.4,0);
%\draw[link,looseness=0.75] (0,0) to[out=0,in=180] ++(0.6,0.3) -- ++(0.4,0);
\draw[link] (0,-0.3) -- ++(1,0);
\end{scope}

\begin{scope}[xshift=8cm] % isotopy step 2
\draw[link] (0.7,-0.1) -- ++(-0.7,0) to[out=180,in=270,looseness=1] ++(-0.5,0.7) to[out=90,in=180] ++(0.5,0.6);
\draw[linkred] (0.85,0.5) arc (90:270:0.3 and 0.55);
\draw[link] (0,0.3) to[out=180,in=180] ++(0,0.4) -- ++(1.2,0) to[out=0,in=0] ++(0,-0.3) -- ++(-0.2,0) to[out=180,in=180] ++(0,-0.2);
\draw[link] (1.3,0.9) to[out=180,in=270] (0.7,1.4) to[out=90,in=0,looseness=1] ++(-0.4,0.2) to[out=180,in=90] ++(-1.2,-0.9) to[out=270,in=180] ++(0.9,-1);
\draw[link] (0,1.2) -- ++(0.3,0) to[out=0,in=180,looseness=1] ++(0.8,0.4) to[out=0,in=90] ++(1.2,-0.9) to[out=270,in=0] ++(-1.3,-1);
%\draw[link] (1.3,0.9) to[out=180,in=270] (0.7,1.4); % redraw a crossing
\draw[linkred] (0.85,0.5) arc (90:-90:0.3 and 0.55) node[below,red,inner sep=2pt] {\small$\kappa$};
\draw[link] (1,0.2) -- ++(0.3,0) to [out=0,in=270] ++(0.3,0.35) to[out=90,in=0] ++(-0.3,0.35);
%%\draw[very thick,fill=white] (0,-0.45) rectangle (1,0.45);
%%\node at (0.5,0) {\Large$\beta$};
\draw[link,looseness=0.75] (0,0.3) to[out=0,in=180] ++(0.6,-0.2) -- ++(0.1,0) to[out=0,in=0,looseness=2] ++(0,-0.2);
%\draw[link,looseness=0.75] (0,0) to[out=0,in=180] ++(0.6,0.3) -- ++(0.4,0);
\draw[link] (0,-0.3) -- ++(1,0);
\end{scope}

\begin{scope}[yshift=-3cm,xshift=-0.25cm] % isotopy step 3
\draw[link] (0.7,-0.1) -- ++(-0.7,0) to[out=180,in=270,looseness=1] ++(-0.5,0.7) to[out=90,in=180] ++(0.5,0.6);
\draw[link] (1.6,-0.05) to[out=270,in=0,looseness=1] ++(-0.8,-0.25) -- ++(-0.5,0) to[out=180,in=180,looseness=2] ++(0,-0.3);
\draw[linkred] (0.85,0.5) arc (90:270:0.3 and 0.65);
\draw[link] (0,0.3) to[out=180,in=180] ++(0,0.4) -- ++(1.2,0) to[out=0,in=0] ++(0,-0.3) -- ++(-0.2,0) to[out=180,in=180] ++(0,-0.2);
\draw[link] (0,1.2) -- ++(0.3,0) to[out=0,in=180,looseness=1] ++(0.8,0.4) to[out=0,in=90] ++(1.2,-0.9) to[out=270,in=0] ++(-1.3,-1.3);
\draw[linkred] (0.85,0.5) arc (90:-90:0.3 and 0.65) node[below,red,inner sep=2pt] {\small$\kappa$};
\draw[link] (1,0.2) -- ++(0.3,0) to [out=0,in=90,looseness=1] ++(0.3,-0.25);
%%\draw[very thick,fill=white] (0,-0.45) rectangle (1,0.45);
%%\node at (0.5,0) {\Large$\beta$};
\draw[link,looseness=0.75] (0,0.3) to[out=0,in=180] ++(0.6,-0.2) -- ++(0.1,0) to[out=0,in=0,looseness=2] ++(0,-0.2);
\draw[link] (0.3,-0.6) -- ++(0.7,0);
\end{scope}

\begin{scope}[yshift=-3cm,xshift=3.25cm] % isotopy step 4
\draw[link] (0.7,-0.1) -- ++(-0.7,0) to[out=180,in=180,looseness=1.5] ++(0,-1.15) -- ++(1,0) to[out=0,in=0,looseness=2] ++(0,0.65);
\draw[link] (1.6,-0.05) to[out=270,in=0,looseness=1] ++(-0.8,-0.25) -- ++(-0.5,0) to[out=180,in=180,looseness=2] ++(0,-0.3);
\draw[linkred] (0.85,0.5) arc (90:270:0.3 and 0.65);
\draw[link] (0,0.3) to[out=180,in=180] ++(0,0.4) -- ++(1.2,0) to[out=0,in=0] ++(0,-0.3) -- ++(-0.2,0) to[out=180,in=180] ++(0,-0.2);
\draw[linkred] (0.85,0.5) arc (90:-90:0.3 and 0.65) node[below,red,inner sep=2pt] {\small$\kappa$};
\draw[link] (1,0.2) -- ++(0.3,0) to [out=0,in=90,looseness=1] ++(0.3,-0.25);
%%\draw[very thick,fill=white] (0,-0.45) rectangle (1,0.45);
%%\node at (0.5,0) {\Large$\beta$};
\draw[link,looseness=0.75] (0,0.3) to[out=0,in=180] ++(0.6,-0.2) -- ++(0.1,0) to[out=0,in=0,looseness=2] ++(0,-0.2);
\draw[link] (0.3,-0.6) -- ++(0.7,0);
\end{scope}

\begin{scope}[yshift=-3cm,xshift=6.25cm] % redraw so U is clearly a circle
\draw[link] (0.3,0.5) coordinate (ec) ellipse (0.5 and 1);
\draw[linkred] (ec) ++ (0.3,0.4) -- ++(0.3,0) to[out=0,in=0] ++(0,0.3) -- ++(-1,0) to[out=180,in=180] ++(0,-0.9) -- node[pos=0.3,above,inner sep=2pt,red] {\small$\kappa$} ++(1,0) to[out=0,in=0] ++(0,-0.3) -- ++(-0.3,0) to[out=180,in=180] ++(0,-0.3) -- ++(0.3,0) to[out=0,in=0] ++(0,0.9) -- ++(-0.3,0) to[out=180,in=180] ++(0,0.3);
\foreach \x/\y in {-40/-20,15/35,105/150} {
  \draw[link] (ec) ++(\x:0.5 and 1) arc (\x:\y:0.5 and 1); % fix ellipse crossings
}
\end{scope}

\begin{scope}[yshift=-3cm,xshift=8.75cm] % move kappa, step 1
\draw[link] (0.3,0.5) coordinate (ec) ellipse (0.5 and 1);
\draw[linkred] (ec) ++ (0.3,0.4) -- ++(0.3,0) to[out=0,in=0] ++(0,0.3) to[out=180,in=90,looseness=1] ++(-0.75,-0.6) -- ++(0,-0.6) to[out=270,in=180,looseness=1] ++(0.6,-0.6);
\draw[linkred] (ec) ++ (0.6,-0.5) -- ++(-0.3,0) to[out=180,in=180] ++(0,-0.3) -- ++(0.3,0) to[out=0,in=0] node[pos=0.9,above right,inner sep=1pt,red] {\small$\kappa$} ++(0,0.9) -- ++(-0.3,0) to[out=180,in=180] ++(0,0.3);
\draw[linkred] (ec) ++ (0.45,-1.1) to[out=0,in=0] ++(0,0.6); % fix kappa crossing
\foreach \x/\y in {-40/-20,15/35,255/295} {
  \draw[link] (ec) ++(\x:0.5 and 1) arc (\x:\y:0.5 and 1); % fix ellipse crossings
}
\end{scope}

\begin{scope}[yshift=-6.25cm,xshift=-0.25cm] % move kappa, step 2
\draw[link] (0.3,0.5) coordinate (ec) ellipse (0.5 and 1);
\draw[linkred] (ec) ++ (0.3,0.4) -- ++(0.3,0) to[out=0,in=0,looseness=1] ++(0,0.3) to[out=180,in=90,looseness=1] ++(-0.85,-0.5) -- ++(0,-0.6) to[out=270,in=180,looseness=1] ++(0.7,-0.7);
\draw[linkred] (ec) ++ (0.3,0.4) to[out=180,in=90,looseness=1] ++(-0.35,-0.4) -- ++(0,-0.1) to[out=270,in=180,looseness=1] ++(0.35,-0.4) -- ++(0.8,0) to[out=0,in=0,looseness=2] ++(0,0.3) node[above right,red,inner sep=1pt] {\small$\kappa$}-- ++(-0.8,0) to[out=180,in=180] ++(0,0.3) -- ++(0.3,0);
\draw[linkred] (ec) ++ (0.6,0.1) to[out=0,in=90,looseness=1] ++(0.3,-0.3) -- ++(0,-0.45) to[out=270,in=0,looseness=1] ++(-0.45,-0.45);
\draw[linkred] (ec) ++ (1.1,-0.5) -- ++(-0.4,0); % fix kappa crossing
\draw[linkred] (ec) ++ (0.9,-0.35) -- ++(0,0.2); % fix another crossing
\foreach \x/\y in {-5/35,210/295} {
  \draw[link] (ec) ++(\x:0.5 and 1) arc (\x:\y:0.5 and 1); % fix ellipse crossings
}
\end{scope}

\begin{scope}[yshift=-6.25cm,xshift=2.75cm] % move kappa, step 3
\draw[link] (-0.2,0.5) coordinate (ec) ellipse (0.75 and 1);
\draw[linkred] (ec) ++ (0.3,0.4) -- ++(0.9,0) to[out=0,in=90,looseness=1] ++(0.45,-0.45) -- ++(0,-0.3) to[out=270,in=0,looseness=1] ++(-0.45,-0.45) coordinate (ofix) -- ++(-0.9,0) to[out=180,in=270,looseness=1] ++(-0.6,0.6) -- ++ (0,0.3) to[out=90,in=180,looseness=1] ++(0.6,0.6) -- ++(0.9,0) to[out=0,in=90,looseness=1] ++(0.75,-0.75) -- ++(0,-0.3) to[out=270,in=0,looseness=1] ++(-0.45,-0.75);
\draw[linkred] (ec) ++ (0.3,0.4) to[out=180,in=90,looseness=1] ++(-0.35,-0.4) -- ++(0,-0.1) to[out=270,in=180,looseness=1] ++(0.35,-0.4) -- ++(0.8,0) coordinate (gfix) to[out=0,in=0,looseness=2] ++(0,0.3) node[above right,red,inner sep=1pt] {\small$\kappa$} -- ++(-0.8,0) to[out=180,in=180] ++(0,0.3) -- ++(0.3,0);
\draw[linkred] (ec) ++ (0.6,0.1) to[out=0,in=90,looseness=1] ++(0.45,-0.3) -- ++(0,-0.45) to[out=270,in=180,looseness=1] ++(0.45,-0.45);
\draw[linkred] (gfix) ++(-0.4,0) -- ++(0.4,0) to[out=0,in=0,looseness=2] ++(0,0.3); % fix kappa crossing
\draw[linkred] (ec) ++ (1.05,-0.35) -- ++(0,0.2); % fix another crossing
\draw[linkred] (ofix) -- ++(-0.5,0); % fix another crossing
\foreach \x/\y in {-5/90} {
  \draw[link] (ec) ++(\x:0.75 and 1) arc (\x:\y:0.75 and 1); % fix ellipse crossings
}
\end{scope}

\node at (5.4,-5.75) {\huge$\xrightarrow{\dcover}$};

\begin{scope}[yshift=-5.75cm,xshift=8.15cm] % double cover
\draw[thinlink] (0,0) -- (0,-1.25);
\foreach \rot in {0,180} {
\begin{scope}[rotate=\rot]
% middle two strands
\draw[linkred,looseness=1] (0,-0.15) to[out=0,in=180] ++(0.6,0.3) ++(0,-0.3) to[out=0,in=180] ++(0.6,0.3);
\draw[linkred,looseness=1] (0,0.15) to[out=0,in=180] ++(0.6,-0.3) ++(0,0.3) to[out=0,in=180] ++(0.6,-0.3) to[out=0,in=0,looseness=2] ++(0,-0.3) -- ++(-1.2,0);
% top two strands
\draw[linkred,looseness=2] (0,0.75) -- ++(1.2,0) to[out=0,in=0] ++(0,-0.6);
\draw[linkred,looseness=2] (0,0.45) -- ++(1.2,0) to[out=0,in=0] ++(0,-1.2) -- ++(-1.2,0);
\end{scope}
}
\node[below,red] at (0.6,-0.75) {$\tilde\kappa$};
\draw[thinlink] (0,0) -- (0,1.25);
% fix crossings over the branch locus 
\foreach \y in {-0.45,-0.75} { \draw[linkred] (-0.1,\y) -- (0.1,\y); }
\begin{scope}
\clip (-0.1,-0.3) rectangle (0.1,0);
\draw[linkred,looseness=1] (-0.6,0.15) to[out=0,in=180] ++(0.6,-0.3) to[out=0,in=180] ++(0.6,0.3);
\end{scope}
\end{scope}
\end{tikzpicture}
\caption{Recovering $K_\beta \cong 5_2$ in the case $\beta = x^{-1}$.  In the last step we indicate the axis of symmetry (i.e., the preimage of $U$) for reference.}
\label{fig:braid-5_2}
\end{figure}
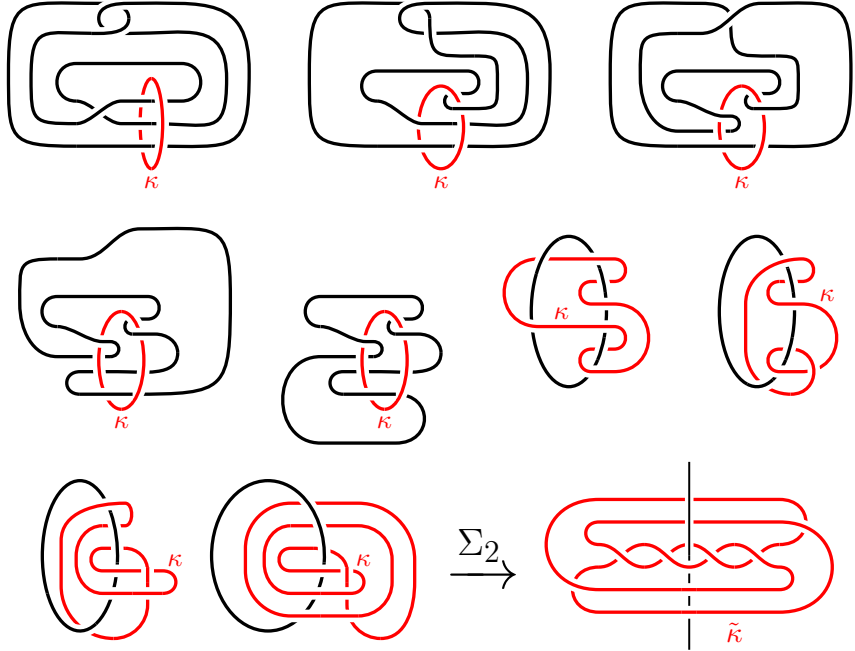
\end{proof}

\begin{proposition} \label{prop:recover-pretzels}
The braids $\beta = y^a x^n y^{-1} x y^{1-a}$ produce $K_\beta \cong P(-3,3,2n+1)$.
\end{proposition}

\begin{proof}
Again by Lemma~\ref{lem:y-conjugation} we need only consider $\beta = y x^n y^{-1} x$.  In Figure~\ref{fig:braid-pretzel} we insert this braid into $\tau \sqcup \kappa$ and perform an isotopy so that the unknot $U = \tau \cup \beta$ clearly bounds a disk, and then in Figure~\ref{fig:braid-pretzel-2} we use this to lift $\kappa$ to the knot $K_\beta = \tilde\kappa$ in the branched double cover $\dcover(U) \cong S^3$.  In the end we are left with a diagram of $P(3,-3,2n+1)$, which is isotopic to $P(-3,3,2n+1)$.
\begin{figure}
\begin{tikzpicture}
\begin{scope} % initial tangle
\draw[linkred] (3,0.6) arc (90:270:0.15 and 0.6);
\draw[link] (0,0.3) to[out=180,in=180] ++(0,0.5) -- ++(3.4,0) to[out=0,in=0] ++(0,-0.5) -- ++(-0.6,0);
\draw[link] (0,0) to[out=180,in=270] ++(-0.6,0.6) to[out=90,in=180] ++(0.9,0.6) to[out=0,in=270,looseness=1] ++(0.4,0.2) to[out=90,in=0,looseness=1] ++(-0.4,0.2) to[out=180,in=90] ++(-1.2,-0.9) to[out=270,in=180] ++(0.9,-1);
\draw[link] (2.8,0) -- ++(0.2,0) to[out=0,in=270] ++(1,0.6) to[out=90,in=0,looseness=1.4] ++(-1.3,0.6) -- ++(-2,0) to[out=180,in=270,looseness=1] ++(-0.4,0.2) to[out=90,in=180,looseness=1] ++(0.4,0.2) -- ++(2,0) to[out=0,in=90,looseness=1.4] ++(1.6,-0.9) to[out=270,in=0] ++(-1.3,-1) -- ++(-0.2,0);
\draw[link] (0.7,1.4) to[out=270,in=0,looseness=1] ++(-0.4,-0.2) -- ++(-0.1,0); % redraw a crossing
\draw[link,looseness=0.75] (0.6,-0.3) -- ++(1,0) ++(0.6,0.3) to[out=0,in=180] ++(0.6,0.3);
\draw[link,looseness=0.75] (0,-0.3) to[out=0,in=180] ++ (0.6,0.3) -- ++(1,0) to[out=0,in=180] ++(0.6,-0.3) -- ++(0.6,0);
\draw[link,looseness=0.75] (0.0,0.3) -- ++(2.2,0) to[out=0,in=180] ++(0.6,-0.3);
\draw[link,looseness=0.75] (0,0) to[out=0,in=180] ++(0.6,-0.3);
\draw[link,looseness=0.75] (1.6,-0.3) to[out=0,in=180] ++(0.6,0.3);
\draw[ultra thick,fill=white] (0.8,0.45) rectangle ++(0.6,-0.6);
\node at (1.1,0.15) {$n$};
\draw[linkred] (3,0.6) arc (90:-90:0.15 and 0.6) node[below,red,inner sep=2pt] {\small$\kappa$};
\end{scope}

\begin{scope}[xshift=6.5cm] % isotopy step 1
\draw[linkred] (3,0.6) arc (90:270:0.15 and 0.6);
\draw[link] (0,0.3) to[out=180,in=180] ++(0,0.5) -- ++(3.4,0) to[out=0,in=0] ++(0,-0.5) -- ++(-0.6,0);
\draw[link] (0,-0.3) to[out=180,in=270] ++(-0.9,1) to[out=90,in=180] ++(1.2,0.7) -- ++(2.4,0) to[out=0,in=90,looseness=1.5] ++(1.6,-0.9);
\draw[link] (2.8,0) -- ++(0.2,0) to[out=0,in=270] ++(1,0.6) to[out=90,in=0,looseness=1.4] ++(-1.3,0.6) -- ++(-2,0) to[out=180,in=270,looseness=1] ++(-0.4,0.2) to[out=90,in=180,looseness=1] ++(0.4,0.2) -- ++(2.3,0) to[out=0,in=90,looseness=1.4] ++(1.6,-0.9) to[out=270,in=0] ++(-1.3,-1) -- ++(-0.5,0);
\draw[link,looseness=0.75] (0,-0.3) to[out=0,in=180] ++ (0.6,0.3) -- ++(1,0) to[out=0,in=180] ++(0.6,-0.3) -- ++(0.6,0);
\draw[link] (4.3,0.5) to[out=270,in=0,looseness=1.25] ++(-0.9,-1.35) -- ++(-1,0) to[out=180,in=180] (2.2,0);
\draw[link,looseness=0.75] (2.2,0) to[out=0,in=180] ++(0.6,0.3);
\draw[link,looseness=0.75] (0.0,0.3) -- ++(2.2,0) to[out=0,in=180] ++(0.6,-0.3);
\draw[ultra thick,fill=white] (0.8,0.45) rectangle ++(0.6,-0.6);
\node at (1.1,0.15) {$n$};
\draw[linkred] (3,0.6) arc (90:-90:0.15 and 0.6) node[left,red,inner sep=3pt] {\small$\kappa$};
\end{scope}

\begin{scope}[yshift=-2.75cm] % isotopy step 2
\draw[linkred] (3,0.6) arc (90:270:0.15 and 0.6);
\draw[link] (0,0.3) to[out=180,in=180] ++(0,0.5) -- ++(3.4,0) to[out=0,in=0] ++(0,-0.5) -- ++(-0.6,0);
\draw[link] (0,0) to[out=180,in=180] ++(0,1.1) -- ++(2.7,0) to[out=0,in=90,looseness=1.1] ++(1.6,-0.6);
\draw[link] (2.8,0) -- ++(1.8,0) to[out=0,in=0,looseness=2] ++(0,-0.3) -- ++(-1.8,0);
\draw[link,looseness=0.75] (0,0) -- ++(1.6,0) to[out=0,in=180] ++(0.6,-0.3) -- ++(0.6,0);
\draw[link] (4.3,0.5) to[out=270,in=0,looseness=1.25] ++(-0.9,-1.35) -- ++(-1,0) to[out=180,in=180] (2.2,0);
\draw[link] (4.6,0) -- ++(-1,0);
\draw[link,looseness=0.75] (2.2,0) to[out=0,in=180] ++(0.6,0.3);
\draw[link,looseness=0.75] (0.0,0.3) -- ++(2.2,0) to[out=0,in=180] ++(0.6,-0.3);
\draw[ultra thick,fill=white] (0.8,0.45) rectangle ++(0.6,-0.6);
\node at (1.1,0.15) {$n$};
\draw[linkred] (3,0.6) arc (90:-90:0.15 and 0.6) node[left,red,inner sep=3pt] {\small$\kappa$};
\end{scope}

\begin{scope}[yshift=-2.75cm,xshift=6.5cm] % isotopy step 3
\draw[linkred] (2.4,0.6) arc (90:270:0.15 and 0.6);
\draw[link] (0,0) to[out=180,in=180] ++(0,1.1) -- ++(2.7,0) to[out=0,in=180,looseness=0.75] ++(1,0.3) to[out=0,in=0] ++(0,-0.6) -- ++(-3.7,0) to[out=180,in=180] ++(0,-0.5);
\draw[link] (2.8,0.3) -- ++(1.2,0) to[out=0,in=0,looseness=2] ++(0,-0.6) -- ++(-1.2,0);
\draw[link,looseness=0.75] (0,0) -- ++(1.6,0) to[out=0,in=180] ++(0.6,-0.3) -- ++(0.6,0);
\draw[link,looseness=0.75] (0.0,0.3) -- ++(2.8,0); % to[out=0,in=180] ++(0.6,-0.3);
\draw[linkred] (2.4,0.6) to[out=0,in=180,looseness=0.75] ++(0.6,-0.55) to[out=0,in=270,looseness=1] ++(0.3,0.55) to[out=90,in=90,looseness=5] ++(0.3,0) to[out=270,in=0,looseness=1.25] ++(-0.6,-0.75) to[out=180,in=0,looseness=0.75] ++(-0.6,-0.45) node[left,red,inner sep=3pt] {\small$\kappa$};
\draw[link] (3,0.3) -- ++(1,0); % redraw some crossings
\draw[link] (3.45,0.8) -- ++(0.25,0);
\draw[ultra thick,fill=white] (0.8,0.45) rectangle ++(0.6,-0.6);
\node at (1.1,0.15) {$n$};
\end{scope}

\begin{scope}[yshift=-5.5cm] % isotopy step 4
\draw[linkred] (2.4,0.6) arc (90:270:0.15 and 0.6) node[red,left,inner sep=3pt] {\small$\kappa$};
\draw[link] (2.8,0.3) -- ++(0.6,0) to[out=0,in=0,looseness=2] ++(0,-0.6) -- ++(-0.6,0);
\draw[link,looseness=0.75] (0.4,0) -- ++(1,0) to[out=0,in=180] ++(0.6,-0.3) -- ++(0.6,0);
\draw[link,looseness=0.75] (0.4,0.3) -- ++(2.2,0); % to[out=0,in=180] ++(0.6,-0.3);
\draw[link,looseness=0.75] (0.4,0.3) to[out=180,in=90] ++(-0.3,-0.15);
\draw[linkred] (2.4,0.6) to[out=0,in=180,looseness=0.75] ++(0.6,-0.55) to[out=0,in=270,looseness=1] ++(0.3,0.55) to[out=90,in=0,looseness=1] ++(-0.6,0.2) -- ++(-2.7,0) to[out=180,in=180] ++(0,-0.5) to[out=0,in=90,looseness=0.75] ++(0.3,-0.15) to[out=270,in=0,looseness=0.75] ++(-0.3,-0.15);
\draw[linkred] (2.4,-0.6) to[out=0,in=180,looseness=0.75] ++(0.6,0.45) to[out=0,in=270] ++(0.6,0.75) to[out=90,in=0,looseness=1] ++(-0.6,0.5) -- ++(-3,0) to[out=180,in=180] ++(0,-1.1); 
\draw[link,looseness=0.75] (0.4,0) to[out=180,in=270] ++(-0.3,0.15);
\draw[link] (2.8,0.3) -- ++(0.6,0) to[out=0,in=0,looseness=2] ++(0,-0.6); % redraw some crossings
\draw[ultra thick,fill=white] (0.8,0.45) rectangle ++(0.6,-0.6);
\node at (1.1,0.15) {$n$};
\end{scope}

\begin{scope}[yshift=-5.5cm,xshift=4.75cm] % isotopy step 4
\draw[linkred] (2.1,0.6) arc (90:270:0.15 and 0.6) node[red,left,inner sep=3pt] {\small$\kappa$};
\draw[link] (2.5,0.3) -- ++(0.6,0) to[out=0,in=0,looseness=2] ++(0,-0.6) -- ++(-0.6,0);
\draw[link,looseness=0.75] (0.4,0) -- ++(1,0) to[out=0,in=180] ++(0.6,-0.3) -- ++(0.6,0);
\draw[link,looseness=0.75] (0.4,0.3) -- ++(2.2,0); % to[out=0,in=180] ++(0.6,-0.3);
\draw[link,looseness=0.75] (1.4,0.3) to[out=180,in=90] ++(-0.3,-0.15);
\draw[linkred] (2.1,0.6) to[out=0,in=180,looseness=0.75] ++(0.6,-0.55) to[out=0,in=270,looseness=1] ++(0.3,0.55) to[out=90,in=0,looseness=1] ++(-0.6,0.2) -- ++(-2.4,0) to[out=180,in=180] ++(0,-0.5) -- ++(1,0) to[out=0,in=90,looseness=0.75] ++(0.3,-0.15) to[out=270,in=0,looseness=0.75] ++(-0.3,-0.15) -- ++(-1,0);
\draw[linkred] (2.1,-0.6) to[out=0,in=180,looseness=0.75] ++(0.6,0.45) to[out=0,in=270] ++(0.6,0.75) to[out=90,in=0,looseness=1] ++(-0.6,0.5) -- ++(-2.7,0) to[out=180,in=180] ++(0,-1.1); 
\draw[link,looseness=0.75] (1.4,0) to[out=180,in=270] ++(-0.3,0.15);
\draw[link] (2.5,0.3) -- ++(0.6,0) to[out=0,in=0,looseness=2] ++(0,-0.6); % redraw some crossings
\draw[ultra thick,red,fill=white] (0.2,0.45) rectangle ++(0.6,-0.6);
\node[red] at (0.5,0.15) {$n$};
\end{scope}

\begin{scope}[yshift=-5.5cm,xshift=9cm] % straighten out U
\draw[link] (0,0.4) coordinate (ec) ellipse (0.3 and 1);
\draw[linkred] (ec) ++(0.6,0.15) to[out=0,in=0] ++(0,-0.9);
\draw[linkred] (ec) ++ (0.6,0.75) to[out=180,in=180,looseness=7] ++(0,-0.3) -- ++(0.6,0) to[out=0,in=0] ++(0,-0.6) -- ++(-0.6,0) to[out=180,in=270,looseness=0.75] ++(-0.6,0.15) to[out=90,in=180,looseness=0.75] ++(0.6,0.15);
\draw[linkred] (ec) ++ (0.6,-0.75) to[out=180,in=270,looseness=0.75] ++(-0.6,0.15) to[out=90,in=180,looseness=0.75] ++(0.6,0.15) -- ++(0.6,0) node[below,red,inner sep=2pt] {\small$\kappa$} to[out=0,in=0] ++(0,1.2);
\draw[ultra thick,red,fill=white] (ec) ++ (0.6,0.9) rectangle ++(0.6,-0.6);
\path (ec) ++ (0.9,0.6) node[red] {$n$};
\foreach \x/\y in {36/90,0/-40} {
  \draw[link] (ec) ++(\x:0.3 and 1) arc (\x:\y:0.3 and 1); % fix ellipse crossings
}
\end{scope}

\end{tikzpicture}
\caption{Recovering $K_\beta \cong P(-3,3,2n+1)$ in the case $\beta = y x^n y^{-1} x$, part 1: isotoping $U \cup \kappa$ so that $U$ bounds a disk in the plane.  Here each box labeled ``$n$'' contains $n$ signed crossings.}
\label{fig:braid-pretzel}
\end{figure}
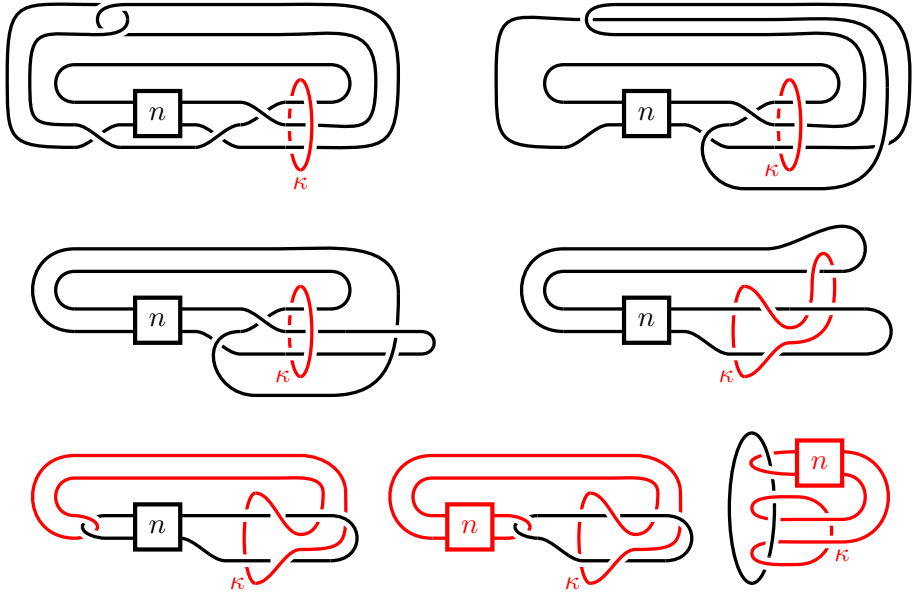
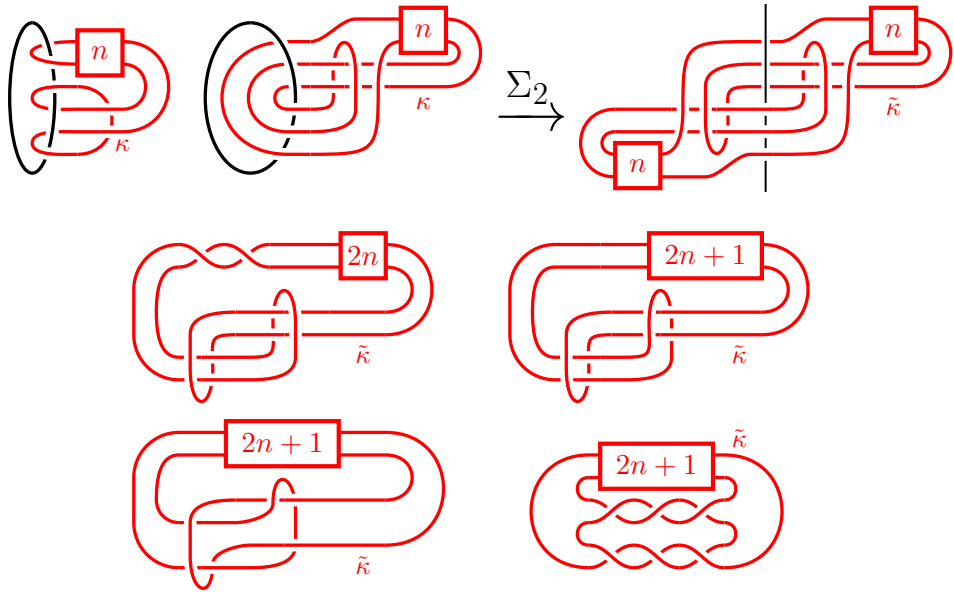
\begin{figure}
\begin{tikzpicture}
\begin{scope} % U straightened out
\draw[link] (0,0) coordinate (ec) ellipse (0.3 and 1);
\draw[linkred] (ec) ++(0.6,0.15) to[out=0,in=0] ++(0,-0.9);
\draw[linkred] (ec) ++ (0.6,0.75) to[out=180,in=180,looseness=7] ++(0,-0.3) -- ++(0.6,0) to[out=0,in=0] ++(0,-0.6) -- ++(-0.6,0) to[out=180,in=270,looseness=0.75] ++(-0.6,0.15) to[out=90,in=180,looseness=0.75] ++(0.6,0.15);
\draw[linkred] (ec) ++ (0.6,-0.75) to[out=180,in=270,looseness=0.75] ++(-0.6,0.15) to[out=90,in=180,looseness=0.75] ++(0.6,0.15) -- ++(0.6,0) node[below,red,inner sep=2pt] {\small$\kappa$} to[out=0,in=0] ++(0,1.2);
\draw[ultra thick,red,fill=white] (ec) ++ (0.6,0.9) rectangle ++(0.6,-0.6);
\path (ec) ++ (0.9,0.6) node[red] {$n$};
\foreach \x/\y in {36/90,0/-40} {
  \draw[link] (ec) ++(\x:0.3 and 1) arc (\x:\y:0.3 and 1); % fix ellipse crossings
}
\end{scope}

\begin{scope}[xshift=3.5cm] % now make kappa look braided near U
\draw[link] (-0.6,0) ellipse (0.6 and 1);
\foreach \y in {0.15,0.45,0.75} {
  \draw[linkred] (-0.1,\y) ++(0.3,0) -- ++(-0.3,0) to[out=180,in=180,looseness=2] ++(0,-2*\y) -- ++(0.3,0);
}
\draw[link] (-0.6,0) ++ (0:0.6 and 1) arc (0:90:0.6 and 1); % fix ellipse crossings
\draw[linkred] (0.2,0.75) to[out=0,in=180,looseness=0.75] ++(0.6,0.3) -- ++(1.2,0) to[out=0,in=0] ++(0,-0.9);
\foreach \y in {0.15,0.45} { \draw[linkred] (0.65,\y) -- ++(1.35,0); }
\draw[linkred] (0.65,0.15) -- ++(0.6,0);
\draw[linkred] (0.2,-0.15) to[out=0,in=270] ++(0.3,0.45) to[out=90,in=90,looseness=5] ++(0.3,0) -- ++(0,-0.45) to[out=270,in=0,looseness=1] ++(-0.6,-0.3);
\foreach \y in {0.15,0.45} { \draw[linkred] (0.2,\y) -- ++(0.45,0); }
\draw[linkred] (0.2,-0.75) to[out=0,in=270] ++(0.9,0.9) to[out=90,in=180,looseness=1] ++(0.3,0.6) -- ++(0.6,0) to[out=0,in=0,looseness=2] ++(0,-0.3);
\node[below,red] at (1.7,0.15) {\small$\kappa$};
\draw[ultra thick,red,fill=white] (1.4,1.2) rectangle ++(0.6,-0.6);
\path (1.7,0.9) node[red] {$n$};
\end{scope}

\node at (6.625,0) {\huge$\xrightarrow{\dcover}$};

\begin{scope}[xshift=9.75cm] % branched cover
\draw[thinlink] (0,0) -- ++(0,-1.25);
\foreach \r in {0,180} {
\begin{scope}[rotate=\r]
\draw[linkred] (0,0.75) -- ++(0.2,0) to[out=0,in=180,looseness=0.75] ++(0.6,0.3) -- ++(1.2,0) to[out=0,in=0] ++(0,-0.9);
\foreach \y in {0.15,0.45} { \draw[linkred] (0.65,\y) -- ++(1.35,0); }
\draw[linkred] (0.65,0.15) -- ++(0.6,0);
\draw[linkred] (0,-0.15) -- ++(0.2,0) to[out=0,in=270] ++(0.3,0.45) to[out=90,in=90,looseness=5] ++(0.3,0) -- ++(0,-0.45) to[out=270,in=0,looseness=1] ++(-0.6,-0.3) -- ++(-0.2,0);
\foreach \y in {0.15,0.45} { \draw[linkred] (0.0,\y) -- ++(0.65,0); }
\draw[linkred] (0,-0.75) -- ++(0.2,0) to[out=0,in=270] ++(0.9,0.9) to[out=90,in=180,looseness=1] ++(0.3,0.6) -- ++(0.6,0) to[out=0,in=0,looseness=2] ++(0,-0.3);
\draw[ultra thick,red,fill=white] (1.4,1.2) rectangle ++(0.6,-0.6);
\path (1.7,0.9) node[red] {$n$};
\end{scope}
}
\node[below,red] at (1.7,0.15) {\small$\tilde\kappa$};
\draw[thinlink] (0,0) -- ++(0,1.25);
\foreach \y in {-0.15,-0.45,-0.75} { \draw[linkred] (-0.1,\y) -- (0.1,\y); } % fix crossings over the branch locus
\end{scope}

\begin{scope}[yshift=-3cm,xshift=2.7cm] % merge the two n-twist boxes
\draw[linkred] (0.45,1.05) -- ++(1.55,0) to[out=0,in=0] ++(0,-1.2);
\foreach \y in {-0.15,0.15} { \draw[linkred] (0.65,\y) -- ++(1.35,0); }
\draw[linkred] (0.65,0.15) -- ++(0.6,0);
\draw[linkred] (-0.75,-0.45) -- ++(0.95,0) to[out=0,in=270] ++(0.3,0.45) to[out=90,in=90,looseness=5] ++(0.3,0) -- ++(0,-0.45) to[out=270,in=0,looseness=1] ++(-0.6,-0.3) -- ++(-0.95,0);
\foreach \y in {0.15,0.45} { \draw[linkred] (0.0,\y-0.3) -- ++(0.65,0); }
\draw[linkred] (0.45,0.75) -- ++(1.55,0) to[out=0,in=0,looseness=2] ++(0,-0.6);
\draw[linkred] (0,-0.15) to[out=180,in=90] ++(-0.3,-0.45) to[out=270,in=270,looseness=5] ++(-0.3,0) -- ++(0,0.45) to[out=90,in=180,looseness=1] ++(0.6,0.3);
\foreach \y in {-0.45,-0.75} { \draw[linkred] (-0.45,\y) -- ++(0.3,0); } % fix some crossings
\draw[linkred] (-0.75,-0.45) to[out=180,in=270,looseness=1] ++(-0.3,0.6) to[out=90,in=180,looseness=1] ++(0.3,0.6);
\draw[linkred] (-0.75,-0.75) to[out=180,in=270,looseness=1] ++(-0.6,0.6) -- ++(0,0.6) to[out=90,in=180,looseness=1] ++(0.6,0.6);
\draw[linkred,looseness=1] (-0.75,0.75) to[out=0,in=180] ++(0.6,0.3) ++(0,-0.3) to[out=0,in=180] ++(0.6,0.3);
\draw[linkred,looseness=1] (-0.75,1.05) to[out=0,in=180] ++(0.6,-0.3) ++(0,0.3) to[out=0,in=180] ++(0.6,-0.3);
\draw[ultra thick,red,fill=white] (1.4,1.2) rectangle ++(0.6,-0.6);
\path (1.7,0.9) node[red] {$2n$};
\node[below,red] at (1.7,-0.15) {\small$\tilde\kappa$};
\end{scope}

\begin{scope}[yshift=-3cm,xshift=7.7cm] % get rid of two half-twists
\draw[linkred] (-0.15,1.05) -- ++(2.15,0) to[out=0,in=0] ++(0,-1.2);
\foreach \y in {0.15,0.45} { \draw[linkred] (0.0,\y-0.3) -- ++(0.65,0); }
\draw[linkred] (-0.75,-0.45) -- ++(0.95,0) to[out=0,in=270] ++(0.3,0.45) to[out=90,in=90,looseness=5] ++(0.3,0) -- ++(0,-0.45) to[out=270,in=0,looseness=1] ++(-0.6,-0.3) -- ++(-0.95,0);
\foreach \y in {-0.15,0.15} { \draw[linkred] (0.65,\y) -- ++(1.35,0); }
\draw[linkred] (-0.15,0.75) -- ++(2.15,0) to[out=0,in=0,looseness=2] ++(0,-0.6);
\draw[linkred] (0,-0.15) to[out=180,in=90] ++(-0.3,-0.45) to[out=270,in=270,looseness=5] ++(-0.3,0) -- ++(0,0.45) to[out=90,in=180,looseness=1] ++(0.6,0.3);
\foreach \y in {-0.45,-0.75} { \draw[linkred] (-0.45,\y) -- ++(0.3,0); } % fix some crossings
\draw[linkred] (-0.75,-0.45) to[out=180,in=270,looseness=1] ++(-0.3,0.6) to[out=90,in=180,looseness=1] ++(0.3,0.6) -- ++(0.6,0);
\draw[linkred] (-0.75,-0.75) to[out=180,in=270,looseness=1] ++(-0.6,0.6) -- ++(0,0.6) to[out=90,in=180,looseness=1] ++(0.6,0.6) -- ++(0.6,0);
\draw[ultra thick,red,fill=white] (0.5,1.2) rectangle ++(1.5,-0.6);
\path (1.25,0.9) node[red] {$2n+1$};
\node[below,red] at (1.7,-0.15) {\small$\tilde\kappa$};
\end{scope}

\begin{scope}[yshift=-5.5cm,xshift=2.7cm] % now make kappa look braided near U
\draw[linkred] (-0.15,1.05) -- ++(2.15,0) to[out=0,in=0] ++(0,-1.5);
\foreach \y in {0.15} { \draw[linkred] (0.0,\y) -- ++(0.65,0); }
\draw[linkred] (-0.75,-0.15) -- ++(0.75,0) to[out=0,in=270,looseness=0.75] ++(0.5,0.25) to[out=90,in=90,looseness=4] ++(0.3,-0.1) -- ++(0,-0.45) to[out=270,in=0,looseness=1] ++(-0.6,-0.3) -- ++(-0.95,0);
++(0.95,0) to[out=0,in=270] ++(0.3,0.45) to[out=90,in=90,looseness=5] ++(0.3,0) -- ++(0,-0.45) to[out=270,in=0,looseness=1] ++(-0.6,-0.3) -- ++(-0.95,0);
\foreach \y in {-0.45,0.15} { \draw[linkred] (0.65,\y) -- ++(1.35,0); }
\draw[linkred] (-0.15,0.75) -- ++(2.15,0) to[out=0,in=0,looseness=2] ++(0,-0.6);
\draw[linkred] (0.65,-0.45) -- (0.2,-0.45) to[out=180,in=90,looseness=0.75] ++(-0.5,-0.25) to[out=270,in=270,looseness=4] ++(-0.3,0.1) -- ++(0,0.45) to[out=90,in=180,looseness=1] ++(0.6,0.3);
\foreach \y in {-0.75} { \draw[linkred] (-0.45,\y) -- ++(0.3,0); } % fix some crossings
\draw[linkred] (-0.75,-0.15) to[out=180,in=270,looseness=1] ++(-0.3,0.3) to[out=90,in=180,looseness=1] ++(0.3,0.6) -- ++(0.6,0);
\draw[linkred] (-0.75,-0.75) to[out=180,in=270,looseness=1] ++(-0.6,0.6) -- ++(0,0.6) to[out=90,in=180,looseness=1] ++(0.6,0.6) -- ++(0.6,0);
\draw[ultra thick,red,fill=white] (-0.125,1.2) rectangle ++(1.5,-0.6);
\path (0.625,0.9) node[red] {$2n+1$};
\node[below,red] at (1.7,-0.45) {\small$\tilde\kappa$};
\end{scope}

\begin{scope}[yshift=-5.5cm,xshift=7.4cm]
\foreach \y in {0.45,-0.15} { \draw[linkred] (0,\y) to[out=180,in=180] ++(0,-0.3) ++ (1.8,0) to[out=0,in=0] ++(0,0.3); }
\draw[linkred] (0,-0.75) to[out=180,in=180] ++(0,1.5) -- ++(1.8,0) node[above right,red,inner sep=2pt] {\small$\tilde\kappa$} to[out=0,in=0] ++(0,-1.5);
\draw[linkred] (0,0.45) -- ++(1.8,0);
\foreach \x in {0,0.6,1.2} { \draw[linkred,looseness=1] (\x,0.15) to[out=0,in=180] ++(0.6,-0.3) (\x,-0.75) to[out=0,in=180] ++(0.6,0.3); }
\foreach \x in {0,0.6,1.2} { \draw[linkred,looseness=1] (\x,-0.15) to[out=0,in=180] ++(0.6,0.3) (\x,-0.45) to[out=0,in=180] ++(0.6,-0.3); }
\draw[ultra thick,red,fill=white] (0.15,0.3) rectangle ++(1.5,0.6);
\path (0.9,0.6) node[red] {$2n+1$};
\end{scope}

\end{tikzpicture}
\caption{Recovering $K_\beta \cong P(-3,3,2n+1)$ in the case $\beta = y x^n y^{-1} x$, part 2: taking branched covers to construct the claimed pretzel knots.}
\label{fig:braid-pretzel-2}
\end{figure}
\end{proof}

\begin{proposition} \label{prop:recover-15n43522}
The braids $\beta = y^a x^3 y^{-1} x^2 y^{1-a}$ produce $K_\beta \cong 15n_{43522}$, possibly up to mirroring.
\end{proposition}

\begin{proof}
In this case, Lemma~\ref{lem:y-conjugation} says that we need only consider $\beta = x^3y^{-1}x^2 y$, as shown in Figure~\ref{fig:braid-15n43522}.  We can repeat the same procedure as in Propositions~\ref{prop:recover-5_2} and \ref{prop:recover-pretzels} to find $K_\beta$, but this is not very enlightening because we find it hard to identify 15-crossing knots from their diagrams.

Instead, we ask SnapPy \cite{snappy} to do the hard work for us: we give it the link $U\cup\kappa$ on the left side of Figure~\ref{fig:braid-15n43522}, do a $(2,0)$-Dehn filling of $U$ (i.e., an orbifold Dehn filling of $U$ with meridional slope, so that $U$ has cone angle $\pi$), and then look at the double covers of the result that are not themselves orbifolds.  SnapPy can produce triangulations of these, and it identifies one of them as the complement of $15n_{43522}$, so this must be $K_\beta$.
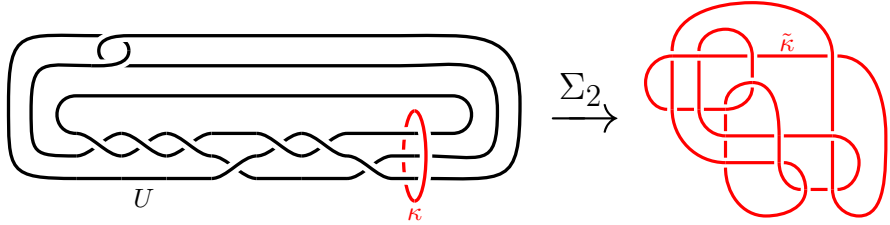
\begin{figure}
\begin{tikzpicture}
\begin{scope} % initial tangle
\draw[linkred] (4.5,0.6) arc (90:270:0.15 and 0.6);
\draw[link] (0,0.3) to[out=180,in=180] ++(0,0.5) -- ++(5,0) to[out=0,in=0] ++(0,-0.5) -- ++(-0.6,0);
\draw[link] (0,0) to[out=180,in=270] ++(-0.6,0.6) to[out=90,in=180] ++(0.9,0.6) to[out=0,in=270,looseness=1] ++(0.4,0.2) to[out=90,in=0,looseness=1] ++(-0.4,0.2) to[out=180,in=90] ++(-1.2,-0.9) to[out=270,in=180] ++(0.9,-1);
\draw[link] (4.4,0) -- ++(0.2,0) to[out=0,in=270] ++(1,0.6) to[out=90,in=0,looseness=1.4] ++(-1.3,0.6) -- ++(-3.6,0) to[out=180,in=270,looseness=1] ++(-0.4,0.2) to[out=90,in=180,looseness=1] ++(0.4,0.2) -- ++(3.6,0) to[out=0,in=90,looseness=1.4] ++(1.6,-0.9) to[out=270,in=0] ++(-1.3,-1) -- ++(-0.2,0);
\draw[link] (0.7,1.4) to[out=270,in=0,looseness=1] ++(-0.4,-0.2) -- ++(-0.1,0); % redraw a crossing
\foreach \x in {0,1,2,4,5} {
 \draw[link,looseness=0.75] (0.6*\x,0) to[out=0,in=180] ++(0.6,0.3);
 \draw[link,looseness=0.75] (0.6*\x,0.3) to[out=0,in=180] ++(0.6,-0.3);
 \draw[link] (0.6*\x,-0.3) -- ++(0.6,0);
}
\draw[link,looseness=0.75] (1.8,0) to[out=0,in=180] ++(0.6,-0.3) (3.6,-0.3) to[out=0,in=180] ++(0.6,0.3) -- ++(0.3,0);
\draw[link,looseness=0.75] (1.8,-0.3) to[out=0,in=180] ++(0.6,0.3) (3.6,0) to[out=0,in=180] ++(0.6,-0.3) -- ++(0.3,0);
\draw[link] (1.8,0.3) -- ++(0.6,0) (3.6,0.3) -- ++(0.9,0);
\draw[linkred] (4.5,0.6) arc (90:-90:0.15 and 0.6) node[below,red,inner sep=2pt] {\small$\kappa$};
\node[below] at (0.9,-0.3) {\small$U$};
\end{scope}

\node at (6.75,0.75) {\huge$\xrightarrow{\dcover}$};

\node[red,above] at (9.45,1.25) {\small$\tilde\kappa$};
\begin{scope}[xshift=7.5cm,yshift=-1cm,scale=0.33,very thick,color=red] % code produced by PLink Viewer
    \draw (5.78, 2.76) .. controls (6.31, 2.76) and (6.66, 2.25) .. (6.66, 1.68);
    \draw (6.66, 1.68) .. controls (6.66, 0.97) and (5.86, 0.61) .. 
          (5.05, 0.61) .. controls (4.09, 0.61) and (3.43, 1.54) .. (3.43, 2.56);
    \draw (3.43, 2.96) .. controls (3.43, 3.19) and (3.43, 3.41) .. (3.43, 3.64);
    \draw (3.43, 4.03) .. controls (3.43, 4.26) and (3.43, 4.49) .. (3.43, 4.71);
    \draw (3.43, 5.11) .. controls (3.43, 5.64) and (3.95, 5.99) .. (4.51, 5.99);
    \draw (4.51, 5.99) .. controls (5.39, 5.99) and (5.59, 4.87) .. (5.59, 3.84);
    \draw (5.59, 3.84) .. controls (5.59, 3.48) and (5.59, 3.12) .. (5.59, 2.76);
    \draw (5.59, 2.76) .. controls (5.59, 2.20) and (5.94, 1.68) .. (6.46, 1.68);
    \draw (6.86, 1.68) .. controls (7.09, 1.68) and (7.31, 1.68) .. (7.54, 1.68);
    \draw (7.94, 1.68) .. controls (8.46, 1.68) and (8.81, 2.20) .. 
          (8.81, 2.76) .. controls (8.81, 3.35) and (8.33, 3.84) .. (7.74, 3.84);
    \draw (7.74, 3.84) .. controls (7.09, 3.84) and (6.43, 3.84) .. (5.78, 3.84);
    \draw (5.39, 3.84) .. controls (4.74, 3.84) and (4.08, 3.84) .. (3.43, 3.84);
    \draw (3.43, 3.84) .. controls (2.84, 3.84) and (2.36, 4.32) .. (2.36, 4.91);
    \draw (2.36, 4.91) .. controls (2.36, 5.56) and (2.36, 6.21) .. (2.36, 6.87);
    \draw (2.36, 7.26) .. controls (2.36, 7.79) and (2.87, 8.14) .. 
          (3.43, 8.14) .. controls (4.03, 8.14) and (4.51, 7.66) .. (4.51, 7.06);
    \draw (4.51, 7.06) .. controls (4.51, 6.77) and (4.51, 6.48) .. (4.51, 6.19);
    \draw (4.51, 5.79) .. controls (4.51, 5.26) and (3.99, 4.91) .. (3.43, 4.91);
    \draw (3.43, 4.91) .. controls (3.14, 4.91) and (2.85, 4.91) .. (2.56, 4.91);
    \draw (2.16, 4.91) .. controls (1.93, 4.91) and (1.71, 4.91) .. (1.48, 4.91);
    \draw (1.08, 4.91) .. controls (0.56, 4.91) and (0.21, 5.43) .. 
          (0.21, 5.99) .. controls (0.21, 6.58) and (0.69, 7.06) .. (1.28, 7.06);
    \draw (1.28, 7.06) .. controls (1.64, 7.06) and (2.00, 7.06) .. (2.36, 7.06);
    \draw (2.36, 7.06) .. controls (3.01, 7.06) and (3.66, 7.06) .. (4.31, 7.06);
    \draw (4.71, 7.06) .. controls (5.65, 7.06) and (6.60, 7.06) .. (7.54, 7.06);
    \draw (7.94, 7.06) .. controls (9.32, 7.06) and (9.89, 5.43) .. 
          (9.89, 3.84) .. controls (9.89, 2.37) and (9.89, 0.61) .. 
          (8.81, 0.61) .. controls (8.22, 0.61) and (7.74, 1.09) .. (7.74, 1.68);
    \draw (7.74, 1.68) .. controls (7.74, 2.34) and (7.74, 2.99) .. (7.74, 3.64);
    \draw (7.74, 4.03) .. controls (7.74, 5.04) and (7.74, 6.05) .. (7.74, 7.06);
    \draw (7.74, 7.06) .. controls (7.74, 8.50) and (6.13, 9.21) .. 
          (4.51, 9.21) .. controls (2.92, 9.21) and (1.28, 8.65) .. (1.28, 7.26);
    \draw (1.28, 6.87) .. controls (1.28, 6.21) and (1.28, 5.56) .. (1.28, 4.91);
    \draw (1.28, 4.91) .. controls (1.28, 3.72) and (2.25, 2.76) .. (3.43, 2.76);
    \draw (3.43, 2.76) .. controls (4.08, 2.76) and (4.74, 2.76) .. (5.39, 2.76);
\end{scope}
\end{tikzpicture}
\caption{The braid $\beta = x^3y^{-1}x^2y$ leads to $K_\beta \cong 15n_{43522}$.}
\label{fig:braid-15n43522}
\end{figure}
\end{proof}

\begin{remark}
SnapPy looks for isometries between a given pair of hyperbolic manifolds by first attempting to produce a canonical triangulation of each, and then comparing the resulting triangulations combinatorially.  Thus when it succeeds, as in the proof of Proposition~\ref{prop:recover-15n43522}, the result is certifiably true: it has found identical triangulations of each, and it does not need any numerical approximation to verify that the triangulations agree.
\end{remark}

As explained at the beginning of this section, this completes the proof of Theorem~\ref{thm:M_F-E24}.

\section{The $(2,4)$-cable of the trefoil} \label{sec:trefoil-24}

In this section, we determine all knots $K\subset S^3$ which arise from the second case of Theorem~\ref{thm:identify-y-c}, in which $M_F$ is the complement of the $(2,4)$-cable of the right-handed trefoil.  Our goal is to prove the following:

\begin{theorem} \label{thm:M_F-C24-T23}
Let $K\subset S^3$ be a nearly fibered knot with genus-1 Seifert surface $F$, and suppose that \[M_F \cong S^3 \setminus N(C_{2,4}(T_{2,3})).\]  Then $K$ is one of the twisted Whitehead doubles
\[ \Wh^{+}(T_{2,3},2)\textrm{ or }\Wh^{-}(T_{2,3},2). \]
\end{theorem}

Just as in Section \ref{sec:unknot-24}, we observe that under the hypotheses of Theorem \ref{thm:M_F-C24-T23}, the sutured Seifert surface complement $S^3(F)$ admits an involution $\iota$, illustrated in Figure~\ref{fig:E-C24-T23-quotient-1}, realizing this complement as the branched double cover of a sutured $3$-ball along a tangle $\tau$, as shown in Figure~\ref{fig:E-C24-T23-quotient-2}.  The exact same reasoning as in the previous section then implies the following analogue of Lemma~\ref{lem:E24-as-branched-cover}:

\begin{figure}
\begin{tikzpicture}
% Draw S^3(F)
\begin{scope}[xshift=3.775cm,yshift=-1cm]
\path[fill=gray!10] (-3.6,-2.1) rectangle (4.7,6.35);
\draw[blue,very thick] (-3.35,2) -- (4.35,2); % axis of involution
\draw[blue,-latex] (4.35,2) ++ (-135:0.1 and 0.4) arc (-135:165:0.1 and 0.4);
\path[blue] (4.35,2) ++ (0,0.4) node[above] {\small$\iota$};

\begin{scope}[every path/.append style={looseness=0.75}] % twist region
\foreach \y in {4,8/3,4/3} {
\foreach \x in {-0.6,0} {
  \path[fill=white] (\x,\y) to[out=270,in=90] ++(1.3,-4/3) -- ++(0.4,0) to[out=90,in=270] ++(-1.3,4/3) -- cycle;
  \draw (\x,\y) to[out=270,in=90] ++(1.3,-4/3) ++(0.4,0) to[out=90,in=270] ++(-1.3,4/3);
}
\foreach \x in {1.1,1.7} {
  \path[fill=white] (\x,\y) to[out=270,in=90] ++(-1.3,-4/3) -- ++(-0.4,0) to[out=90,in=270] ++(1.3,4/3) -- cycle;
  \draw (\x,\y) to[out=270,in=90] ++(-1.3,-4/3) ++ (-0.4,0) to[out=90,in=270] ++(1.3,4/3);
}
}
\end{scope}

\begin{scope}[every path/.append style={looseness=2}] % left side of trefoil
\path[fill=white] (-0.6,4) to[out=90,in=90] ++(-1,0) -- ++(0,-4) to[out=270,in=270] ++(1,0) -- ++(0.4,0) to[out=270,in=270] ++(-1.8,0) -- ++(0,4) to[out=90,in=90] ++(1.8,0);
\draw (-0.6,4) to[out=90,in=90] ++(-1,0) -- ++(0,-4) to[out=270,in=270] ++(1,0) ++(0.4,0) to[out=270,in=270] ++(-1.8,0) -- ++(0,4) to[out=90,in=90] ++(1.8,0);

\path[fill=white] (0,4) to[out=90,in=90] ++(-2.2,0) -- ++(0,-4) to[out=270,in=270] ++(2.2,0) -- ++(0.4,0) to[out=270,in=270] ++(-3,0) -- ++(0,4) to[out=90,in=90] ++(3,0);
\draw (0,4) to[out=90,in=90] ++(-2.2,0) -- ++(0,-4) to[out=270,in=270] ++(2.2,0) ++(0.4,0) to[out=270,in=270] ++(-3,0) -- ++(0,4) to[out=90,in=90] ++(3,0);
\end{scope}

\begin{scope}[every path/.append style={looseness=2}] % right side of trefoil
\path[fill=white] (1.1,4) to[out=90,in=90] ++(2.2,0) to[out=270,in=90,looseness=1] ++(-1.1,-4/3) -- ++(0.4,0) to[out=90,in=270,looseness=1] ++(1.1,4/3) to[out=90,in=90] ++(-3,0);
\draw (1.1,4) to[out=90,in=90] ++(2.2,0) to[out=270,in=90,looseness=1] ++(-1.1,-4/3) ++(0.4,0) to[out=90,in=270,looseness=1] ++(1.1,4/3) to[out=90,in=90] ++(-3,0);
\path[fill=white] (1.7,4) to[out=90,in=90] ++(0.5,0) to[out=270,in=90,looseness=1] ++(1.1,-4/3) -- ++(0.4,0) to[out=90,in=270,looseness=1] ++(-1.1,4/3) to[out=90,in=90] ++(-1.3,0);
\draw (1.7,4) to[out=90,in=90] ++(0.5,0) to[out=270,in=90,looseness=1] ++(1.1,-4/3) ++(0.4,0) to[out=90,in=270,looseness=1] ++(-1.1,4/3) to[out=90,in=90] ++(-1.3,0);

\path[fill=white] (1.7,0) to[out=270,in=270] ++(0.5,0) to[out=90,in=270,looseness=1] ++(1.1,4/3) -- ++(0.4,0) to[out=270,in=90,looseness=1] ++(-1.1,-4/3) to[out=270,in=270] ++(-1.3,0);
\draw (1.7,0) to[out=270,in=270] ++(0.5,0) to[out=90,in=270,looseness=1] ++(1.1,4/3) ++(0.4,0) to[out=270,in=90,looseness=1] ++(-1.1,-4/3) to[out=270,in=270] ++(-1.3,0);
\path[fill=white] (1.1,0) to[out=270,in=270] ++(2.2,0) to[out=90,in=270,looseness=1] ++(-1.1,4/3) -- ++(0.4,0) to[out=270,in=90,looseness=1] ++(1.1,-4/3) to[out=270,in=270] ++(-3,0);
\draw (1.1,0) to[out=270,in=270] ++(2.2,0) to[out=90,in=270,looseness=1] ++(-1.1,4/3) ++(0.4,0) to[out=270,in=90,looseness=1] ++(1.1,-4/3) to[out=270,in=270] ++(-3,0);
\end{scope}

\draw[red,ultra thick] (2.95,2.5) arc (90:270:0.15 and 0.5); % back of K
\begin{scope}[every path/.append style={looseness=2}] % rest of exterior
\path[fill=white] (2.2,1.8) rectangle (3.7,2.2);
\path[fill=white] (2.2,4/3) rectangle (2.6,8/3);
\path[fill=white] (3.3,4/3) rectangle (3.7,8/3);
\path[fill=white] (2.7,2.2) arc (270:180:0.1) -- ++(0,-0.1) -- ++(0.1,0); % NW corner of N(alpha)
\path[fill=white] (3.2,2.2) arc (-90:0:0.1) -- ++(0,-0.1) -- ++(-0.1,0); % NE corner of N(alpha)
\path[fill=white] (2.7,1.8) arc (90:180:0.1) -- ++(0,0.1) -- ++(0.1,0); % SW corner of N(alpha)
\path[fill=white] (3.2,1.8) arc (90:0:0.1) -- ++(0,0.1) -- ++(-0.1,0); % SE corner of N(alpha)
\draw (2.2,4/3) -- ++(0,4/3) (3.7,4/3) -- ++(0,4/3);
\draw (2.6,8/3) -- (2.6,2.3) arc (180:270:0.1) -- ++(0.5,0) arc (270:360:0.1) -- (3.3,8/3);
\draw (2.6,4/3) -- (2.6,1.7) arc (180:90:0.1) -- ++(0.5,0) arc (90:0:0.1) -- (3.3,4/3);
\end{scope}
\begin{scope}
\clip (2.35,1.8) rectangle (3.15,2.2);
\draw[ultra thick] (2.95,2) ++(0,0.22) arc (90:-90:0.05 and 0.22) node[midway,left,inner sep=1pt] {\tiny$s(\gamma)$};
\end{scope}
\draw[red,ultra thick] (2.95,2.5) node[above] {$K$} arc (90:-90:0.15 and 0.5);
\end{scope}
\end{tikzpicture}
\caption{The involution $\iota$ of $S^3(F) \cong M_F\setminus N(\alpha)$  in the case where $M_F \cong S^3 \setminus N(C_{2,4}(T_{2,3}))$, given by $180^\circ$ rotation about the horizontal axis (in blue). The meridian of $\alpha$ (in red) is isotopic in $S^3(F)$ to a pushoff of $K$.}
\label{fig:E-C24-T23-quotient-1}
\end{figure}
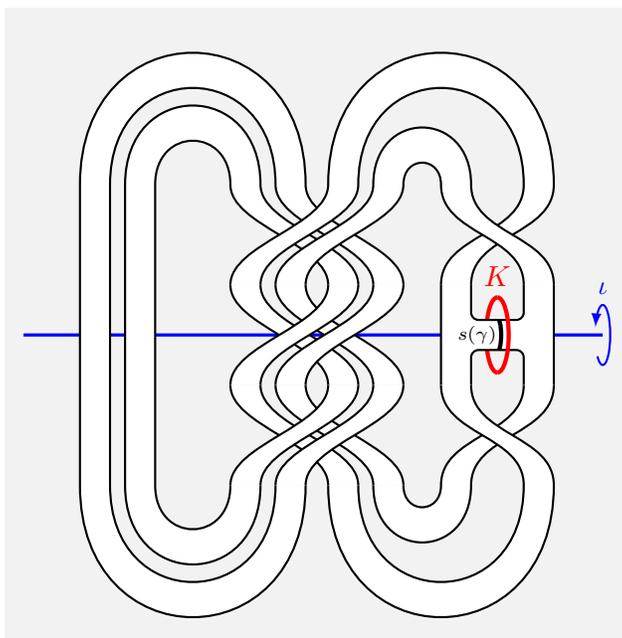
\begin{figure}
\begin{tikzpicture}
%%% quotient and isotopy
% Quotient by the involution
\begin{scope}[yshift=-9cm,scale=2/3]
\path[fill=gray!10] (-3.6,0) rectangle (4.7,6.1);
\draw[blue,very thick,looseness=1.5] (-2.85,2) -- (4,2) arc (90:0:0.25) -- ++(0,-0.9) arc (360:270:0.25) -- ++(-6.85,0) arc (270:180:0.25) -- ++(0,0.9) arc (180:90:0.25); % branch locus

\begin{scope}[every path/.append style={looseness=0.75}] % twist region
\foreach \y in {4} {
\foreach \x in {-0.6,0} {
  \path[fill=white] (\x,\y) to[out=270,in=90] ++(1.3,-5/3) -- ++(0.4,0) to[out=90,in=270] ++(-1.3,5/3) -- cycle;
  \draw (\x,\y) to[out=270,in=90] ++(1.3,-5/3) ++(0.4,0) to[out=90,in=270] ++(-1.3,5/3);
}
\foreach \x in {1.1,1.7} {
  \path[fill=white] (\x,\y) to[out=270,in=90] ++(-1.3,-5/3) -- ++(-0.4,0) to[out=90,in=270] ++(1.3,5/3) -- cycle;
  \draw (\x,\y) to[out=270,in=90] ++(-1.3,-5/3) ++ (-0.4,0) to[out=90,in=270] ++(1.3,5/3);
}
}
\end{scope}
\begin{scope}[every path/.append style={looseness=2}]
\foreach \x in {1.7,1.1} {
  \draw[fill=white] (\x,7/3) -- ++(0,-1/3) to[out=270,in=270] ++(-1.7,0) -- ++(0,1/3) ++(0.4,0) -- ++(0,-1/3) to[out=270,in=270] ++(0.9,0) -- ++(0,1/3);
}
\draw[blue,very thick] (0.55,2) -- (-1,2); % redraw part of axis
\end{scope}

\begin{scope}[every path/.append style={looseness=2}] % left side of trefoil
\path[fill=white] (-0.6,4) to[out=90,in=90] ++(-1,0) -- ++(0,-2) arc (360:180:0.2) -- ++(0,2) to[out=90,in=90] ++(1.8,0);
\draw (-0.6,4) to[out=90,in=90] ++(-1,0) -- ++(0,-2) arc (360:180:0.2) -- ++(0,2) to[out=90,in=90] ++(1.8,0);

\path[fill=white] (0,4) to[out=90,in=90] ++(-2.2,0) -- ++(0,-2) arc (360:180:0.2) -- ++(0,2) to[out=90,in=90] ++(3,0);
\draw (0,4) to[out=90,in=90] ++(-2.2,0) -- ++(0,-2) arc (360:180:0.2) -- ++(0,2) to[out=90,in=90] ++(3,0);
\end{scope}

\draw[red,ultra thick] (2.95,2.5) arc (90:270:0.15 and 0.5); % back of K
\begin{scope}[every path/.append style={looseness=2}] % rest of exterior
\path[fill=white] (2.3,1.8) rectangle (3.6,2.2);
\path[fill=white] (2.7,2.2) arc (270:180:0.1) -- ++(0,-0.1) -- ++(0.1,0); % NW corner of N(alpha)
\path[fill=white] (3.2,2.2) arc (-90:0:0.1) -- ++(0,-0.1) -- ++(-0.1,0); % NE corner of N(alpha)
\draw (2.6,8/3) -- (2.6,2.3) arc (180:270:0.1) -- ++(0.5,0) arc (270:360:0.1) -- (3.3,8/3);
\end{scope}

\begin{scope}[every path/.append style={looseness=2}] % right side of trefoil
\path[fill=white] (1.1,4) to[out=90,in=90] ++(2.2,0) to[out=270,in=90,looseness=1] ++(-1.1,-4/3) -- ++(0,1.9-8/3) arc (180:270:0.1) -- ++(0.3,0) -- ++(0,8/3-1.8) to[out=90,in=270,looseness=1] ++(1.1,4/3) to[out=90,in=90] ++(-3,0);
\draw (1.1,4) to[out=90,in=90] ++(2.2,0) to[out=270,in=90,looseness=1] ++(-1.1,-4/3) -- ++(0,1.9-8/3) arc (180:270:0.1) -- ++(0.3,0) ++(0,0.5) -- ++(0,8/3-2.3) to[out=90,in=270,looseness=1] ++(1.1,4/3) to[out=90,in=90] ++(-3,0);
\path[fill=white] (1.7,4) to[out=90,in=90] ++(0.5,0) to[out=270,in=90,looseness=1] ++(1.1,-4/3) -- ++(0,1.8-8/3) -- ++(0.3,0) arc (270:360:0.1) -- ++(0,8/3-1.9) to[out=90,in=270,looseness=1] ++(-1.1,4/3) to[out=90,in=90] ++(-1.3,0);
\draw (1.7,4) to[out=90,in=90] ++(0.5,0) to[out=270,in=90,looseness=1] ++(1.1,-4/3) ++(0,1.8-8/3) -- ++(0.3,0) arc (270:360:0.1) -- ++(0,8/3-1.9) to[out=90,in=270,looseness=1] ++(-1.1,4/3) to[out=90,in=90] ++(-1.3,0);
%\path[fill=white] (1.7,0) to[out=270,in=270] ++(0.5,0) to[out=90,in=270,looseness=1] ++(1.1,4/3) -- ++(0.4,0) to[out=270,in=90,looseness=1] ++(-1.1,-4/3) to[out=270,in=270] ++(-1.3,0);
%\draw (1.7,0) to[out=270,in=270] ++(0.5,0) to[out=90,in=270,looseness=1] ++(1.1,4/3) ++(0.4,0) to[out=270,in=90,looseness=1] ++(-1.1,-4/3) to[out=270,in=270] ++(-1.3,0);
%\path[fill=white] (1.1,0) to[out=270,in=270] ++(2.2,0) to[out=90,in=270,looseness=1] ++(-1.1,4/3) -- ++(0.4,0) to[out=270,in=90,looseness=1] ++(1.1,-4/3) to[out=270,in=270] ++(-3,0);
%\draw (1.1,0) to[out=270,in=270] ++(2.2,0) to[out=90,in=270,looseness=1] ++(-1.1,4/3) ++(0.4,0) to[out=270,in=90,looseness=1] ++(1.1,-4/3) to[out=270,in=270] ++(-3,0);
\draw (2.6,1.8) -- (3.3,1.8);
\draw (2.6,8/3) -- (2.6,2.3) arc (180:270:0.1) -- (3.2,2.2) arc (270:360:0.1) -- (3.3,8/3);
\end{scope}
\begin{scope}
\clip (2.8,1.8) rectangle (3.2,2.2);
\draw[ultra thick] (2.95,2.22) arc (90:-90:0.04 and 0.22);% node[midway,left,inner sep=1pt] {\tiny$\gamma$};
\end{scope}
\draw[red,ultra thick] (2.95,2.5) arc (90:-90:0.15 and 0.5);
\end{scope}

% Do an isotopy
\begin{scope}[xshift=7cm,yshift=-10cm]
\tikzset{linkb/.style = { gray!10, double = blue, line width = 1.75pt, double distance = 1.25pt, looseness=1.75 }}
\path[fill=gray!10] (-3.35,0) rectangle (4.05,6.1);
\begin{scope}[yshift=-0.2cm]
\draw[red,ultra thick] (2.5,2) ++(0,0.55) arc (90:270:0.2 and 0.55);
% start filling in complement
\path[fill=white] (2,2.4) -- ++(1,0) arc (90:-90:0.15 and 0.4) -- ++(-1,0) arc (-90:90:0.15 and 0.4);
% parts of branch locus near complement, with left-handed twists
\draw[linkb,looseness=1.5] (-2.75,2) ++(0,-0.2) to[out=180,in=180] ++(0,-1) -- ++(6,0) to[out=0,in=0] ++(0,1) -- ++(-0.25,0);
\draw[linkb,looseness=0.75] (3,2.2) to[out=0,in=270,looseness=1.5] (3.4,2.4) to[out=90,in=270] ++(0.2,0.3) ++(-0.2,0) to[out=90,in=270] ++(0.2,0.3); % right side
\draw[linkb,looseness=0.75] (3,2) to[out=0,in=270,looseness=1.5] (3.6,2.4) to[out=90,in=270] ++(-0.2,0.3) ++(0.2,0) to[out=90,in=270] ++(-0.2,0.3);
\begin{scope} % draw complement, after isotopy
\draw[thin] (3,2.4) arc (90:270:0.15 and 0.4);
\draw[fill=white,fill opacity=0.66] (2,2.4) -- ++(1,0) arc (90:-90:0.15 and 0.4) -- ++(-1,0) arc (-90:90:0.15 and 0.4);
\draw[fill=white] (2,2) ellipse (0.15 and 0.4);
\end{scope}
\draw[blue,very thick,looseness=0.75] (2,2) to[out=180,in=270,looseness=1.5] (1.4,2.4); % left side
\draw[blue,very thick,looseness=0.75] (2,2.2) to[out=180,in=270,looseness=1.5] (1.6,2.4);
\draw[blue,very thick] (2,1.8) -- ++(-0.6,0);
\draw[linkb,looseness=0.75] (1.4,2.4) to[out=90,in=270] ++(0.2,0.3) ++(-0.2,0) to[out=90,in=270] ++(0.2,0.3);
\draw[linkb,looseness=0.75] (1.6,2.4) to[out=90,in=270] ++(-0.2,0.3) ++(0.2,0) to[out=90,in=270] ++(-0.2,0.3);
% crossing near complement
\foreach \x in {1.4,1.6} { \draw[linkb] (\x,3) to[out=90,in=270,looseness=0.75] ++(1.4,1.5); }
\foreach \x in {3.4,3.6} { \draw[linkb] (\x,3) to[out=90,in=270,looseness=0.75] ++(-1.4,1.5); }

\draw[linkb] (1.4,1.8) -- (0,1.8); % extend strand coming from braid complement
\draw[linkb] (1,2) to[out=270,in=270,looseness=2.25] ++(-1.4,0) (0.8,2) to[out=270,in=270,looseness=2.5] ++(-1,0);
\draw[linkb] (0.4,2) to[out=270,in=270,looseness=2.25] ++(-1.4,0) (0.2,2) to[out=270,in=270,looseness=2.5] ++(-1,0);
\foreach \x in {0.2,0.4,0.8,1} { \draw[linkb] (\x,2) to[out=90,in=270,looseness=0.75] ++(-1.4,1.5); }
\foreach \x in {-0.2,-0.4,-0.8,-1} { \draw[linkb] (\x,2) to[out=90,in=270,looseness=0.75] ++(1.4,1.5); }
\draw[linkb] (0,1.8) -- (-1.2,1.8); % extend strand coming from braid complement

% join strands to curve on bottom left
\draw[linkb] (-1,3.5) to[out=90,in=90] ++(-0.8,0) -- ++(0,-1.5);
\draw[linkb] (-1.2,3.5) to[out=90,in=90] ++(-0.4,0) -- ++(0,-1.5);
\draw[linkb] (-0.4,3.5) to[out=90,in=90] ++(-2,0) -- ++(0,-1.5);
\draw[linkb] (-0.6,3.5) to[out=90,in=90] ++(-1.6,0) -- ++(0,-1.5);
\draw[linkb] (-2.4,2) to[out=270,in=0] ++(-0.2,-0.2) -- ++(-0.15,0);
\draw[linkb] (-2.2,2) to[out=270,in=270] ++(0.4,0);
\draw[linkb] (-1.6,2) to[out=270,in=180] ++(0.2,-0.2) -- ++(0.4,0);

% add crossings joining middle to right
\draw[linkb] (0.4,3.5) -- ++(0,1) to[out=90,in=90,looseness=1.75] (3.0,4.5);
\draw[linkb] (0.6,3.5) -- ++(0,1) to[out=90,in=90,looseness=1.75] (2.8,4.5);
\draw[linkb] (1.0,3.5) -- ++(0,1) to[out=90,in=90,looseness=1.75] (2.2,4.5);
\draw[linkb] (1.2,3.5) -- ++(0,1) to[out=90,in=90,looseness=1.75] (2.0,4.5);

\node[blue,left,inner sep=2pt] at (0.4,4.5) {$\tau$};

\draw[red,ultra thick] (2.5,2) ++(0,0.55) node[above] {$\kappa$} arc (90:-90:0.2 and 0.55);
\begin{scope} % draw image of suture
\clip (2.4,1.6) rectangle (2.8,2.4);
\draw[ultra thick] (2.5,2.42) arc (90:-90:0.095 and 0.42);% node[midway,left,inner sep=1pt] {\tiny$\gamma$};
\end{scope}

\end{scope}
\end{scope}

% \draw[-latex] (8.6,3.45) -- node[right] {\small{isotopy}} ++(0,-0.7); % leftover from the N(T_{2,4}) case
\end{tikzpicture}
\caption{Taking the quotient of $S^3(F)$  by the involution $\iota$ from  Figure~\ref{fig:E-C24-T23-quotient-1}, followed by an isotopy.  The quotient has branch locus $\tau$ (blue), and a curve $\kappa$ (red) which lifts to  $K$.}
\label{fig:E-C24-T23-quotient-2}
\end{figure}
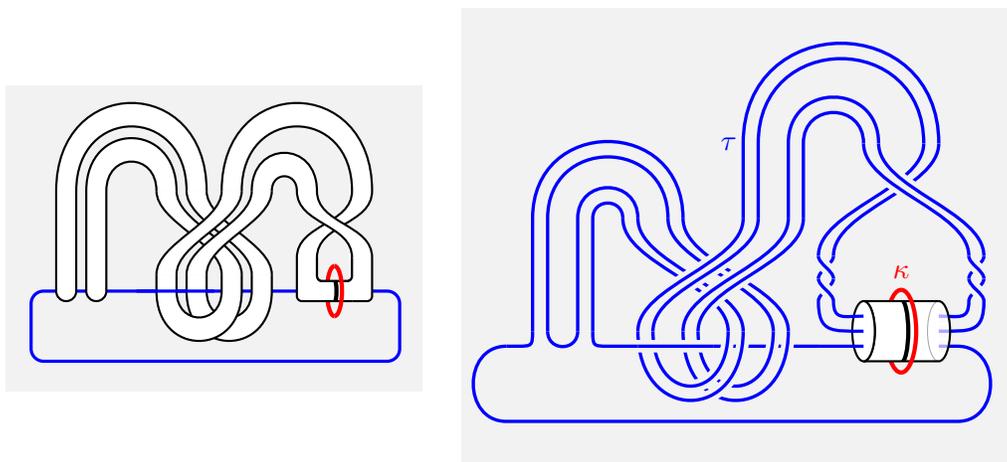

\begin{figure}
\begin{tikzpicture}
\begin{scope}
\draw[linkred,ultra thick] (3.1,2) ++(0,0.55) arc (90:270:0.15 and 0.55);
% parts of branch locus near complement, with left-handed twists
\draw[link,looseness=1.5] (-2.4,2) ++(0,-0.2) to[out=270,in=180] ++(1,-1) -- ++(4.65,0) to[out=0,in=0] ++(0,1) -- ++(-0.25,0);
\draw[link,looseness=0.75] (2,2.2) -- ++(1,0) to[out=0,in=270,looseness=1.5] (3.4,2.4) to[out=90,in=270] ++(0.2,0.3) ++(-0.2,0) to[out=90,in=270] ++(0.2,0.3); % right side
\draw[link,looseness=0.75] (2,2) -- ++(1,0) to[out=0,in=270,looseness=1.5] (3.6,2.4) to[out=90,in=270] ++(-0.2,0.3) ++(0.2,0) to[out=90,in=270] ++(-0.2,0.3);
\draw[link,looseness=0.75] (2,2) to[out=180,in=270,looseness=1.5] (1.4,2.4); % left side
\draw[link,looseness=0.75] (2,2.2) to[out=180,in=270,looseness=1.5] (1.6,2.4);
\draw[link] (3,1.8) -- ++(-1.6,0);
\draw[link,looseness=0.75] (1.4,2.4) to[out=90,in=270] ++(0.2,0.3) ++(-0.2,0) to[out=90,in=270] ++(0.2,0.3);
\draw[link,looseness=0.75] (1.6,2.4) to[out=90,in=270] ++(-0.2,0.3) ++(0.2,0) to[out=90,in=270] ++(-0.2,0.3);
% crossing near complement
\foreach \x in {1.4,1.6} { \draw[link] (\x,3) to[out=90,in=270,looseness=0.75] ++(1.4,1.5); }
\foreach \x in {3.4,3.6} { \draw[link] (\x,3) to[out=90,in=270,looseness=0.75] ++(-1.4,1.5); }

\draw[link] (1.4,1.8) -- (0,1.8); % extend strand coming from braid complement
\draw[link] (1,2) to[out=270,in=270,looseness=2.25] ++(-1.4,0) (0.8,2) to[out=270,in=270,looseness=2.5] ++(-1,0);
\draw[link] (0.4,2) to[out=270,in=270,looseness=2.25] ++(-1.4,0) (0.2,2) to[out=270,in=270,looseness=2.5] ++(-1,0);
\foreach \x in {0.2,0.4,0.8,1} { \draw[link] (\x,2) to[out=90,in=270,looseness=0.75] ++(-1.4,1.5); }
\foreach \x in {-0.2,-0.4,-0.8,-1} { \draw[link] (\x,2) to[out=90,in=270,looseness=0.75] ++(1.4,1.5); }
\draw[link] (0,1.8) -- (-1.2,1.8); % extend strand coming from braid complement

% join strands to curve on bottom left
\draw[link] (-1,3.5) to[out=90,in=90] ++(-0.8,0) -- ++(0,-1.5);
\draw[link] (-1.2,3.5) to[out=90,in=90] ++(-0.4,0) -- ++(0,-1.5);
\draw[link] (-0.4,3.5) to[out=90,in=90] ++(-2,0) -- ++(0,-1.5);
\draw[link] (-0.6,3.5) to[out=90,in=90] ++(-1.6,0) -- ++(0,-1.5);
\draw[link] (-2.4,2) -- ++(0,-0.2);
\draw[link] (-2.2,2) to[out=270,in=270] ++(0.4,0);
\draw[link] (-1.6,2) to[out=270,in=180] ++(0.2,-0.2) -- ++(0.4,0);

% add crossings joining middle to right
\draw[link] (0.4,3.5) -- ++(0,1) to[out=90,in=90,looseness=1.75] (3.0,4.5);
\draw[link] (0.6,3.5) -- ++(0,1) to[out=90,in=90,looseness=1.75] (2.8,4.5);
\draw[link] (1.0,3.5) -- ++(0,1) to[out=90,in=90,looseness=1.75] (2.2,4.5);
\draw[link] (1.2,3.5) -- ++(0,1) to[out=90,in=90,looseness=1.75] (2.0,4.5);

\node[left,inner sep=2pt] at (0.4,4.5) {$\tau$};
\begin{scope}[xshift=-0.2cm]
\draw[very thick,fill=white] (2,1.55) rectangle (3,2.45);
\node at (2.5,2) {\Large$\beta$};
\end{scope}
\draw[linkred,ultra thick] (3.1,2) ++(0,0.55) node[red,above] {$\kappa$} arc (90:-90:0.15 and 0.55);
\end{scope}

%%%

\begin{scope}[xshift=7cm]
\draw[linkred,ultra thick] (3.1,2) ++(0,0.55) arc (90:270:0.15 and 0.55);
% parts of branch locus near complement, with left-handed twists
\draw[link,looseness=1.5] (-2.4,2) ++(0,-0.2) to[out=270,in=180] ++(1,-1) -- ++(4.65,0) to[out=0,in=0] ++(0,1) -- ++(-0.25,0);
\draw[link,looseness=0.75] (2,2.2) -- ++(1,0) to[out=0,in=270,looseness=1.5] (3.4,2.4) to[out=90,in=270] ++(0.2,0.3) ++(-0.2,0) to[out=90,in=270] ++(0.2,0.3); % right side
\draw[link,looseness=0.75] (2,2) -- ++(1,0) to[out=0,in=270,looseness=1.5] (3.6,2.4) to[out=90,in=270] ++(-0.2,0.3) ++(0.2,0) to[out=90,in=270] ++(-0.2,0.3);
\draw[link,looseness=0.75] (2,2) to[out=180,in=270,looseness=1.5] (1.4,2.4); % left side
\draw[link,looseness=0.75] (2,2.2) to[out=180,in=270,looseness=1.5] (1.6,2.4);
\draw[link] (3,1.8) -- ++(-1.6,0);
\draw[link,looseness=0.75] (1.4,2.4) to[out=90,in=270] ++(0.2,0.3) ++(-0.2,0) to[out=90,in=270] ++(0.2,0.3);
\draw[link,looseness=0.75] (1.6,2.4) to[out=90,in=270] ++(-0.2,0.3) ++(0.2,0) to[out=90,in=270] ++(-0.2,0.3);
% crossing near complement
\foreach \x in {1.4,1.6} { \draw[link] (\x,3) to[out=90,in=270,looseness=0.75] ++(1.4,1.5); }
\foreach \x in {3.4,3.6} { \draw[link] (\x,3) to[out=90,in=270,looseness=0.75] ++(-1.4,1.5); }

\draw[link] (1.4,1.8) -- (0,1.8); % extend strand coming from braid complement
\draw[link] (1,2) to[out=270,in=270,looseness=2.25] ++(-1.4,0);
\draw[link] (0.2,2) to[out=270,in=270,looseness=2.5] ++(-1,0);
\foreach \x in {0.2,1} { \draw[link] (\x,2) to[out=90,in=270,looseness=0.75] ++(-1.4,1.5); }
\foreach \x in {-0.4,-0.8} { \draw[link] (\x,2) to[out=90,in=270,looseness=0.75] ++(1.4,1.5); }
\draw[link] (0,1.8) -- (-1.2,1.8); % extend strand coming from braid complement

% join strands to curve on bottom left
\draw[link] (-1.2,3.5) to[out=90,in=90] ++(-0.4,0) -- ++(0,-1.5);
\draw[link] (-0.4,3.5) to[out=90,in=90] ++(-2,0) -- ++(0,-1.5);
\draw[link] (-2.4,2) -- ++(0,-0.2);
\draw[link] (-1.6,2) to[out=270,in=180] ++(0.2,-0.2) -- ++(0.4,0);

% add crossings joining middle to right
%\draw[link] (0.4,4.5) to[out=90,in=90,looseness=1.75] (3.0,4.5);
\draw[link] (0.6,3.5) -- ++(0,1) to[out=90,in=90,looseness=1.75] (2.8,4.5);
\draw[link] (1.0,3.5) -- ++(0,1) to[out=90,in=90,looseness=1.75] (2.2,4.5);
%\draw[link] (1.2,4.5) to[out=90,in=90,looseness=1.75] (2.0,4.5);

% close up tangle with new arc next to the "\tau" label
%\draw[link] (0.4,4.5) to[out=270,in=270] (1.2,4.5);
\draw[link] (2,4.5) to[out=90,in=90] (3,4.5);
\draw[link] (2.8,4.5) to[out=90,in=90,looseness=1.75] (0.6,4.5); % redraw crossing

\node[left,inner sep=2pt] at (0.6,4.5) {$\tau$};
\begin{scope}[xshift=-0.2cm]
\draw[very thick,fill=white] (2,1.55) rectangle (3,2.45);
\node at (2.5,2) {\Large$\beta$};
\end{scope}
\draw[linkred,ultra thick] (3.1,2) ++(0,0.55) node[red,above] {$\kappa$} arc (90:-90:0.15 and 0.55);
\end{scope}

%%%

\begin{scope}[yshift=-5.5cm]
\draw[linkred,ultra thick] (3.1,2) ++(0,0.55) arc (90:270:0.15 and 0.55);
% parts of branch locus near complement, with left-handed twists
\draw[link,looseness=1.5] (-2.4,2) ++(0,-0.2) to[out=270,in=180] ++(1,-1) -- ++(4.65,0) to[out=0,in=0] ++(0,1) -- ++(-0.25,0);
\draw[link,looseness=0.75] (2,2.2) -- ++(1,0) to[out=0,in=270,looseness=1.5] (3.4,2.4) to[out=90,in=0,looseness=1.5] ++(-0.4,0.4); % right side
\draw[link,looseness=0.75] (2,2) -- ++(1,0) to[out=0,in=270,looseness=1.5] (3.6,2.4) -- ++(0,0.6);
\draw[link,looseness=0.75] (2,2) to[out=180,in=270,looseness=1.5] (1.4,2.4); % left side
\draw[link,looseness=0.75] (2,2.2) to[out=180,in=270,looseness=1.5] (1.6,2.4);
\draw[link] (3,1.8) -- ++(-1.6,0);
\draw[link,looseness=0.75] (1.4,2.4) -- ++(0,0.6);
\draw[link,looseness=0.75] (1.6,2.4) to[out=90,in=180,looseness=1.5] ++(0.4,0.4) -- ++(1,0);
% crossing near complement
\foreach \x in {1.4} { \draw[link] (\x,3) to[out=90,in=270,looseness=0.75] ++(1.4,1.5); }
\foreach \x in {3.6} { \draw[link] (\x,3) to[out=90,in=270,looseness=0.75] ++(-1.4,1.5); }

\draw[link] (1.4,1.8) -- (0,1.8); % extend strand coming from braid complement
\draw[link] (1,2) to[out=270,in=270,looseness=2.25] ++(-1.4,0);
\draw[link] (0.2,2) to[out=270,in=270,looseness=2.5] ++(-1,0);
\foreach \x in {0.2,1} { \draw[link] (\x,2) to[out=90,in=270,looseness=0.75] ++(-1.4,1.5); }
\foreach \x in {-0.4,-0.8} { \draw[link] (\x,2) to[out=90,in=270,looseness=0.75] ++(1.4,1.5); }
\draw[link] (0,1.8) -- (-1.2,1.8); % extend strand coming from braid complement

% join strands to curve on bottom left
\draw[link] (-1.2,3.5) to[out=90,in=90] ++(-0.4,0) -- ++(0,-1.5);
\draw[link] (-0.4,3.5) to[out=90,in=90] ++(-2,0) -- ++(0,-1.5);
\draw[link] (-2.4,2) -- ++(0,-0.2);
\draw[link] (-1.6,2) to[out=270,in=180] ++(0.2,-0.2) -- ++(0.4,0);

% add crossings joining middle to right
\draw[link] (0.6,3.5) -- ++(0,1) to[out=90,in=90,looseness=1.75] (2.8,4.5);
\draw[link] (1.0,3.5) -- ++(0,1) to[out=90,in=90,looseness=1.75] (2.2,4.5);

\node[left,inner sep=2pt] at (0.6,4.5) {$\tau$};
\begin{scope}[xshift=-0.2cm]
\draw[very thick,fill=white] (2,1.55) rectangle (3,2.45);
\node at (2.5,2) {\Large$\beta$};
\end{scope}
\draw[linkred,ultra thick] (3.1,2) ++(0,0.55) arc (90:-90:0.15 and 0.55) node[red,below] {$\kappa$};
\end{scope}

%%%

\begin{scope}[yshift=-5.5cm,xshift=7cm]
\draw[linkred,ultra thick] (3.1,2) ++(0,0.55) arc (90:270:0.15 and 0.55);
% parts of branch locus near complement, with left-handed twists
\draw[link,looseness=1.5] (2,0.8) -- ++(1.25,0) to[out=0,in=0] ++(0,1) -- ++(-0.25,0);
\draw[link,looseness=0.75] (2,2.2) -- ++(1,0) to[out=0,in=270,looseness=1.5] (3.4,2.4) to[out=90,in=0,looseness=1.5] ++(-0.4,0.4); % right side
\draw[link,looseness=0.75] (2,2) -- ++(1,0) to[out=0,in=270,looseness=1.5] (3.6,2.4) -- ++(0,0.6);
\draw[link,looseness=0.75] (2,2) to[out=180,in=270,looseness=1.5] (1.4,2.4); % left side
\draw[link,looseness=0.75] (2,2.2) to[out=180,in=270,looseness=1.5] (1.6,2.4);
\draw[link] (3,1.8) -- ++(-1.6,0);
\draw[link,looseness=0.75] (1.4,2.4) -- ++(0,0.6);
\draw[link,looseness=0.75] (1.6,2.4) to[out=90,in=180,looseness=1.5] ++(0.4,0.4) -- ++(1,0);
% crossing near complement
\foreach \x in {1.4} { \draw[link] (\x,3) to[out=90,in=270,looseness=0.75] ++(1.4,1.5); }
\foreach \x in {3.6} { \draw[link] (\x,3) to[out=90,in=270,looseness=0.75] ++(-1.4,1.5); }

\draw[link] (1.4,1.8) -- (0,1.8); % extend strand coming from braid complement
\draw[link,looseness=2] (-0.8,2) -- ++(0,-0.2) to[out=270,in=270] ++(1,0) to[out=90,in=90] ++(1,0) -- ++(0,-0.2) to[out=270,in=180,looseness=1] ++(0.8,-0.8);
\draw[link] (1,1.8) -- (1.4,1.8); % fix a crossing
\draw[link] (0,1.8) -- (-1.2,1.8); % extend strand coming from braid complement
\draw[link] (-1.2,2) to[out=270,in=270,looseness=2.5] ++(1,-0);
\draw[link] (-0.2,2) to[out=90,in=90,looseness=2] ++(-1.5,0.3) to[out=270,in=180,looseness=0.75] ++(1,-0.5);
\foreach \x in {-0.4,-0.8} { \draw[link] (\x-0.4,2) to[out=90,in=270,looseness=0.75] ++(1.8,1.5); }

% add crossings joining middle to right
\draw[link] (0.6,3.5) -- ++(0,1) to[out=90,in=90,looseness=1.75] (2.8,4.5);
\draw[link] (1.0,3.5) -- ++(0,1) to[out=90,in=90,looseness=1.75] (2.2,4.5);

\node[left,inner sep=2pt] at (0.6,4.5) {$\tau$};
\begin{scope}[xshift=-0.2cm]
\draw[very thick,fill=white] (2,1.55) rectangle (3,2.45);
\node at (2.5,2) {\Large$\beta$};
\end{scope}
\draw[linkred,ultra thick] (3.1,2) ++(0,0.55) arc (90:-90:0.15 and 0.55) node[red,below] {$\kappa$};
\end{scope}

\end{tikzpicture}
\caption{An isotopy of the tangle $\tau \cup \beta$ in the complement of $\kappa$.}
\label{fig:tau-T23-isotopy}
\end{figure}
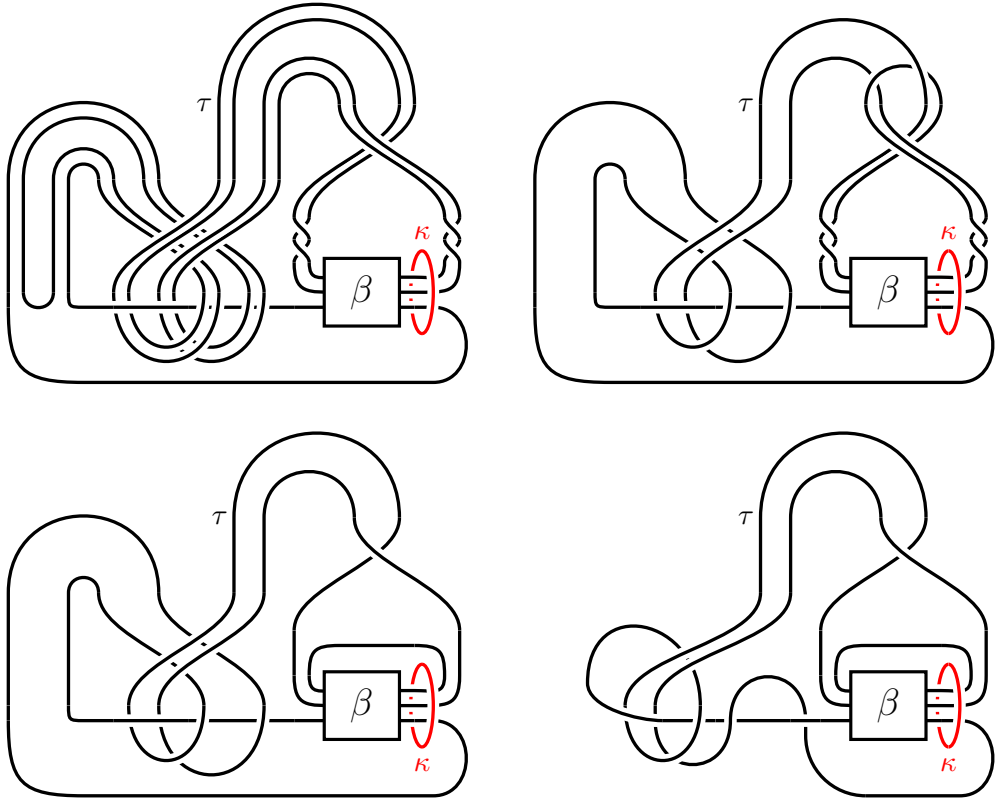

\begin{lemma} \label{lem:C24-T23-as-branched-cover}
Suppose that $K\subset S^3$ is a nearly fibered knot with genus-1 Seifert surface $F$, and that \[M_F \cong S^3 \setminus N(C_{2,4}(T_{2,3})).\]  Then there is a tangle $\tau$ and a 3-braid $\beta \in B_3$, depicted in Figure~\ref{fig:tau-T23-isotopy}, such that $\tau\cup\beta$ is an unknot in $S^3$, and such that the lift
\[ \tilde\kappa \subset \dcover(\tau\cup\beta) \cong S^3 \]
of the pictured curve $\kappa$ is isotopic to $K$.
\end{lemma}

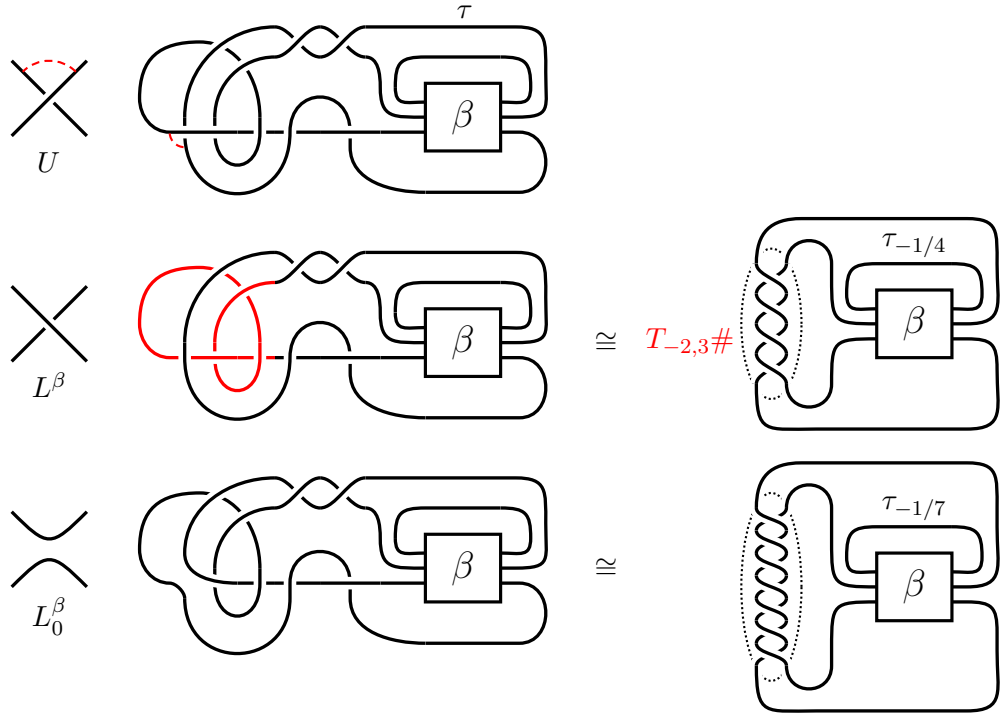
\begin{figure}
\begin{tikzpicture}

\begin{scope}
%\draw[linkred,ultra thick] (3.1,2) ++(0,0.55) arc (90:270:0.15 and 0.55);
% parts of branch locus near complement
\draw[link,looseness=1.5] (2,1) -- ++(1.25,0) to[out=0,in=0] ++(0,0.8) -- ++(-0.25,0); % below the braid
\draw[link,looseness=0.75] (2,2.2) -- ++(1,0) to[out=0,in=270,looseness=1.5] (3.4,2.4) to[out=90,in=0,looseness=1.5] ++(-0.4,0.4); % right side
\draw[link,looseness=0.75] (2,2) -- ++(1.2,0) to[out=0,in=270,looseness=1.75] (3.6,2.6) to[out=90,in=0,looseness=1.5] ++(-0.4,0.6) -- ++(-2,0);
\draw[link,looseness=0.75] (2,2) to[out=180,in=270,looseness=1.5] (1.4,2.4) to[out=90,in=0,looseness=1] (1.2,2.8); % left side
\draw[link,looseness=0.75] (2,2.2) to[out=180,in=270,looseness=1.5] (1.6,2.4);
\draw[link] (3,1.8) -- ++(-1.6,0);
\draw[link,looseness=0.75] (1.6,2.4) to[out=90,in=180,looseness=1.5] ++(0.4,0.4) -- ++(1,0);

% add a left-handed twist above the braid
\draw[link,looseness=0.75] (1.2,2.8) to[out=180,in=0] ++(-0.6,0.4) ++(0,-0.4) to[out=180,in=0] ++(-0.6,0.4);
\draw[link,looseness=0.75] (1.2,3.2) to[out=180,in=0] ++(-0.6,-0.4) ++(0,0.4) to[out=180,in=0] ++(-0.6,-0.4);

\draw[link] (1.4,1.8) -- (0,1.8); % extend strand coming from braid complement
\draw[link,looseness=2] (-1.2,2) -- ++(0,-0.2) to[out=270,in=270] ++(1.4,0) to[out=90,in=90] ++(0.8,0) -- ++(0,-0.2) to[out=270,in=180,looseness=1] ++(1,-0.6);
\draw[link] (0.6,1.8) -- (1.4,1.8); % fix a crossing
\draw[link] (0,1.8) -- (-0.5,1.8); % extend strand coming from braid complement
\draw[link] (-0.8,2) -- ++(0,-0.2) to[out=270,in=270,looseness=2.5] ++(0.6,-0) -- ++(0,0.2);
\draw[link] (-0.2,2) to[out=90,in=0,looseness=1] ++(-0.8,1) to[out=180,in=90,looseness=1.25] ++(-0.8,-0.75) to[out=270,in=180,looseness=1] ++(0.4,-0.45) -- ++(0.9,0);

\draw[link,looseness=1] (0,3.2) to[out=180,in=90] (-1.2,2);
\draw[link,looseness=1] (0,2.8) to[out=180,in=90] (-0.8,2);

\node[above,inner sep=3pt] at (2.525,3.2) {$\tau$};
\begin{scope}%[xshift=-0.2cm]
\draw[very thick,fill=white] (2,1.55) rectangle (3,2.45);
\node at (2.5,2) {\Large$\beta$};
\end{scope}
%\draw[linkred,ultra thick] (3.1,2) ++(0,0.55) arc (90:-90:0.15 and 0.55) node[red,below] {$\kappa$};
\draw[red,densely dashed] (-1.2,1.8) ++ (185:0.2) arc (185:270:0.2); % mark crossing for skein relation

\begin{scope}[xshift=-3.5cm,yshift=1.75cm]
\draw[link] (1,0) -- ++(-1,1);
\draw[link] (0,0) -- ++(1,1);
\draw[red,densely dashed] (0.5,0.5) ++ (45:0.5) arc (45:135:0.5);
\node[below] at (0.5,0) {$U^{\vphantom{\beta}}$};
\end{scope}
\end{scope}

%%%

\begin{scope}[yshift=-3cm]
%\draw[linkred,ultra thick] (3.1,2) ++(0,0.55) arc (90:270:0.15 and 0.55);
% parts of branch locus near complement
\draw[link,looseness=1.5] (2,1) -- ++(1.25,0) to[out=0,in=0] ++(0,0.8) -- ++(-0.25,0); % below the braid
\draw[link,looseness=0.75] (2,2.2) -- ++(1,0) to[out=0,in=270,looseness=1.5] (3.4,2.4) to[out=90,in=0,looseness=1.5] ++(-0.4,0.4); % right side
\draw[link,looseness=0.75] (2,2) -- ++(1.2,0) to[out=0,in=270,looseness=1.75] (3.6,2.6) to[out=90,in=0,looseness=1.5] ++(-0.4,0.6) -- ++(-2,0);
\draw[link,looseness=0.75] (2,2) to[out=180,in=270,looseness=1.5] (1.4,2.4) to[out=90,in=0,looseness=1] (1.2,2.8); % left side
\draw[link,looseness=0.75] (2,2.2) to[out=180,in=270,looseness=1.5] (1.6,2.4);
\draw[link] (3,1.8) -- ++(-1.6,0);
\draw[link,looseness=0.75] (1.6,2.4) to[out=90,in=180,looseness=1.5] ++(0.4,0.4) -- ++(1,0);

% add a left-handed twist above the braid
\draw[link,looseness=0.75] (1.2,2.8) to[out=180,in=0] ++(-0.6,0.4) ++(0,-0.4) to[out=180,in=0] ++(-0.6,0.4);
\draw[link,looseness=0.75] (1.2,3.2) to[out=180,in=0] ++(-0.6,-0.4) ++(0,0.4) to[out=180,in=0] ++(-0.6,-0.4);

\draw[link] (1.4,1.8) -- (0,1.8); % extend strand coming from braid complement
\draw[linkred] (0,1.8) -- (-0.5,1.8); % extend strand coming from braid complement
\draw[linkred] (-0.8,2) -- ++(0,-0.2) to[out=270,in=270,looseness=2.5] ++(0.6,-0) -- ++(0,0.2);
\draw[linkred] (-0.2,2) to[out=90,in=0,looseness=1] ++(-0.8,1) to[out=180,in=90,looseness=1.25] ++(-0.8,-0.75) to[out=270,in=180,looseness=1] ++(0.4,-0.45) -- ++(0.9,0);
\draw[link,looseness=2] (-1.2,2) -- ++(0,-0.2) to[out=270,in=270] ++(1.4,0) to[out=90,in=90] ++(0.8,0) -- ++(0,-0.2) to[out=270,in=180,looseness=1] ++(1,-0.6);
\draw[link] (0.6,1.8) -- (1.4,1.8); % fix a crossing

\draw[link,looseness=1] (0,3.2) to[out=180,in=90] (-1.2,2);
\draw[linkred,looseness=1] (0,2.8) to[out=180,in=90] (-0.8,2);

%\node[above,inner sep=3pt] at (2.525,3.2) {$\tau$};
\begin{scope}%[xshift=-0.2cm]
\draw[very thick,fill=white] (2,1.55) rectangle (3,2.45);
\node at (2.5,2) {\Large$\beta$};
\end{scope}
%\draw[linkred,ultra thick] (3.1,2) ++(0,0.55) arc (90:-90:0.15 and 0.55) node[red,below] {$\kappa$};

\begin{scope}[xshift=-3.5cm,yshift=1.75cm]
\draw[link] (0,0) -- ++(1,1);
\draw[link] (1,0) -- ++(-1,1);
\node[below] at (0.5,0) {$L^\beta$};
\end{scope}
\end{scope}

%%%

\begin{scope}[yshift=-3cm,xshift=6cm]
\node at (-1.6,2) {$\cong$};
\node[left,red] at (0.25,2) {$T_{-2,3}\#$};
\begin{scope}[yshift=0.25cm]
\node[above,inner sep=2pt] at (2.5,2.8) {$\tau_{-1/4}$};
\draw[densely dotted] (0.6,2) ellipse (0.4 and 1);
%\draw[linkred,ultra thick] (3.1,2) ++(0,0.55) arc (90:270:0.15 and 0.55);
% parts of branch locus near complement
\draw[link,looseness=0.75] (2,2.2) -- ++(1,0) to[out=0,in=270,looseness=1.5] (3.4,2.4) to[out=90,in=0,looseness=1.5] ++(-0.4,0.4); % right side
\draw[link,looseness=0.75] (2,2) -- ++(1.2,0) to[out=0,in=270,looseness=1.5] (3.6,2.8) to[out=90,in=0,looseness=1.5] ++(-0.4,0.6) -- ++(-2.2,0) to[out=180,in=90,looseness=1.25] (0.4,2.8);
\draw[link] (2,2) to[out=180,in=270] (1.4,2.6) -- ++(0,0.2) to[out=90,in=90] (0.8,2.8); % left side
\draw[link,looseness=0.75] (2,2.2) to[out=180,in=270,looseness=1.5] (1.6,2.4) to[out=90,in=180,looseness=1.5] ++(0.4,0.4) -- ++(1,0);

% add left-handed twists
\foreach \y in {0,...,3} {
 \draw[link,looseness=0.75] (0.8,2.8-0.4*\y) to[out=270,in=90] ++(-0.4,-0.4);
 \draw[link,looseness=0.75] (0.4,2.8-0.4*\y) to[out=270,in=90] ++(0.4,-0.4);
}

\draw[link] (0.8,1.2) to[out=270,in=270] (1.4,1.2) to[out=90,in=180] (2,1.8);
\draw[link] (0.4,1.2) to[out=270,in=180] (1.2,0.6) -- ++(2,0) to[out=0,in=270,looseness=1.5] (3.6,1.4) to[out=90,in=0,looseness=1.5] ++(-0.4,0.4) -- ++(-1.2,0); % below the braid

\begin{scope}%[xshift=-0.2cm]
\draw[very thick,fill=white] (2,1.55) rectangle (3,2.45);
\node at (2.5,2) {\Large$\beta$};
\end{scope}
%\draw[linkred,ultra thick] (3.1,2) ++(0,0.55) arc (90:-90:0.15 and 0.55) node[red,below] {$\kappa$};
\end{scope}
\end{scope}

%%%

\begin{scope}[yshift=-6cm]
%\draw[linkred,ultra thick] (3.1,2) ++(0,0.55) arc (90:270:0.15 and 0.55);
% parts of branch locus near complement
\draw[link,looseness=1.5] (2,1) -- ++(1.25,0) to[out=0,in=0] ++(0,0.8) -- ++(-0.25,0); % below the braid
\draw[link,looseness=0.75] (2,2.2) -- ++(1,0) to[out=0,in=270,looseness=1.5] (3.4,2.4) to[out=90,in=0,looseness=1.5] ++(-0.4,0.4); % right side
\draw[link,looseness=0.75] (2,2) -- ++(1.2,0) to[out=0,in=270,looseness=1.75] (3.6,2.6) to[out=90,in=0,looseness=1.5] ++(-0.4,0.6) -- ++(-2,0);
\draw[link,looseness=0.75] (2,2) to[out=180,in=270,looseness=1.5] (1.4,2.4) to[out=90,in=0,looseness=1] (1.2,2.8); % left side
\draw[link,looseness=0.75] (2,2.2) to[out=180,in=270,looseness=1.5] (1.6,2.4);
\draw[link] (3,1.8) -- ++(-1.6,0);
\draw[link,looseness=0.75] (1.6,2.4) to[out=90,in=180,looseness=1.5] ++(0.4,0.4) -- ++(1,0);

% add a left-handed twist above the braid
\draw[link,looseness=0.75] (1.2,2.8) to[out=180,in=0] ++(-0.6,0.4) ++(0,-0.4) to[out=180,in=0] ++(-0.6,0.4);
\draw[link,looseness=0.75] (1.2,3.2) to[out=180,in=0] ++(-0.6,-0.4) ++(0,0.4) to[out=180,in=0] ++(-0.6,-0.4);

\draw[link] (1.4,1.8) -- (0,1.8); % extend strand coming from braid complement
\draw[link] (0,1.8) -- (-0.5,1.8); % extend strand coming from braid complement
\draw[link] (-0.8,2) -- ++(0,-0.2) to[out=270,in=270,looseness=2.5] ++(0.6,-0) -- ++(0,0.2);
\draw[link] (-0.2,2) to[out=90,in=0,looseness=1] ++(-0.8,1) to[out=180,in=90,looseness=1.25] ++(-0.8,-0.75) to[out=270,in=180,looseness=1] ++(0.4,-0.45) to[out=0,in=90,looseness=0.75] ++(0.2,-0.2); % ends at (-1.2,1.6)
\draw[link,looseness=1.8] (-1.2,1.6) to[out=270,in=270] ++(1.4,0) -- ++(0,0.2) to[out=90,in=90] ++(0.8,0) -- ++(0,-0.2) to[out=270,in=180,looseness=1] ++(1,-0.6);
\draw[link] (0.6,1.8) -- (1.4,1.8); % fix a crossing

\draw[link,looseness=1] (0,3.2) to[out=180,in=90] (-1.2,2.2) to[out=270,in=180] (-0.5,1.8);
\draw[link,looseness=1] (0,2.8) to[out=180,in=90] (-0.8,2);

%\node[above,inner sep=3pt] at (2.525,3.2) {$\tau$};
\begin{scope}%[xshift=-0.2cm]
\draw[very thick,fill=white] (2,1.55) rectangle (3,2.45);
\node at (2.5,2) {\Large$\beta$};
\end{scope}
%\draw[linkred,ultra thick] (3.1,2) ++(0,0.55) arc (90:-90:0.15 and 0.55) node[red,below] {$\kappa$};
%\draw[red,densely dashed] (-1.2,1.8) ++ (185:0.2) arc (185:270:0.2); % mark crossing for skein relation

\begin{scope}[xshift=-3.5cm,yshift=1.75cm]
\draw[link] (0,0) to[out=45,in=135] ++(1,0) (0,1) to[out=-45,in=-135] ++(1,0);
\node[below] at (0.5,0) {$L_0^\beta$};
\end{scope}
\end{scope}

%%%

\begin{scope}[yshift=-6cm,xshift=6cm]
\node at (-1.6,2) {$\cong$};
\begin{scope}[yshift=-0.25cm]
\node[above,inner sep=2pt] at (2.5,2.8) {$\tau_{-1/7}$};
\draw[densely dotted] (0.6,2) ellipse (0.4 and 1.25);
%\draw[linkred,ultra thick] (3.1,2) ++(0,0.55) arc (90:270:0.15 and 0.55);
% parts of branch locus near complement
\draw[link,looseness=0.75] (2,2.2) -- ++(1,0) to[out=0,in=270,looseness=1.5] (3.4,2.4) to[out=90,in=0,looseness=1.5] ++(-0.4,0.4); % right side
\draw[link,looseness=0.75] (2,2) -- ++(1.2,0) to[out=0,in=270,looseness=1.5] (3.6,2.8) -- ++(0,0.25) to[out=90,in=0,looseness=1.5] ++(-0.4,0.6) -- ++(-2.2,0) to[out=180,in=90,looseness=1.25] (0.4,3.05);
\draw[link] (2,2) to[out=180,in=270] (1.4,2.6) -- ++(0,0.45) to[out=90,in=90] ++(-0.6,0); % left side
\draw[link,looseness=0.75] (2,2.2) to[out=180,in=270,looseness=1.5] (1.6,2.4) to[out=90,in=180,looseness=1.5] ++(0.4,0.4) -- ++(1,0);

% add left-handed twists
\foreach \y in {0,...,6} {
 \draw[link,looseness=0.75] (0.8,3.05-0.3*\y) to[out=270,in=90] ++(-0.4,-0.3);
 \draw[link,looseness=0.75] (0.4,3.05-0.3*\y) to[out=270,in=90] ++(0.4,-0.3);
}

\draw[link] (0.8,0.95) to[out=270,in=270] (1.4,0.95) -- ++(0,0.25) to[out=90,in=180] (2,1.8);
\draw[link] (0.4,0.95) to[out=270,in=180] (1.2,0.35) -- ++(2,0) to[out=0,in=270,looseness=1.5] (3.6,1.15) -- ++(0,0.25) to[out=90,in=0,looseness=1.5] ++(-0.4,0.4) -- ++(-1.2,0); % below the braid

\begin{scope}%[xshift=-0.2cm]
\draw[very thick,fill=white] (2,1.55) rectangle (3,2.45);
\node at (2.5,2) {\Large$\beta$};
\end{scope}
%\draw[linkred,ultra thick] (3.1,2) ++(0,0.55) arc (90:-90:0.15 and 0.55) node[red,below] {$\kappa$};
\end{scope}
\end{scope}

\end{tikzpicture}
\caption{A crossing change and $0$-resolution of $\tau\cup\beta$ at the indicated crossing.}
\label{fig:C24-T23-resolution}
\end{figure}

With Lemma~\ref{lem:C24-T23-as-branched-cover} at hand, we are left to determine which braids $\beta$ cause $\tau\cup\beta$ to be unknotted.  Supposing that it is indeed an unknot $U$, we choose a crossing in Figure~\ref{fig:C24-T23-resolution}, indicated by a red dashed arc, and produce two link diagrams $L^\beta$ and $L^\beta_0$ by changing that crossing and by taking its $0$-resolution, respectively.  We can see in Figure~\ref{fig:C24-T23-resolution} that
\begin{align*}
L^\beta &\cong T_{-2,3} \# (\tau_{-1/4} \cup \beta), &
L^\beta_0 & \cong (\tau_{-1/7} \cup \beta)
\end{align*}
where $\tau_{-1/4}$ and $\tau_{-1/7}$ are tangle diagrams differing only in the circled rational sub-tangles, having $-4$ and $-7$ half-twists respectively.

\begin{lemma} \label{lem:tau-twists-unknotted}
If $\tau\cup\beta$ is an unknot, then so are $\tau_{-1/4}\cup\beta$ and $\tau_{-1/7}\cup\beta$.
\end{lemma}

\begin{proof}
Just as in Section~\ref{sec:unknot-24}, the Montesinos trick tells us that there is a curve $\gamma \subset \dcover(U) \cong S^3$ and an integer $n\in\Z$ such that
\begin{align*}
\dcover(L^\beta) &\cong S^3_{(2n+1)/2}(\gamma), &
\dcover(L^\beta_0) &\cong S^3_n(\gamma).
\end{align*}
Since $\dcover(L^\beta)$ arises as non-integral surgery on a knot $\gamma\subset S^3$, it must be irreducible \cite{gordon-luecke-integral}.  But we also know that
\[ \dcover(L^\beta) \cong L(3,2) \# \dcover(\tau_{-1/4}\cup\beta), \]
and if this is irreducible then the second summand must be $S^3$, so then $\tau_{-1/4}\cup\beta$ must be unknotted \cite{waldhausen-involution}.

Now that we have $\dcover(L^\beta) \cong L(3,2) \cong S^3_{3/2}(U)$ arising from a non-integral surgery on $\gamma$, of slope $\frac{2n+1}{2}$, we know that $\gamma$ must be an unknot or a torus knot \cite{cgls}.  In fact it cannot be a nontrivial torus knot, since otherwise no surgery would produce a lens space of order 3 \cite{moser}.  So $\gamma$ is an unknot, and then we must have $\frac{2n+1}{2} = \frac{3}{2}$, or $n=1$.  But in this case we have
\[ \dcover(L^\beta_0) \cong S^3_n(\gamma) \cong S^3_1(U) \cong S^3, \]
so again by \cite{waldhausen-involution} we can conclude that $\tau_{-1/7} \cup \beta \cong L^\beta_0$ is an unknot.
\end{proof}

\begin{lemma} \label{lem:beta-closure}
If $\tau\cup\beta$ is an unknot, then the link $\tau_{1/0}\cup\beta$ depicted in Figure~\ref{fig:tau-0-tangle} is an unknot, and the 3-braid closure $\hat\beta$ is a 2-component unlink.
\end{lemma}

\begin{proof}
We take the tangles $\tau_{-1/4}$ and $\tau_{-1/7}$ in Figure~\ref{fig:C24-T23-resolution} and replace their circled twist regions with rational tangles of slopes $\frac{1}{0}$ or $\frac{0}{1}$ to get the tangles $\tau_{1/0}$ and $\tau_{0/1}$ depicted in Figure~\ref{fig:tau-0-tangle}, observing that
\[ \tau_{0/1} \cup \beta \cong\hat\beta. \]
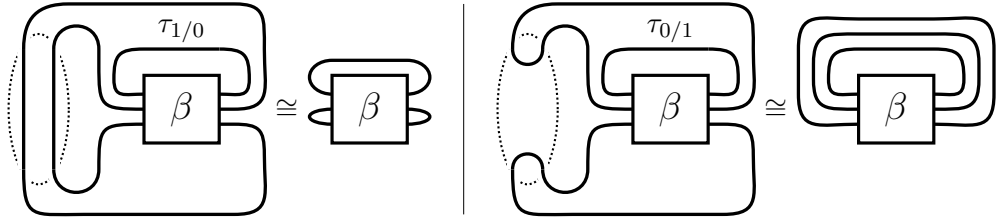
\begin{figure}
\begin{tikzpicture}
\begin{scope}
\begin{scope}
\node[above,inner sep=2pt] at (2.5,2.8) {$\tau_{1/0}$};
\draw[densely dotted] (0.6,2) ellipse (0.4 and 1);
\draw[link,looseness=0.75] (2,2.2) -- ++(1,0) to[out=0,in=270,looseness=1.5] (3.4,2.4) to[out=90,in=0,looseness=1.5] ++(-0.4,0.4); % right side
\draw[link,looseness=0.75] (2,2) -- ++(1.2,0) to[out=0,in=270,looseness=1.5] (3.6,2.8) to[out=90,in=0,looseness=1.5] ++(-0.4,0.6) -- ++(-2.2,0) to[out=180,in=90,looseness=1.25] (0.4,2.8);
\draw[link] (2,2) to[out=180,in=270] (1.4,2.6) -- ++(0,0.2) to[out=90,in=90] (0.8,2.8); % left side
\draw[link,looseness=0.75] (2,2.2) to[out=180,in=270,looseness=1.5] (1.6,2.4) to[out=90,in=180,looseness=1.5] ++(0.4,0.4) -- ++(1,0);

% add rational tangle
\draw[link] (0.4,2.8) -- ++(0,-1.6);
\draw[link] (0.8,2.8) -- ++(0,-1.6);

\draw[link] (0.8,1.2) to[out=270,in=270] (1.4,1.2) to[out=90,in=180] (2,1.8);
\draw[link] (0.4,1.2) to[out=270,in=180] (1.2,0.6) -- ++(2,0) to[out=0,in=270,looseness=1.5] (3.6,1.4) to[out=90,in=0,looseness=1.5] ++(-0.4,0.4) -- ++(-1.2,0); % below the braid

\begin{scope}%[xshift=-0.2cm]
\draw[very thick,fill=white] (2,1.55) rectangle (3,2.45);
\node at (2.5,2) {\Large$\beta$};
\end{scope}
\end{scope}

\node at (3.9,2) {$\cong$};

\begin{scope}[xshift=2.5cm]
\draw[link] (2,2.2) arc (270:90:0.3 and 0.225) -- ++(1,0)  arc (90:-90:0.3 and 0.225);
\draw[link] (2,2) arc (90:270:0.3 and 0.1) -- ++(1,0) arc (-90:90:0.3 and 0.1);
\begin{scope}
\draw[very thick,fill=white] (2,1.55) rectangle (3,2.45);
\node at (2.5,2) {\Large$\beta$};
\end{scope}
\end{scope}
\end{scope}

% left picture: max x = 5.8
% right picture: min x = 0.2
\draw[very thin] (6.25,3.4) -- ++(0,-2.8);

\begin{scope}[xshift=6.5cm]
\begin{scope}
\node[above,inner sep=2pt] at (2.5,2.8) {$\tau_{0/1}$};
\draw[densely dotted] (0.6,2) ellipse (0.4 and 1);
\draw[link,looseness=0.75] (2,2.2) -- ++(1,0) to[out=0,in=270,looseness=1.5] (3.4,2.4) to[out=90,in=0,looseness=1.5] ++(-0.4,0.4); % right side
\draw[link,looseness=0.75] (2,2) -- ++(1.2,0) to[out=0,in=270,looseness=1.5] (3.6,2.8) to[out=90,in=0,looseness=1.5] ++(-0.4,0.6) -- ++(-2.2,0) to[out=180,in=90,looseness=1.25] (0.4,2.8);
\draw[link] (2,2) to[out=180,in=270] (1.4,2.6) -- ++(0,0.2) to[out=90,in=90] (0.8,2.8); % left side
\draw[link,looseness=0.75] (2,2.2) to[out=180,in=270,looseness=1.5] (1.6,2.4) to[out=90,in=180,looseness=1.5] ++(0.4,0.4) -- ++(1,0);

% add rational tangle
\draw[link] (0.4,2.8) to[out=270,in=270] ++(0.4,0);
\draw[link] (0.4,1.2) to[out=90,in=90] ++(0.4,0);

\draw[link] (0.8,1.2) to[out=270,in=270] (1.4,1.2) to[out=90,in=180] (2,1.8);
\draw[link] (0.4,1.2) to[out=270,in=180] (1.2,0.6) -- ++(2,0) to[out=0,in=270,looseness=1.5] (3.6,1.4) to[out=90,in=0,looseness=1.5] ++(-0.4,0.4) -- ++(-1.2,0); % below the braid

\begin{scope}%[xshift=-0.2cm]
\draw[very thick,fill=white] (2,1.55) rectangle (3,2.45);
\node at (2.5,2) {\Large$\beta$};
\end{scope}
\end{scope}

\node at (3.9,2) {$\cong$};

\begin{scope}[xshift=3cm]
\draw[link,looseness=1.5] (2,2.2) -- ++(1,0) to[out=0,in=270] ++(0.4,0.2) to[out=90,in=0] ++(-0.4,0.4) -- ++(-1,0) to[out=180,in=90] ++(-0.4,-0.4) to[out=270,in=180] (2,2.2); % top strand
\draw[link,looseness=1.5] (2,2) -- ++(1,0) to[out=0,in=270] ++(0.6,0.4) -- ++(0,0.2) to[out=90,in=0] ++(-0.6,0.4) -- ++(-1,0) to[out=180,in=90] ++(-0.6,-0.4) -- ++(0,-0.2) to[out=270,in=180] (2,2); % middle strand
\draw[link,looseness=1.5] (2,1.8) -- ++(1,0) to[out=0,in=270] ++(0.8,0.4) -- ++(0,0.6) to[out=90,in=0] ++(-0.8,0.4) -- ++(-1,0) to[out=180,in=90] ++(-0.8,-0.4) -- ++(0,-0.6) to[out=270,in=180] (2,1.8); % bottom strand

\begin{scope}
\draw[very thick,fill=white] (2,1.55) rectangle (3,2.45);
\node at (2.5,2) {\Large$\beta$};
\end{scope}
\end{scope}
\end{scope}

\end{tikzpicture}
\caption{Two rational tangle replacements produce the links $\tau_{1/0}\cup\beta$ and $\tau_{0/1}\cup\beta \cong \hat\beta$.}
\label{fig:tau-0-tangle}
\end{figure}%
Lemma~\ref{lem:tau-twists-unknotted} says that $\tau_{-1/4} \cup \beta$ and $\tau_{-1/7} \cup \beta$ are both unknotted, so their branched double covers satisfy
\[ \dcover(\tau_{-1/4} \cup \beta) \cong \dcover(\tau_{-1/7} \cup \beta) \cong S^3. \]
In particular, if we remove the circled rational subtangles from either unknot, then the branched double cover of what remains is a knot complement $S^3 \setminus N(L)$, and it has two different Dehn fillings (corresponding to the rational tangles of slopes $-\frac{1}{4}$ and $-\frac{1}{7}$) which both produce $S^3$.  Then $L$ must be an unknot \cite[Theorem~2]{gordon-luecke-complement}, and the fillings that produce $\dcover(\tau_{-1/4}\cup\beta)$ and $\dcover(\tau_{-1/7}\cup\beta)$ must have slopes $\frac{1}{n}$ and $\frac{1}{n-3}$ for some $n\in\Z$.

It follows that if we replace these rational tangles with one of slope $\frac{1}{0}$, then this corresponds to a Dehn filling of $S^3 \setminus N(L)$ of slope $\frac{1}{n+4}$, and then
\[ \dcover(\tau_{1/0} \cup \beta) \cong S^3_{1/(n+4)}(L) \cong S^3 \]
since $L$ is unknotted.  We apply Waldhausen's result \cite{waldhausen-involution} once again to see that $\tau_{1/0} \cup \beta$ is an unknot.

Similarly, if we instead use the rational tangle that produces $\tau_{0/1} \cup \beta$, then the corresponding Dehn filling of $S^3 \setminus N(L)$ is at distance one from both the $\frac{1}{n}$- and $\frac{1}{n-3}$-fillings, so it must have slope $\frac{0}{1}$.  In other words, we have shown that
\[ \dcover(\hat\beta) \cong \dcover(\tau_{0/1} \cup \beta) \cong S^3_0(L) \cong S^1\times S^2. \]
But the only link in $S^3$ with branched double cover $S^1\times S^2$ is the two-component unlink \cite{tollefson}, so this determines $\hat\beta$ up to isotopy.
\end{proof}

We can now apply methods from Section~\ref{sec:unknot-24} to determine all of the possible braids $\beta$ to which Lemma~\ref{lem:C24-T23-as-branched-cover} might apply.

\begin{proposition} \label{prop:T23-braids}
If $\tau \cup \beta$ is unknotted, where $\tau$ is the tangle shown in Figure~\ref{fig:tau-T23-isotopy}, then
\[ \beta = y^a x^{\pm1} y^{-a} \]
for some $a\in\Z$.
\end{proposition}

\begin{proof}
Lemma~\ref{lem:beta-closure} tells us that the knot $\tau_{1/0} \cup \beta$ on the left side of Figure~\ref{fig:tau-0-tangle} is an unknot, with branched double cover $S^3$.  Using the representation $\rho: B_3 \to SL_2(\Z)$ from \eqref{eq:B3-action}, which was defined by
\begin{align*}
\rho(x) &= \begin{pmatrix} 1 & 1 \\ 0 & 1 \end{pmatrix}, &
\rho(y) &= \begin{pmatrix} 1 & 0 \\ -1 & 1 \end{pmatrix}, 
\end{align*}
we apply Lemma~\ref{lem:rho-beta-from-l-beta} with $(p,q,\bar{q},r)=(1,1,1,0)$ to see that
\[ \rho(\beta) = (-1)^e \begin{pmatrix} 1 & 0 \\ k & 1 \end{pmatrix} \begin{pmatrix} 1 & 1 \\ 0 & 1 \end{pmatrix} \begin{pmatrix} 1 & 0 \\ \ell & 1 \end{pmatrix} = \rho(\Delta^{4d+2e} y^{-k} x y^{-\ell}) \]
for some integers $e\in\{0,1\}$ and $d,k,\ell$.  (Note that the knot labeled $L^\beta$ in Lemma~\ref{lem:rho-beta-from-l-beta}, as depicted in Figure~\ref{fig:E-T24-resolutions}, is our $\tau_{1/0}\cup\beta$, and that the two cases \eqref{eq:rho-beta-factor-1} and \eqref{eq:rho-beta-factor-2} of Lemma~\ref{lem:rho-beta-from-l-beta} coincide since $q=\bar{q}$.)  In fact, we recall from Lemma~\ref{lem:kernel-rho} that $\ker(\rho)$ is generated by $\Delta^4$, so we must have
\[ \beta = \Delta^{4d+2e} y^{-k} x y^{-\ell}. \]

Now we use the other conclusion of Lemma~\ref{lem:beta-closure}, namely that the 3-braid closure $\hat\beta$ is a 2-component unlink.  Viewing this as the $(2,0)$-torus link, Birman and Menasco \cite{birman-menasco-iii} proved that $\beta$ must be conjugate to either $y$ or $y^{-1}$, so that its exponent sum is $\pm1$ and 
\[ \tr \rho(\beta) = \tr \rho(y^{\pm1}) = 2. \]
But we can also compute that
\begin{align*}
\tr \rho(\beta) = \tr\rho(y^k\beta y^{-k}) &= (-1)^e \tr\left( \begin{pmatrix} 1 & 1 \\ 0 & 1 \end{pmatrix} \begin{pmatrix} 1 & 0 \\ k+\ell & 1 \end{pmatrix} \right) \\
&= (-1)^e (k+\ell+2).
\end{align*}
Thus $(k+\ell,e)$ is either $(0,0)$ or $(-4,1)$.

Suppose first that $(k+\ell,e) = (0,0)$.  Then $\beta = \Delta^{4d} y^{\ell} x y^{-\ell}$ for some integer $d$.  In this case its exponent sum is $12d+1$, and since this is equal to $\pm1$ we must have $d=0$.

In the remaining case we have $(k+\ell,e)=(-4,1)$, so $\beta = \Delta^{4d+2} y^{\ell} \cdot y^4x \cdot y^{-\ell}$ for some $d$.  The exponent sum is $12d+11 = \pm1$, so then $d=-1$ and we have
\[ y^{-\ell}\beta y^\ell = \Delta^{-2} y^4 x. \]
We now use the braid relation $xyx=yxy$ to see that
\[ y^2(\Delta^2 x^{-1})y^{-2} = y^2(y\cdot xyx\cdot y)y^{-2} = y^2(y^2 x y^2)y^{-2} = y^4 x, \]
and since $\Delta^2$ is central it follows that
\[ y^{-\ell}\beta y^\ell = \Delta^{-2} y^4 x = y^2 x^{-1} y^{-2} \]
or
\[ \beta = y^{\ell+2} x^{-1} y^{-(\ell+2)}. \]
This completes the proof.
\end{proof}

We now determine the knots $K_\beta$ that arise in Lemma~\ref{lem:C24-T23-as-branched-cover}.

\begin{lemma} \label{lem:T23-only-2}
Suppose that $K$ satisfies the hypotheses of Lemma~\ref{lem:C24-T23-as-branched-cover}, and write $K = K_\beta$ where $K$ arises as the lift of the curve $\kappa$ in the branched double cover of the unknot $U = \tau \cup \beta$.  Then $K$ is isotopic to either $K_x$ or $K_{x^{-1}}$.
\end{lemma}

\begin{proof}
By Proposition~\ref{prop:T23-braids} we know that $\beta = y^a x^\epsilon y^{-a}$, where $a\in\Z$ and $\epsilon=\pm1$.  These are illustrated in Figure~\ref{fig:T23-Kbeta-1}, where we have started with a slight isotopy of the unknot $U = \tau\cup\beta$ from Figure~\ref{fig:C24-T23-resolution}.
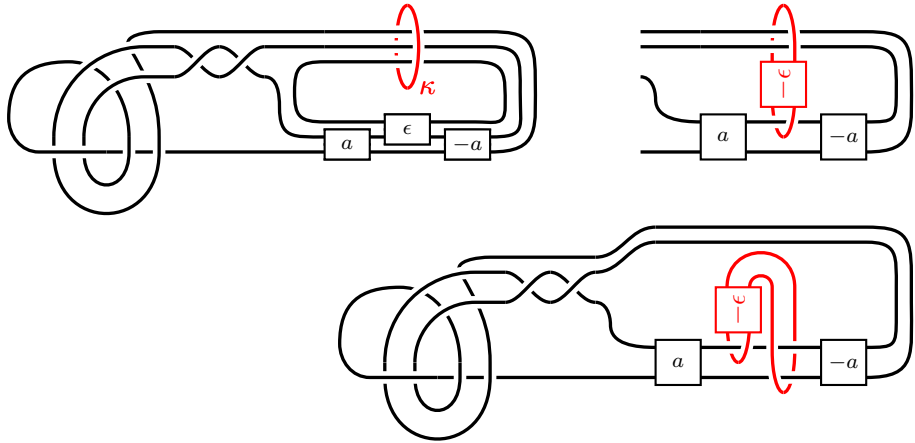
\begin{figure}
\begin{tikzpicture}

\begin{scope}
\draw[linkred,ultra thick] (3.1,3.2) ++(0,0.55) arc (90:270:0.15 and 0.55);
% braid region
\draw[link,looseness=1.5] (2,1.8) -- ++(2.2,0) to[out=0,in=270] (4.8,2.8) to[out=90,in=0] ++(-0.6,0.6) -- ++(-2.2,0); % bottom strand
\draw[link,looseness=1.5] (2,2.2) -- ++(2,0) to[out=0,in=270] (4.4,2.4) to[out=90,in=0] ++(-0.4,0.6) -- ++(-2,0) to[out=180,in=90] ++(-0.4,-0.4) to[out=270,in=180] ++(0.4,-0.4); % top strand
\draw[link,looseness=0.75] (2,2) -- ++(2.2,0) to[out=0,in=270,looseness=1.75] (4.6,2.6) to[out=90,in=0,looseness=1.5] ++(-0.4,0.6) -- ++(-2.2,0); % middle strand

% add a left-handed twist above the braid
\draw[link,looseness=0.75] (1.2,2.8) to[out=180,in=0] ++(-0.6,0.4) ++(0,-0.4) to[out=180,in=0] ++(-0.6,0.4);
\draw[link,looseness=0.75] (1.2,3.2) to[out=180,in=0] ++(-0.6,-0.4) ++(0,0.4) to[out=180,in=0] ++(-0.6,-0.4);
\draw[link] (2,3.2) -- (1.2,3.2); % coming into top right of twist
\draw[link] (2,2) to[out=180,in=270,looseness=1.5] (1.4,2.4) to[out=90,in=0,looseness=1] (1.2,2.8); % coming into bottom right of twist

\draw[link] (2,1.8) -- (-0.9,1.8); % extend strand coming from bottom of braid region
\draw[link,looseness=2] (-1.6,2) -- ++(0,-0.2) to[out=270,in=270] ++(1.4,0) -- ++(0,0.2) to[out=90,in=300,looseness=1] ++(-0.4,1.2) to[out=120,in=180,looseness=1] ++(0.4,0.2) -- ++(2.2,0); % end up at top left of braid region
\draw[link] (-1.2,2) -- ++(0,-0.2) to[out=270,in=270,looseness=2.5] ++(0.6,-0) -- ++(0,0.2);
\draw[link] (-0.6,2) to[out=90,in=0,looseness=1] ++(-0.8,1) to[out=180,in=90,looseness=1.25] ++(-0.8,-0.75) to[out=270,in=180,looseness=1] ++(0.4,-0.45) -- ++(0.9,0);

\draw[link,looseness=1] (0,3.2) -- ++(-0.4,0) to[out=180,in=90] (-1.6,2); % come out left end of twist
\draw[link,looseness=1] (0,2.8) -- ++(-0.4,0) to[out=180,in=90] (-1.2,2);

% twist boxes
\draw[fill=white] (2,1.7) rectangle (2.6,2.1);
\node at (2.3,1.9) {\scriptsize$a$};
\draw[fill=white] (2.8,1.9) rectangle (3.4,2.3);
\node at (3.1,2.1) {\scriptsize$\epsilon$};
\draw[fill=white] (3.6,1.7) rectangle (4.2,2.1);
\node at (3.9,1.9) {\scriptsize$-a$};

\draw[linkred,ultra thick] (3.1,3.2) ++(0,0.55) arc (90:-90:0.15 and 0.55) node[red,right] {$\kappa$};
\end{scope}

\begin{scope}
\draw[linkred,ultra thick] (3.1,3.2) ++(0,0.55) arc (90:270:0.15 and 0.55);
% braid region
\draw[link,looseness=1.5] (2,1.8) -- ++(2.2,0) to[out=0,in=270] (4.8,2.8) to[out=90,in=0] ++(-0.6,0.6) -- ++(-2.2,0); % bottom strand
\draw[link,looseness=1.5] (2,2.2) -- ++(2,0) to[out=0,in=270] (4.4,2.4) to[out=90,in=0] ++(-0.4,0.6) -- ++(-2,0) to[out=180,in=90] ++(-0.4,-0.4) to[out=270,in=180] ++(0.4,-0.4); % top strand
\draw[link,looseness=0.75] (2,2) -- ++(2.2,0) to[out=0,in=270,looseness=1.75] (4.6,2.6) to[out=90,in=0,looseness=1.5] ++(-0.4,0.6) -- ++(-2.2,0); % middle strand

% add a left-handed twist above the braid
\draw[link,looseness=0.75] (1.2,2.8) to[out=180,in=0] ++(-0.6,0.4) ++(0,-0.4) to[out=180,in=0] ++(-0.6,0.4);
\draw[link,looseness=0.75] (1.2,3.2) to[out=180,in=0] ++(-0.6,-0.4) ++(0,0.4) to[out=180,in=0] ++(-0.6,-0.4);
\draw[link] (2,3.2) -- (1.2,3.2); % coming into top right of twist
\draw[link] (2,2) to[out=180,in=270,looseness=1.5] (1.4,2.4) to[out=90,in=0,looseness=1] (1.2,2.8); % coming into bottom right of twist

\draw[link] (2,1.8) -- (-0.9,1.8); % extend strand coming from bottom of braid region
\draw[link,looseness=2] (-1.6,2) -- ++(0,-0.2) to[out=270,in=270] ++(1.4,0) -- ++(0,0.2) to[out=90,in=300,looseness=1] ++(-0.4,1.2) to[out=120,in=180,looseness=1] ++(0.4,0.2) -- ++(2.2,0); % end up at top left of braid region
\draw[link] (-1.2,2) -- ++(0,-0.2) to[out=270,in=270,looseness=2.5] ++(0.6,-0) -- ++(0,0.2);
\draw[link] (-0.6,2) to[out=90,in=0,looseness=1] ++(-0.8,1) to[out=180,in=90,looseness=1.25] ++(-0.8,-0.75) to[out=270,in=180,looseness=1] ++(0.4,-0.45) -- ++(0.9,0);

\draw[link,looseness=1] (0,3.2) -- ++(-0.4,0) to[out=180,in=90] (-1.6,2); % come out left end of twist
\draw[link,looseness=1] (0,2.8) -- ++(-0.4,0) to[out=180,in=90] (-1.2,2);

% twist boxes
\draw[fill=white] (2,1.7) rectangle (2.6,2.1);
\node at (2.3,1.9) {\scriptsize$a$};
\draw[fill=white] (2.8,1.9) rectangle (3.4,2.3);
\node at (3.1,2.1) {\scriptsize$\epsilon$};
\draw[fill=white] (3.6,1.7) rectangle (4.2,2.1);
\node at (3.9,1.9) {\scriptsize$-a$};

\draw[linkred,ultra thick] (3.1,3.2) ++(0,0.55) arc (90:-90:0.15 and 0.55) node[red,right] {$\kappa$};
\end{scope}

%%%
\begin{scope}[xshift=5cm]
\draw[linkred,ultra thick] (3.1,3.2) ++(0,0.55) arc (90:180:0.15 and 0.55) -- ++(0,-0.65) arc (180:270:0.15 and 0.55);
% braid region
\draw[link,looseness=1.5] (2,1.8) -- ++(2.2,0) to[out=0,in=270] (4.8,2.8) to[out=90,in=0] ++(-0.6,0.6) -- ++(-2.2,0); % bottom strand
%\draw[link,looseness=1.5] (2,2.2) -- ++(2,0) to[out=0,in=270] (4.4,2.4) to[out=90,in=0] ++(-0.4,0.6) -- ++(-2,0) to[out=180,in=90] ++(-0.4,-0.4) to[out=270,in=180] ++(0.4,-0.4); % top strand
\draw[link,looseness=0.75] (2,2.2) -- ++(2.2,0) to[out=0,in=270,looseness=1.75] (4.6,2.6) to[out=90,in=0,looseness=1.5] ++(-0.4,0.6) -- ++(-2.2,0); % middle strand

\foreach \y in {1.8,3.2,3.4} { \draw[link] (2,\y) -- ++(-0.8,0); }
\draw[link] (2,2.2) to[out=180,in=270,looseness=1] (1.4,2.5) to[out=90,in=0,looseness=1] (1.2,2.8); % coming into bottom right of twist

% twist boxes
\draw[fill=white] (2,1.7) rectangle (2.6,2.3);
\node at (2.3,2) {\scriptsize$a$};
%\draw[fill=white] (2.8,1.9) rectangle (3.4,2.3);
%\node at (3.1,2.1) {\scriptsize$\epsilon$};
\draw[fill=white] (3.6,1.7) rectangle (4.2,2.3);
\node at (3.9,2) {\scriptsize$-a$};

\draw[linkred,ultra thick] (3.1,3.2) ++(0,0.55) arc (90:0:0.15 and 0.55) -- ++(0,-0.65) arc (0:-90:0.15 and 0.55);
\draw[red,fill=white] (2.8,2.4) rectangle (3.4,3);
\node[red,rotate=90] at (3.1,2.7) {\small$-\epsilon$};
\end{scope}

\begin{scope}[xshift=4.4cm,yshift=-3cm]
\draw[linkred,ultra thick] (2.95,3) -- ++(0,-0.45) arc (180:270:0.15 and 0.55); % back of kappa
\draw[linkred,ultra thick] (2.95,3) to[out=90,in=90] ++(0.9,0) -- ++ (0,-0.85) arc (0:-90:0.15 and 0.55);
% braid region
\draw[link,looseness=1.5] (2,1.8) -- ++(2.8,0) to[out=0,in=270] ++(0.6,1) -- ++(0,0.4) to[out=90,in=0] ++(-0.6,0.6) -- ++(-2.8,0); % bottom strand
\draw[link,looseness=0.75] (2,2.2) -- ++(2.8,0) to[out=0,in=270,looseness=1.75] ++(0.4,0.4) -- ++(0,0.4) to[out=90,in=0,looseness=1.5] ++(-0.4,0.6) -- ++(-2.8,0); % middle strand

\foreach \y in {1.8} { \draw[link] (2,\y) -- ++(-0.8,0); }
\foreach \y in {3.6,3.8} {\draw[link,looseness=1] (2,\y) to[out=180,in=0] ++(-0.8,-0.4); }
\draw[link] (2,2.2) to[out=180,in=270,looseness=1] (1.4,2.5) to[out=90,in=0,looseness=1] (1.2,2.8); % coming into bottom right of twist

\draw[linkred,ultra thick] (3.25,3) -- ++(0,-0.45) arc (0:-90:0.15 and 0.55); % front of kappa
\draw[linkred,ultra thick] (3.25,3) to[out=90,in=90] ++(0.3,0) -- ++(0,-0.85) arc (180:270:0.15 and 0.55);
\draw[red,fill=white] (2.8,2.4) rectangle (3.4,3);
\node[red,rotate=90] at (3.1,2.7) {\small$-\epsilon$};

% add rest of diagram
% add a left-handed twist above the braid
\draw[link,looseness=0.75] (1.2,2.8) to[out=180,in=0] ++(-0.6,0.4) ++(0,-0.4) to[out=180,in=0] ++(-0.6,0.4);
\draw[link,looseness=0.75] (1.2,3.2) to[out=180,in=0] ++(-0.6,-0.4) ++(0,0.4) to[out=180,in=0] ++(-0.6,-0.4);

\draw[link] (2,1.8) -- (-0.9,1.8); % extend strand coming from bottom of braid region
\draw[link,looseness=2] (-1.6,2) -- ++(0,-0.2) to[out=270,in=270] ++(1.4,0) -- ++(0,0.2) to[out=90,in=300,looseness=1] ++(-0.4,1.2) to[out=120,in=180,looseness=1] ++(0.4,0.2) -- ++(1.4,0); % end up at top left of braid region
\draw[link] (-1.2,2) -- ++(0,-0.2) to[out=270,in=270,looseness=2.5] ++(0.6,-0) -- ++(0,0.2);
\draw[link] (-0.6,2) to[out=90,in=0,looseness=1] ++(-0.8,1) to[out=180,in=90,looseness=1.25] ++(-0.8,-0.75) to[out=270,in=180,looseness=1] ++(0.4,-0.45) -- ++(0.9,0);

\draw[link,looseness=1] (0,3.2) -- ++(-0.4,0) to[out=180,in=90] (-1.6,2); % come out left end of twist
\draw[link,looseness=1] (0,2.8) -- ++(-0.4,0) to[out=180,in=90] (-1.2,2);

% twist boxes
\draw[fill=white] (2,1.7) rectangle (2.6,2.3);
\node at (2.3,2) {\scriptsize$a$};
\draw[fill=white] (4.2,1.7) rectangle (4.8,2.3);
\node at (4.5,2) {\scriptsize$-a$};
\end{scope}

\end{tikzpicture}
\caption{An isotopy of the unknot $U = \tau \cup \beta$, where $\beta = y^a x^\epsilon y^{-a}$.}
\label{fig:T23-Kbeta-1}
\end{figure}
The bottom of Figure~\ref{fig:T23-Kbeta-1} makes it clear that up to isotopy the knot $K_\beta$ only depends on $\epsilon$ and the parity of $a$, because the tangle relation
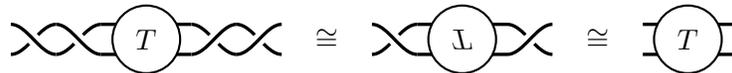
\[ \begin{tikzpicture}[link/.append style={looseness=0.75}]
\foreach \x in {0,0.6} {
  \draw[link] (\x,-0.2) to[out=0,in=180] ++(0.6,0.4);
  \draw[link] (\x,0.2) to[out=0,in=180] ++(0.6,-0.4);
}
\foreach \x in {2.4,3} {
  \draw[link] (\x,0.2) to[out=0,in=180] ++(0.6,-0.4);
  \draw[link] (\x,-0.2) to[out=0,in=180] ++(0.6,0.4);
}
\draw[link] (1.2,-0.2) -- ++(1.2,0) (1.2,0.2) -- ++(1.2,0);
\draw[fill=white] (1.8,0) circle (0.45);
\node at (1.8,0) {$T$};

\foreach \x in {4.8} {
  \draw[link] (\x,-0.2) to[out=0,in=180] ++(0.6,0.4);
  \draw[link] (\x,0.2) to[out=0,in=180] ++(0.6,-0.4);
}
\foreach \x in {6.6} {
  \draw[link] (\x,0.2) to[out=0,in=180] ++(0.6,-0.4);
  \draw[link] (\x,-0.2) to[out=0,in=180] ++(0.6,0.4);
}
\draw[link] (5.4,-0.2) -- ++(1.2,0) (5.4,0.2) -- ++(1.2,0);
\draw[fill=white] (6,0) circle (0.45);
\node[yscale=-1] at (6,0) {$T$};

\draw[link] (8.4,-0.2) -- ++(1.2,0) (8.4,0.2) -- ++(1.2,0);
\draw[fill=white] (9,0) circle (0.45);
\node at (9,0) {$T$};

\node at (4.2,0) {$\cong$};
\node at (7.8,0) {$\cong$};
%\node at (9.75,0) {$\vphantom{\cong}.$};
\end{tikzpicture} \]
lets us identify the links $U\cup \kappa$ for $\beta=y^{a+2}x^\epsilon y^{-(a+2)}$ and for $\beta=y^ax^\epsilon y^{-a}$ up to isotopy.  Thus we need only consider the cases $a=0$ and $a=1$.

Starting from the bottom of Figure~\ref{fig:T23-Kbeta-1}, we simplify part of the corresponding diagrams by an isotopy in Figures~\ref{fig:T23-tangle-simplify-0} and \ref{fig:T23-tangle-simplify-1}, corresponding to $a=0$ and $a=1$ respectively.
\begin{figure}
\begin{tikzpicture}

\begin{scope}
\draw[linkred,ultra thick] (2.35,2) arc (270:180:0.15 and 0.4) to[out=90,in=180,looseness=1] ++(0.5,0.5); % end at (3.3,2.9)
\draw[linkred,ultra thick] (2.7,2.9) -- ++(1,0) arc (90:0:0.6 and 0.15);

% right end of braid region
\draw[link,looseness=1.5] (3.3,1.8) -- ++(0.1,0) to[out=0,in=270] ++(0.6,1) to[out=90,in=0] ++(-0.6,0.6) -- ++(-2.2,0); % bottom strand
\draw[link,looseness=0.75] (3.3,2.2) to[out=0,in=270,looseness=1.75] ++(0.4,0.4) to[out=90,in=0,looseness=1.5] ++(-0.4,0.6) -- ++(-2.1,0); % middle strand

\draw[link] (3.3,1.8) -- ++(-2.1,0);
\draw[link] (3.3,2.2) -- (2,2.2) to[out=180,in=270,looseness=1] (1.4,2.5) to[out=90,in=0,looseness=1] (1.2,2.8); % coming into bottom right of twist

% front of kappa and twist box
\draw[linkred, ultra thick] (2.35,2) arc (270:360:0.15 and 0.4) to[out=90,in=180,looseness=1] ++(0.2,0.2) -- ++(1,0) arc (270:360:0.6 and 0.15);
\draw[red,fill=white] (2.8,2.5) rectangle ++(0.6,0.5);
\node[red] at (3.1,2.75) {\small$-\epsilon$};

% add rest of diagram
% add a left-handed twist above the braid
\draw[link,looseness=0.75] (1.2,2.8) to[out=180,in=0] ++(-0.6,0.4) ++(0,-0.4) to[out=180,in=0] ++(-0.6,0.4);
\draw[link,looseness=0.75] (1.2,3.2) to[out=180,in=0] ++(-0.6,-0.4) ++(0,0.4) to[out=180,in=0] ++(-0.6,-0.4);

\draw[link] (1.2,1.8) -- ++(-1.2,0); % extend strands coming from top and bottom of braid region
\draw[link] (1.2,3.4) -- ++(-1.2,0);
\end{scope}

\begin{scope}[xshift=5cm]
\draw[linkred,ultra thick] (0.95,2) arc (270:180:0.15 and 0.4) to[out=90,in=180,looseness=1] ++(0.5,0.5); % end at (3.3,2.9)
\draw[linkred,ultra thick,looseness=0.75] (1.3,2.9) -- ++(1.2,0) coordinate (rtwist) to[out=0,in=180] ++(0.5,-0.3) ++(0,0.3) to[out=0,in=180] ++(0.5,-0.3) ++(0,0.3) arc (90:0:0.6 and 0.15);

% right end of braid region
\draw[link,looseness=1.5] (0,1.8) -- ++(3.1,0) -- ++(0.1,0) to[out=0,in=270] ++(0.6,1) to[out=90,in=0] ++(-0.6,0.6) -- ++(-3.2,0); % bottom strand
\draw[link,looseness=0.75] (1.9,2.2) to[out=0,in=270,looseness=1.75] ++(0.4,0.4) to[out=90,in=0,looseness=1.5] ++(-0.4,0.6) -- ++(-1.9,0); % middle strand
\draw[link] (1.9,2.2) -- ++(-1.1,0) to[out=180,in=270,looseness=1] ++(-0.6,0.3) to[out=90,in=0,looseness=1] ++(-0.2,0.3); % bottom half of middle strand

% front of kappa and twist box
\draw[linkred, ultra thick,looseness=0.75] (0.95,2) arc (270:360:0.15 and 0.4) to[out=90,in=180,looseness=1] ++(0.2,0.2) -- ++(1.2,0) coordinate (rtwist) to[out=0,in=180] ++(0.5,0.3) ++(0,-0.3) to[out=0,in=180] ++(0.5,0.3) ++(0,-0.3) arc (270:360:0.6 and 0.15);
\draw[red,fill=white] (1.4,2.5) rectangle ++(0.6,0.5);
\node[red] at (1.7,2.75) {\small$-\epsilon$};
\end{scope}

%%%

\begin{scope}[xshift=2cm]
\draw[->] (-2,-1.15) coordinate (arrowbase) ++(0.2,0.4) -- node[above,sloped] {\small$\epsilon=+1$} ++(1.25,0.8);
\draw[->] (arrowbase) ++(0.2,-0.4) -- node[below,sloped] {\small$\epsilon=-1$} ++(1.25,-0.8);

\begin{scope}[yshift=-2.5cm]
\draw[linkred,ultra thick] (0.95,2) arc (270:180:0.15 and 0.4) to[out=90,in=180,looseness=1] ++(0.5,0.5); % end at (3.3,2.9)
\draw[linkred,ultra thick,looseness=0.75] (1.3,2.9) -- ++(0.2,0) to[out=0,in=180] ++(0.5,-0.3) ++(0,0.3) -- ++(0.5,0) coordinate (rtwist) to[out=0,in=180] ++(0.5,-0.3) ++(0,0.3) to[out=0,in=180] ++(0.5,-0.3) ++(0,0.3) arc (90:0:0.6 and 0.15);

% right end of braid region
\draw[link,looseness=1.5] (0,1.8) -- ++(3.1,0) -- ++(0.1,0) to[out=0,in=270] ++(0.6,1) to[out=90,in=0] ++(-0.6,0.6) -- ++(-3.2,0); % bottom strand
\draw[link,looseness=0.75] (1.9,2.2) to[out=0,in=270,looseness=1.75] ++(0.4,0.4) to[out=90,in=0,looseness=1.5] ++(-0.4,0.6) -- ++(-1.9,0); % middle strand
\draw[link] (1.9,2.2) -- ++(-1.1,0) to[out=180,in=270,looseness=1] ++(-0.6,0.3) to[out=90,in=0,looseness=1] ++(-0.2,0.3); % bottom half of middle strand

% front of kappa and twist box
\draw[linkred, ultra thick,looseness=0.75] (0.95,2) arc (270:360:0.15 and 0.4) to[out=90,in=180,looseness=1] ++(0.2,0.2) -- ++(0.2,0) to[out=0,in=180] ++(0.5,0.3) ++(0,-0.3) -- ++(0.5,0) coordinate (rtwist) to[out=0,in=180] ++(0.5,0.3) ++(0,-0.3) to[out=0,in=180] ++(0.5,0.3) ++(0,-0.3) arc (270:360:0.6 and 0.15);
%\draw[red,fill=white] (1.4,2.5) rectangle ++(0.6,0.5);
%\node[red] at (1.7,2.75) {\small$-\epsilon$};
\end{scope}

\begin{scope}[yshift=-2.5cm,xshift=5cm]
% right end of braid region
\draw[linkred] (1.2,3.05) ++(-0.5,-0.15) coordinate (redl1) arc (180:270:0.3 and 0.15) arc (90:0:0.3 and 0.15) ++ (-0.6,-0.3) arc (180:270:0.5 and 0.15) coordinate (redr1);
\draw[linkred] (1.2,3.05) to[out=0,in=180,looseness=0.75] ++(0.9,-0.65) ++(0,0.4) arc (90:0:0.6 and 0.2);
\draw[link,looseness=1] (0,2.8) to[out=0,in=90] ++(0.2,-0.4) to[out=270,in=180] ++(0.4,-0.4) coordinate (loopbr) to[out=0,in=270] ++(0.4,0.6) to[out=90,in=0] ++(-0.4,0.6) -- ++(-0.6,0); % middle portion of U, from height 2 to 3.2. Maximum x = 1
\draw[linkred,looseness=0.75] (redl1) arc (180:90:0.5 and 0.15);
\draw[linkred,looseness=0.75] (redl1) ++ (0.6,-0.3) arc (0:-90:0.3 and 0.15) arc (90:180:0.3 and 0.15);

\draw[link,looseness=1.5] (0,1.8) -- ++(1.7,0) -- ++(0.1,0) to[out=0,in=270] ++(0.6,1) to[out=90,in=0] ++(-0.6,0.6) -- ++(-1.8,0); % outer strand
\draw[linkred,looseness=0.75] (redr1) to[out=0,in=180] ++(0.9,0.65) ++(0,-0.4) arc (-90:0:0.6 and 0.2);
\node[below] at (1.2,1.8) {\scriptsize$\beta = x$};
\end{scope}

\begin{scope}[yshift=-5cm]
\draw[linkred,ultra thick] (0.95,2) arc (270:180:0.15 and 0.4) to[out=90,in=180,looseness=1] ++(0.5,0.5); % end at (3.3,2.9)
\draw[linkred,ultra thick,looseness=0.75] (1.3,2.9) -- ++(0.2,0) ++(0,-0.3) to[out=0,in=180] ++(0.5,0.3) -- ++(0.5,0) coordinate (rtwist) to[out=0,in=180] ++(0.5,-0.3) ++(0,0.3) to[out=0,in=180] ++(0.5,-0.3) ++(0,0.3) arc (90:0:0.6 and 0.15);

% right end of braid region
\draw[link,looseness=1.5] (0,1.8) -- ++(3.1,0) -- ++(0.1,0) to[out=0,in=270] ++(0.6,1) to[out=90,in=0] ++(-0.6,0.6) -- ++(-3.2,0); % bottom strand
\draw[link,looseness=0.75] (1.9,2.2) to[out=0,in=270,looseness=1.75] ++(0.4,0.4) to[out=90,in=0,looseness=1.5] ++(-0.4,0.6) -- ++(-1.9,0); % middle strand
\draw[link] (1.9,2.2) -- ++(-1.1,0) to[out=180,in=270,looseness=1] ++(-0.6,0.3) to[out=90,in=0,looseness=1] ++(-0.2,0.3); % bottom half of middle strand

% front of kappa and twist box
\draw[linkred, ultra thick,looseness=0.75] (0.95,2) arc (270:360:0.15 and 0.4) to[out=90,in=180,looseness=1] ++(0.2,0.2) -- ++(0.2,0) ++(0,0.3) to[out=0,in=180] ++(0.5,-0.3) -- ++(0.5,0) coordinate (rtwist) to[out=0,in=180] ++(0.5,0.3) ++(0,-0.3) to[out=0,in=180] ++(0.5,0.3) ++(0,-0.3) arc (270:360:0.6 and 0.15);
\end{scope}

\begin{scope}[yshift=-5cm,xshift=5cm]
% right end of braid region
\draw[linkred] (1.2,3.05) arc (90:180:0.5 and 0.15) coordinate (redl1) ++(0.6,-0.3) arc (0:-90:0.3 and 0.15) arc (90:180:0.3 and 0.15) arc (180:270:0.5 and 0.15);
\draw[linkred] (1.2,3.05) to[out=0,in=180,looseness=0.75] ++(0.9,-0.65);
\draw[link,looseness=1] (0,2.8) to[out=0,in=90] ++(0.2,-0.4) to[out=270,in=180] ++(0.4,-0.4) coordinate (loopbr) to[out=0,in=270] ++(0.4,0.6) to[out=90,in=0] ++(-0.4,0.6) -- ++(-0.6,0); % middle portion of U, from height 2 to 3.2. Maximum x = 1
\draw[linkred,looseness=0.75] (redl1) arc (180:270:0.3 and 0.15) arc (90:0:0.3 and 0.15) ++(-0.6,-0.3) arc (180:270:0.5 and 0.15) to[out=0,in=180] ++(0.9,0.65) coordinate (redr1) to[out=0,in=180] ++(0.6,-0.4) ++(0,0.4) arc (90:0:0.6 and 0.2);

\draw[link,looseness=1.5] (0,1.8) -- ++(2.3,0) -- ++(0.1,0) to[out=0,in=270] ++(0.6,1) to[out=90,in=0] ++(-0.6,0.6) -- ++(-2.4,0); % outer strand
\draw[linkred,looseness=0.75] (redr1) ++(0,-0.4) to[out=0,in=180] ++(0.6,0.4) ++(0,-0.4) arc (-90:0:0.6 and 0.2);
\node[below] at (1.5,1.8) {\scriptsize$\beta = x^{-1}$};
\end{scope}

\draw[thin,densely dashed] (4.75,1.1) rectangle (8.45,-3.75);
\end{scope}
\end{tikzpicture}
\caption{Simplifying the case $a=0$, where $\beta=x^{\pm1}$, by an isotopy.}
\label{fig:T23-tangle-simplify-0}
\end{figure}
%%%%%
\begin{figure}
\begin{tikzpicture}

\begin{scope}
\draw[linkred,ultra thick] (2.95,2) arc (270:180:0.15 and 0.4) to[out=90,in=180,looseness=1] ++(0.5,0.5); % end at (3.3,2.9)
\draw[linkred,ultra thick] (3.3,2.9) -- ++(1,0) arc (90:0:0.6 and 0.15);

% right end of braid region
\draw[link,looseness=1.5] (3.9,1.8) -- ++(0.1,0) to[out=0,in=270] ++(0.6,1) to[out=90,in=0] ++(-0.6,0.6) -- ++(-2.8,0); % bottom strand
\draw[link,looseness=0.75] (3.9,2.2) to[out=0,in=270,looseness=1.75] ++(0.4,0.4) to[out=90,in=0,looseness=1.5] ++(-0.4,0.6) -- ++(-2.7,0); % middle strand

\foreach \y in {1.8} { \draw[link] (2,\y) -- ++(-0.8,0); }
%\foreach \y in {3.6,3.8} {\draw[link,looseness=1] (2,\y) to[out=180,in=0] ++(-0.8,0); }
\draw[link] (2,2.2) to[out=180,in=270,looseness=1] (1.4,2.5) to[out=90,in=0,looseness=1] (1.2,2.8); % coming into bottom right of twist

% twists for a=1
\begin{scope}[link/.append style={looseness=0.75}]
\draw[link] (2,1.8) to[out=0,in=180] ++(0.6,0.4) -- ++(0.7,0) to[out=0,in=180] ++(0.6,-0.4); % end at x=3.9
\draw[link] (2,2.2) to[out=0,in=180] ++(0.6,-0.4) -- ++(0.7,0) to[out=0,in=180] ++(0.6,0.4);
\end{scope}

% front of kappa and twist box
\draw[linkred, ultra thick] (2.95,2) arc (270:360:0.15 and 0.4) to[out=90,in=180,looseness=1] ++(0.2,0.2) -- ++(1,0) arc (270:360:0.6 and 0.15);
\draw[red,fill=white] (3.4,2.5) rectangle ++(0.6,0.5);
\node[red] at (3.7,2.75) {\small$-\epsilon$};

% add rest of diagram
% add a left-handed twist above the braid
\draw[link,looseness=0.75] (1.2,2.8) to[out=180,in=0] ++(-0.6,0.4) ++(0,-0.4) to[out=180,in=0] ++(-0.6,0.4);
\draw[link,looseness=0.75] (1.2,3.2) to[out=180,in=0] ++(-0.6,-0.4) ++(0,0.4) to[out=180,in=0] ++(-0.6,-0.4);

\draw[link] (1.2,1.8) -- ++(-1.2,0); % extend strands coming from top and bottom of braid region
\draw[link] (1.2,3.4) -- ++(-1.2,0);
%\draw[link,orange,looseness=2] (-1.6,2) -- ++(0,-0.2) to[out=270,in=270] ++(1.4,0) -- ++(0,0.2) to[out=90,in=300,looseness=1] ++(-0.4,1.2) to[out=120,in=180,looseness=1] ++(0.4,0.2) -- ++(1.4,0); % end up at top left of braid region
%\draw[link,yellow] (-1.2,2) -- ++(0,-0.2) to[out=270,in=270,looseness=2.5] ++(0.6,-0) -- ++(0,0.2);
%\draw[link,green] (-0.6,2) to[out=90,in=0,looseness=1] ++(-0.8,1) to[out=180,in=90,looseness=1.25] ++(-0.8,-0.75) to[out=270,in=180,looseness=1] ++(0.4,-0.45) -- ++(0.9,0);

%\draw[link,looseness=1] (0,3.2) -- ++(-0.4,0) to[out=180,in=90] (-1.6,2); % come out left end of twist
%\draw[link,looseness=1] (0,2.8) -- ++(-0.4,0) to[out=180,in=90] (-1.2,2);
\end{scope}

\begin{scope}[xshift=5.25cm]
\draw[linkred,ultra thick] (2.15,1.6) arc (270:180:0.15 and 0.4) -- ++(0,0.4) to[out=90,in=180,looseness=1] ++(0.5,0.5); % end at (3.3,2.9)
\draw[linkred,ultra thick] (2.5,2.9) -- ++(1,0) arc (90:0:0.6 and 0.15);

% right end of braid region
\draw[link,looseness=1.5] (3.1,1.8) -- ++(0.1,0) to[out=0,in=270] ++(0.6,1) to[out=90,in=0] ++(-0.6,0.6) -- ++(-2.8,0); % bottom strand
\draw[link,looseness=0.75] (3.1,2.2) to[out=0,in=270,looseness=1.75] ++(0.4,0.4) to[out=90,in=0,looseness=1.5] ++(-0.4,0.6) -- ++(-1.9,0); % middle strand

\foreach \y in {1.8} { \draw[link] (2,\y) -- ++(-0.8,0); }
%\foreach \y in {3.6,3.8} {\draw[link,looseness=1] (2,\y) to[out=180,in=0] ++(-0.8,0); }
\draw[link] (2,2.2) to[out=180,in=270,looseness=1] (1.4,2.5) to[out=90,in=0,looseness=1] (1.2,2.8); % coming into bottom right of twist

\draw[link] (2,1.8) -- ++(1.1,0); % bottom strand of braid

% front of kappa and twist box
\draw[linkred, ultra thick] (2.15,1.6) arc (270:360:0.15 and 0.4) -- ++(0,0.4) to[out=90,in=180,looseness=1] ++(0.2,0.2) -- ++(1,0) arc (270:360:0.6 and 0.15);
\draw[red,fill=white] (2.6,2.5) rectangle ++(0.6,0.5);
\node[red] at (2.9,2.75) {\small$-\epsilon$};

\draw[link] (2,2.2) -- ++(1.1,0); % top strand of braid

% add a left-handed twist above the braid
\draw[link,looseness=0.75] (1.2,2.8) to[out=180,in=0] ++(-0.6,0.4) ++(0,-0.4) to[out=180,in=0] ++(-0.6,0.4);
\draw[link,looseness=0.75] (1.2,3.2) to[out=180,in=0] ++(-0.6,-0.4) ++(0,0.4) to[out=180,in=0] ++(-0.6,-0.4);

\draw[link] (1.2,1.8) -- ++(-1.2,0); % extend strands coming from top and bottom of braid region
\draw[link] (1.2,3.4) -- ++(-1.2,0);
\end{scope}

\begin{scope}[xshift=9.75cm]
\draw[linkred,ultra thick,looseness=1] (3.9,2.95) arc (0:90:0.5 and 0.15) -- ++(-0.7,0) to[out=180,in=90] ++(-0.25,-0.5) to[out=270,in=180] ++(0.25,-0.5) -- ++(0.7,0) arc (270:360:0.5 and 0.15) coordinate (rrb);

%outer strand of braid region
\draw[link,looseness=1.5] (0,1.8) -- ++(3,0) to[out=0,in=270] ++(0.6,1) to[out=90,in=0] ++(-0.6,0.6) -- ++(-3,0); % bottom strand

\draw[link] (2,1.8) -- ++(-0.8,0);
\draw[link,looseness=1] (1.2,2.8) to[out=0,in=90] (1.4,2.4) to[out=270,in=180] (1.8,2) to[out=0,in=270] (2.2,2.6) to[out=90,in=0] (1.8,3.2) -- (1.2,3.2); % coming into bottom right of twist

% front of kappa and twist box
\draw[linkred,ultra thick] (rrb) arc (0:90:0.5 and 0.15) -- ++(-0.95,0) arc (270:90:0.5 and 0.2) -- ++(0.95,0) arc (270:360:0.5 and 0.15);
\draw[red,fill=white] (2.75,2) rectangle ++(0.6,0.5);
\node[red] at (3.05,2.25) {\small$-\epsilon$};

% redraw part of twist loop on left
\draw[link,looseness=1] (1.8,2) to[out=0,in=270] (2.2,2.6);

% add a left-handed twist above the braid
\draw[link,looseness=0.75] (1.2,2.8) to[out=180,in=0] ++(-0.6,0.4) ++(0,-0.4) to[out=180,in=0] ++(-0.6,0.4);
\draw[link,looseness=0.75] (1.2,3.2) to[out=180,in=0] ++(-0.6,-0.4) ++(0,0.4) to[out=180,in=0] ++(-0.6,-0.4);
\end{scope}

%%%

\begin{scope}[yshift=-3.75cm]
\draw[linkred,ultra thick,looseness=1] (3.9,2.95) arc (0:90:0.5 and 0.15) -- ++(-0.7,0) to[out=180,in=90] ++(-0.25,-0.5) to[out=270,in=180] ++(0.25,-0.5) -- ++(0.7,0) arc (270:360:0.5 and 0.15) coordinate (rrb);

%outer strand of braid region
\draw[link,looseness=1.5] (0,1.8) -- ++(3,0) to[out=0,in=270] ++(0.6,1) to[out=90,in=0] ++(-0.6,0.6) -- ++(-3,0); % bottom strand

\draw[link] (2,1.8) -- ++(-0.8,0);
\draw[link,looseness=1] (0,2.8) to[out=0,in=90] ++(0.2,-0.4) to[out=270,in=180] ++(0.4,-0.4) coordinate (loopbr) to[out=0,in=270] ++(0.4,0.6) to[out=90,in=0] ++(-0.4,0.6) -- ++(-0.6,0); % coming into bottom right of twist

% front of kappa and twist box
\draw[linkred,ultra thick] (rrb) arc (0:90:0.5 and 0.15) -- ++(-0.95,0) ++(-1.2,0) arc (270:90:0.5 and 0.2) coordinate (rtwist) ++ (1.2,0) -- ++(0.95,0) arc (270:360:0.5 and 0.15);
\draw[linkred,ultra thick,looseness=0.75] (rtwist) to[out=0,in=180] ++(0.6,-0.4) ++(0,0.4) to[out=0,in=180] ++(0.6,-0.4);
\draw[linkred,ultra thick,looseness=0.75] (rtwist) ++ (0,-0.4) to[out=0,in=180] ++(0.6,0.4) ++(0,-0.4) to[out=0,in=180] ++(0.6,0.4);
\draw[red,fill=white] (2.75,2) rectangle ++(0.6,0.5);
\node[red] at (3.05,2.25) {\small$-\epsilon$};

% redraw part of twist loop on left
\draw[link,looseness=1] (loopbr) to[out=0,in=270] ++(0.4,0.6);

% add a left-handed twist above the braid
%\draw[link,looseness=0.75] (1.2,2.8) to[out=180,in=0] ++(-0.6,0.4) ++(0,-0.4) to[out=180,in=0] ++(-0.6,0.4);
%\draw[link,looseness=0.75] (1.2,3.2) to[out=180,in=0] ++(-0.6,-0.4) ++(0,0.4) to[out=180,in=0] ++(-0.6,-0.4);

\draw[->] (4,2.6) ++(0.2,0.4) -- node[above,sloped] {\small$\epsilon=+1$} ++(1.25,0.8);
\draw[->] (4,2.6) ++(0.2,-0.4) -- node[below,sloped] {\small$\epsilon=-1$} ++(1.25,-0.8);
\end{scope}

\begin{scope}[xshift=0.5cm]

%%% epsilon=1
\begin{scope}[yshift=-2.5cm,xshift=5.25cm]
\draw[linkred,ultra thick,looseness=1] (3.9,2.95) arc (0:90:0.5 and 0.15) -- ++(-0.7,0) to[out=180,in=90] ++(-0.25,-0.5) to[out=270,in=180] ++(0.25,-0.5) coordinate (repsilon) ++(0.6,0) -- ++(0.1,0) arc (270:360:0.5 and 0.15) coordinate (rrb);

%outer strand of braid region
\draw[link,looseness=1.5] (0,1.8) -- ++(3,0) to[out=0,in=270] ++(0.6,1) to[out=90,in=0] ++(-0.6,0.6) -- ++(-3,0); % bottom strand

\draw[link] (2,1.8) -- ++(-0.8,0);
\draw[link,looseness=1] (0,2.8) to[out=0,in=90] ++(0.2,-0.4) to[out=270,in=180] ++(0.4,-0.4) coordinate (loopbr) to[out=0,in=270] ++(0.4,0.6) to[out=90,in=0] ++(-0.4,0.6) -- ++(-0.6,0); % coming into bottom right of twist

% front of kappa and twist box
\draw[linkred,ultra thick] (rrb) arc (0:90:0.5 and 0.15) -- ++(-0.1,0) ++(-0.6,0) -- ++(-0.25,0) ++(-1.2,0) arc (270:90:0.5 and 0.2) coordinate (rtwist) ++ (1.2,0) -- ++(0.95,0) arc (270:360:0.5 and 0.15);
\begin{scope}[linkred/.append style={looseness=0.75}]
\draw[linkred,ultra thick] (rtwist) to[out=0,in=180] ++(0.6,-0.4) ++(0,0.4) to[out=0,in=180] ++(0.6,-0.4);
\draw[linkred,ultra thick] (rtwist) ++ (0,-0.4) to[out=0,in=180] ++(0.6,0.4) ++(0,-0.4) to[out=0,in=180] ++(0.6,0.4);
%\draw[red,fill=white] (2.75,2) rectangle ++(0.6,0.5);
%\node[red] at (3.05,2.25) {\small$-\epsilon$};
\draw[linkred,ultra thick] (repsilon) ++(0,0.3) to[out=0,in=180] ++(0.6,-0.3);
\draw[linkred,ultra thick] (repsilon) to[out=0,in=180] ++(0.6,0.3);
\end{scope}

% redraw part of twist loop on left
\draw[link,looseness=1] (loopbr) to[out=0,in=270] ++(0.4,0.6);
\end{scope}

\begin{scope}[yshift=-2.5cm,xshift=9.75cm]
\draw[link,looseness=1] (0,2.8) to[out=0,in=90] ++(0.2,-0.4) to[out=270,in=180] ++(0.4,-0.4) coordinate (loopbr) to[out=0,in=270] ++(0.4,0.6) to[out=90,in=0] ++(-0.4,0.6) -- ++(-0.6,0); % coming into bottom right of twist

% front of kappa and twist box
\draw[linkred,ultra thick] (1.25,2.4) arc (270:90:0.5 and 0.2) coordinate (rtwist);
\path (1.85,2.8) to[out=0,in=180] ++(0.3,0.25) coordinate (linkingpt);
\draw[linkred,ultra thick] (linkingpt) arc (90:0:0.5 and 0.15) ++ (0,-0.6) arc (0:90:0.3 and 0.15) arc (270:180:0.3 and 0.15);
\draw[link,looseness=1.5] (0,1.8) -- ++(1.8,0) to[out=0,in=270] ++(0.6,1) to[out=90,in=0] ++(-0.6,0.6) -- ++(-1.8,0); % outer strand of braid region
\begin{scope}[linkred/.append style={looseness=0.75}]
\draw[linkred,ultra thick] (rtwist) to[out=0,in=180] ++(0.9,-0.65);
\draw[linkred,ultra thick] (rtwist) ++(0,-0.4) to[out=0,in=180] ++(0.9,0.65);
\draw[linkred,ultra thick] (linkingpt) ++(0.5,-0.15) arc (0:-90:0.3 and 0.15) arc (90:180:0.3 and 0.15);
\draw[linkred,ultra thick] (linkingpt) ++(0.5,-0.75) arc (0:-90:0.5 and 0.15);
\node[below] at (1.2,1.8) {\scriptsize$\beta = yxy^{-1}$};
\end{scope}

% redraw part of twist loop on left
\draw[link,looseness=1] (loopbr) to[out=0,in=270] ++(0.4,0.6);
\end{scope}

%%% epsilon=-1
\begin{scope}[yshift=-5cm,xshift=5.25cm]
\draw[linkred,ultra thick,looseness=1] (3.9,2.95) arc (0:90:0.5 and 0.15) -- ++(-0.7,0) to[out=180,in=90] ++(-0.25,-0.5) to[out=270,in=180] ++(0.25,-0.5) coordinate (repsilon) ++(0.6,0) -- ++(0.1,0) arc (270:360:0.5 and 0.15) coordinate (rrb);

%outer strand of braid region
\draw[link,looseness=1.5] (0,1.8) -- ++(3,0) to[out=0,in=270] ++(0.6,1) to[out=90,in=0] ++(-0.6,0.6) -- ++(-3,0); % bottom strand

%\draw[link] (2,1.8) -- ++(-0.8,0);
\draw[link,looseness=1] (0,2.8) to[out=0,in=90] ++(0.2,-0.4) to[out=270,in=180] ++(0.4,-0.4) coordinate (loopbr) to[out=0,in=270] ++(0.4,0.6) to[out=90,in=0] ++(-0.4,0.6) -- ++(-0.6,0); % coming into bottom right of twist

% front of kappa and twist box
\draw[linkred,ultra thick] (rrb) arc (0:90:0.5 and 0.15) -- ++(-0.1,0) ++(-0.6,0) -- ++(-0.25,0) ++(-1.2,0) arc (270:90:0.5 and 0.2) coordinate (rtwist) ++ (1.2,0) -- ++(0.95,0) arc (270:360:0.5 and 0.15);
\begin{scope}[linkred/.append style={looseness=0.75}]
\draw[linkred,ultra thick] (rtwist) to[out=0,in=180] ++(0.6,-0.4) ++(0,0.4) to[out=0,in=180] ++(0.6,-0.4);
\draw[linkred,ultra thick] (rtwist) ++ (0,-0.4) to[out=0,in=180] ++(0.6,0.4) ++(0,-0.4) to[out=0,in=180] ++(0.6,0.4);
%\draw[red,fill=white] (2.75,2) rectangle ++(0.6,0.5);
%\node[red] at (3.05,2.25) {\small$-\epsilon$};
\draw[linkred,ultra thick] (repsilon) to[out=0,in=180] ++(0.6,0.3);
\draw[linkred,ultra thick] (repsilon) ++(0,0.3) to[out=0,in=180] ++(0.6,-0.3);
\end{scope}

% redraw part of twist loop on left
\draw[link,looseness=1] (loopbr) to[out=0,in=270] ++(0.4,0.6);
\end{scope}

\begin{scope}[yshift=-5cm,xshift=9.75cm]
\draw[link,looseness=1] (0,2.8) to[out=0,in=90] ++(0.2,-0.4) to[out=270,in=180] ++(0.4,-0.4) coordinate (loopbr) to[out=0,in=270] ++(0.4,0.6) to[out=90,in=0] ++(-0.4,0.6) -- ++(-0.6,0); % coming into bottom right of twist

% front of kappa and twist box
\draw[linkred,ultra thick] (1.25,2.4) arc (270:90:0.5 and 0.2) coordinate (rtwist);
\path (1.85,2.8) to[out=0,in=180] ++(0.9,0.25) coordinate (linkingpt);
\draw[linkred,ultra thick,looseness=0.75] (linkingpt) ++(0.5,-0.15) arc (0:-90:0.3 and 0.15) arc (90:270:0.3 and 0.15);
\draw[linkred,ultra thick] (linkingpt) ++(0.5,-0.75) arc (0:-90:0.5 and 0.15);
\draw[link,looseness=1.5] (0,1.8) -- ++(2.4,0) to[out=0,in=270] ++(0.6,1) to[out=90,in=0] ++(-0.6,0.6) -- ++(-2.4,0); % outer strand of braid region
\draw[linkred,ultra thick] (linkingpt) arc (90:0:0.5 and 0.15) ++(0,-0.6) arc (0:90:0.3 and 0.15) arc (270:180:0.3 and 0.15);
\begin{scope}[linkred/.append style={looseness=0.75}]
\draw[linkred,ultra thick] (rtwist) to[out=0,in=180] ++(0.6,-0.4) ++(0,0.4) to[out=0,in=180] ++(0.9,-0.65);
\draw[linkred,ultra thick] (rtwist) ++(0,-0.4) to[out=0,in=180] ++(0.6,0.4) ++(0,-0.4) to[out=0,in=180] ++(0.9,0.65);
\node[below] at (1.5,1.8) {\scriptsize$\beta = yx^{-1}y^{-1}$};
\end{scope}

% redraw part of twist loop on left
\draw[link,looseness=1] (loopbr) to[out=0,in=270] ++(0.4,0.6);
\end{scope}

\draw[thin, densely dashed] (9.5,1.1) rectangle (13.2,-3.75);

\end{scope}
\end{tikzpicture}
\caption{Simplifying the case $a=1$, where $\beta = yx^{\pm1}y^{-1}$, by an isotopy.}
\label{fig:T23-tangle-simplify-1}
\end{figure}
%%%%%
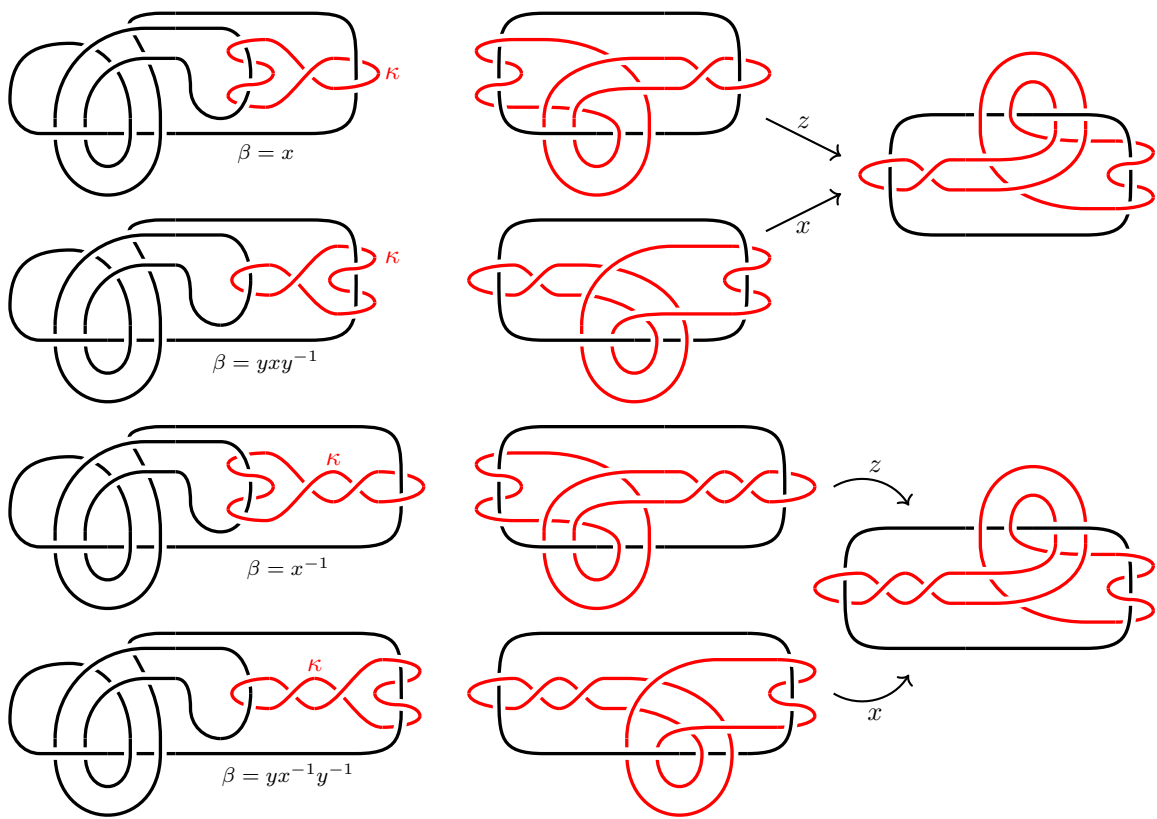
\begin{figure}
\begin{tikzpicture}

\foreach \y in {0cm,2.75cm,5.5cm,8.25cm} {
\begin{scope}[yshift=-\y]
\draw[link] (0,1.8) -- (-0.9,1.8); % extend strand coming from bottom of braid region
\draw[link,looseness=2] (-1.6,2) -- ++(0,-0.2) to[out=270,in=270] ++(1.4,0) -- ++(0,0.2) to[out=90,in=300,looseness=1] ++(-0.4,1.2) to[out=120,in=180,looseness=1] ++(0.4,0.2) -- ++(0.2,0); % end up at top left of braid region
\draw[link] (-1.2,2) -- ++(0,-0.2) to[out=270,in=270,looseness=2.5] ++(0.6,-0) -- ++(0,0.2);
\draw[link] (-0.6,2) to[out=90,in=0,looseness=1] ++(-0.8,1) to[out=180,in=90,looseness=1.25] ++(-0.8,-0.75) to[out=270,in=180,looseness=1] ++(0.4,-0.45) -- ++(0.9,0);

\draw[link,looseness=1] (0,3.2) -- ++(-0.4,0) to[out=180,in=90] (-1.6,2); % come out left end of twist
\draw[link,looseness=1] (0,2.8) -- ++(-0.4,0) to[out=180,in=90] (-1.2,2);
\end{scope}
}

\begin{scope}[yshift=0cm]
% right end of braid region
\draw[linkred] (1.2,3.05) ++(-0.5,-0.15) coordinate (redl1) arc (180:270:0.3 and 0.15) arc (90:0:0.3 and 0.15) ++ (-0.6,-0.3) arc (180:270:0.5 and 0.15) coordinate (redr1);
\draw[linkred] (1.2,3.05) to[out=0,in=180,looseness=0.75] ++(0.9,-0.65) ++(0,0.4) arc (90:0:0.6 and 0.2);
\draw[link,looseness=1] (0,2.8) to[out=0,in=90] ++(0.2,-0.4) to[out=270,in=180] ++(0.4,-0.4) coordinate (loopbr) to[out=0,in=270] ++(0.4,0.6) to[out=90,in=0] ++(-0.4,0.6) -- ++(-0.6,0); % middle portion of U, from height 2 to 3.2. Maximum x = 1
\draw[linkred,looseness=0.75] (redl1) arc (180:90:0.5 and 0.15);
\draw[linkred,looseness=0.75] (redl1) ++ (0.6,-0.3) arc (0:-90:0.3 and 0.15) arc (90:180:0.3 and 0.15);
\draw[link,looseness=1.5] (0,1.8) -- ++(1.7,0) -- ++(0.1,0) to[out=0,in=270] ++(0.6,1) to[out=90,in=0] ++(-0.6,0.6) -- ++(-1.8,0); % outer strand
\draw[linkred,looseness=0.75] (redr1) to[out=0,in=180] ++(0.9,0.65) ++(0,-0.4) arc (-90:0:0.6 and 0.2);
\node[below] at (1.2,1.8) {\scriptsize$\beta = x$};
\node[right,red,inner sep=2pt] at (2.7,2.6) {\small$\kappa$};
\end{scope}

\begin{scope}[yshift=0cm,xshift=6.5cm]
\draw[link] (0,1.8) -- (-0.9,1.8); % undercrossing part of U, should be drawn first
\draw[linkred,looseness=2] (-1.6,2) -- ++(0,-0.2) to[out=270,in=270] ++(1.4,0) -- ++(0,0.2) to[out=90,in=0,looseness=1] ++(-1.4,1.05) -- ++(-0.4,0) coordinate (ktl); % top left of kappa, end at (-2,3.05)
\draw[linkred] (-1.2,2) -- ++(0,-0.2) to[out=270,in=270,looseness=2.5] ++(0.6,0) to[out=90,in=0,looseness=0.75] ++(-1,0.35) -- ++(-0.4,0); % end at (-2,2.15)

\draw[linkred,looseness=1] (0,2.8) -- ++(-0.4,0) to[out=180,in=90] (-1.6,2); % connect to right tangle
\draw[linkred,looseness=1] (0,2.4) -- ++(-0.4,0) to[out=180,in=90] (-1.2,2);

\draw[linkred,looseness=0.75] (0,2.8) -- ++(0.2,0) to[out=0,in=180] ++(0.6,-0.4) ++(0,0.4) arc (90:0:0.6 and 0.2); % right crossing

\draw[linkred] (ktl) ++(-0.5,-0.15) arc (180:270:0.3 and 0.15) arc (90:0:0.3 and 0.15) ++ (-0.6,-0.3) arc (180:270:0.5 and 0.15);
\draw[link,looseness=1.5] (0,1.8) -- ++(0.4,0) to[out=0,in=270] ++(0.6,0.8) to[out=90,in=0] ++(-0.6,0.8) -- ++(-0.4,0); % U, right
\draw[link,looseness=1.5] (-0.9,1.8) -- ++(-0.7,0) to[out=180,in=270] ++(-0.6,0.8) to[out=90,in=180] ++(0.6,0.8) -- ++(1.6,0); % U, left
\draw[linkred] (ktl) arc (90:180:0.5 and 0.15) ++ (0.6,-0.3) arc (0:-90:0.3 and 0.15) arc (90:180:0.3 and 0.15);

\draw[linkred,looseness=0.75] (0,2.4) -- ++(0.2,0) to[out=0,in=180] ++(0.6,0.4) ++(0,-0.4) arc (-90:0:0.6 and 0.2); % right crossing
\end{scope}

\begin{scope}[yshift=-2.75cm]
\draw[link,looseness=1] (0,2.8) to[out=0,in=90] ++(0.2,-0.4) to[out=270,in=180] ++(0.4,-0.4) coordinate (loopbr) to[out=0,in=270] ++(0.4,0.6) to[out=90,in=0] ++(-0.4,0.6) -- ++(-0.6,0); % coming into bottom right of twist
%
% front of kappa and twist box
\draw[linkred,ultra thick] (1.25,2.4) arc (270:90:0.5 and 0.2) coordinate (rtwist);
\path (1.85,2.8) to[out=0,in=180] ++(0.3,0.25) coordinate (linkingpt);
\draw[linkred,ultra thick] (linkingpt) arc (90:0:0.5 and 0.15) ++ (0,-0.6) arc (0:90:0.3 and 0.15) arc (270:180:0.3 and 0.15);
\draw[link,looseness=1.5] (0,1.8) -- ++(1.8,0) to[out=0,in=270] ++(0.6,1) to[out=90,in=0] ++(-0.6,0.6) -- ++(-1.8,0); % outer strand of braid region
\begin{scope}[linkred/.append style={looseness=0.75}]
\draw[linkred,ultra thick] (rtwist) to[out=0,in=180] ++(0.9,-0.65);
\draw[linkred,ultra thick] (rtwist) ++(0,-0.4) to[out=0,in=180] ++(0.9,0.65);
\draw[linkred,ultra thick] (linkingpt) ++(0.5,-0.15) arc (0:-90:0.3 and 0.15) arc (90:180:0.3 and 0.15);
\draw[linkred,ultra thick] (linkingpt) ++(0.5,-0.75) arc (0:-90:0.5 and 0.15);
\node[below] at (1.2,1.8) {\scriptsize$\beta = yxy^{-1}$};
\end{scope}
% redraw part of twist loop on left
\draw[link,looseness=1] (loopbr) to[out=0,in=270] ++(0.4,0.6);
\node[right,red,inner sep=2pt] at (2.7,2.9) {\small$\kappa$};
\end{scope}

\begin{scope}[yshift=-2.75cm,xshift=7cm]
\draw[link] (0,1.8) -- (-0.9,1.8); % undercrossing part of U, should be drawn first
\draw[linkred] (-1.2,1.8) to[out=270,in=270,looseness=2.5] ++(0.6,0) to[out=90,in=0,looseness=0.75] ++(-1,0.6) -- ++(-0.3,0) to[out=180,in=0,looseness=0.75] ++(-0.6,0.4); % end at (-2,2.4) plus twist
\draw[linkred,looseness=2] (-1.6,2) -- ++(0,-0.2) to[out=270,in=270] ++(1.4,0) to[out=90,in=0,looseness=1] ++(-1.4,1) -- ++(-0.3,0) to[out=180,in=0,looseness=0.75] ++(-0.6,-0.4); % top left of kappa, end at (-2,2.8) plus twist

\draw[linkred,looseness=1] (0.4,3.05) coordinate (ktl) -- ++(-0.8,0) to[out=180,in=90] (-1.6,2); % connect to right tangle
\draw[linkred,looseness=1] (0.4,2.15) -- ++(-0.8,0) to[out=180,in=90] (-1.2,1.8);

\draw[linkred] (-2.5,2.4) arc (270:180:0.6 and 0.2);
\draw[linkred] (ktl) arc (90:0:0.5 and 0.15) ++(-0.6,-0.3) arc (180:270:0.3 and 0.15) arc (90:0:0.3 and 0.15);
\draw[link,looseness=1.5] (0,1.8) to[out=0,in=270] ++(0.6,0.8) to[out=90,in=0] ++(-0.6,0.8); % U, right
\draw[link,looseness=1.5] (-0.9,1.8) -- ++(-1.2,0) to[out=180,in=270] ++(-0.6,0.8) to[out=90,in=180] ++(0.6,0.8) -- ++(2.1,0); % U, left
\draw[linkred] (ktl) ++(0.5,-0.15) arc (0:-90:0.3 and 0.15) arc (90:180:0.3 and 0.15) ++(0.6,-0.3) arc (0:-90:0.5 and 0.15);
\draw[linkred] (-2.5,2.8) arc (90:180:0.6 and 0.2);
\end{scope}

\begin{scope}[yshift=3.85cm,xshift=10.5cm,rotate=180]
\draw[link] (0,1.8) -- (-0.9,1.8); % undercrossing part of U, should be drawn first
\draw[linkred,looseness=2] (-1.6,2) -- ++(0,-0.2) to[out=270,in=270] ++(1.4,0) -- ++(0,0.2) to[out=90,in=0,looseness=1] ++(-1.4,1.05) -- ++(-0.4,0) coordinate (ktl); % top left of kappa, end at (-2,3.05)
\draw[linkred] (-1.2,2) -- ++(0,-0.2) to[out=270,in=270,looseness=2.5] ++(0.6,0) to[out=90,in=0,looseness=0.75] ++(-1,0.35) -- ++(-0.4,0); % end at (-2,2.15)

\draw[linkred,looseness=1] (0,2.8) -- ++(-0.4,0) to[out=180,in=90] (-1.6,2); % connect to right tangle
\draw[linkred,looseness=1] (0,2.4) -- ++(-0.4,0) to[out=180,in=90] (-1.2,2);

\draw[linkred,looseness=0.75] (0,2.8) -- ++(0.2,0) to[out=0,in=180] ++(0.6,-0.4) ++(0,0.4) arc (90:0:0.6 and 0.2); % right crossing

\draw[linkred] (ktl) ++(-0.5,-0.15) arc (180:270:0.3 and 0.15) arc (90:0:0.3 and 0.15) ++ (-0.6,-0.3) arc (180:270:0.5 and 0.15);
\draw[link,looseness=1.5] (0,1.8) -- ++(0.4,0) to[out=0,in=270] ++(0.6,0.8) to[out=90,in=0] ++(-0.6,0.8) -- ++(-0.4,0); % U, right
\draw[link,looseness=1.5] (-0.9,1.8) -- ++(-0.7,0) to[out=180,in=270] ++(-0.6,0.8) to[out=90,in=180] ++(0.6,0.8) -- ++(1.6,0); % U, left
\draw[linkred] (ktl) arc (90:180:0.5 and 0.15) ++ (0.6,-0.3) arc (0:-90:0.3 and 0.15) arc (90:180:0.3 and 0.15);

\draw[linkred,looseness=0.75] (0,2.4) -- ++(0.2,0) to[out=0,in=180] ++(0.6,0.4) ++(0,-0.4) arc (-90:0:0.6 and 0.2); % right crossing
\end{scope}
\draw[->] (7.85,2) -- node[midway,above,inner sep=3pt] {\small$z$} ++(1,-0.5);
\draw[->] (7.85,0.5) -- node[midway,below,inner sep=3pt] {\small$x$} ++(1,0.5);

\begin{scope}[yshift=-5.5cm]
% right end of braid region
\draw[linkred] (1.2,3.05) arc (90:180:0.5 and 0.15) coordinate (redl1) ++(0.6,-0.3) arc (0:-90:0.3 and 0.15) arc (90:180:0.3 and 0.15) arc (180:270:0.5 and 0.15);
\draw[linkred] (1.2,3.05) to[out=0,in=180,looseness=0.75] ++(0.9,-0.65);
\draw[link,looseness=1] (0,2.8) to[out=0,in=90] ++(0.2,-0.4) to[out=270,in=180] ++(0.4,-0.4) coordinate (loopbr) to[out=0,in=270] ++(0.4,0.6) to[out=90,in=0] ++(-0.4,0.6) -- ++(-0.6,0); % middle portion of U, from height 2 to 3.2. Maximum x = 1
\draw[linkred,looseness=0.75] (redl1) arc (180:270:0.3 and 0.15) arc (90:0:0.3 and 0.15) ++(-0.6,-0.3) arc (180:270:0.5 and 0.15) to[out=0,in=180] ++(0.9,0.65) coordinate (redr1) to[out=0,in=180] ++(0.6,-0.4) ++(0,0.4) arc (90:0:0.6 and 0.2);

\draw[link,looseness=1.5] (0,1.8) -- ++(2.3,0) -- ++(0.1,0) to[out=0,in=270] ++(0.6,1) to[out=90,in=0] ++(-0.6,0.6) -- ++(-2.4,0); % outer strand
\draw[linkred,looseness=0.75] (redr1) ++(0,-0.4) to[out=0,in=180] ++(0.6,0.4) ++(0,-0.4) arc (-90:0:0.6 and 0.2);
\node[below] at (1.5,1.8) {\scriptsize$\beta = x^{-1}$};
\node[above,red,inner sep=2pt] at (redr1) {\small$\kappa$};
\end{scope}

\begin{scope}[yshift=-5.5cm,xshift=6.5cm]
\draw[link] (0,1.8) -- (-0.9,1.8); % undercrossing part of U, should be drawn first
\draw[linkred,looseness=2] (-1.6,2) -- ++(0,-0.2) to[out=270,in=270] ++(1.4,0) -- ++(0,0.2) to[out=90,in=0,looseness=1] ++(-1.4,1.05) -- ++(-0.4,0) coordinate (ktl); % top left of kappa, end at (-2,3.05)
\draw[linkred] (-1.2,2) -- ++(0,-0.2) to[out=270,in=270,looseness=2.5] ++(0.6,0) to[out=90,in=0,looseness=0.75] ++(-1,0.35) -- ++(-0.4,0); % end at (-2,2.15)

\draw[linkred,looseness=1] (0,2.8) -- ++(-0.4,0) to[out=180,in=90] (-1.6,2); % connect to right tangle
\draw[linkred,looseness=1] (0,2.4) -- ++(-0.4,0) to[out=180,in=90] (-1.2,2);

\draw[linkred,looseness=0.75] (0,2.8) -- ++(0.2,0) to[out=0,in=180] ++(0.6,-0.4) ++(0,0.4) to[out=0,in=180] ++(0.6,-0.4) ++(0,0.4) arc (90:0:0.6 and 0.2); % right crossing

\draw[linkred] (ktl) arc (90:180:0.5 and 0.15) ++ (0.6,-0.3) arc (0:-90:0.3 and 0.15) arc (90:180:0.3 and 0.15);
\draw[link,looseness=1.5] (0,1.8) -- ++(1,0) to[out=0,in=270] ++(0.6,0.8) to[out=90,in=0] ++(-0.6,0.8) -- ++(-1,0); % U, right
\draw[link,looseness=1.5] (-0.9,1.8) -- ++(-0.7,0) to[out=180,in=270] ++(-0.6,0.8) to[out=90,in=180] ++(0.6,0.8) -- ++(1.6,0); % U, left
\draw[linkred] (ktl) ++(-0.5,-0.15) arc (180:270:0.3 and 0.15) arc (90:0:0.3 and 0.15) ++ (-0.6,-0.3) arc (180:270:0.5 and 0.15);

\draw[linkred,looseness=0.75] (0,2.4) -- ++(0.2,0) to[out=0,in=180] ++(0.6,0.4) ++(0,-0.4) to[out=0,in=180] ++(0.6,0.4) ++(0,-0.4) arc (-90:0:0.6 and 0.2); % right crossing
\end{scope}

\begin{scope}[yshift=-8.25cm]
\draw[link,looseness=1] (0,2.8) to[out=0,in=90] ++(0.2,-0.4) to[out=270,in=180] ++(0.4,-0.4) coordinate (loopbr) to[out=0,in=270] ++(0.4,0.6) to[out=90,in=0] ++(-0.4,0.6) -- ++(-0.6,0); % coming into bottom right of twist
%
% front of kappa and twist box
\draw[linkred,ultra thick] (1.25,2.4) arc (270:90:0.5 and 0.2) coordinate (rtwist);
\path (1.85,2.8) to[out=0,in=180] ++(0.9,0.25) coordinate (linkingpt);
\draw[linkred,ultra thick,looseness=0.75] (linkingpt) ++(0.5,-0.15) arc (0:-90:0.3 and 0.15) arc (90:270:0.3 and 0.15);
\draw[linkred,ultra thick] (linkingpt) ++(0.5,-0.75) arc (0:-90:0.5 and 0.15);
\draw[link,looseness=1.5] (0,1.8) -- ++(2.4,0) to[out=0,in=270] ++(0.6,1) to[out=90,in=0] ++(-0.6,0.6) -- ++(-2.4,0); % outer strand of braid region
\draw[linkred,ultra thick] (linkingpt) arc (90:0:0.5 and 0.15) ++(0,-0.6) arc (0:90:0.3 and 0.15) arc (270:180:0.3 and 0.15);
\begin{scope}[linkred/.append style={looseness=0.75}]
\draw[linkred,ultra thick] (rtwist) to[out=0,in=180] ++(0.6,-0.4) ++(0,0.4) node[above,red,inner sep=2pt] {\small$\kappa$} to[out=0,in=180] ++(0.9,-0.65);
\draw[linkred,ultra thick] (rtwist) ++(0,-0.4) to[out=0,in=180] ++(0.6,0.4) ++(0,-0.4) to[out=0,in=180] ++(0.9,0.65);
\node[below] at (1.5,1.8) {\scriptsize$\beta = yx^{-1}y^{-1}$};
\end{scope}
% redraw part of twist loop on left
\draw[link,looseness=1] (loopbr) to[out=0,in=270] ++(0.4,0.6);
\end{scope}

\begin{scope}[yshift=-8.25cm,xshift=7.6cm]
\draw[link] (0,1.8) -- (-0.9,1.8); % undercrossing part of U, should be drawn first
\draw[linkred] (-1.2,1.8) to[out=270,in=270,looseness=2.5] ++(0.6,0) to[out=90,in=0,looseness=0.75] ++(-1,0.6) -- ++(-0.3,0) to[out=180,in=0,looseness=0.75] ++(-0.6,0.4) ++(0,-0.4) to[out=180,in=0,looseness=0.75] ++(-0.6,0.4); % end at (-2,2.4) plus twists
\draw[linkred,looseness=2] (-1.6,2) -- ++(0,-0.2) to[out=270,in=270] ++(1.4,0) to[out=90,in=0,looseness=1] ++(-1.4,1) -- ++(-0.3,0) to[out=180,in=0,looseness=0.75] ++(-0.6,-0.4) ++(0,0.4) to[out=180,in=0,looseness=0.75] ++(-0.6,-0.4); % top left of kappa, end at (-2,2.8) plus twists

\draw[linkred,looseness=1] (0.4,3.05) coordinate (ktl) -- ++(-0.8,0) to[out=180,in=90] (-1.6,2); % connect to right tangle
\draw[linkred,looseness=1] (0.4,2.15) -- ++(-0.8,0) to[out=180,in=90] (-1.2,1.8);

\draw[linkred] (-3.1,2.4) arc (270:180:0.6 and 0.2);
\draw[linkred] (ktl) ++(0.5,-0.15) arc (0:-90:0.3 and 0.15) arc (90:180:0.3 and 0.15) ++(0.6,-0.3) arc (0:-90:0.5 and 0.15);
\draw[link,looseness=1.5] (0,1.8) to[out=0,in=270] ++(0.6,0.8) to[out=90,in=0] ++(-0.6,0.8); % U, right
\draw[link,looseness=1.5] (-0.9,1.8) -- ++(-1.8,0) to[out=180,in=270] ++(-0.6,0.8) to[out=90,in=180] ++(0.6,0.8) -- ++(2.7,0); % U, left
\draw[linkred] (ktl) arc (90:0:0.5 and 0.15) ++(-0.6,-0.3) arc (180:270:0.3 and 0.15) arc (90:0:0.3 and 0.15);
\draw[linkred] (-3.1,2.8) arc (90:180:0.6 and 0.2);
\end{scope}

\begin{scope}[yshift=-1.65cm,xshift=10.5cm,rotate=180]
\draw[link] (0,1.8) -- (-0.9,1.8); % undercrossing part of U, should be drawn first
\draw[linkred,looseness=2] (-1.6,2) -- ++(0,-0.2) to[out=270,in=270] ++(1.4,0) -- ++(0,0.2) to[out=90,in=0,looseness=1] ++(-1.4,1.05) -- ++(-0.4,0) coordinate (ktl); % top left of kappa, end at (-2,3.05)
\draw[linkred] (-1.2,2) -- ++(0,-0.2) to[out=270,in=270,looseness=2.5] ++(0.6,0) to[out=90,in=0,looseness=0.75] ++(-1,0.35) -- ++(-0.4,0); % end at (-2,2.15)

\draw[linkred,looseness=1] (0,2.8) -- ++(-0.4,0) to[out=180,in=90] (-1.6,2); % connect to right tangle
\draw[linkred,looseness=1] (0,2.4) -- ++(-0.4,0) to[out=180,in=90] (-1.2,2);

\draw[linkred,looseness=0.75] (0,2.8) -- ++(0.2,0) to[out=0,in=180] ++(0.6,-0.4) ++(0,0.4) to[out=0,in=180] ++(0.6,-0.4) ++(0,0.4) arc (90:0:0.6 and 0.2); % right crossing

\draw[linkred] (ktl) arc (90:180:0.5 and 0.15) ++ (0.6,-0.3) arc (0:-90:0.3 and 0.15) arc (90:180:0.3 and 0.15);
\draw[link,looseness=1.5] (0,1.8) -- ++(1,0) to[out=0,in=270] ++(0.6,0.8) to[out=90,in=0] ++(-0.6,0.8) -- ++(-1,0); % U, right
\draw[link,looseness=1.5] (-0.9,1.8) -- ++(-0.7,0) to[out=180,in=270] ++(-0.6,0.8) to[out=90,in=180] ++(0.6,0.8) -- ++(1.6,0); % U, left
\draw[linkred] (ktl) ++(-0.5,-0.15) arc (180:270:0.3 and 0.15) arc (90:0:0.3 and 0.15) ++ (-0.6,-0.3) arc (180:270:0.5 and 0.15);

\draw[linkred,looseness=0.75] (0,2.4) -- ++(0.2,0) to[out=0,in=180] ++(0.6,0.4) ++(0,-0.4) to[out=0,in=180] ++(0.6,0.4) ++(0,-0.4) arc (-90:0:0.6 and 0.2); % right crossing
\end{scope}
\draw[->] (8.75,-2.9) to[bend left=45] node[midway,above,inner sep=3pt] {\small$z$} ++(1,-0.25);
\draw[->] (8.75,-5.65) to[bend right=45] node[midway,below,inner sep=3pt] {\small$x$} ++(1,0.25);

\end{tikzpicture}
\caption{The diagrams $U \cup \kappa$ for $\beta = x^{\pm1}$ and $\beta=yx^{\pm1}y^{-1}$.  Each arrow represents a $180^\circ$ rotation about the $z$-axis or the $x$-axis according to its label, where we view the page as the $xy$-plane.}
\label{fig:T23-four-diagrams}
\end{figure}
In Figure~\ref{fig:T23-four-diagrams} we further isotope the diagrams for each $U \cup \kappa$, starting from the simplifications in Figures~\ref{fig:T23-tangle-simplify-0} and \ref{fig:T23-tangle-simplify-1}, and we see that the corresponding links for $\beta=x$ and $\beta=yxy^{-1}$ are isotopic to each other in a way which carries $U$ to $U$ and $\kappa$ to $\kappa$, as are the links for $\beta=x^{-1}$ and $\beta=yxy^{-1}$.  It follows that
\[ K_{x\vphantom{y^a}} \cong K_{y^a x y^{-a}} \quad\text{and}\quad K_{x^{-1}} \cong K_{y^a x^{-1} y^{-a}} \]
for all $a\in\Z$, since it is true for $a=1$ and since for fixed $\epsilon$ the knot $K_{y^ax^\epsilon y^{-a}}$ depends only on the parity of $a$.  Thus every $K_\beta$ must be isotopic to either $K_x$ or $K_{x^{-1}}$ as claimed.
\end{proof}

\begin{proposition} \label{prop:recover-wh-plus}
We have $K_x \cong \Wh^+(T_{2,3},2).$
\end{proposition}

\begin{proof}
We take the link with components $U = \tau \cup \beta$ (where $\beta=x$) and $\kappa$ from the top row of Figure~\ref{fig:T23-four-diagrams}, and isotope it into a convenient position in the first half of Figure~\ref{fig:wh-plus-rht}.  Having done so, in the remainder of Figure~\ref{fig:wh-plus-rht} we then take the branched double cover with respect to the unknot $U$, lifting $\kappa$ to the knot $\tilde\kappa = K_x$ as we do so, and then isotope it further until it is recognizable as the 2-twisted, positively clasped Whitehead double of $T_{2,3}$.
\begin{figure}
\begin{tikzpicture}
\begin{scope}
\draw[link] (0,1.8) -- (-0.9,1.8); % undercrossing part of U, should be drawn first
\draw[linkred,looseness=2] (-1.6,2) -- ++(0,-0.2) to[out=270,in=270] ++(1.4,0) -- ++(0,0.2) to[out=90,in=0,looseness=1] ++(-1.4,1.05) -- ++(-0.4,0) coordinate (ktl); % top left of kappa, end at (-2,3.05)
\draw[linkred] (-1.2,2) -- ++(0,-0.2) to[out=270,in=270,looseness=2.5] ++(0.6,0) to[out=90,in=0,looseness=0.75] ++(-1,0.35) -- ++(-0.4,0); % end at (-2,2.15)

\draw[linkred,looseness=1] (0,2.8) -- ++(-0.4,0) to[out=180,in=90] (-1.6,2); % connect to right tangle
\draw[linkred,looseness=1] (0,2.4) -- ++(-0.4,0) to[out=180,in=90] (-1.2,2);

\draw[linkred,looseness=0.75] (0,2.8) -- ++(0.2,0) to[out=0,in=180] ++(0.6,-0.4) ++(0,0.4) arc (90:0:0.6 and 0.2); % right crossing

\draw[linkred] (ktl) ++(-0.5,-0.15) arc (180:270:0.3 and 0.15) arc (90:0:0.3 and 0.15) ++ (-0.6,-0.3) arc (180:270:0.5 and 0.15);
\draw[link,looseness=1.5] (0,1.8) -- ++(0.4,0) to[out=0,in=270] ++(0.6,0.8) to[out=90,in=0] ++(-0.6,0.8) -- ++(-0.4,0); % U, right
\draw[link,looseness=1.5] (-0.9,1.8) -- ++(-0.7,0) to[out=180,in=270] ++(-0.6,0.8) to[out=90,in=180] ++(0.6,0.8) -- ++(1.6,0); % U, left
\draw[linkred] (ktl) arc (90:180:0.5 and 0.15) ++ (0.6,-0.3) arc (0:-90:0.3 and 0.15) arc (90:180:0.3 and 0.15);

\draw[linkred,looseness=0.75] (0,2.4) -- ++(0.2,0) to[out=0,in=180] ++(0.6,0.4) ++(0,-0.4) arc (-90:0:0.6 and 0.2) node[right,red,inner sep=2pt] {\small$\kappa$}; % right crossing
\end{scope}

\begin{scope}[xshift=6cm]
\draw[link] (0,1.8) -- ++(-0.9,0); % undercrossing part of U, should be drawn first
\draw[linkred,looseness=2] (-1.6,2) -- ++(0,-0.2) to[out=270,in=270] ++(1.4,0) -- ++(0,1.4) to[out=90,in=90] ++(0.8,0) coordinate (newfold); % top left of kappa, used to end at (-2,3.05)
\draw[linkred] (-1.2,2) -- ++(0,-0.2) to[out=270,in=270,looseness=2.5] ++(0.6,0) to[out=90,in=90,looseness=2] ++(-1.4,0) coordinate (y1) to[out=270,in=270] ++(2.2,0);

\draw[linkred,looseness=1] (0,2.8) -- ++(-0.4,0) to[out=180,in=90] (-1.6,2); % connect to right tangle
\draw[linkred,looseness=1] (0,2.4) -- ++(-0.4,0) to[out=180,in=90] (-1.2,2);

\draw[linkred,looseness=0.75] (0,2.8) -- ++(1,0) to[out=0,in=180] ++(0.6,-0.4) ++(0,0.4) arc (90:0:0.6 and 0.2); % right crossing

\draw[link,looseness=1.5] (0,1.8) -- ++(1.2,0) to[out=0,in=270] ++(0.6,0.8) -- ++(0,0.6) to[out=90,in=0] ++(-0.6,0.8) -- ++(-2,0); % U, right

\draw[linkred,looseness=0.75] (0,2.4) -- ++(1,0) to[out=0,in=180] ++(0.6,0.4) ++(0,-0.4) arc (-90:0:0.6 and 0.2) node[right,red,inner sep=2pt] {\small$\kappa$}; % right crossing
\draw[linkred] (newfold) -- ++(0,-1.2);
\begin{scope}
\clip ($(y1)+(0.5,-0.5)$) rectangle ($(newfold)+(0,0.25)$);
\draw[linkred] (y1) to[out=270,in=270] ++(2.2,0) -- ++(0,1) to[out=90,in=0,looseness=1.25] ++(-0.4,0.4) -- ++(-0.7,0) coordinate (y3);
\end{scope}
\draw[linkred] (y3) to[out=180,in=90,looseness=1.25] ++(-1.5,-1.2) -- ++(0,-0.2) coordinate (y4);
\draw[linkred] (y4) to[out=270,in=270] ++(3,0) -- ++(0,0.2);

\draw[link,looseness=1.5] (-0.9,1.8) -- ++(-1.5,0) to[out=180,in=270] ++(-0.6,0.8) -- ++(0,0.6) to[out=90,in=180] ++(0.6,0.8) -- ++(2.4,0); % U, left
\end{scope}

\begin{scope}[yshift=-4.25cm]
\draw[link] (0.3,1.8) -- ++(-1.2,0); % undercrossing part of U, should be drawn first
\draw[linkred,looseness=2] (-1.6,2) -- ++(0,-0.2) to[out=270,in=270] ++(1.4,0) -- ++(0,1.4) to[out=90,in=90] ++(0.8,0) coordinate (newfold); % top left of kappa, used to end at (-2,3.05)
\draw[linkred] (-1.2,2) -- ++(0,-0.2) to[out=270,in=270,looseness=2.5] ++(0.6,0) to[out=90,in=90,looseness=2] ++(-1.4,0) coordinate (y1) to[out=270,in=270] ++(2.2,0);

\draw[linkred,looseness=1] (0,2.8) -- ++(-0.4,0) to[out=180,in=90] (-1.6,2); % connect to right tangle
\draw[linkred,looseness=1] (0,2.4) -- ++(-0.4,0) to[out=180,in=90] (-1.2,2);

\draw[linkred,looseness=0.75] (0,2.8) -- ++(0.6,0) to[out=0,in=90,looseness=1] ++(0.4,-0.4) -- ++(0,-0.4) coordinate (outerbottom); % right crossing

\draw[link,looseness=1.5] (0,1.8) -- ++(1,0) to[out=0,in=270] ++(0.6,0.8) -- ++(0,1) to[out=90,in=0] ++(-0.6,0.8) -- ++(-2.1,0); % U, right

\draw[linkred] (outerbottom) -- ++(0,-0.2) to[out=270,in=270] ++(-3.8,0) -- ++(0,0.2) to[out=90,in=180,looseness=1.25] ++(1.9,2) -- ++(0.4,0) to[out=0,in=90,looseness=1.25] ++(1.5,-0.8);
\draw[linkred,looseness=0.75] (0,2.4) -- ++(0.6,0) to[out=0,in=270,looseness=1] ++(0.4,0.4) -- ++(0,0.4) node[right,red,inner sep=2pt] {\small$\kappa$}; % right crossing
\draw[linkred] (newfold) -- ++(0,-1.2);
\begin{scope}
\clip ($(y1)+(0.5,-0.5)$) rectangle ($(newfold)+(0,0.25)$);
\draw[linkred] (y1) to[out=270,in=270] ++(2.2,0) -- ++(0,1) to[out=90,in=0,looseness=1.25] ++(-0.4,0.4) -- ++(-0.7,0) coordinate (y3);
\end{scope}
\draw[linkred] (y3) to[out=180,in=90,looseness=1.25] ++(-1.5,-1.2) -- ++(0,-0.2) coordinate (y4);
\draw[linkred] (y4) to[out=270,in=270] ++(3,0) -- ++(0,0.2);

\draw[link,looseness=1.5] (-0.9,1.8) -- ++(-1.8,0) to[out=180,in=270] ++(-0.6,0.8) -- ++(0,1) to[out=90,in=180] ++(0.6,0.8) -- ++(2.4,0); % U, left

\node at (2.75,2.3) {\huge$\xrightarrow{\dcover}$};
\end{scope}

\begin{scope}[yshift=-4.25cm,xshift=6.5cm]
\begin{scope}
\draw[linkred,looseness=2] (-1.6,2) ++(1.4,0) -- ++(0,1.2) to[out=90,in=90] ++(0.8,0) coordinate (newfold); % top left of kappa, used to end at (-2,3.05)
\draw[linkred] (-1.2,2) ++(0.6,0) to[out=90,in=90,looseness=2] ++(-1.4,0) coordinate (y1) ++(2.2,-0.2);
\draw[linkred,looseness=1] (0,2.8) -- ++(-0.4,0) to[out=180,in=90] (-1.6,2); % connect to right tangle
\draw[linkred,looseness=1] (0,2.4) -- ++(-0.4,0) to[out=180,in=90] (-1.2,2);
\draw[linkred,looseness=0.75] (0,2.8) -- ++(0.6,0) to[out=0,in=90,looseness=1] ++(0.4,-0.4) -- ++(0,-0.4) coordinate (outerbottom); % right crossing
\draw[linkred] (outerbottom) ++(-3.8,0) to[out=90,in=180,looseness=1.25] ++(1.9,2) -- ++(0.4,0) to[out=0,in=90,looseness=1.25] ++(1.5,-0.8);
\draw[linkred,looseness=0.75] (0,2.4) -- ++(0.6,0) to[out=0,in=270,looseness=1] ++(0.4,0.4) -- ++(0,0.4) node[right,red,inner sep=2pt] {\small$\tilde\kappa$}; % right crossing
\draw[linkred] (newfold) -- ++(0,-1.2);
\draw[linkred] (y1) ++(2.2,0) -- ++(0,0.8) to[out=90,in=0,looseness=1.25] ++(-0.4,0.4) -- ++(-0.7,0) coordinate (y3);
\draw[linkred] (y3) to[out=180,in=90,looseness=1.25] ++(-1.5,-1.2) coordinate (y4);
\draw[linkred] (y4) ++(3,0) -- ++(0,0.2);
\end{scope}
\begin{scope}[rotate around={180:(-0.9,2)}]
\draw[linkred,looseness=2] (-1.6,2) ++(1.4,0) -- ++(0,1.2) to[out=90,in=90] ++(0.8,0) coordinate (newfold); % top left of kappa, used to end at (-2,3.05)
\draw[linkred] (-1.2,2) ++(0.6,0) to[out=90,in=90,looseness=2] ++(-1.4,0) coordinate (y1) ++(2.2,-0.2);
\draw[linkred,looseness=1] (0,2.8) -- ++(-0.4,0) to[out=180,in=90] (-1.6,2); % connect to right tangle
\draw[linkred,looseness=1] (0,2.4) -- ++(-0.4,0) to[out=180,in=90] (-1.2,2);
\draw[linkred,looseness=0.75] (0,2.8) -- ++(0.6,0) to[out=0,in=90,looseness=1] ++(0.4,-0.4) -- ++(0,-0.4) coordinate (outerbottom);%++(0,0.4) arc (90:0:0.6 and 0.2); % right crossing
\draw[linkred] (outerbottom) ++(-3.8,0) to[out=90,in=180,looseness=1.25] ++(1.9,2) -- ++(0.4,0) to[out=0,in=90,looseness=1.25] ++(1.5,-0.8);
\draw[linkred,looseness=0.75] (0,2.4) -- ++(0.6,0) to[out=0,in=270,looseness=1] ++(0.4,0.4) -- ++(0,0.4); %to[out=0,in=180] ++(0.6,0.4) ++(0,-0.4) arc (-90:0:0.6 and 0.2); % right crossing
\draw[linkred] (newfold) -- ++(0,-1.2);
\draw[linkred] (y1) ++(2.2,0) -- ++(0,0.8) to[out=90,in=0,looseness=1.25] ++(-0.4,0.4) -- ++(-0.7,0) coordinate (y3);
\draw[linkred] (y3) to[out=180,in=90,looseness=1.25] ++(-1.5,-1.2) coordinate (y4);
\draw[linkred] (y4) ++(3,0) -- ++(0,0.2);
\end{scope}
\end{scope}

\begin{scope}[xshift=-1cm] % shift bottom row

\begin{scope}[yshift=-9cm]
\begin{scope}
\draw[linkred,looseness=2] (-1.6,2) ++(1.4,0) -- ++(0,1.2) coordinate (flipo) to[out=90,in=90,looseness=2.5] ++(-0.4,0) -- ++(0,-0.2); % top left of kappa, used to end at (-2,3.05)
\path (-1.2,2) ++(0.6,0) ++(-1.4,0) coordinate (y1) ++(2.2,-0.2);
\draw[linkred] (-0.6,2) -- ++(0,1) coordinate (flipy);
\draw[linkred,looseness=1] (0,2.8) -- ++(-0.4,0) to[out=180,in=90] (-1.6,2); % connect to right tangle
\draw[linkred,looseness=1] (0,2.4) -- ++(-0.4,0) to[out=180,in=90] (-1.2,2);
\path (0,2.8) -- ++(1,0) to[out=0,in=270,looseness=1] ++(0.4,0.4) coordinate (flipob) ++(-0.4,-1.2) coordinate (outerbottom); % right crossing
\draw[linkred,looseness=0.75] (0,2.8) -- ++(0.6,0)  node[above,red,inner sep=2pt] {\small$\tilde\kappa$} to[out=0,in=90,looseness=1] ++(0.4,-0.4) -- ++(0,-0.4); % right crossing
\draw[linkred,looseness=1.25] (0,2.4) -- ++(1,0) to[out=0,in=90,looseness=1] ++(0.4,-0.4) to[out=270,in=0] ++(-2.3,-2.4) -- ++(-0.4,0) to[out=180,in=270] ++(-1.9,1.2);
\draw[linkred] (y1) ++(2.2,0) -- ++(0,0.8) to[out=90,in=0,looseness=1.25] ++(-0.4,0.4) -- ++(-0.7,0) coordinate (y3);
\draw[linkred] (y3) to[out=180,in=90,looseness=1.25] ++(-1.5,-1.2) coordinate (y4);
\begin{scope}
\clip ($(flipo)+(-0.2,-0.2)$) rectangle ++(-0.5,0.5);
\draw[linkred] (flipo) to[out=90,in=90,looseness=2.5] ++(-0.4,0) -- ++(0,-0.2);
\end{scope}
\end{scope}
\begin{scope}[rotate around={180:(-0.9,2)}]
\draw[linkred,looseness=2] (-1.6,2) ++(1.4,0) -- ++(0,1.2) to[out=90,in=90] ++(0.8,0) coordinate (newfold); % top left of kappa, used to end at (-2,3.05)
\draw[linkred] (-1.2,2) ++(0.6,0) to[out=90,in=90,looseness=2] ++(-1.4,0) coordinate (y1) ++(2.2,-0.2);
\draw[linkred,looseness=1] (0,2.8) -- ++(-0.4,0) to[out=180,in=90] (-1.6,2); % connect to right tangle
\draw[linkred,looseness=1] (0,2.4) -- ++(-0.4,0) to[out=180,in=90] (-1.2,2);
\path (0,2.8) -- ++(0.6,0) to[out=0,in=90,looseness=1] ++(0.4,-0.4) -- ++(0,-0.4) coordinate (outerbottom); % right crossing
\draw[linkred] (outerbottom) ++(-3.8,0) to[out=90,in=180,looseness=1.25] ++(1.9,2) -- ++(0.4,0) to[out=0,in=90,looseness=1.25] ++(1.5,-0.8);
\draw[linkred,looseness=0.75] (0,2.8) -- ++(1,0) to[out=0,in=270,looseness=1] ++(0.4,0.4) -- ++(0,0.2); % right crossing
\draw[linkred,looseness=0.75] (0,2.4) -- ++(0.6,0) to[out=0,in=270,looseness=1] ++(0.4,0.4) -- ++(0,0.4); % right crossing
\draw[linkred] (newfold) -- ++(0,-1.2);
\end{scope}
\end{scope}

\begin{scope}[yshift=-9cm,xshift=5.25cm]
\begin{scope}
\draw[linkred,looseness=2] (-1.6,2) ++(1.4,0) -- ++(0,1.2) coordinate (flipo) to[out=90,in=90,looseness=2.5] ++(-0.4,0) -- ++(0,-0.2);
\path (-1.2,2) ++(0.6,0) ++(-1.4,0) coordinate (y1) ++(2.2,-0.2);
\draw[linkred] (-0.6,2) -- ++(0,1) coordinate (flipy);
\draw[linkred,looseness=1] (0,2.8) -- ++(-0.4,0) to[out=180,in=90] (-1.6,2); % connect to right tangle
\draw[linkred,looseness=1] (0,2.4) -- ++(-0.4,0) to[out=180,in=90] (-1.2,2);
\draw[linkred,looseness=0.75] (0,2.8) -- ++(0.6,0) node[above,red,inner sep=2pt] {\small$\tilde\kappa$} to[out=0,in=90,looseness=1] ++(0.4,-0.4) -- ++(0,-0.4); % right crossing
\draw[linkred,looseness=1.25] (0,2.4) -- ++(1,0) to[out=0,in=90,looseness=1] ++(0.4,-0.4) to[out=270,in=0] ++(-2.3,-2.4) -- ++(-0.4,0) to[out=180,in=270] ++(-1.9,1.2);
\path (y1) ++(2.2,0) -- ++(0,0.8) to[out=90,in=0,looseness=1.25] ++(-0.4,0.4) -- ++(-0.7,0) coordinate (y3);
\draw[linkred] (y3) -- ++(0.7,0) to[out=0,in=0] ++(0,0.5) -- ++(-0.7,0) to[out=180,in=90,looseness=1.25] ++(-1.5,-1.7) coordinate (fixo) ++(0,-1.1) to[out=270,in=270] ++(1.2,0);
\draw[linkred] (y3) to[out=180,in=90,looseness=1.25] ++(-1.1,-1.2);
\begin{scope}
\clip ($(flipo)+(-0.2,-0.2)$) rectangle ++(-0.5,0.5);
\draw[linkred] (flipo) to[out=90,in=90,looseness=2.5] ++(-0.4,0) -- ++(0,-0.2);
\end{scope}
\end{scope}
\begin{scope}[rotate around={180:(-0.9,2)}]
\draw[linkred,looseness=2] (-1.6,2) ++(1.4,0) -- ++(0,1.2) to[out=90,in=90] ++(0.4,0) coordinate (newfold); 
\draw[linkred] (-0.6,2) -- ++(0,1.1);
\draw[linkred,looseness=1] (0,2.8) -- ++(-0.4,0) to[out=180,in=90] (-1.6,2); % connect to right tangle
\draw[linkred,looseness=1] (0,2.4) -- ++(-0.4,0) to[out=180,in=90] (-1.2,2);
\path (0,2.8) -- ++(0.6,0) to[out=0,in=90,looseness=1] ++(0.4,-0.4) -- ++(0,-0.4) coordinate (outerbottom); % right crossing
\draw[linkred] (outerbottom) ++(-3.8,0) to[out=90,in=180,looseness=1.25] ++(1.9,2) -- ++(0.4,0) to[out=0,in=90,looseness=1.25] ++(1.5,-0.8);
\draw[linkred,looseness=0.75] (0,2.8) -- ++(1,0) to[out=0,in=270,looseness=1] ++(0.4,0.4) -- ++(0,0.2); % right crossing
\draw[linkred,looseness=0.75] (0,2.4) -- ++(0.6,0) to[out=0,in=270,looseness=1] ++(0.4,0.4) -- ++(0,0.4); % right crossing
\draw[linkred] (newfold) -- ++(0,-1.2);
\end{scope}
\draw[linkred] (fixo) -- ++(0,-1.1);
\end{scope}

\begin{scope}[yshift=-6.5cm,xshift=8.5cm] % nice (imho) drawing of Wh^+(T(2,3),2)
\draw[linkred] (0.4,0.15) arc (90:180:0.6 and 0.15) ++ (0.4,0) arc (360:270:0.6 and 0.15); % clasp
\draw[linkred] (-0.4,0.15) arc (90:0:0.6 and 0.15) ++ (-0.4,0) arc (180:270:0.6 and 0.15);

\draw[linkred] (-1.35,-1.5) to[out=60,in=150] (0,-0.9) to[out=330,in=240] (1.15,-1) to[out=60,in=0,looseness=1] (0.4,-0.15); % clasp top right to one of the middle strands
\draw[linkred] (-1.15,-1.5) to[out=60,in=150] (0,-1.1) to[out=330,in=240] (1.35,-1) to[out=60,in=0,looseness=1] (0.4,0.15);

\draw[linkred] (-0.4,0.15) to[out=180,in=120,looseness=1] (-1.35,-1) to[out=300,in=210] (0,-1.1) to[out=30,in=120] (1.15,-1.5); % clasp top left to the other middle strands
\draw[linkred] (-0.4,-0.15) to[out=180,in=120,looseness=1] (-1.15,-1) to[out=300,in=210] (0,-0.9) to[out=30,in=120] (1.35,-1.5);

\begin{scope} % fix middle crossing
\clip (-0.5,-1.5) rectangle (0.5,-0.5);
\draw[linkred] (-1.35,-1.5) to[out=60,in=150] (0,-0.9) to[out=330,in=240] (1.15,-1);
\draw[linkred] (-1.15,-1.5) to[out=60,in=150] (0,-1.1) to[out=330,in=240] (1.35,-1);
\end{scope}

% bottom strand, minus clasp
\draw[linkred,looseness=1] (-1.35,-1.5) to[out=240,in=180] (-0.6,-2.15) ++ (1.2,0) to[out=0,in=300] (1.35,-1.5);
\draw[linkred,looseness=1] (-1.15,-1.5) to[out=240,in=180] (-0.6,-1.85) ++ (1.2,0) to[out=0,in=300] (1.15,-1.5);

% clasp
\draw[linkred,looseness=0.75] (-0.6,-1.85) to[out=0,in=180] ++(0.6,-0.3) ++(0,0.3) to[out=0,in=180] ++(0.6,-0.3); % clasp
\draw[linkred,looseness=0.75] (-0.6,-2.15) to[out=0,in=180] ++(0.6,0.3) ++(0,-0.3) to[out=0,in=180] ++(0.6,0.3);
\node[below,red,inner sep=3pt] at (0,-2.15) {\small$\tilde\kappa$};
\end{scope}

\end{scope}
\end{tikzpicture}
\caption{A proof that $K_x \cong \Wh^+(T_{2,3},2)$, beginning with the link $U \cup \kappa$ from the top row of Figure~\ref{fig:T23-four-diagrams} and ending with the lift $K_x = \tilde\kappa$ of $\kappa$ to $\dcover(U) \cong S^3$.}
\label{fig:wh-plus-rht}
\end{figure}
\end{proof}

\begin{proposition} \label{prop:recover-wh-minus}
We have $K_{x^{-1}} \cong \Wh^-(T_{2,3},2).$
\end{proposition}

\begin{proof}
Just as in Proposition~\ref{prop:recover-wh-plus}, we take the link with components $U = \tau \cup \beta$ and $\kappa$, this time with $\beta=x^{-1}$, as pictured in the third row of Figure~\ref{fig:T23-four-diagrams}.  In Figure~\ref{fig:wh-minus-rht} we carry out an isotopy, take the branched double cover with respect to the unknot $U$, and then lift $\kappa$ to the knot $\tilde\kappa = K_{x^{-1}}$, which we recognize after further isotopy as the 2-twisted, negatively clasped Whitehead double of $T_{2,3}$.
\begin{figure}
\begin{tikzpicture}
\begin{scope}
\draw[link] (0,1.8) -- (-0.9,1.8); % undercrossing part of U, should be drawn first
\draw[linkred,looseness=2] (-1.6,2) -- ++(0,-0.2) to[out=270,in=270] ++(1.4,0) -- ++(0,0.2) to[out=90,in=0,looseness=1] ++(-1.4,1.05) -- ++(-0.4,0) coordinate (ktl); % top left of kappa, end at (-2,3.05)
\draw[linkred] (-1.2,2) -- ++(0,-0.2) to[out=270,in=270,looseness=2.5] ++(0.6,0) to[out=90,in=0,looseness=0.75] ++(-1,0.35) -- ++(-0.4,0); % end at (-2,2.15)

\draw[linkred,looseness=1] (0,2.8) -- ++(-0.4,0) to[out=180,in=90] (-1.6,2); % connect to right tangle
\draw[linkred,looseness=1] (0,2.4) -- ++(-0.4,0) to[out=180,in=90] (-1.2,2);

\draw[linkred,looseness=0.75] (0,2.8) -- ++(0.2,0) to[out=0,in=180] ++(0.6,-0.4) ++(0,0.4) to[out=0,in=180] ++(0.6,-0.4) ++(0,0.4) arc (90:0:0.6 and 0.2) node[right,red,inner sep=2pt] {\small$\kappa$}; % right crossings

\draw[linkred] (ktl) arc (90:180:0.5 and 0.15) ++ (0.6,-0.3) arc (0:-90:0.3 and 0.15) arc (90:180:0.3 and 0.15);
\draw[link,looseness=1.5] (0,1.8) -- ++(1,0) to[out=0,in=270] ++(0.6,0.8) to[out=90,in=0] ++(-0.6,0.8) -- ++(-1,0); % U, right
\draw[link,looseness=1.5] (-0.9,1.8) -- ++(-0.7,0) to[out=180,in=270] ++(-0.6,0.8) to[out=90,in=180] ++(0.6,0.8) -- ++(1.6,0); % U, left
\draw[linkred] (ktl) ++(-0.5,-0.15) arc (180:270:0.3 and 0.15) arc (90:0:0.3 and 0.15) ++ (-0.6,-0.3) arc (180:270:0.5 and 0.15);

\draw[linkred,looseness=0.75] (0,2.4) -- ++(0.2,0) to[out=0,in=180] ++(0.6,0.4) ++(0,-0.4) to[out=0,in=180] ++(0.6,0.4) ++(0,-0.4) arc (-90:0:0.6 and 0.2); % right crossings
\end{scope}

\begin{scope}[xshift=7cm]
\draw[link] (-0.9,1.8) -- ++(2,0); % undercrossing part of U, should be drawn first
\draw[linkred,looseness=2] (-1.6,2) -- ++(0,-0.2) to[out=270,in=270] ++(1.4,0) -- ++(0,1) to[out=90,in=0,looseness=1] ++(-0.4,0.4) -- ++(-1,0) coordinate (ktl); % top left of kappa, end at (-2,3.05)
\draw[linkred] (-1.2,2) -- ++(0,-0.2) to[out=270,in=270,looseness=2.5] ++(0.6,0) to[out=90,in=270] ++(0,1.2) to[out=90,in=90] ++(0.8,0) coordinate (yloop);

\draw[linkred,looseness=1] (0,2.8) -- ++(-0.4,0) to[out=180,in=90] (-1.6,2); % connect to right tangle
\draw[linkred,looseness=1] (0,2.4) -- ++(-0.4,0) to[out=180,in=90] (-1.2,2);

\draw[linkred,looseness=0.75] (0,2.8) -- ++(0.8,0) to[out=0,in=180] ++(0.6,-0.4) ++(0,0.4) to[out=0,in=180] ++(0.6,-0.4) ++(0,0.4) arc (90:0:0.6 and 0.2) node[right,red,inner sep=2pt] {\small$\kappa$}; % right crossings

\draw[link,looseness=1.5] (1.1,1.8) -- ++(0.5,0) to[out=0,in=270] ++(0.6,0.8) -- ++(0,0.8) to[out=90,in=0] ++(-0.6,0.8) -- ++(-1.6,0); % U, right
\draw[linkred,looseness=0.75] (0,2.4) -- ++(0.8,0) to[out=0,in=180] ++(0.6,0.4) ++(0,-0.4) to[out=0,in=180] ++(0.6,0.4) ++(0,-0.4) arc (-90:0:0.6 and 0.2); % right crossings
\draw[linkred] (ktl) to[out=180,in=90,looseness=1] ++(-0.8,-0.8) -- ++(0,-0.6) to[out=270,in=270] ++(3,0) -- ++(0,1.2) to[out=90,in=0,looseness=1] ++(-0.8,0.8);
\draw[linkred] (yloop) -- ++(0,-1.2) to[out=270,in=270] ++(-2.2,0) -- ++(0,0.6) to[out=90,in=180,looseness=1] ++(1.8,1.4);
\draw[link,looseness=1.5] (-0.9,1.8) -- ++(-1.5,0) to[out=180,in=270] ++(-0.6,0.8) -- ++(0,0.8) to[out=90,in=180] ++(0.6,0.8) -- ++(2.4,0); % U, left
\end{scope}

\begin{scope}[yshift=-4.25cm]
\draw[link,looseness=1.5] (-0.9,1.8) -- ++(2.5,0) to[out=0,in=270] ++(0.6,0.8) -- ++(0,1.2) to[out=90,in=0] ++(-0.6,0.8) -- ++(-1.6,0); % U, right

\draw[linkred,looseness=2] (-1.6,2) -- ++(0,-0.2) to[out=270,in=270] ++(1.4,0) -- ++(0,1) to[out=90,in=0,looseness=1] ++(-0.4,0.4) -- ++(-1,0) coordinate (ktl); % top left of kappa, end at (-2,3.05)
\draw[linkred] (-1.2,2) -- ++(0,-0.2) to[out=270,in=270,looseness=2.5] ++(0.6,0) to[out=90,in=270] ++(0,1.2) to[out=90,in=90] ++(0.8,0) coordinate (yloop);

\draw[linkred,looseness=1] (0,2.8) -- ++(-0.4,0) to[out=180,in=90] (-1.6,2); % connect to right tangle
\draw[linkred,looseness=1] (0,2.4) -- ++(-0.4,0) to[out=180,in=90] (-1.2,2);

\draw[linkred,looseness=0.75] (0,2.8) -- ++(0.8,0) to[out=0,in=180] ++(0.6,-0.3) ++(0,0.3) to[out=0,in=90] ++(0.4,-0.3) to[out=270,in=90,looseness=0.5] ++(-0.8,-0.5) coordinate (outerloop) -- ++(0,-0.1);
\draw[linkred] (outerloop) -- ++(0,-0.2) to[out=270,in=270,looseness=1.75] ++(-3.8,0) -- ++(0,1) to[out=90,in=180,looseness=1] ++(1.4,1.4) -- ++(1.2,0); % right crossings

\draw[linkred,looseness=0.75] (0,2.4) -- ++(0.8,0) to[out=0,in=180] ++(0.6,0.4) ++(0,-0.3) to[out=0,in=270] ++(0.4,0.3) node[right,red,inner sep=2pt] {\small$\kappa$} to[out=90,in=0,looseness=1] ++(-1.4,1.4) -- ++(-0.6,0); % right crossings
\draw[linkred] (ktl) to[out=180,in=90,looseness=1] ++(-0.8,-0.8) -- ++(0,-0.6) to[out=270,in=270] ++(3,0) -- ++(0,1.2) to[out=90,in=0,looseness=1] ++(-0.8,0.8);
\draw[linkred] (yloop) -- ++(0,-1.2) to[out=270,in=270] ++(-2.2,0) -- ++(0,0.6) to[out=90,in=180,looseness=1] ++(1.8,1.4);

\draw[link,looseness=1.5] (-0.9,1.8) -- ++(-1.9,0) to[out=180,in=270] ++(-0.6,0.8) -- ++(0,1.2) to[out=90,in=180] ++(0.6,0.8) -- ++(2.8,0); % U, left
\node at (2.9,2.3) {\huge$\xrightarrow{\dcover}$};
\end{scope}

\begin{scope}[yshift=-4.25cm,xshift=7cm]
\begin{scope}
\draw[linkred,looseness=2] (-1.6,2) ++ (1.4,0) -- ++(0,0.8) to[out=90,in=0,looseness=1] ++(-0.4,0.4) -- ++(-1,0) coordinate (ktl); % top left of kappa, end at (-2,3.05)
\draw[linkred] (-1.2,2) ++ (0.6,0) -- ++(0,1) to[out=90,in=90] ++(0.8,0) coordinate (yloop);
\draw[linkred,looseness=1] (0,2.8) -- ++(-0.4,0) to[out=180,in=90] (-1.6,2); % connect to right tangle
\draw[linkred,looseness=1] (0,2.4) -- ++(-0.4,0) to[out=180,in=90] (-1.2,2);
\draw[linkred,looseness=0.75] (0,2.8) -- ++(0.8,0) to[out=0,in=180] ++(0.6,-0.3) ++(0,0.3) to[out=0,in=90] ++(0.4,-0.3) to[out=270,in=90,looseness=0.5] ++(-0.8,-0.5) coordinate (outerloop);
\draw[linkred] (outerloop) ++(-3.8,0) -- ++(0,0.8) to[out=90,in=180,looseness=1] ++(1.4,1.4) -- ++(1.2,0); % right crossings
\draw[linkred,looseness=0.75] (0,2.4) -- ++(0.8,0) to[out=0,in=180] ++(0.6,0.4) ++(0,-0.3) to[out=0,in=270] ++(0.4,0.3) node[right,red,inner sep=2pt] {\small$\tilde\kappa$} to[out=90,in=0,looseness=1] ++(-1.4,1.4) -- ++(-0.6,0); % right crossings
\draw[linkred] (ktl) to[out=180,in=90,looseness=1] ++(-0.8,-0.8) -- ++(0,-0.4) ++(3,0) -- ++(0,1) to[out=90,in=0,looseness=1] ++(-0.8,0.8);
\draw[linkred] (yloop) -- ++(0,-1) ++(-2.2,0) -- ++(0,0.4) to[out=90,in=180,looseness=1] ++(1.8,1.4);
\end{scope}
\begin{scope}[rotate around={180:(-0.9,2)}]
\draw[linkred,looseness=2] (-1.6,2) ++ (1.4,0) -- ++(0,0.8) to[out=90,in=0,looseness=1] ++(-0.4,0.4) -- ++(-1,0) coordinate (ktl); % top left of kappa, end at (-2,3.05)
\draw[linkred] (-1.2,2) ++ (0.6,0) -- ++(0,1) to[out=90,in=90] ++(0.8,0) coordinate (yloop);
\draw[linkred,looseness=1] (0,2.8) -- ++(-0.4,0) to[out=180,in=90] (-1.6,2); % connect to right tangle
\draw[linkred,looseness=1] (0,2.4) -- ++(-0.4,0) to[out=180,in=90] (-1.2,2);
\draw[linkred,looseness=0.75] (0,2.8) -- ++(0.8,0) to[out=0,in=180] ++(0.6,-0.3) ++(0,0.3) to[out=0,in=90] ++(0.4,-0.3) to[out=270,in=90,looseness=0.5] ++(-0.8,-0.5) coordinate (outerloop);
\draw[linkred] (outerloop) ++(-3.8,0) -- ++(0,0.8) to[out=90,in=180,looseness=1] ++(1.4,1.4) -- ++(1.2,0); % right crossings
\draw[linkred,looseness=0.75] (0,2.4) -- ++(0.8,0) to[out=0,in=180] ++(0.6,0.4) ++(0,-0.3) to[out=0,in=270] ++(0.4,0.3) to[out=90,in=0,looseness=1] ++(-1.4,1.4) -- ++(-0.6,0); % right crossings
\draw[linkred] (ktl) to[out=180,in=90,looseness=1] ++(-0.8,-0.8) -- ++(0,-0.4) ++(3,0) -- ++(0,1) to[out=90,in=0,looseness=1] ++(-0.8,0.8);
\draw[linkred] (yloop) -- ++(0,-1) ++(-2.2,0) -- ++(0,0.4) to[out=90,in=180,looseness=1] ++(1.8,1.4);
\end{scope}
\end{scope}

\begin{scope}[xshift=-1cm] % shift bottom row

\begin{scope}[yshift=-9.5cm]
\begin{scope}
\path (-0.2,2) -- ++(0,0.8) to[out=90,in=0,looseness=1] ++(-0.4,0.4) -- ++(-1,0) coordinate (ktl); % top left of kappa, end at (-2,3.05)
\draw[linkred,looseness=2] (-0.2,2) -- ++(0,1.6) arc (0:90:0.2 and 0.6);
\draw[linkred] (-1.2,2) ++ (0.6,0) -- ++(0,1) ++(0.8,0) coordinate (yloop);
\draw[linkred,looseness=1] (0,2.8) -- ++(-0.4,0) to[out=180,in=90] (-1.6,2); % connect to right tangle
\draw[linkred,looseness=1] (0,2.4) -- ++(-0.4,0) to[out=180,in=90] (-1.2,2);
\draw[linkred,looseness=0.75] (0,2.8) -- ++(0.8,0) to[out=0,in=180] ++(0.6,-0.3) ++(0,0.3) to[out=0,in=90] ++(0.4,-0.4) to[out=270,in=90,looseness=0.5] ++(-0.4,-0.4) ++(-0.4,0) coordinate (outerloop);
\path[linkred] (outerloop) ++(-3.8,0) -- ++(0,0.8) to[out=90,in=180,looseness=1] ++(1.4,1.4) -- ++(1.2,0); % right crossings
\draw[linkred,looseness=0.75] (0,2.4) -- ++(0.8,0) to[out=0,in=180] ++(0.6,0.4) ++(0,-0.3) to[out=0,in=90] ++(0.4,-0.6) node[right,red,inner sep=2pt] {\small$\tilde\kappa$} to[out=270,in=0,looseness=1.25] ++(-1.8,-2.5) -- ++(-1.8,0) to[out=180,in=270,looseness=1] ++(-1.8,0.8) -- ++(0,0.2); % right crossings
\draw[linkred] (ktl) ++(-0.8,-0.8) ++(0,-0.4) ++(3,0) -- ++(0,1) to[out=90,in=0,looseness=1] ++(-0.8,0.8);
\draw[linkred] (yloop) ++(0,-1) ++(-2.2,0) -- ++(0,0.4) to[out=90,in=180,looseness=1] ++(1.8,1.4);
\draw[linkred] (-1.2,2) ++ (0.6,0) ++(0,1) -- ++(0,0.6) arc (180:90:0.2 and 0.6);
\end{scope}
\begin{scope}[rotate around={180:(-0.9,2)}]
\draw[linkred,looseness=2] (-1.6,2) ++ (1.4,0) -- ++(0,0.8) to[out=90,in=0,looseness=1] ++(-0.4,0.4) -- ++(-1,0) coordinate (ktl); % top left of kappa, end at (-2,3.05)
\draw[linkred] (-1.2,2) ++ (0.6,0) -- ++(0,1) to[out=90,in=90] ++(0.8,0) coordinate (yloop);
\draw[linkred,looseness=1] (0,2.8) -- ++(-0.4,0) to[out=180,in=90] (-1.6,2); % connect to right tangle
\draw[linkred,looseness=1] (0,2.4) -- ++(-0.4,0) to[out=180,in=90] (-1.2,2);
\draw[linkred,looseness=0.75] (0,2.8) -- ++(0.8,0) to[out=0,in=180] ++(0.6,-0.3) ++(0,0.3) to[out=0,in=270] ++(0.4,0.8) ++(-0.8,-1.6) coordinate (outerloop);
\draw[linkred] (outerloop) ++(-4.2,0) -- ++(0,0.8) to[out=90,in=180,looseness=1] ++(1.4,1.4) -- ++(1.6,0); % right crossings
\draw[linkred,looseness=0.75] (0,2.4) -- ++(0.8,0) to[out=0,in=180] ++(0.6,0.4) ++(0,-0.3) to[out=0,in=270] ++(0.4,0.3) to[out=90,in=0,looseness=1] ++(-1.4,1.4) -- ++(-0.6,0); % right crossings
\draw[linkred] (ktl) to[out=180,in=90,looseness=1] ++(-0.8,-0.8) -- ++(0,-0.4);
\draw[linkred] (yloop) -- ++(0,-1);
\end{scope}
\end{scope}

\begin{scope}[yshift=-9.5cm,xshift=6cm]
\begin{scope}
\path (-0.2,2) -- ++(0,0.8) to[out=90,in=0,looseness=1] ++(-0.4,0.4) -- ++(-1,0) coordinate (ktl); % top left of kappa, end at (-2,3.05)
\draw[linkred,looseness=2] (-0.2,2) -- ++(0,1.4) arc (0:90:0.2 and 0.6);
\draw[linkred] (-1.2,2) ++ (0.6,0) -- ++(0,1) ++(0.8,0) coordinate (yloop);
\draw[linkred,looseness=1] (0,2.8) -- ++(-0.4,0) to[out=180,in=90] (-1.6,2); % connect to right tangle
\draw[linkred,looseness=1] (0,2.4) -- ++(-0.4,0) to[out=180,in=90] (-1.2,2);
\draw[linkred,looseness=0.75] (0,2.8) -- ++(0.8,0) to[out=0,in=180] ++(0.6,-0.3) ++(0,0.3) to[out=0,in=90] ++(0.4,-0.4) to[out=270,in=90,looseness=0.5] ++(-0.4,-0.4) ++(-0.4,0) coordinate (outerloop);
\path[linkred] (outerloop) ++(-3.8,0) -- ++(0,0.8) to[out=90,in=180,looseness=1] ++(1.4,1.4) -- ++(1.2,0); % right crossings
\draw[linkred,looseness=0.75] (0,2.4) -- ++(0.8,0) to[out=0,in=180] ++(0.6,0.4) ++(0,-0.3) to[out=0,in=90] ++(0.4,-0.6) node[right,red,inner sep=2pt] {\small$\tilde\kappa$} to[out=270,in=0,looseness=1.25] ++(-1.8,-2.5) -- ++(-1.8,0) to[out=180,in=270,looseness=1] ++(-1.8,0.8) -- ++(0,0.2); % right crossings
\draw[linkred] (yloop) ++(0,-1) ++(-2.2,0) -- ++(0,0.4) to[out=90,in=180,looseness=1] ++(1.6,1.4) arc (-90:90:0.6 and 0.2) to[out=180,in=90,looseness=1.25] ++(-2,-1.8);
\draw[linkred] (-1.2,2) ++ (0.6,0) ++(0,1) -- ++(0,0.4) arc (180:90:0.2 and 0.6);
\end{scope}
\begin{scope}[rotate around={180:(-0.9,2)}]
\draw[linkred] (-0.2,2) -- ++(0,1) to[out=90,in=90] ++(0.8,0) coordinate (ploop);
\draw[linkred] (-1.2,2) ++ (0.6,0) -- ++(0,1) to[out=90,in=90] ++(0.8,0) coordinate (yloop);
\draw[linkred,looseness=1] (0,2.8) -- ++(-0.4,0) to[out=180,in=90] (-1.6,2); % connect to right tangle
\draw[linkred,looseness=1] (0,2.4) -- ++(-0.4,0) to[out=180,in=90] (-1.2,2);
\draw[linkred,looseness=0.75] (0,2.8) -- ++(0.8,0) to[out=0,in=180] ++(0.6,-0.3) ++(0,0.3) to[out=0,in=270] ++(0.4,0.8) ++(-0.8,-1.6) coordinate (outerloop);
\draw[linkred] (outerloop) ++(-4.2,0) -- ++(0,0.8) to[out=90,in=180,looseness=1] ++(1.4,1.4) -- ++(1.6,0); % right crossings
\draw[linkred,looseness=0.75] (0,2.4) -- ++(0.8,0) to[out=0,in=180] ++(0.6,0.4) ++(0,-0.3) to[out=0,in=270] ++(0.4,0.3) to[out=90,in=0,looseness=1] ++(-1.4,1.4) -- ++(-0.6,0); % right crossings
\draw[linkred] (ploop) -- ++(0,-1.4);
\draw[linkred] (yloop) -- ++(0,-1);
\end{scope}
\end{scope}

\begin{scope}[yshift=-6.5cm,xshift=10cm] % nice (imho) drawing of Wh^-(T(2,3),2)
\draw[linkred] (-0.4,0.15) arc (90:0:0.6 and 0.15) ++ (-0.4,0) arc (180:270:0.6 and 0.15); % clasp
\draw[linkred] (0.4,0.15) arc (90:180:0.6 and 0.15) ++ (0.4,0) arc (360:270:0.6 and 0.15);

\draw[linkred] (-1.35,-1.5) to[out=60,in=150] (0,-0.9) to[out=330,in=240] (1.15,-1) to[out=60,in=0,looseness=1] (0.4,-0.15); % clasp top right to one of the middle strands
\draw[linkred] (-1.15,-1.5) to[out=60,in=150] (0,-1.1) to[out=330,in=240] (1.35,-1) to[out=60,in=0,looseness=1] (0.4,0.15);

\draw[linkred] (-0.4,0.15) to[out=180,in=120,looseness=1] (-1.35,-1) to[out=300,in=210] (0,-1.1) to[out=30,in=120] (1.15,-1.5); % clasp top left to the other middle strands
\draw[linkred] (-0.4,-0.15) to[out=180,in=120,looseness=1] (-1.15,-1) to[out=300,in=210] (0,-0.9) to[out=30,in=120] (1.35,-1.5);

\begin{scope} % fix middle crossing
\clip (-0.5,-1.5) rectangle (0.5,-0.5);
\draw[linkred] (-1.35,-1.5) to[out=60,in=150] (0,-0.9) to[out=330,in=240] (1.15,-1);
\draw[linkred] (-1.15,-1.5) to[out=60,in=150] (0,-1.1) to[out=330,in=240] (1.35,-1);
\end{scope}

% bottom strand, minus clasp
\draw[linkred,looseness=1] (-1.35,-1.5) to[out=240,in=180] (-0.6,-2.15) ++ (1.2,0) to[out=0,in=300] (1.35,-1.5);
\draw[linkred,looseness=1] (-1.15,-1.5) to[out=240,in=180] (-0.6,-1.85) ++ (1.2,0) to[out=0,in=300] (1.15,-1.5);

% clasp
\draw[linkred,looseness=0.75] (-0.6,-1.85) to[out=0,in=180] ++(0.6,-0.3) ++(0,0.3) to[out=0,in=180] ++(0.6,-0.3); % clasp
\draw[linkred,looseness=0.75] (-0.6,-2.15) to[out=0,in=180] ++(0.6,0.3) ++(0,-0.3) to[out=0,in=180] ++(0.6,0.3);
\node[below,red,inner sep=3pt] at (0,-2.15) {\small$\tilde\kappa$};
\end{scope}

\end{scope}
\end{tikzpicture}
\caption{A proof that $K_{x^{-1}} \cong \Wh^-(T_{2,3},2)$, beginning with the link $U \cup \kappa$ from the third row of Figure~\ref{fig:T23-four-diagrams} and ending with the lift $K_{x^{-1}} = \tilde\kappa$ of $\kappa$ to $\dcover(U) \cong S^3$.}
\label{fig:wh-minus-rht}
\end{figure}
\end{proof}

We can now finish the proof of Theorem~\ref{thm:M_F-C24-T23}, and then conclude Theorem~\ref{thm:main-hfk}.

\begin{proof}[Proof of Theorem~\ref{thm:M_F-C24-T23}]
We apply Lemma~\ref{lem:C24-T23-as-branched-cover}, according to which $K$ is the lift of $\kappa$ in the branched double cover of the unknot $U = \tau \cup \beta$.  Although there are infinitely many such $\beta$ (see Proposition~\ref{prop:T23-braids}), Lemma~\ref{lem:T23-only-2} says that in fact $K$ must arise from this construction for either $\beta = x$ or $\beta = x^{-1}$.  In the case $\beta=x$, Proposition~\ref{prop:recover-wh-plus} says that $K \cong \Wh^+(T_{2,3},2)$, and if instead we have $\beta=x^{-1}$ then $K \cong \Wh^-(T_{2,3},2)$ by Proposition~\ref{prop:recover-wh-minus}.  This completes the proof.
\end{proof}

\begin{proof}[Proof of Theorem~\ref{thm:main-hfk}]
Letting $F$ be a genus-1 Seifert surface for $K$, we proved in Theorem~\ref{thm:identify-y-c} that up to replacing $K$ with its mirror, the manifold $M_F$ must be the complement of the $(2,4)$-cable of either the unknot or the right-handed trefoil.  In the unknot case, Theorem~\ref{thm:M_F-E24} says that $K$ is one of
\[ 5_2, \, 15n_{43522}, \, \text{or} \, P(-3,3,2n+1) \]
for some $n\in\Z$.  Likewise, in the trefoil case, Theorem~\ref{thm:M_F-C24-T23} tells us that $K$ is either
\[ \Wh^+(T_{2,3},2) \,\text{or}\, \Wh^-(T_{2,3},2). \]
Thus either $K$ or its mirror must be one of the knots listed above.
\end{proof}

\section{Detection results for Khovanov homology} \label{sec:kh-to-hfk}
Our goal in this section is to prove the detection results for reduced Khovanov homology stated in Theorems \ref{thm:main-kh} and \ref{thm:main-kh-pretzel}.  We will do so after establishing some preliminary results.  We continue to work with coefficients in $\Q$ throughout this section.
 
Recall that both   reduced Khovanov homology and knot Floer homology admit bigradings, which can be collapsed to a single $\delta$-grading, defined for these two theories by \begin{align*}
\gr_\delta &= \tfrac{1}{2}\gr_q - \gr_h,\\
\gr_\delta &= \gr_m - \gr_a,
\end{align*} respectively.  We say that either invariant is \emph{thin} if it is supported in a unique $\delta$-grading. Given a knot $K\subset S^3$, Dowlin's spectral sequence \cite{dowlin} \[\Khr(K) \implies \hfkhat(\mirror{K})\] from reduced Khovanov homology to knot Floer homology respects the $\delta$-gradings on either side, up to an overall shift.  This implies the following:

\begin{lemma}\label{lem:dim-equal}
Let $K\subset S^3$ be a knot for which $\Khr(K)$ is thin. Then $\hfkhat(K)$ is thin and \[ \dim\hfkhat(K) = \dim\Khr(K) = \det(K).\] \end{lemma}

\begin{proof}
Suppose that $\Khr(K)$ is thin. Then the fact that Dowlin's spectral sequence respects the $\delta$-grading up to an overall shift, together with the symmetry \cite{osz-knot}\[\hfkhat_m(K,a) \cong \hfkhat_{-m}(\mirror{K},-a),\]  implies that $\hfkhat(K)$ is also thin. Recall that the graded Euler characteristics of reduced Khovanov homology and knot Floer homology recover the Jones and Alexander polynomials, respectively \cite{khovanov-patterns, osz-knot}:
\begin{align}
\label{eqn:jones}V_K(t) &= \sum_{h,q} (-1)^h t^{q/2} \dim \Khr^{h,q}(K),\\
\label{eqn:alex}\Delta_K(t) &= \sum_{m,a} (-1)^m t^a \dim \hfkhat_m(K,a).
\end{align}
Supposing that $\Khr(K)$ and $\hfkhat(K)$ are supported in $\delta$-gradings $\delta_1$ and $\delta_2$, respectively, it follows that 
\begin{align*}
V_K(-1) &= (-1)^{\delta_1} \dim \Khr(K),\\
\Delta_K(-1)&=(-1)^{\delta_2}\dim\hfkhat(K),
\end{align*}
and thus \[\dim\hfkhat(K) = |\Delta_K(-1)| = \det(K) =|V_K(-1)| =\dim\Khr(K),\] as claimed.
\end{proof}

The next result pertains to the geography of knot Floer homology. For this result, recall that for any knot $K\subset S^3$, there are two differentials on knot Floer homology, 
 \begin{align*}
 \xi &=  \xi^1 +  \xi^2 + \dots + \\
 \omega &=  \omega^1 +  \omega^2 + \dots +,
 \end{align*}  where $\xi^i$ and $ \omega^i$ are, respectively,   sums of maps of the form
 \begin{align}
 \label{eqn:grading-shift-1}\xi^i_a& :\hfkhat_m(K,a)\to\hfkhat_{m-1}(K,a-i)\\
  \label{eqn:grading-shift-2}\omega^i_a& :\hfkhat_m(K,a)\to\hfkhat_{m-1}(K,a+i).
 \end{align}Indeed, given a doubly-pointed Heegaard diagram  for the knot $K\subset S^3$, \[(\Sigma,\alpha,\beta,z,w),\] the differential $\partial$ in the  Heegaard Floer complex \[\cfhat(S^3)=\cfhat(\Sigma,\alpha,\beta,w)\]  is a sum $\partial  = d_0 + d_1,$ where $d_0$ counts those  disks that avoid the basepoint $z$, and $d_1$ counts the rest. Then \[\hfkhat(K) \cong H_*(\cfhat(\Sigma,\alpha,\beta,w), d_0),\] and $\xi$ is the differential on this homology induced by $d_1$. The map $\omega$ is defined in the same way but with the roles of $z$ and $w$ swapped.
It follows from the definition that the homology with respect to either differential recovers the  Heegaard Floer homology of $S^3$, \begin{equation}\label{eqn:hom}H_*(\hfkhat(K),\xi)\cong H_*(\hfkhat(K),\omega)\cong\Q.\end{equation} Furthermore, the components $\xi^1$ and $\omega^1$ anticommute. (This follows from Ozsv{\'a}th--Szab{\'o}'s original construction of  $\mathit{CFK}^\infty(K)$ in \cite{osz-knot}; it is also stated explicitly in \cite[Equation~(3.7)]{bls} where our $\xi^1$ and $\omega^1$ correspond to their $\Psi^p$ and $\Omega^p$.)

When $\hfkhat(K)$ is thin, we have that $\xi = \xi^1$ and $\omega = \omega^1$ according to the grading shifts in \eqref{eqn:grading-shift-1} and \eqref{eqn:grading-shift-2}. In particular, \[\xi \omega = -\omega\xi.\] Moreover, in this case, the two homology groups in \eqref{eqn:hom} are supported in Alexander gradings $\tau(K)$ and $-\tau(K)$, respectively, where $\tau(K)$ is the Ozsv{\'a}th--Szab{\'o} tau invariant \cite{osz-tau}. With this background in place, we may now prove the following:
 
\begin{lemma}\label{lem:thin-fibered}
Let $K\subset S^3$ be a  knot of genus $g\geq 1$ for which $\hfkhat(K)$ is thin. Then \[\dim\hfkhat(K,g) \leq \dim\hfkhat(K,g-1).\] If in addition $K$ is fibered with $|\tau(K)|<g$, then this is a strict inequality.
\end{lemma}

\begin{proof}
Suppose that $g\geq 1$ and  $\hfkhat(K)$ is thin. Then $\xi = \xi^1$ and $\omega = \omega^1$ and $\xi\omega = -\omega\xi$. If \[\dim\hfkhat(K,g) > \dim\hfkhat(K,g-1),\] then we have also that \[\dim\hfkhat(K,-g) >\dim\hfkhat(K,1-g),\] by conjugation symmetry. The complex $(\hfkhat(K),\xi)$, given by \[\hfkhat(K,g) \xrightarrow{\xi_g} \hfkhat(K,g-1) \xrightarrow{\xi_{g-1}} \dots  \xrightarrow{\xi_{2-g}}\hfkhat(K,1-g) \xrightarrow{\xi_{1-g}} \hfkhat(K,-g),\] then has nontrivial homology in both of the Alexander gradings $g$ and $-g$, meaning that \[ \dim H_*(\hfkhat(K),\xi) \geq 2,\] a contradiction. This proves the first claim.

Now suppose that $K$ is also fibered, and assume for a contradiction that $|\tau(K)|<g$ but  \[\dim\hfkhat(K,g) = \dim\hfkhat(K,g-1) =1.\] The fact that $\tau(K)\neq \pm g$ implies that  the complexes $(\hfkhat(K),\xi)$ and $(\hfkhat(K),\omega)$ both have    trivial homology in Alexander grading $g$. This  implies  that the  components  \[\hfkhat(K,g)\xrightarrow{\xi_g}\hfkhat(K,g-1)\xrightarrow{\omega_{g-1}}\hfkhat(K,g)\] of $\xi$ and $\omega$ are both nontrivial, and hence so is their composition, since \[\hfkhat(K,g)\cong\hfkhat(K,g-1)\cong \Q.\]  Letting $x$ be a generator of $\hfkhat(K,g)$,  this shows that  $\omega(\xi(x)) \neq 0$. On the other hand, $\xi(\omega(x)) = \xi(0)=0$, which contradicts the fact that $\xi\omega = -\omega\xi.$
\end{proof}
 
We  now prove Theorem \ref{thm:main-kh}, which states that reduced Khovanov homology detects  $5_2$.

\begin{proof}[Proof of Theorem \ref{thm:main-kh}] Suppose that \[\Khr(K)\cong\Khr(5_2)\] as bigraded vector spaces. Note that $\Khr(5_2)$ is thin since $5_2$ is alternating \cite{lee-endomorphism}. It then follows from Lemma \ref{lem:dim-equal} that the knot Floer homology of $K$ is thin, and that \[\dim\hfkhat(K) =  \det(5_2) = 7.\] Let $g\geq 1$ be the genus of $K$, and let us first suppose that $K$ is not fibered. Then \[\dim \hfkhat(K,\pm g) \geq 2.\] Together with the fact from Lemma \ref{lem:thin-fibered} that \[\dim\hfkhat(K,g) \leq \dim\hfkhat(K,g-1),\] and the fact that the total dimension is 7, this implies  that $g=1$ and the sequence \[(\dim\hfkhat(K,a)\mid -1\leq a \leq 1)=(2,3,2).\] In particular, $K$ is a nearly fibered knot of genus 1, and it follows from Theorem \ref{thm:main-hfk} and Table \ref{fig:hfk-table} that $K$ is either $5_2$ or $\mirror{5_2}$. But reduced Khovanov homology distinguishes $5_2$ from its mirror, so we have that $K = 5_2$, as desired.

Finally, let us suppose for a contradiction that $K$ is fibered. First, note that \begin{equation}\label{eqn:tau-in}|\tau(K)|<g.\end{equation} Indeed, if $|\tau(K)| =  g$ instead, then either $K$ or its mirror is strongly quasipositive \cite[Theorem~1.2]{hedden-positivity}. In this case, \cite[Proposition~4]{plamenevskaya-kh} implies that Rasmussen's invariant \cite{rasmussen-s} satisfies  $s(K) = \pm 2g$. Since $\Khr(K)$ is thin, it is    supported in the $\delta$-grading \[\tfrac{1}{2}s(K) = \pm g,\] as argued at the end of \cite[Proof of Theorem~1]{bdlls}. Since $\Khr(5_2)$ is supported in $\delta$-grading 1, it follows that $g=1$. Then $K$ is a  fibered knot of genus 1, and hence   a trefoil or the figure eight, but this violates our assumption that \[\Khr(K)\cong \Khr(5_2).\] 
The strict inequality in \eqref{eqn:tau-in} therefore holds. 

It then follows from Lemma \ref{lem:thin-fibered} that \[1=\dim\hfkhat(K,g) < \dim\hfkhat(K,g-1).\] The fact that $\dim\hfkhat(K)=7$ then implies that either $g=1$, which cannot happen (since $K$ is not a trefoil or the figure eight, as discussed above), or else $g>1$ and  \[\hfkhat(K,a) \cong \begin{cases}
\Q& \textrm{if } a=\pm g\\
\Q^2& \textrm{if }  a=\pm (g-1)\\
\Q&\textrm{if } a=0\\
0& \textrm{otherwise}.
\end{cases}\] Let us assume the latter holds. Note in this case that if $g>2$ then the complex $(\hfkhat(K),\xi)$ must have nontrivial homology in both Alexander gradings $g-1$ and $1-g$, meaning that \[ \dim H_*(\hfkhat(K),\xi) \geq 2,\] a contradiction. Therefore, $g=2$ and \[(\dim\hfkhat(K,a)\mid -2\leq a \leq 2)=(1,2,1,2,1).\] 

The complexes $(\hfkhat(K),\xi)$ and $(\hfkhat(K),\omega)$ therefore take the forms \[\Q_2\xrightarrow{\xi_{2}}\Q^2_1\xrightarrow{\xi_{1}}\Q_0\xrightarrow{\xi_{0}}\Q^2_{-1}\xrightarrow{\xi_{-1}}\Q_{-2}\] and \[\Q_2\xleftarrow{\omega_{1}}\Q^2_1\xleftarrow{\omega_{0}}\Q_0\xleftarrow{\omega_{-1}}\Q^2_{-1}\xleftarrow{\omega_{-2}}\Q_{-2},\] respectively, where the subscripts indicate the Alexander grading. The fact that \[\tau(K) \neq \pm g = \pm 2\] implies that the homologies of these complexes are trivial in Alexander gradings $\pm 2$. This implies that  the components $\xi_2,\xi_{-1},\omega_{-2},$ and $\omega_1$ are all nontrivial. Moreover, $\xi_1$ and $\xi_0$ cannot  both be nontrivial, as this would imply that their composition is nontrivial, which would violate $\xi^2 = 0$. Let us assume without loss of generality that \[\xi_1 \neq 0\,\textrm{ and }\,\xi_0= 0.\] Let $x$ be an element of $\hfkhat(K,{-1})$ for which $\xi_{-1}(x)\neq 0$. Then \[\omega(\xi(x)) = \omega_{-2}(\xi_{-1}(x))\neq 0,\] while \[\xi(\omega(x)) = \xi_0(\omega_{-1}(x)) = 0,\] contradicting the fact that $\omega\xi = -\xi\omega$. We have therefore ruled out the possibility that $K$ is fibered, completing the proof of Theorem \ref{thm:main-kh}. \end{proof}

\begin{remark}
One can use a similar argument to prove the slightly stronger result that if $\Khr(K)$ is 7-dimensional and supported in a unique $\delta$-grading $d$ then, up to taking mirrors, either $K=5_2$, or else $d=3$ and \[\hfkhat(K)\cong\hfkhat(T_{2,7})\] as bigraded vector spaces. Though relatively straightforward, proving this takes quite a bit of room, so we do not pursue it here.
\end{remark}

Finally, we prove Theorem \ref{thm:main-kh-pretzel}, which states that reduced Khovanov homology together with the degree of the Alexander polynomial detects each  pretzel knot $P(-3,3,2n+1)$.

\begin{proof}[Proof of Theorem \ref{thm:main-kh-pretzel}]
Suppose that \[\Khr(K)\cong \Khr(P(-3,3,2n+1))\] as bigraded vector spaces, and that $\Delta_K(t)$ has degree one. Then $K$ is not fibered. Starkston proved \cite[Theorem~4.1]{starkston-pretzel} that the reduced Khovanov homology of this pretzel  is thin. It then follows from Lemma \ref{lem:dim-equal} that the knot Floer homology of $K$ is thin, and that \[\dim\hfkhat(K) =  \det(P(-3,3,2n+1)) = 9.\] Since $\hfkhat(K)$ is thin and $\Delta_K(t)$ has degree one, we conclude from \eqref{eqn:alex} and the genus detection \eqref{eqn:genus-detection} that $g(K)=1$. Since $K$ is not fibered, we have that \[\dim \hfkhat(K,\pm 1) \geq 2.\] Together with the fact from Lemma \ref{lem:thin-fibered} that \[\dim\hfkhat(K,1) \leq \dim\hfkhat(K,0),\] and the fact that the total dimension is 9, this implies that the sequence \[(\dim\hfkhat(K,a)\mid -1\leq a \leq 1)=(2,5,2) \textrm{ or } (3,3,3).\] But in the latter case, we would have \[\Delta_K(t) = \pm (3t-3+3t^{-1}),\] which would imply that $\Delta_K(1) = \pm 3$, but $\Delta_K(1) = 1$ for any knot $K\subset S^3$. Therefore,\[\dim\hfkhat(K,1) = 2,\] and hence $K$ is   nearly fibered of genus 1. The fact that $\hfkhat(K)$ is thin and 9-dimensional then means, by Theorem \ref{thm:main-hfk} and  Table \ref{fig:hfk-table}, that $K$ must be   a pretzel  knot $P(-3,3,2m+1)$ for some $m\in \Z$ (the mirror of any such pretzel is another such pretzel). But \[\Khr(P(-3,3,2m+1)) \not\cong \Khr(P(-3,3,2n+1))\] for $m\neq n$, by \cite[Theorem~4.1]{starkston-pretzel} or the more general \cite[Theorem~3.2]{hedden-watson}. We conclude that $K = P(-3,3,2n+1),$ as desired. \end{proof}

\section{Detection results for HOMFLY homology} \label{sec:homfly}

As mentioned in \S\ref{ssec:results},   reduced HOMFLY homology,  defined by Khovanov--Rozansky in \cite{khovanov-rozansky-2}, assigns to a knot $K\subset S^3$ a triply-graded vector space over $\Q$, \[\bar{H}(K) = \bigoplus_{i,j,k}\bar{H}^{i,j,k}(K),\] which determines the HOMFLY polynomial of $K$ by the relation \[ P_K(a,q) = \sum_{i,j,k} (-1)^{(k-j)/2} a^j q^i \dim \bar{H}^{i,j,k}(K). \] Our goal in this section is to prove Theorem \ref{thm:main-homfly-detection}, which says that reduced HOMFLY homology detects  each pretzel knot $P(-3,3,2n+1)$. We begin with the following computation:

\begin{lemma} \label{lem:compute-pretzel-homfly}
We have $\dim \bar{H}(P(-3,3,2n+1)) = 9$ for all $n \in \Z$.
\end{lemma}

In order to prove this lemma, let us first recall that Khovanov--Rozansky also defined  for each integer $N\geq 1$ a reduced  $\slfrak_N$ homology  theory \cite{khovanov-rozansky-1}, which assign to a knot $K\subset S^3$ a bigraded vector space over $\Q$, \[\bar{H}_N(K) = \bigoplus_{i,j}\bar{H}^{i,j}_N(K).\] Khovanov homology is related to the $\slfrak_2$ theory by the following change in gradings,\begin{equation}\label{eqn:kh-sl2}\Khr^{h,q}(K) \cong \bar{H}^{q,-h}_2(K).\end{equation}  Rasmussen proved in  \cite[Theorem~2]{rasmussen-differentials} that there is a spectral sequence  which starts at $\bar{H}(K)$ and converges to $ \bar{H}_N(K),$ for each $N\geq 1$. Moreover, when this spectral sequence collapses at the first page, as it  does  for $N$ sufficiently large, the reduced HOMFLY homology determines the $\slfrak_N$ theory  \cite[Theorem 1]{rasmussen-differentials} by
\begin{equation} \label{eq:kr-from-homfly}
\bar{H}^{I,J}_N(K) \cong \bigoplus_{\substack{i+Nj=I\\(k-j)/2=J}} \bar{H}^{i,j,k}(K).
\end{equation}
In particular,  $\dim \bar{H}(K) = \dim \bar{H}_N(K)$ for $N\gg 0$. 

\begin{proof}[Proof of Lemma \ref{lem:compute-pretzel-homfly}]
Let us write \[K_n = P(-3,3,2n+1)\] for convenience.  First, note that $K_0$ is the 2-bridge knot $6_1$. It therefore follows from  \cite[Theorem~1]{rasmussen-2bridge}  that $K_0$ is \emph{$N$-thin} for all $N > 4$, which implies by \cite[Corollary~4.3]{rasmussen-2bridge} that
\[ \dim \bar{H}_N(K_0) = \det(K_0) = 9 \quad\text{for all } N>4. \]
Next, observe that $K_0$ can be obtained via band surgery on the 2-stranded pretzel link $P(-3,3)$, which is a split link (in fact, a 2-component unlink), as shown in Figure~\ref{fig:pretzel-band}.
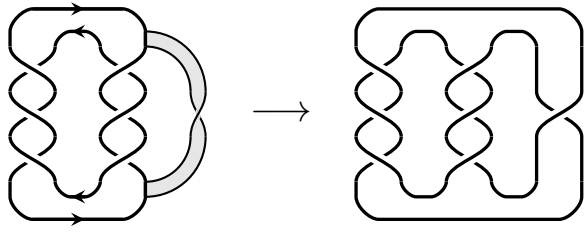
\begin{figure}
\begin{tikzpicture}[link/.append style={looseness=0.75}]
\begin{scope}
\foreach \i in {0,1,2} {
  \draw[link] (0,0.6*\i) to[out=90,in=270] ++(0.6,0.6);
  \draw[link] (1.8,0.6*\i) to[out=90,in=270] ++(-0.6,0.6);
}
\foreach \i in {0,1,2} {
  \draw[link] (0.6,0.6*\i) to[out=90,in=270] ++(-0.6,0.6);
  \draw[link] (1.2,0.6*\i) to[out=90,in=270] ++(0.6,0.6);
}
\draw[link] (0,1.8) -- ++(0,0.2) to[out=90,in=180] ++(0.3,0.3) -- ++(1.2,0) to[out=0,in=90] ++(0.3,-0.3) -- ++(0,-0.2);
\draw[link] (0.6,1.8) to[out=90,in=180] ++(0.2,0.2) -- ++(0.2,0) to[out=0,in=90] ++(0.2,-0.2);
\draw[link] (0,0) -- ++(0,-0.2) to[out=270,in=180] ++(0.3,-0.3) -- ++(1.2,0) to[out=0,in=270] ++(0.3,0.3) -- ++(0,0.2);
\draw[link] (0.6,0) to[out=270,in=180] ++(0.2,-0.2) -- ++(0.2,0) to[out=0,in=270] ++(0.2,0.2);
\begin{scope}[very thick,decoration={markings,mark=at position 0.575 with {\arrow{stealth}}}]
    \draw[postaction={decorate}] (0.3,2.3) -- ++(1.2,0);
    \draw[postaction={decorate}] (0.3,-0.5) -- ++(1.2,0);
\end{scope}
\begin{scope}[very thick,decoration={markings,mark=at position 0.9 with {\arrow{stealth}}}]
    \draw[postaction={decorate}] (1,2) -- ++(-0.2,0);
    \draw[postaction={decorate}] (1,-0.2) -- ++(-0.2,0);
\end{scope}
\path[fill=gray!20] (1.8,2) to[out=0,in=90] ++(0.8,-0.8) to[out=270,in=60] ++(-0.1,-0.3) to[out=120,in=270] ++(-0.1,0.3) to[out=90,in=0] ++(-0.6,0.6) -- ++(0,0.2);
\path[fill=gray!20] (1.8,-0.2) to[out=0,in=270] ++(0.8,0.8) to[out=90,in=-60] ++(-0.1,0.3) to[out=-120,in=90] ++(-0.1,-0.3) to[out=270,in=0] ++(-0.6,-0.6) -- ++(0,-0.2);
\draw (1.8,2) to[out=0,in=90] ++(0.8,-0.8);
\draw (1.8,1.8) to[out=0,in=90] ++(0.6,-0.6);
\draw (1.8,0) to[out=0,in=270] ++(0.6,0.6);
\draw (1.8,-0.2) to[out=0,in=270] ++(0.8,0.8);
\draw[line width=0.775pt] (2.4,1.2) to[out=270,in=90] (2.6,0.6);
\begin{scope} % fix a crossing
\clip[rotate around={60:(2.5,0.9)}] (2.4,0.85) rectangle (2.6,0.95);
\draw[gray!20,line width=0.825] (2.4,1.2) to[out=270,in=90] (2.6,0.6);
\end{scope}
\draw (2.6,1.2) to[out=270,in=90] (2.4,0.6);
\draw[line width=1.25pt] (1.8,-0.2) -- ++(0,0.2) (1.8,1.8) -- ++(0,0.2);
\end{scope}

\node at (3.6,0.9) {\Large$\longrightarrow$};

\begin{scope}[xshift=4.6cm]
\foreach \i in {0,1,2} {
  \draw[link] (0,0.6*\i) to[out=90,in=270] ++(0.6,0.6);
  \draw[link] (1.8,0.6*\i) to[out=90,in=270] ++(-0.6,0.6);
}
\foreach \i in {0,1,2} {
  \draw[link] (0.6,0.6*\i) to[out=90,in=270] ++(-0.6,0.6);
  \draw[link] (1.2,0.6*\i) to[out=90,in=270] ++(0.6,0.6);
}
\draw[link] (2.4,1.8) -- ++(0,-0.6) to[out=270,in=90] ++(0.6,-0.6) -- ++(0,-0.6);
\draw[link] (3,1.8) -- ++(0,-0.6) to[out=270,in=90] ++(-0.6,-0.6) -- ++(0,-0.6);
\draw[link] (0,1.8) -- ++(0,0.2) to[out=90,in=180] ++(0.3,0.3) -- ++(2.4,0) to[out=0,in=90] ++(0.3,-0.3) -- ++(0,-0.2);
\draw[link] (0.6,1.8) to[out=90,in=180] ++(0.2,0.2) -- ++(0.2,0) to[out=0,in=90] ++(0.2,-0.2);
\draw[link] (1.8,1.8) to[out=90,in=180] ++(0.2,0.2) -- ++(0.2,0) to[out=0,in=90] ++(0.2,-0.2);
\draw[link] (0.6,0) to[out=270,in=180] ++(0.2,-0.2) -- ++(0.2,0) to[out=0,in=270] ++(0.2,0.2);
\draw[link] (1.8,0) to[out=270,in=180] ++(0.2,-0.2) -- ++(0.2,0) to[out=0,in=270] ++(0.2,0.2);
\draw[link] (0,0) -- ++(0,-0.2) to[out=270,in=180] ++(0.3,-0.3) -- ++(2.4,0) to[out=0,in=270] ++(0.3,0.3) -- ++(0,0.2);
\end{scope}

\end{tikzpicture}
\caption{Building $P(-3,3,1)$ by attaching a band to a 2-component unlink.}
\label{fig:pretzel-band}
\end{figure}
Each $K_n$ can then be obtained from $K_0$ by adding $n$ full twists to that band, so a theorem of Wang \cite[Proposition~1.7]{wang-split} says that for any $N \geq 2$, the dimension
\[ \dim \bar{H}_N(K_n) \]
is independent of $n$.  Thus, for any $n\in\Z$, the above computation for $K_0$  tells us that
\[ \dim \bar{H}_N(K_n) = 9 \quad\text{for all } N>4, \]
and hence that $\dim \bar{H}(K_n) = 9$, as desired.
\end{proof}

With this computation in hand, we may now prove Theorem~\ref{thm:main-homfly-detection}.

\begin{proof}[Proof of Theorem~\ref{thm:main-homfly-detection}]
Suppose that \[\bar{H}(K) \cong \bar{H}(P(-3,3,2n+1))\] as triply-graded vector spaces.  Then $K$ has the same HOMFLY polynomial as $P(-3,3,2n+1)$. Since the HOMFLY polynomial  specializes to the Alexander polynomial, we   have that \[ \Delta_K(t) = \Delta_{P(-3,3,2n+1)}(t) = -2t + 5 - 2t^{-1}. \]
In particular,   \[\dim \bar{H}_2(K) = \dim\Khr(K) \geq \det(K) = |\Delta_K(-1)| = 9.\]
Since we also know from the computation in Lemma \ref{lem:compute-pretzel-homfly} that
\[ \dim \bar{H}(K) = \dim \bar{H}(P(-3,3,2n+1)) = 9, \]  it follows that the spectral sequence from $\bar{H}(K)$ to $\bar{H}_2(K)$ must collapse at the first page. Therefore, $\bar{H}(K)$ determines $\bar{H}_2(K)$ as in  \eqref{eq:kr-from-homfly}. In particular, it follows that \[\bar{H}_2(K) \cong \bar{H}_2(P(-3,3,2n+1))\] as bigraded vector spaces. Then we have   by \eqref{eqn:kh-sl2} that \[\Khr(K) \cong \Khr(P(-3,3,2n+1))\] as bigraded vector spaces. Since $K$ has the same Alexander polynomial and reduced Khovanov homology as $P(-3,3,2n+1)$, Theorem~\ref{thm:main-kh-pretzel} says that $K=P(-3,3,2n+1)$.
\end{proof}

\appendix
\section{Computations of knot Floer homology} \label{sec:hfk-15n43522}

In this appendix, we explain the knot Floer homology calculations recorded in Table \ref{fig:hfk-table}.  The computation for $5_2$ follows from the fact that it is alternating \cite[Theorem~1.3]{osz-alternating}. For the pretzel knots $P(-3,3,2n+1)$, we apply \cite[Theorem~1.3]{osz-hfk-mutation} (but see also \cite[Theorem~1]{hedden-watson}).  For the twisted Whitehead doubles, Hedden \cite[Theorem~1.2]{hedden-wh} computed their knot Floer homology over $\Z/2\Z$, but his results work over arbitrary fields.  This leaves only the knot $15n_{43522}$, which will occupy the remainder of this appendix.

\begin{proposition} \label{prop:hfk-15n43522}
We have that
\[ \hfkhat(15n_{43522},a;\Q) \cong \begin{cases} \Q_{(0)}^2 & a=1 \\ \Q_{(-1)}^4 \oplus \Q_{(0)}^{\vphantom{4}} & a=0 \\ \Q_{(-2)}^2 & a=-1, \end{cases} \]
where the subscripts denote Maslov gradings.
\end{proposition}

To start, we can carry out the same computation with coefficients in a finite field using a program by Zolt\'an Szab\'o \cite{szabo-program}, and over $\F=\Z/2\Z$ we find that
\[ \hfkhat(15n_{43522},a;\F) \cong \begin{cases} \F_{(0)}^2 & a=1 \\ \F_{(-1)}^4 \oplus \F_{(0)}^{\vphantom{4}} & a=0 \\ \F_{(-2)}^2 & a=-1. \end{cases} \]
Proposition \ref{prop:hfk-15n43522} will then follow from the universal coefficient theorem if we can show that $\hfkhat(15n_{43522};\Z)$ has no 2-torsion.

Suppose, for a contradiction, that there is 2-torsion in some Alexander grading $a$. Then by the universal coefficient theorem, it must contribute $\F$ summands to consecutive homological (i.e., Maslov) gradings of $\hfkhat(15n_{43522},a;\F)$.  By inspection, it can only possibly contribute $\F_{(-1)}\oplus \F_{(0)}$ to $\hfkhat(15n_{43522},0;\F)$, and therefore
\[ \hfkhat(15n_{43522},a;\Q) \cong \begin{cases} \Q_{(0)}^2 & a=1 \\ \Q_{(-1)}^3 & a=0 \\ \Q_{(-2)}^2 & a=-1. \end{cases} \]
That is, \[\hfkhat(15n_{43522};\Q) \cong \hfkhat(\mirror{5_2};\Q)\] as bigraded vector spaces.  Since this knot Floer homology is thin, we have that \[\cfkinfty(15n_{43522};\Q) \cong \cfkinfty(\mirror{5_2};\Q)\] up to filtered chain homotopy equivalence \cite[Lemma~5]{petkova-thin}. Since the complex $\cfkinfty(K)$ determines    \cite{osz-knot} the Heegaard Floer homology of  $n$-surgery on a knot $K\subset S^3$  for integers \[n \geq 2g(K)-1=1,\]  it follows that
\begin{align*}
\dim \hfhat(S^3_1(15n_{43522});\Q) &= \dim \hfhat(S^3_1(\mirror{5_2});\Q) \\
&= \dim \hfhat(-\Sigma(2,3,11);\Q) = 3.
\end{align*}
We will use this together with the following lemma to get a contradiction.

\begin{lemma} \label{lem:1-surgery-dim-5}
If $K \subset S^3$ is a knot of genus at least 2, then $\dim \hfhat(S^3_{\pm1}(K);\Q) \geq 5$.
\end{lemma}

\begin{proof}
By the surgery exact triangles 
\[ \dots \to \hfhat(S^3;\Q) \to \hfhat(S^3_0(K);\Q) \to \hfhat(S^3_1(K);\Q) \to \dots, \] and
\[ \dots \to \hfhat(S^3;\Q) \to \hfhat(S^3_{-1}(K);\Q) \to \hfhat(S^3_0(K);\Q) \to \dots, \]
it suffices to show that $\dim \hfhat(S^3_0(K);\Q) \geq 6$.

Let $\spinc_i \in \spc(S^3_0(K))$ be the $\spc$ structure with
\[ \langle c_1(\spinc_i), [\hat\Sigma]\rangle = 2i, \]
where $\hat\Sigma \subset S^3_0(K)$ is a capped-off Seifert surface for $K$.  Then according to \cite[Corollary~4.5]{osz-knot} and the way in which knot Floer homology detects the genus $g = g(K)$, which is at least 2, we have
\[ \hfp(S^3_0(K),\spinc_{g-1};\Q) \cong \hfkhat(K,g;\Q) \not\cong 0. \]
Likewise, \[\hfp(S^3_0(K),\spinc_{1-g};\Q) \not\cong 0,\] by the conjugation symmetry of Heegaard Floer homology.  Furthermore, $\hfp(S^3_0(K),\spinc_0;\Q)$ is nontrivial because $\spinc_0$ is torsion (see \cite[\S10.6]{osz-hf1}).

We now recall from \cite[Proposition~2.1]{osz-hf2} that $\hfhat(Y,\spinc)$ is nonzero if and only if $\hfp(Y,\spinc)$ is nonzero, so we have shown that
\[\hfhat(S^3_0(K),\spinc_i)\not\cong 0\]
for each $i=g-1,0,1-g$.  In fact, each of these $\spc$ summands has Euler characteristic zero \cite[Proposition~5.1]{osz-hf2} and hence even dimension, so the total dimension of $\hfhat(S^3_0(K))$ must be at least $2+2+2=6$, as claimed.
\end{proof}

\begin{proof}[Proof of Proposition~\ref{prop:hfk-15n43522}]
Supposing otherwise, we have already argued that
\[ \dim \hfhat(S^3_1(15n_{43522});\Q) = 3. \]
We now observe the following coincidences in SnapPy \cite{snappy}:
\begin{verbatim}
In[1]: M1 = Manifold('K15n43522(1,1)')
In[2]: N1 = Manifold('9_42(-1,1)')
In[3]: M1.is_isometric_to(N1)
Out[3]: True
In[4]: M2 = Manifold('K15n43522(-1,1)')
In[5]: N2 = Manifold('8_20(-1,1)')
In[6]: M2.is_isometric_to(N2)
Out[6]: True
\end{verbatim}
In other words, if $\texttt{K15n43522}$, $\texttt{8\char`_20}$, and $\texttt{9\char`_42}$ denote each of $15n_{43522}$, $8_{20}$, and $9_{42}$ with the fixed chirality given by SnapPy (which may or may not be mirror to their usual chiralities), then we have
\begin{align*}
S^3_1(\texttt{K15n43522}) &\cong \pm S^3_{-1}(\texttt{9\char`_42}), &
S^3_{-1}(\texttt{K15n43522}) &\cong \pm S^3_{-1}(\texttt{8\char`_20}).
\end{align*}
But $8_{20}$ and $9_{42}$ both have genus 2, so we can apply Lemma~\ref{lem:1-surgery-dim-5} to conclude that
\[ \dim \hfhat(S^3_{\pm1}(15n_{43522});\Q) \geq 5 \]
and we have a contradiction.
\end{proof}

\bibliographystyle{alpha}
\bibliography{References}

\end{document}